\documentclass[a4paper,10pt]{article}
\usepackage[margin=0.80in,bottom=0.90in,top=0.85in]{geometry}

\usepackage[english]{babel}
\usepackage{chngcntr}
\usepackage[utf8]{inputenc}
\usepackage{amsmath,mathrsfs}
\usepackage{indentfirst}
\usepackage{fancyhdr}
\usepackage{amssymb}
\usepackage{amsthm}
\usepackage[dvips]{graphicx}
\usepackage{color}
\usepackage{xcolor}
\usepackage{cancel}
\usepackage{subcaption}
\usepackage[dvipsnames]{xcolor}
\usepackage{eucal}
\usepackage{latexsym}
\usepackage{chngcntr}
\usepackage{faktor}
\usepackage{microtype}
\usepackage{newlfont}
\usepackage{enumitem}
\usepackage{geometry}
\usepackage{nicematrix}
\usepackage{tikz}
\usepackage{todonotes}

\usepackage{amsmath, amsfonts, amsthm,amssymb,  bbm}

\usepackage{tikz-cd}
\usepackage{multicol}
\usepackage{cancel}
\usepackage{mathtools}
\usepackage{graphicx}

\usepackage{xcolor}
\definecolor{linkblue}{RGB}{20,40,90}
\usepackage[
  colorlinks=true,
  linkcolor=linkblue,
  citecolor=linkblue,
  urlcolor=linkblue
]{hyperref}

\newcommand{\e}{\varepsilon}
\newcommand{\cO}{\mathcal{O}}
\newcommand{\ch}{c_{\tth}}
\newcommand{\chk}{c_{\tth, \kappa}}
\newcommand{\rhkb}{r_{\tth,\kappa,b}}

\newcommand{\ckb}{c_{\kappa,b}}

\newcommand{\tth}{\mathtt{h}}
\newcommand{\ttf}{\mathtt{f}_\e}
\newcommand{\pa}{\partial}
\newcommand{\cH}{\mathcal{H}}
\newcommand{\de}{\textrm{d}}
\newcommand{\tf}{\mathtt{f}}
\newcommand{\cA}{{\mathcal A}}
\newcommand{\cJ}{{\mathcal J}}

\newcommand{\N}{{\mathbb{N}}}
\newcommand{\R}{{\mathbb{R}}}
\newcommand{\Z}{{\mathbb{Z}}}
\newcommand{\T}{{\mathbb{T}}}

\newcommand{\vet}[2]{\begin{bmatrix}#1 \\ #2 \end{bmatrix}}
\newcommand{\cB}{\mathcal{B}}

\newcommand{\fR}{\mathfrak{R}}

\newcommand{\im}{\textup{i}}

\newcommand{\cL}{ L}

\newcommand{\vertiii}[1]{{\left\vert\kern-0.25ex\left\vert\kern-0.25ex\left\vert #1 
    \right\vert\kern-0.25ex\right\vert\kern-0.25ex\right\vert}}
\newcommand{\opnorm}[1]{{\vert\kern-0.25ex\vert\kern-0.25ex\vert #1 
    \vert\kern-0.25ex\vert\kern-0.25ex\vert}}

\theoremstyle{plain}
\newtheorem{lemma}{Lemma}
\newtheorem{theorem}[lemma]{Theorem}
\newtheorem{proposition}[lemma]{Proposition}
\newtheorem{remark}[lemma]{Remark}
\newtheorem{definition}[lemma]{Definition}

\theoremstyle{definition} 

\numberwithin{equation}{section}
\counterwithin{lemma}{section}

\title{Benjamin--Feir spectrum of hydroelastic Stokes waves}

\begin{document}

\author{Ting-Yang Hsiao\footnote{International School for Advanced Studies (SISSA), Via Bonomea 265, 34136, Trieste, Italy. \newline
	\textit{Emails:} \texttt{thsiao@sissa.it}, \texttt{yezhang@sissa.it}},  Zirui Li\footnote{University of Illinois Urbana-Champaign,
1409 W Green St, Urbana, 61801, Illinois, USA.\newline
	\textit{Emails:} \texttt{ziruili2@illinois.edu}, \texttt{cz43@illinois.edu}}, Ye Zhang$^*$,  Chengbin Zhu}
 

\maketitle
\begin{abstract}
    We determine the complete Benjamin-Feir spectrum near the origin for
small-amplitude hydroelastic Stokes waves of the two-dimensional
finite-depth irrotational Euler equations with surface tension and elastic
bending. For the non-resonant Stokes branch and away from an intrinsic
characteristic-collision surface $\mathfrak D$, we resolve all four Bloch
eigenvalues bifurcating from the origin in the long-wave Floquet regime.
Exploiting the Hamiltonian and reversible structure of the problem, we
reduce the linearized Bloch operator to the four-dimensional spectral
subspace bifurcating from the generalized kernel at the origin and
conjugate the resulting matrix to the direct sum of a Benjamin-Feir block
and a long-wave block. The long-wave pair remains purely imaginary,
whereas the Benjamin-Feir pair is governed by an explicit closed-form
instability index $\operatorname{Ind}(\mathtt{h},\kappa,b)$: a positive index
produces a local figure-eight spectral curve with nonzero real part, while
a negative index implies that all four small eigenvalues remain purely
imaginary. Together with the Wilton-type resonance loci and the
characteristic-collision surface $\mathfrak D$, this index yields a
three-parameter spectral-stability diagram in the depth $\mathtt{h}$, surface
tension $\kappa$, and bending rigidity $b$. The diagram recovers the
classical pure-gravity critical-depth limit and, on the zero-bending
boundary, the gravity--capillary stability diagram. It also reveals a
genuinely hydroelastic phenomenon: all Wilton-type resonances disappear
whenever $b\geq 1/14$ or $\kappa\geq 1/2$. This provides the first
complete rigorous characterization of the local Benjamin-Feir spectrum
for a hydroelastic free-boundary problem.
\end{abstract}

\tableofcontents

\section{Introduction} \label{sec1}
Hydroelastic waves arise when the motion of an inviscid fluid is coupled to the deformation of an elastic interface or an elastic structure in contact with the fluid. Classical examples include waves propagating beneath a continuous ice cover, waves interacting with very large floating flexible structures, and laboratory waves on water covered by thin elastic sheets, as illustrated in Figure~\ref{Figure 1}. In each of these settings, elasticity contributes a restoring mechanism that acts alongside gravity and, depending on the model, capillarity or membrane tension. As a result, the dispersion relation is altered in a fundamental way, often producing wave regimes that do not occur in the classical free-surface problem. This interplay between fluid motion and elasticity has long made hydroelasticity a natural meeting point of fluid mechanics, elasticity, and nonlinear wave theory \cite{SDWRL, Squire, KPVB, To2008}.

From the physical side, the subject has a long history. Early studies of wave propagation in ice already recognized that floating elastic media support flexural–gravity phenomena rather than purely gravitational waves \cite{EC}. Later work on waves generated by moving loads and moving sources on floating ice clarified how elastic bending and gravity combine to determine the far-field response, and established a paradigm that still shapes modern wave–ice theory \cite{DHS}. Over the last three decades, the subject has broadened considerably. Review articles by Squire and coauthors consolidated the sea-ice and wave–ice literature and emphasized the geophysical importance of wave propagation, attenuation, breakup, and fracture in ice-covered seas \cite{SDWRL, Squire}. In parallel, laboratory experiments on fluids covered by elastic sheets made it possible to observe hydroelastic dispersion, nonlinear frequency shifts, three-wave interactions, and wake generation in controlled conditions \cite{DBF, DBeF, OTLLRDS}. Recent field and remote-sensing work has further highlighted the relevance of flexural–gravity waves in sea ice \cite{DJMF}, while the recent work of \cite[Prasad--Behera]{prasad2026hydroelastic} further illustrates how currents and
elastic rigidity affect the response of flexural--gravity waves to disturbances, reinforcing the relevance of
stability questions for hydroelastic wavetrains. Together, these developments make the modulational stability of periodic hydroelastic wavetrains a physically compelling issue. Once elasticity changes the dispersion relation, it is natural to ask how it changes the Benjamin--Feir mechanism.

The Benjamin--Feir instability is among the central instability mechanisms in dispersive wave theory. In the classical deep-water theory of Benjamin and Feir, a nearly monochromatic Stokes wavetrain is unstable with respect to long-wave sideband perturbations \cite[Benjamin--Feir]{BF}. Closely related instability criteria were obtained in broader nonlinear dispersive settings by \cite[Benjamin]{Benjamin}, \cite[Lighthill]{Li}, and \cite[Whitham]{Whitham}, and were later recast in Hamiltonian and modulation-theoretic terms \cite[Zakharov]{Zak68,Zak1}, \cite[Zakharov--Kharitonov]{ZK}. In the water-wave context, the question may be formulated spectrally as follows: after linearization about a small-amplitude periodic traveling wave and Floquet decomposition in the longitudinal variable, do eigenvalues near the origin leave the imaginary axis as the Floquet exponent is turned on? The difficulty, already visible in the classical problem, is that the relevant eigenvalues bifurcate from a defective eigenvalue at the origin. This makes the instability a genuinely delicate small-spectrum problem rather than a routine perturbative calculation.

For the full water-wave equations, rigorous progress came much later than the original physics. A decisive milestone was achieved by Bridges and Mielke, who gave the first rigorous proof of Benjamin--Feir instability in finite depth by combining Hamiltonian structure with center-manifold reduction \cite[Bridges--Mielke]{BrM}. Subsequent work deepened the picture in several directions. Numerical investigations by Deconinck and Oliveras gave a detailed spectral view of periodic gravity waves in the full Euler equations \cite[Deconinck--Oliveras]{DO}. Nguyen and Strauss later proved modulational instability in deep water by a different and self-contained method \cite[Nguyen--Strauss]{NS}. Most relevant for the present paper, Berti, Maspero, and Ventura obtained complete descriptions of the Benjamin--Feir spectrum in deep water and in finite depth, identifying the full quartet of small eigenvalues and clarifying the exact local spectral geometry near the origin \cite[Berti--Maspero--Ventura]{BMV1,BMV3}. More recently, Hur and Yang developed an Evans-function approach that not only recovers the Benjamin--Feir instability near the origin but also detects instability away from the origin and in related capillary-gravity settings \cite[Hur--Yang]{HY,HY2}. Further developments in the instability theory for water waves include the study of transverse perturbations \cite{CNS, CNS2,  JRSY, HTW}, nonlinear instability \cite{ChenSu}. Complementary spectral investigations have also addressed high-frequency instabilities of two-dimensional pure-gravity Stokes waves \cite{HY, BMV4, BCMV} and the numerical structure of the spectrum of periodic water waves \cite{Mc1, Mc2, CD, BDS, CDT}. Further dynamical and stability questions for hydrodynamic waves have
been investigated in \cite{BMP2025, LMMT, BGMS, H2024, HW2026}. Taken together, these works provide a substantial theory of modulational and spectral instability for classical water waves.

Hydroelastic waves followed a different mathematical trajectory. The earliest rigorous hydroelastic work concentrated on modeling and existence rather than on modulational spectral theory. Toland introduced a fully nonlinear hydroelastic model with a Hamiltonian flavor and proved the existence of steady periodic hydroelastic waves in \cite[Toland]{To2008}. Baldi and Toland developed a related existence theory for steady periodic waves under nonlinear elastic membranes \cite[Baldi--Toland]{BTo}. Plotnikov and Toland subsequently emphasized the modeling side, placing hydroelastic wave theory in the framework of nonlinear shell theory and later extending it to strain-gradient settings \cite[Plotnikov--Toland]{PTo, PTo2}. Further
existence theories include global bifurcation and ripple results for interfacial hydroelastic waves
\cite[Akers--Ambrose--Sulon]{AAS1,AAS2}, solitary-wave existence results
\cite[Milewski--Vanden-Broeck--Wang]{MVBW}, \cite[Guyenne--P{\u a}r{\u a}u]{GPa},
\cite[Groves--Hewer--Wahl{\'e}n]{GHW}, and spatial-dynamics constructions of solitary
hydroelastic surface waves \cite[Ahmad--Groves]{AG2024}. More recently, rigorous existence theory has been
extended to rotational hydroelastic flows in finite depth and to flows with stagnation points \cite[Martin--P{\u a}r{\u a}u]{MP2026}. In parallel, numerical and asymptotic studies began to examine the stability of hydroelastic traveling waves. The paper of Trichtchenko, Milewski, Părău, and Vanden-Broeck is
particularly relevant to the present study: it studies the stability of periodic traveling flexural–gravity waves in two dimensions using a combination of asymptotic reductions and numerical spectral computations \cite{TMPVB}.

\begin{figure}[t] 
\centering
\resizebox{\textwidth}{!}{%
\begin{tikzpicture}[
    >=Stealth,
    axis/.style={->, thick},
    label/.style={font=\small}
]

\def\L{12}
\def\h{1.8}
\def\ticevis{0.10}

\fill[cyan!35]
    plot[domain=0:\L, samples=200, smooth]
        (\x,{0.35*sin(70*\x)+0.08*sin(140*\x)})
    -- (\L,-\h)
    -- (0,-\h)
    -- cycle;

\fill[cyan!60!blue, opacity=0.9]
    plot[domain=0:\L, samples=200, smooth]
        (\x,{0.35*sin(70*\x)+0.08*sin(140*\x)+\ticevis})
    --
    plot[domain=\L:0, samples=200, smooth]
        (\x,{0.35*sin(70*\x)+0.08*sin(140*\x)})
    -- cycle;

\draw[thick, cyan!80!blue]
    plot[domain=0:\L, samples=200, smooth]
        (\x,{0.35*sin(70*\x)+0.08*sin(140*\x)+\ticevis});

\draw[thin, cyan!80!blue]
    plot[domain=0:\L, samples=200, smooth]
        (\x,{0.35*sin(70*\x)+0.08*sin(140*\x)});

\draw[thick] (0,-\h) -- (\L,-\h);
\node[label, below] at (6,-\h) {$y=-\tth$};

\draw[thin, gray] (0,0) -- (\L,0);

\draw[axis] (-0.3,0) -- (\L+0.7,0) node[right] {$x$};
\draw[axis] (0,-\h-0.4) -- (0,1.25) node[above] {$y$};

\draw[->, black, thick] (5.9,0.95) -- (5.9,0.37);
\node[label, black, above] at (6.15,0.95) {$y=\eta(t,x)$};

\draw[->, very thick] (0.9,0.92) -- (2.7,0.92);
\node[label, above] at (1.8,0.92) {Hydroelastic Wave};

\draw[<->, thick] (-0.55,0) -- (-0.55,-\h);
\node[label, left] at (-0.55,-0.9) {$\tth$};

\node[label, anchor=west] at (9.3,1.02) {Thin Ice Cover};

\begin{scope}[shift={(8.05,0.87)}]
    \fill[cyan!60!blue, opacity=0.9]
        plot[domain=0:1.15, samples=80, smooth]
            (\x,{0.02*sin(360*\x)+\ticevis})
        --
        plot[domain=1.15:0, samples=80, smooth]
            (\x,{0.02*sin(360*\x)})
        -- cycle;

    \draw[thick, cyan!80!blue]
        plot[domain=0:1.15, samples=80, smooth]
            (\x,{0.02*sin(360*\x)+\ticevis});

    \draw[thin, cyan!80!blue]
        plot[domain=0:1.15, samples=80, smooth]
            (\x,{0.02*sin(360*\x)});
\end{scope}

\end{tikzpicture}%
}
\caption{Schematic illustration of a finite-depth hydroelastic fluid domain with bottom $y=-\tth$ and upper free surface $y=\eta(t,x)$.}
\label{Figure 1}
\end{figure}
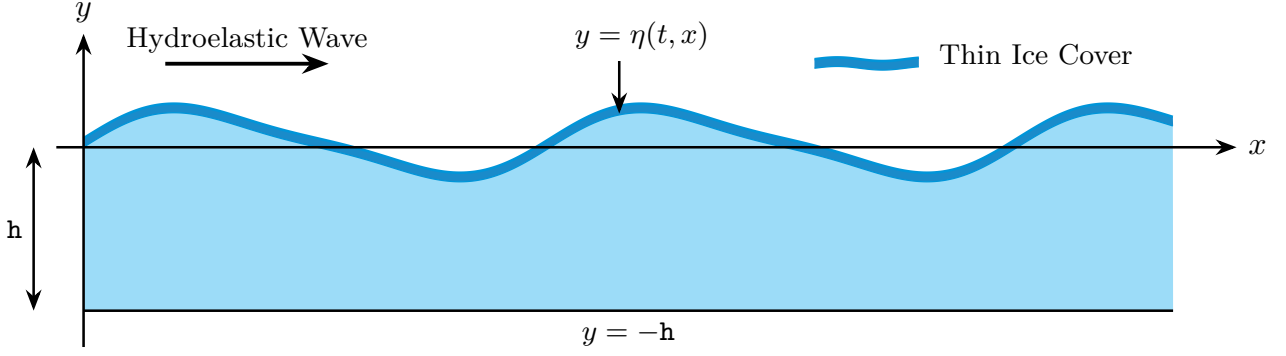

What remained missing, however, was a theorem for hydroelastic waves that
plays the same role as the complete Benjamin--Feir spectral theorems now
available for pure water waves.  The hydroelastic results discussed above
establish existence, formulate models, investigate well-posedness, or analyze
instability through asymptotic or numerical methods, but they do not provide a
complete rigorous resolution of the four small Floquet eigenvalues for a
hydroelastic Stokes wave together with a sharp criterion for the appearance of
nonzero real part.  A recent preprint by \cite[Wan--Yang]{WY} proves nonlinear
modulational instability in deep pure hydroelasticity through a focusing cubic
NLS approximation.  This is an important complementary development, but it
does not replace a complete linear spectral theorem of the type established
here.  The present paper fills this gap for finite-depth hydroelastic Stokes
waves with both capillarity and elastic bending.

\paragraph{Main results and spectral geometry.}
We work away from two types of exceptional parameters.  First, for each
integer $n \geq 2$, we define the $n$-th Wilton-type resonant set by
\begin{equation}\label{def:fRn}
    \fR_n
    :=
    \left\{
        (\tth,\kappa,b) \in \R_{>0}\times\R_{\geq 0} \times \R_{>0}
        \colon \kappa+\frac{n^4-\rho_n(\tth)}{n^2-\rho_n(\tth)} b=\frac{\rho_n(\tth)-1}{n^2-\rho_n(\tth)}
    \right\},
    \quad \rho_n(\tth):=\frac{n\tanh(\tth)}{\tanh(n\tth)},\quad n \geq 2,
\end{equation}
and we set
\begin{equation}\label{def:fR}
    \fR
    :=
    \bigcup_{\substack{n\in\N\,, n\geq 2}} \fR_n .
\end{equation}
At parameter values in $\fR$, the fundamental Fourier mode resonates with a higher harmonic, so that the small-amplitude Stokes branch is no longer generated by a simple eigenvalue of the linearized traveling-wave problem. Conversely, away from $\fR$, we establish a real-analytic one-parameter family $\varepsilon \mapsto (\eta_\varepsilon,\psi_\varepsilon,c_\varepsilon)$ of small-amplitude hydroelastic Stokes waves (see Appendix~\ref{sec:AppG}). This analytic dependence on the amplitude parameter provides the nonlinear foundation for the Bloch–Kato perturbation analysis developed below.

Second, we define the block-decoupling degeneracy set by
\begin{equation}\label{def:fD}
    \mathfrak D
    :=
    \left\{
        (\tth,\kappa,b) \in \R_{>0}\times\R_{\geq 0} \times \R_{>0}
        \colon
        \textup{D}_{\tth,\kappa,b}=0
    \right\},
    \qquad
    \textup{D}_{\tth,\kappa,b}:=\tth-\frac{1}{4}\mathsf{e}^2_{12}.
\end{equation}
The coefficient $\mathsf{e}_{12}$ is an explicit closed-form function of
$(\tth,\kappa,b)$, defined in \eqref{e11 f11}. Indeed, $\textup{D}_{\tth,\kappa,b}=0$ is equivalent to the coincidence of the Benjamin--Feir transport velocity $c_0-\frac{1}{2}\mathsf{e}_{12}$ with one of the long-wave characteristic velocities $c_0 \mp \sqrt{\tth}$; hence $\mathfrak{D}$ marks an intrinsic leading-order spectral resonance at which the Sylvester equation underlying the $2+2$ block separation loses invertibility. Accordingly, throughout the main spectral result we work on the parameter region
\begin{equation}\label{eq:intro-admissible-set}
    \mathscr P
    :=
    \bigl(\R_{>0}\times\R_{\geq 0}\times\R_{>0}\bigr)
    \setminus
    \bigl(\fR\cup\mathfrak D\bigr).
\end{equation}
It is important to distinguish the excluded sets $\fR$ and $\mathfrak D$
from the Benjamin--Feir stability boundary.  The former consist of
parameter values at which either the construction of the underlying Stokes
branch or the finite-dimensional spectral reduction degenerates, whereas
the latter is determined, on $\mathscr P$, by the sign of an
explicitly computable Benjamin--Feir instability index.

For $(\tth,\kappa,b)\in\mathscr P$, the four spectral curves issuing
from the origin separate into a Benjamin--Feir pair
$\lambda_1^{\pm}(\mu,\varepsilon)$ and a long-wave pair
$\lambda_0^{\pm}(\mu,\varepsilon)$.  In the modulational scaling
$\mu=\varepsilon\nu$, with $\nu$ bounded and $\varepsilon>0$ sufficiently
small, Theorem~\ref{Complete BF thm} yields
\begin{align}
\lambda_1^{\pm}(\varepsilon\nu,\varepsilon)
    ={}& \im\left(c_0-\frac{1}{2}\mathsf{e}_{12}\right) \varepsilon\nu
    \pm \frac{\varepsilon^2\nu}{8}
    \sqrt{
        8\mathsf{e}_{22}\mathsf{e}_{\mathrm{WB}}- \mathsf{e}^2_{22}\nu^2
        + \mathcal O(\varepsilon)
    }
    + \im\cO(\nu\varepsilon^3),                                      \label{eq:intro-lambda1}
\\
\lambda_0^{\pm}(\varepsilon\nu,\varepsilon)
    ={}& \im\bigl(c_0 \mp \sqrt{\tth}\bigr) \varepsilon\nu
    + \im\mathcal O(\nu\varepsilon^2),
    \qquad
    \lambda_0^{\pm}(\varepsilon\nu,\varepsilon)\in \im \mathbb R,      \label{eq:intro-lambda0}
\end{align}
where $c_0:=\sqrt{(1+\kappa+b)\tanh(\tth)}$ denotes the phase velocity of the hydroelastic Stokes waves, while $\frac{1}{2}\mathsf{e}_{12}$ corresponds to the group velocity. Here $\mathsf{e}_{12}$, $\mathsf{e}_{22}$ and $\mathsf{e}_{\mathrm{WB}}$ are explicit closed-form functions of
$(\tth,\kappa,b)$, recorded in \eqref{e11 f11} and \eqref{eWB}. Formulae~\eqref{eq:intro-lambda1}--\eqref{eq:intro-lambda0}
show that the long-wave pair remains purely imaginary and moves at order
$\mathcal O(\varepsilon)$, while the Benjamin--Feir pair has imaginary part of
order $\mathcal O(\varepsilon)$ and, in the unstable regime $\{\mathrm{Ind}(\tth,\kappa,b)>0\}$ develops real
part of order $\mathcal O(\varepsilon^2)$; see \eqref{eq:intro-BF-index}.

\begin{figure}[h!] 
  \centering
  \includegraphics[scale=0.9]{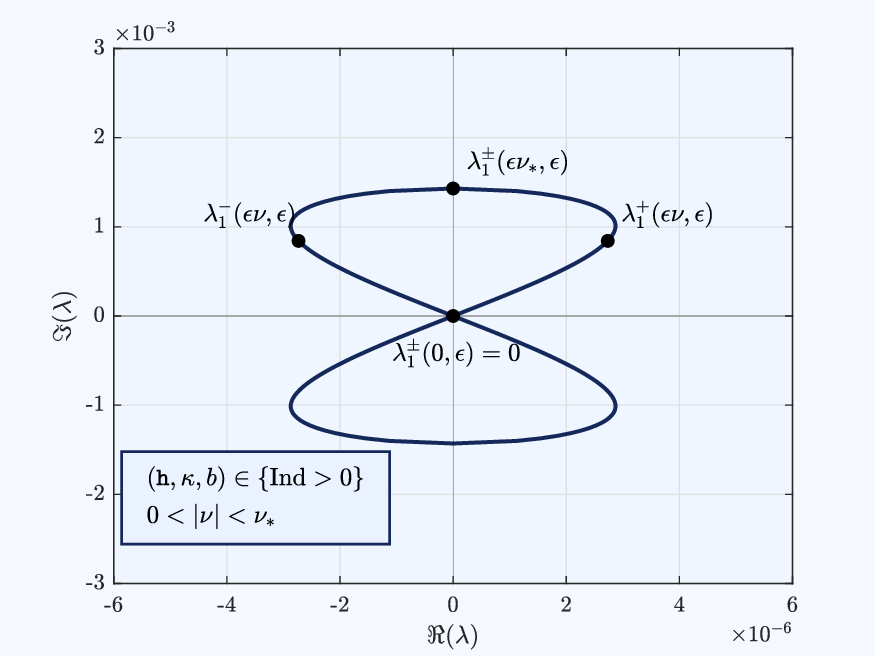}
  \caption{Local structure of the unstable spectrum near $\lambda=0$. The picture illustrates the eigenvalues $\lambda_{1}^{\pm}(\e\nu,\e)$ in the regime $(\tth,\kappa,b)\in \{\mathrm{Ind}(\tth,\kappa,b)>0\}$,
    showing the complex-plane trajectories of $\lambda_{1}^{\pm}(\e\nu,\e)$ for fixed $|\e| \ll 1$ as $\nu$ varies.
    The ``figure-eight'' shape depends on the depth $\tth>0$, capillarity parameter $\kappa\geq 0$, and bending parameter $b>0$.}
    \label{hydroelaticfigureeigth}
\end{figure}

More precisely, the Benjamin--Feir discriminant satisfies
\begin{equation}\label{eq:intro-BF-discriminant}
\Delta_{ \mathrm{BF}}(\tth,\kappa,b;\varepsilon\nu,\varepsilon)
    = \varepsilon^2
      \left(
        8\mathsf{e}_{22}(\tth,\kappa,b)\mathsf{e}_{ \mathrm{WB}}(\tth,\kappa,b)
        - \mathsf{e}_{22}(\tth,\kappa,b)^2\nu^2
        +  \cO(\varepsilon)
      \right).
\end{equation}
It is therefore natural to introduce the leading-order Benjamin--Feir index
\begin{equation}\label{eq:intro-BF-index}
    \mathrm{Ind}(\tth,\kappa,b)
    := 8\mathsf{e}_{22}(\tth,\kappa,b)\mathsf{e}_{
    \mathrm{WB}}(\tth,\kappa,b).
\end{equation}
If $\mathrm{Ind}(\tth,\kappa,b)>0$, then for all sufficiently small
$\varepsilon>0$ there is a nontrivial sideband interval of unstable Floquet
exponents, whose leading-order width is described by
\begin{equation}\label{eq:intro-unstable-band}
    0<|\nu|<\nu_*(\tth,\kappa,b;\e),
    \qquad
    \nu_*(\tth,\kappa,b;\e)
    := \frac{\sqrt{8\mathsf{e}_{22}(\tth,\kappa,b)\mathsf{e}_{ \mathrm{WB}}(\tth,\kappa,b)}}{|\mathsf{e}_{22}(\tth,\kappa,b)|}\left(1+\cO(\e)\right).
\end{equation}
Within this interval, $\lambda_1^{\pm}$ possess opposite nonzero real parts.
As the Floquet exponent varies across the unstable interval, these branches
leave the imaginary axis and return to it at the endpoints; together with the
Hamiltonian and reversible spectral symmetries, they form a local
figure-eight spectral curve near the origin, as displayed in
Figure~\ref{hydroelaticfigureeigth}.  In contrast, if
$\mathrm{Ind}(\tth,\kappa,b)<0$, the four eigenvalues remain purely
imaginary in the corresponding small-amplitude modulational regime.  On the
zero set of $\mathrm{Ind}(\tth,\kappa,b)$ the quadratic leading-order criterion is
degenerate and higher-order information is required.

\begin{figure}[p]
    \centering
    \begin{subfigure}[t]{1\textwidth}
        \centering
        \includegraphics[width=\linewidth]{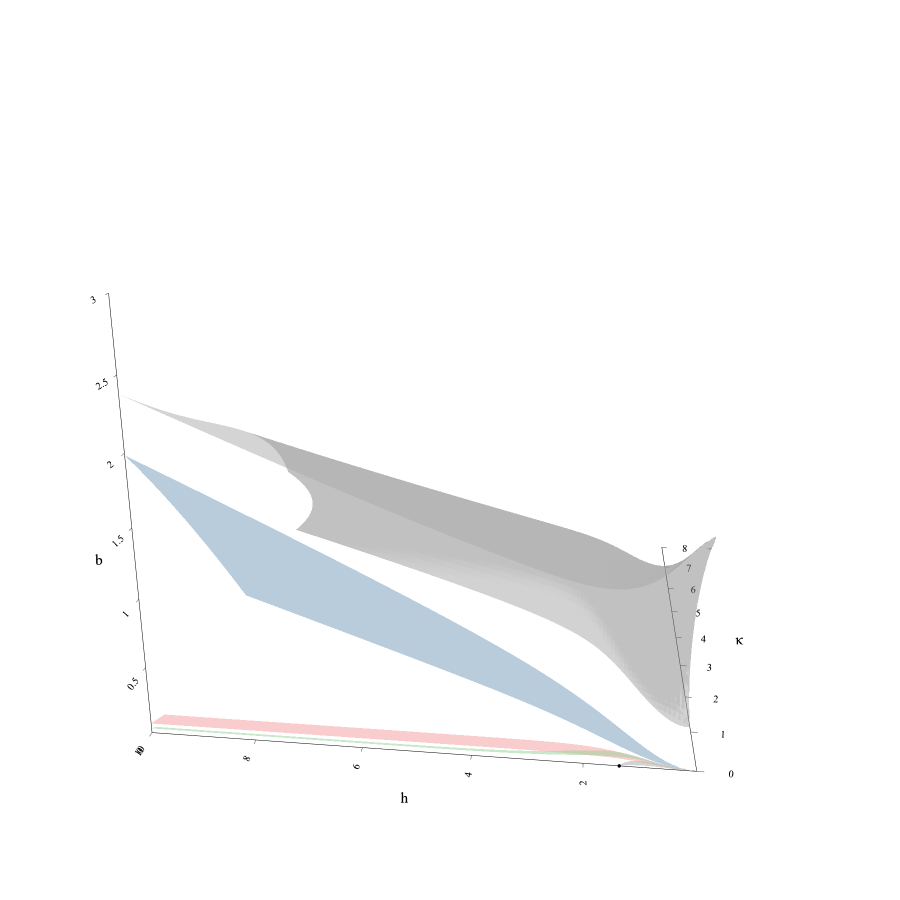}
    \end{subfigure}
    \caption{
The Benjamin--Feir transition surface
$\Sigma_{\mathrm{BF}}$
in the three-dimensional parameter space $(\tth,\kappa,b)$. The green and gray surfaces represent,
respectively, the regular components
$\{\mathsf{e}_{22}(\tth,\kappa,b)=0\}$ and
$\{\mathsf{e}_{\mathrm{WB}}(\tth,\kappa,b)=0\}$.
The blue surface is the block-decoupling degeneracy set
$\mathfrak D=\{\textup{D}_{\tth,\kappa,b}=0\}$, while the red surface is the
second-harmonic Wilton-type resonance set
$\mathfrak R_{2}=\{r_{\tth,\kappa,b}=0\}$ (cf. \eqref{rhkb}). The black point at the intersection of the boundary faces $b=0$ and
$\kappa=0$ marks the pure-gravity critical-depth limit
$(\tth,\kappa,b)=(1.3627827\ldots,0,0)$, recovering the classical
Whitham--Benjamin critical depth
$\tth_{\mathrm{WB}}=1.3627827\ldots$
identified by \cite[Berti--Maspero--Ventura]{BMV3,BMV_ed}.
Moreover, the continuous boundary limit as $b\rightarrow 0^+$ recovers the
gravity--capillary stability diagram; see, for instance, \cite[Hur--Yang]{HY2}, \cite[Sun--Wahlén]{SW}, and \cite[Hsiao--Maspero]{HM2025}.
}
\label{fig:3d-stability-diagrams}
\end{figure}

\begin{figure}[p]
    \centering
    \begin{subfigure}[t]{0.49\textwidth}
        \centering
        \includegraphics[width=\linewidth]{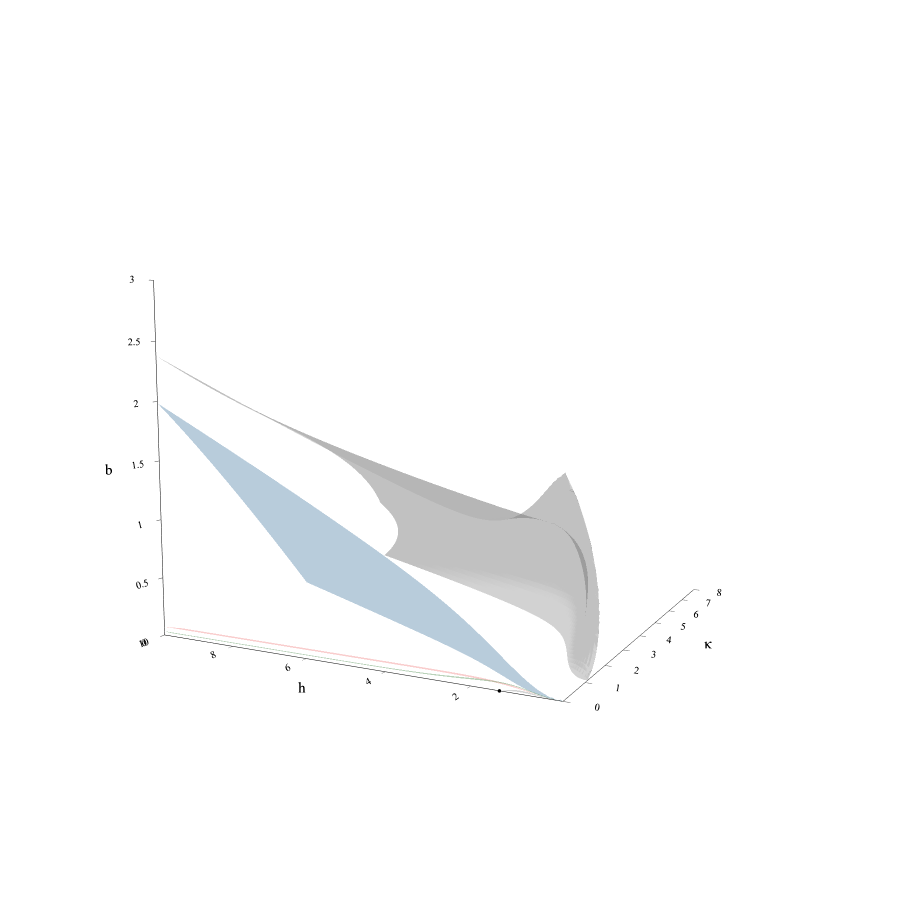}
        \label{fig:3d-stability-5}
    \end{subfigure}
    \hfill
    \begin{subfigure}[t]{0.49\textwidth}
        \centering
        \includegraphics[width=\linewidth]{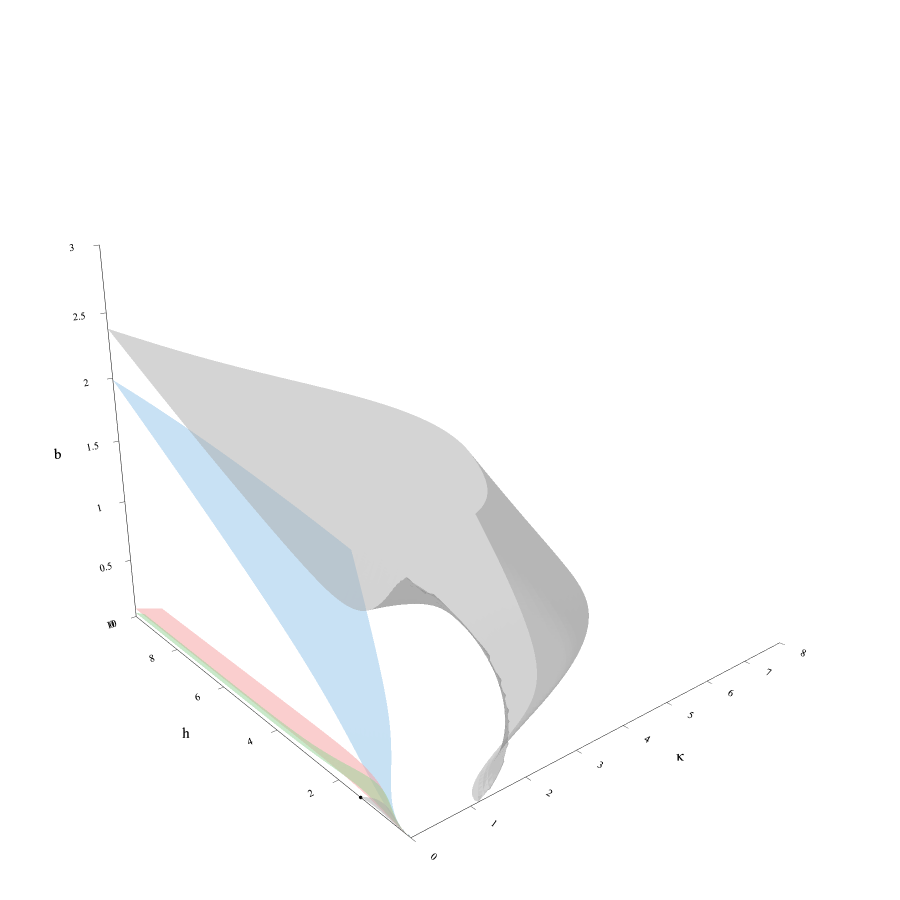}
        \label{fig:3d-stability-6}
    \end{subfigure}
    \begin{subfigure}[t]{0.49\textwidth}
        \centering
        \includegraphics[width=\linewidth]{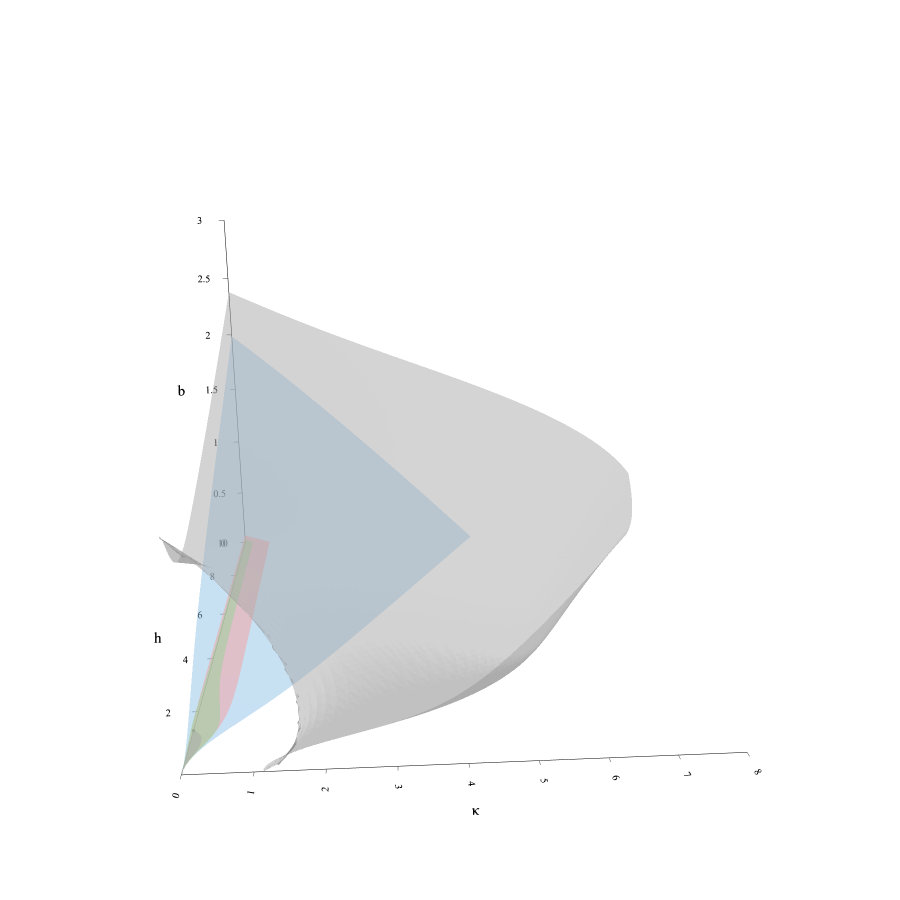}
        \label{fig:3d-stability-7}
    \end{subfigure}
    \hfill
    \begin{subfigure}[t]{0.49\textwidth}
        \centering
        \includegraphics[width=\linewidth]{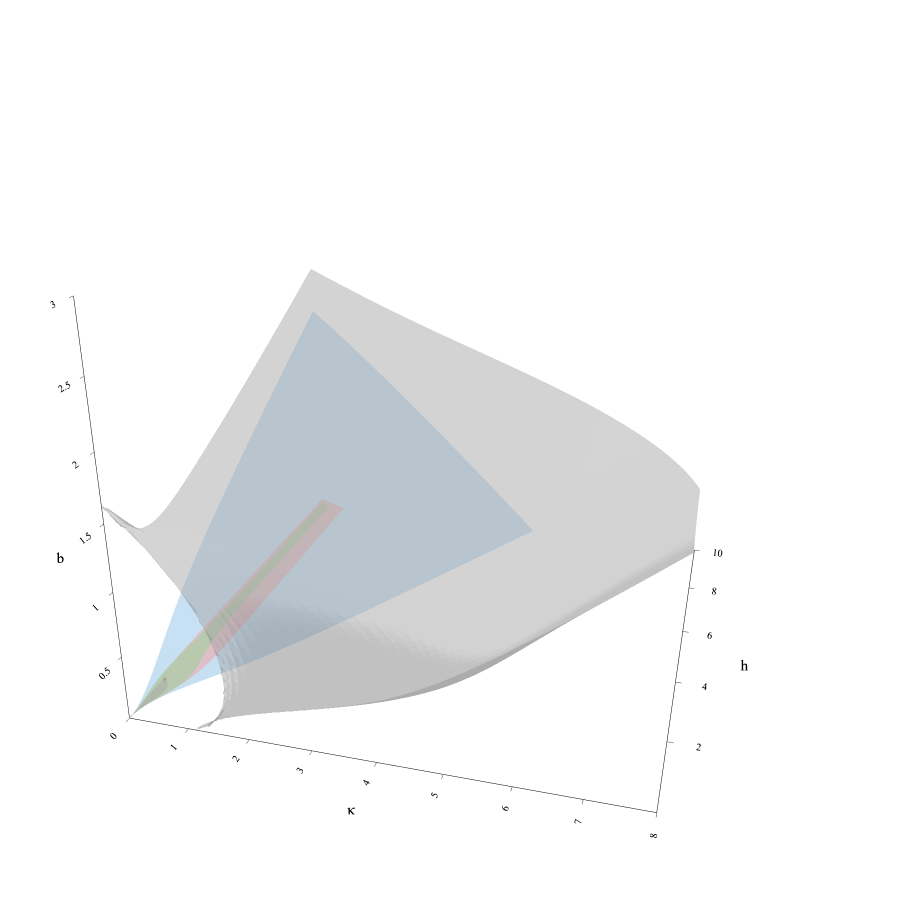}
        \label{fig:3d-stability-8}
    \end{subfigure}
    \caption{The portion of the Benjamin--Feir transition surface $\Sigma_{\mathrm{BF}}$ contained in the truncated physical parameter region $\{(\tth,\kappa,b):0<\tth\le 10,\;0\le\kappa\le 8,\;0<b\le 3\}$. All four panels display the same transition diagram from different
viewing angles. The region exterior to the outer gray component ($\{\mathsf{e}_{\mathrm{WB}}(\tth,\kappa,b)=0\}$) is Benjamin--Feir unstable,
whereas the region between the outer gray component and the blue surface ($\mathfrak D=\{\textup{D}_{\tth,\kappa,b}=0\}$) is stable.
The region bounded by the inner gray component and the blue surface contains both stable and unstable subregions, due to the intersection of the red surface ($\mathfrak R_{2}=\{r_{\tth,\kappa,b}=0\}$) and
green surface ($\{\mathsf{e}_{22}(\tth,\kappa,b)=0\}$).
Finally, the region interior to the inner gray component is stable.}
\label{fig:3d-stability-diagrams2}
\end{figure}

\begin{figure}[p]
    \centering
    \begin{subfigure}[t]{0.49\textwidth}
        \centering
        \includegraphics[width=\linewidth]{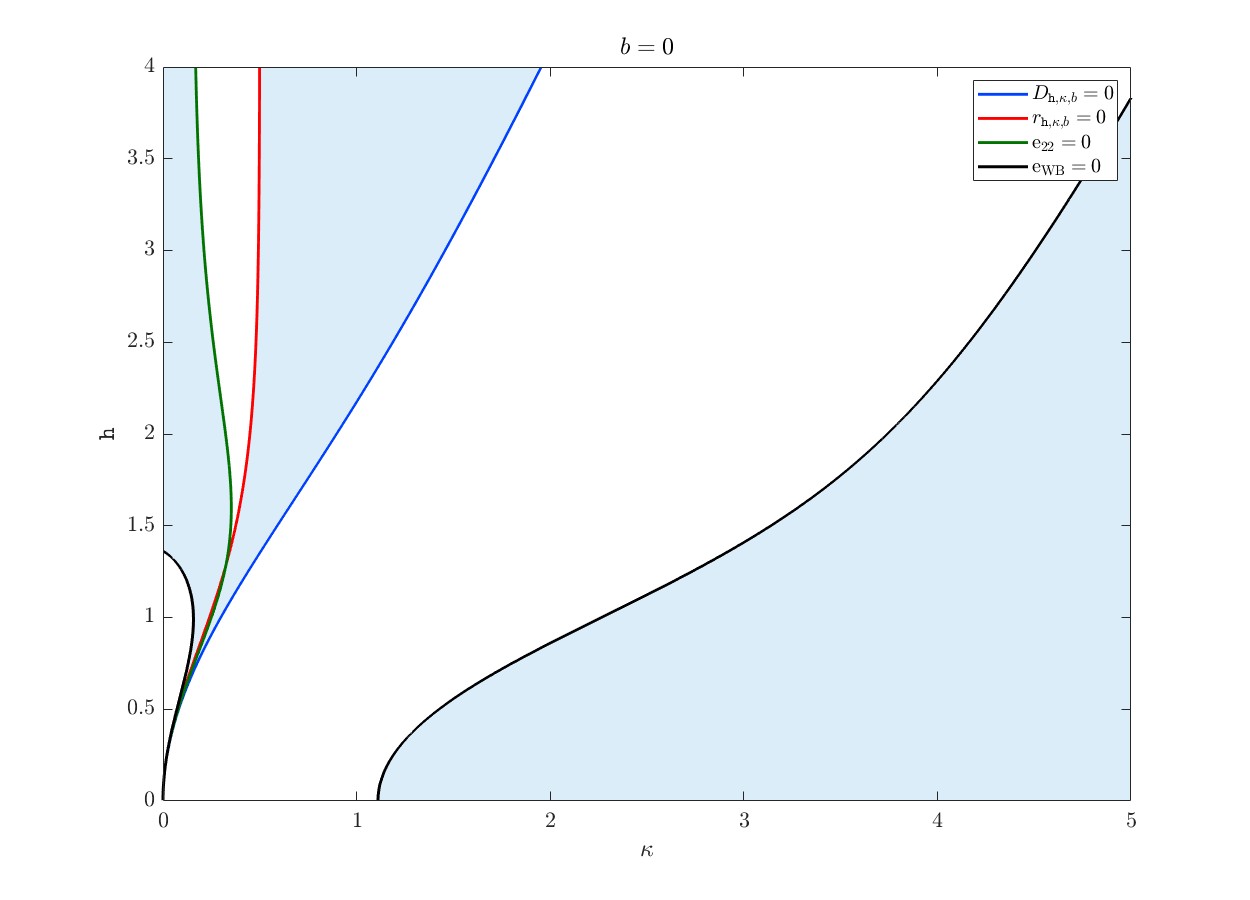}
    \end{subfigure}
    \hfill
    \begin{subfigure}[t]{0.49\textwidth}
        \centering
        \includegraphics[width=\linewidth]{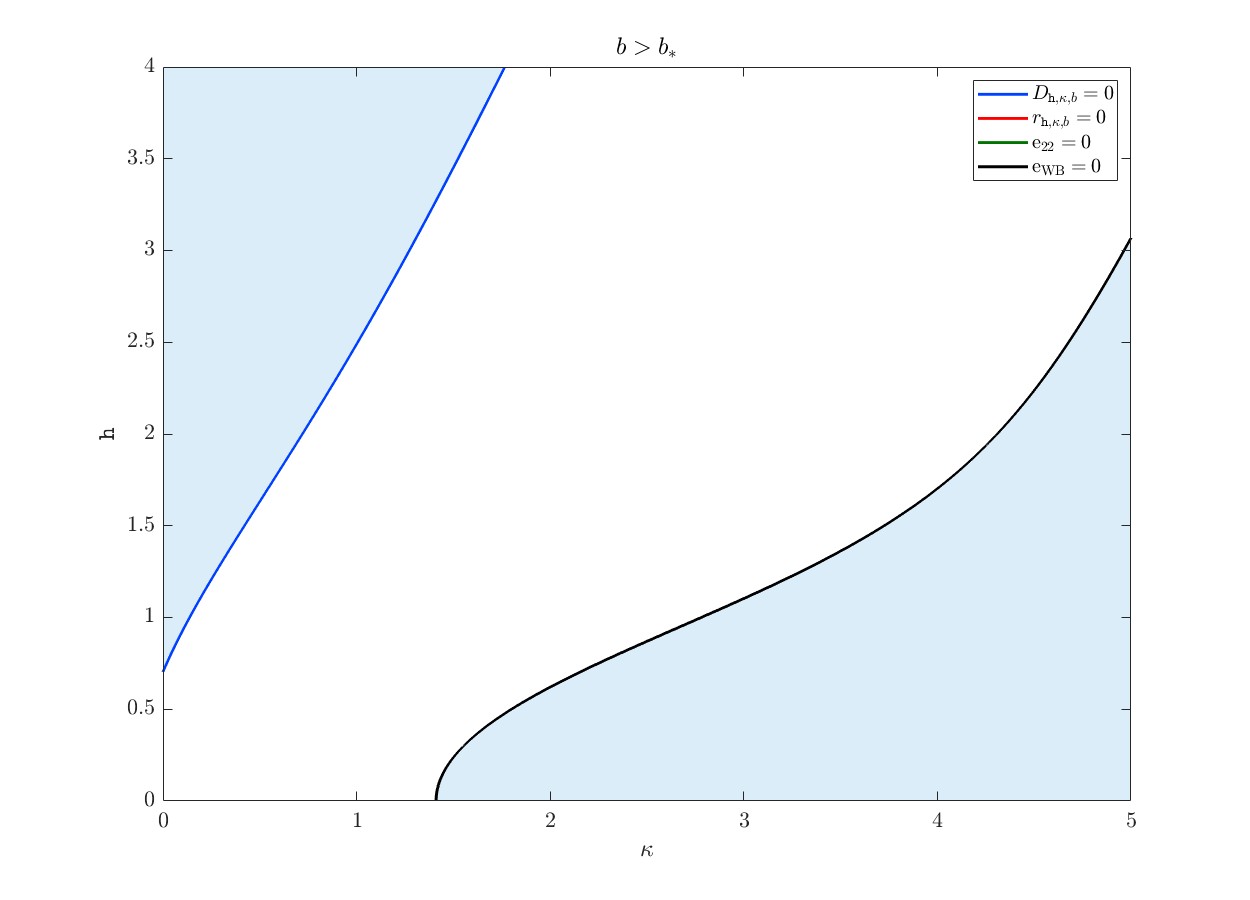}
    \end{subfigure}
    \begin{subfigure}[t]{0.49\textwidth}
        \centering
        \includegraphics[width=\linewidth]{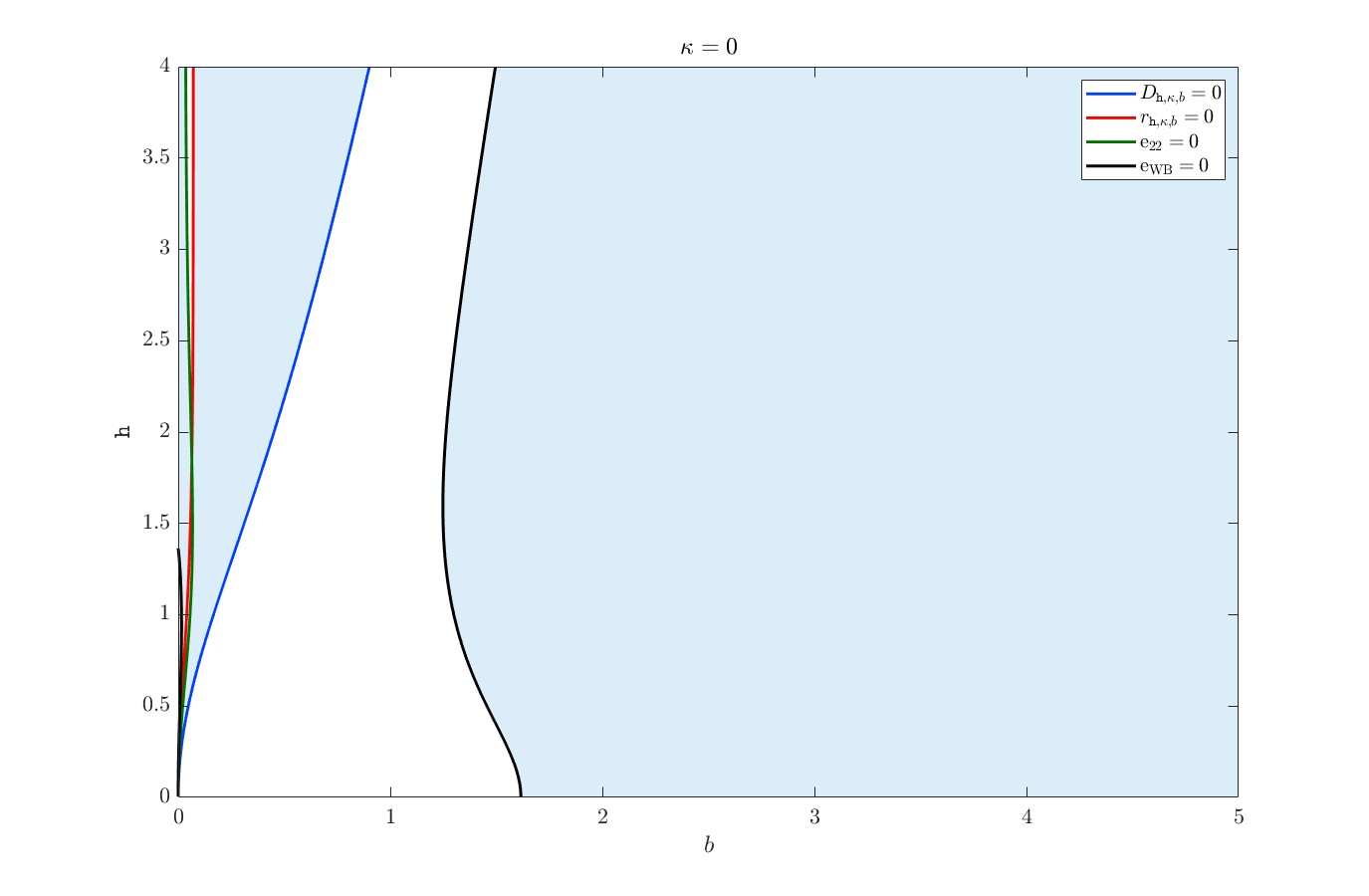}
    \end{subfigure}
    \hfill
    \begin{subfigure}[t]{0.49\textwidth}
        \centering
        \includegraphics[width=\linewidth]{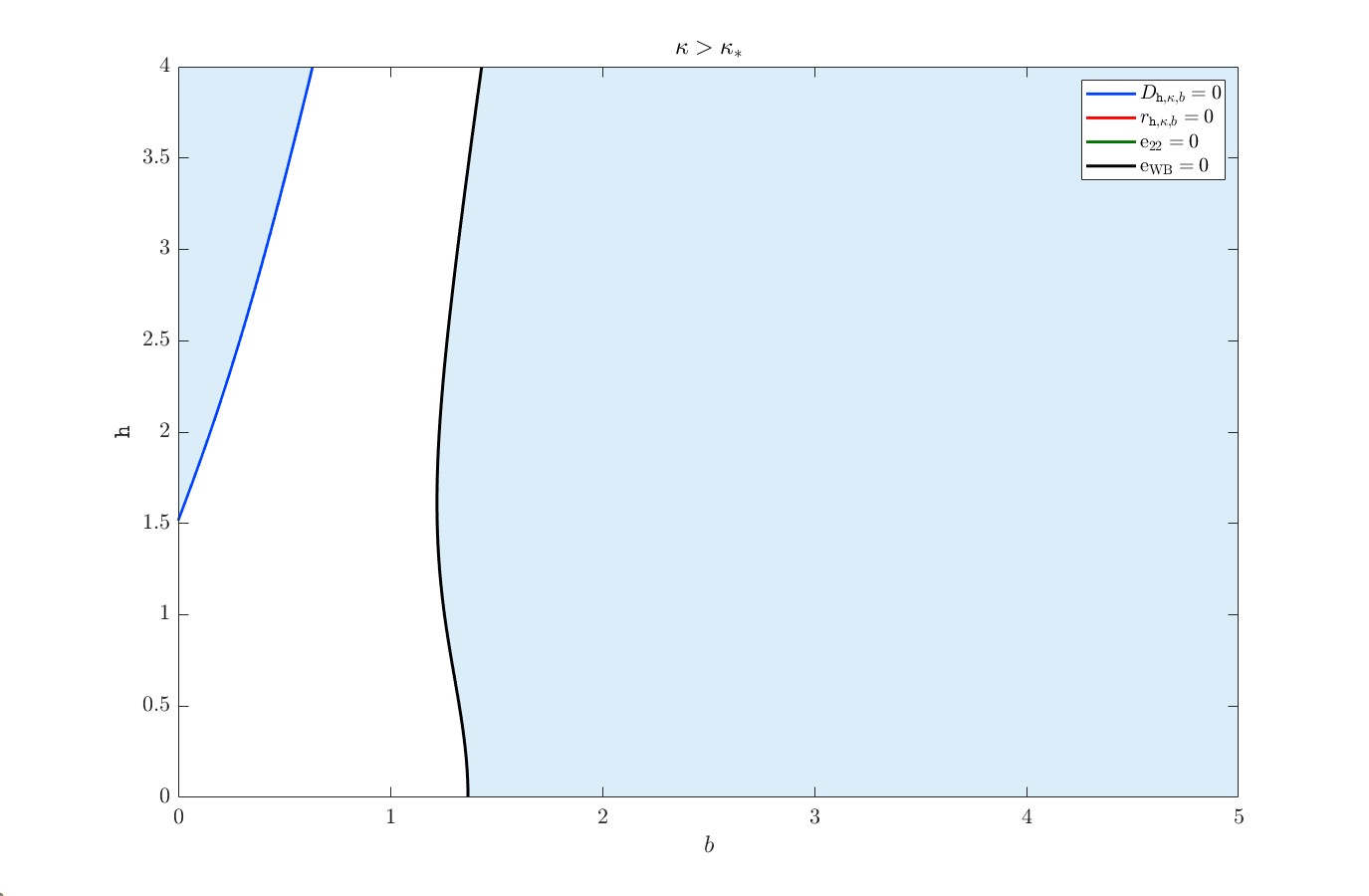}
    \end{subfigure}

    \hfill
 
    \begin{subfigure}[t]{0.49\textwidth}
        \centering
        \includegraphics[width=\linewidth]{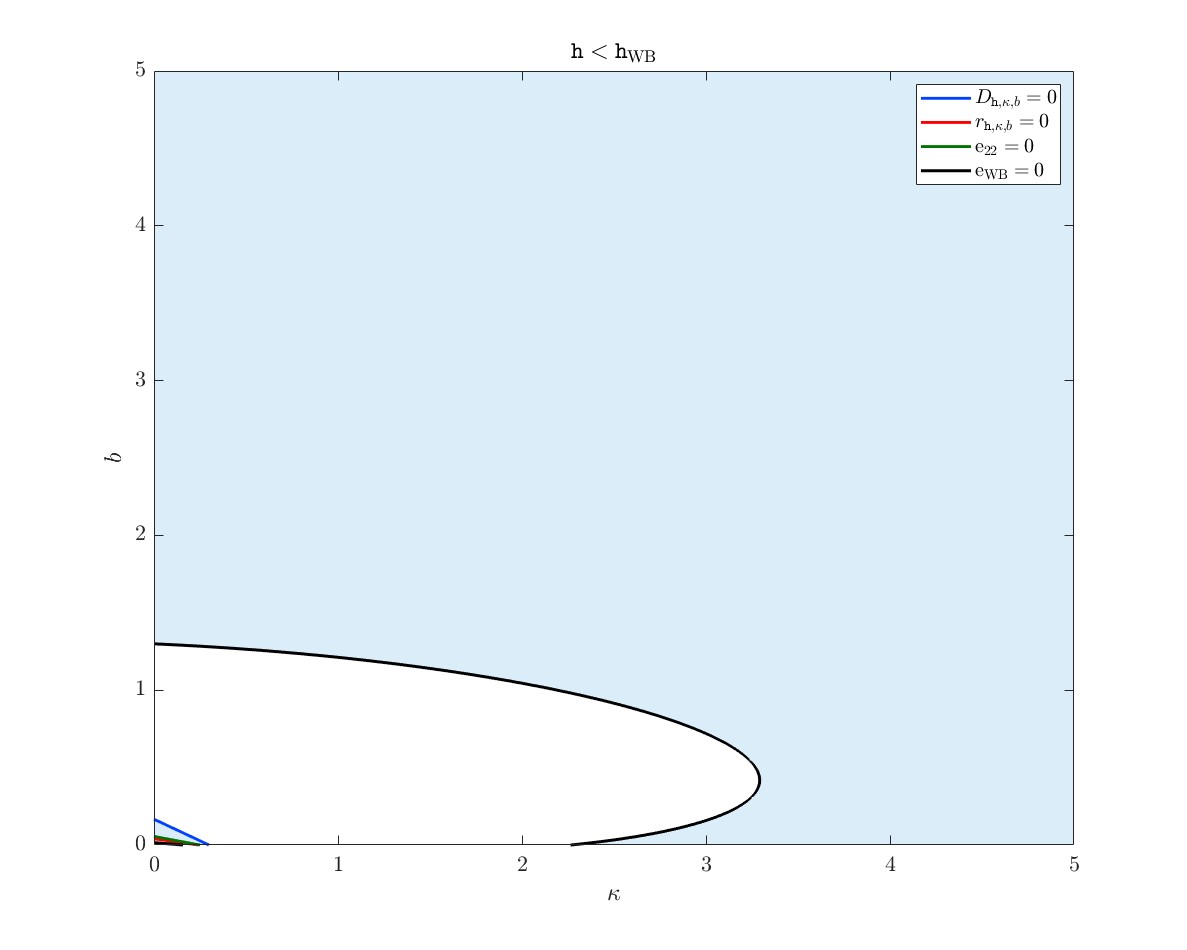}
    \end{subfigure}
    \hfill
    \begin{subfigure}[t]{0.49\textwidth}
        \centering
        \includegraphics[width=\linewidth]{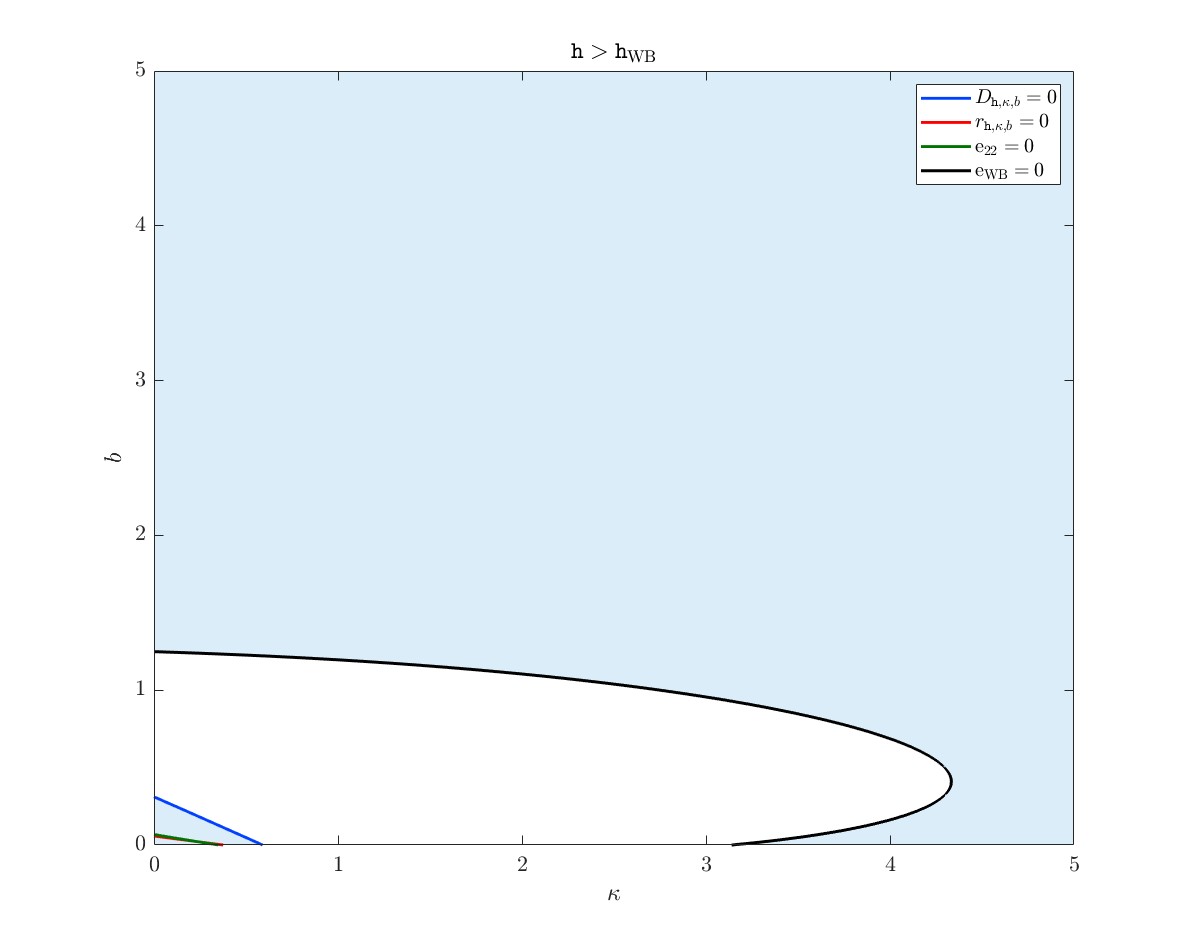}
    \end{subfigure}
    \caption{
Two-dimensional slices of the leading-order Benjamin--Feir stability diagram
together with the exceptional parameter sets.
The light-blue regions indicate Benjamin--Feir instability,
$\operatorname{Ind}(\tth,\kappa,b)=8\mathsf{e}_{22}(\tth,\kappa,b)\mathsf{e}_{\mathrm{WB}}(\tth,\kappa,b)>0$,
whereas the white regions indicate
$\operatorname{Ind}(\tth,\kappa,b)<0$.
The green and black curves denote, respectively, the regular transition curves
$\{\mathsf{e}_{22}=0\}$ and $\{\mathsf{e}_{\mathrm{WB}}=0\}$.
The blue curves denote the block-decoupling degeneracy set
$\mathfrak D=\{\textup{D}_{\tth,\kappa,b}=0\}$, while the red curves denote the
second-harmonic Wilton-type resonance set
$\mathfrak R_2=\{r_{\tth,\kappa,b}=0\}$.
The upper-left panel is the boundary-limit slice $b=0$, which recovers the
gravity--capillary stability diagram (see, for instance, \cite[Figure 1]{HM2025}).
The upper-right panel is a fixed-bending slice
$b=9/100>b_{*}:=1/14$ (cf. \eqref{bc kc}), for which all resonance curves $\mathfrak{R}$ are absent.
The middle-left panel shows the pure flexural--gravity slice $\kappa=0$,
while the middle-right panel is a fixed-capillarity slice
$\kappa=3/5>\kappa_{*}:=1/2$ (cf. \eqref{bc kc}), for which no resonance occurs.
Finally, the bottom panels are fixed-depth slices with
$\tth=1<\tth_{\mathrm{WB}}$ and $\tth=1.5>\tth_{\mathrm{WB}}$, respectively,
where $\tth_{\mathrm{WB}}=1.3627827\ldots$ is the classical
Whitham--Benjamin critical depth.
}
\label{fig:3d-stability-diagrams3}
\end{figure}

We define the Benjamin--Feir unstable set in the admissible parameter
region by
\[
    \mathcal U_{\mathrm{BF}}
    :=
    \left\{
    (\mathtt h,\kappa,b)\in\mathscr P:
    \operatorname{Ind}(\mathtt h,\kappa,b)>0
    \right\}.
\]
Its regular transition boundary is
\[
    \Sigma_{\mathrm{reg}}
    :=
    \partial_{\mathscr P}\mathcal U_{\mathrm{BF}}.
\]
For the purpose of describing the global parameter-space geometry, let us introduce the Benjamin--Feir transition surface
$\Sigma_{\mathrm{BF}}$
\[
    \Sigma_{\mathrm{BF}}
    :=
\overline{\Sigma_{\mathrm{reg}}}^{\mathscr{R}}
    \cup \mathfrak D \cup \mathfrak R_2,
    \qquad
    \mathscr{R}
    :=
    \mathbb R_{>0}\times\mathbb R_{\geq0}\times\mathbb R_{>0}.
\]
Here $\mathfrak D \not\subset \mathscr{P}$ and $\mathfrak R_2\not\subset \mathscr{P}$ are exceptional
parameter surfaces at which, respectively, the finite-dimensional
block-decoupling or the second-harmonic Stokes expansion degenerates.
Nevertheless, they form part of the boundary structure of the stability
diagram and are therefore included in $\Sigma_{\mathrm{BF}}$. 

Since $\mathrm{Ind}(\tth,\kappa,b)$ is available in closed form,
$\Sigma_{\mathrm{BF}}$ can be plotted as a surface in the three-dimensional
$(\tth,\kappa,b)$-space, separating the Benjamin--Feir unstable and stable
regions in the small-amplitude limit; see
Figure~\ref{fig:3d-stability-diagrams} and Figure~\ref{fig:3d-stability-diagrams2}. To illustrate the geometry of the stability and exceptional sets, we
display in Figure~\ref{fig:3d-stability-diagrams3} representative two-dimensional slices
of the three-parameter diagram obtained by fixing, respectively, the bending parameter $b>0$,
the capillarity parameter $\kappa\geq 0$, or the depth $\tth>0$. The boundary-limit slice $b=0$ recovers the gravity--capillary
diagram \cite[Figure 1]{HM2025}, whereas the fixed-bending and fixed-capillarity slices reveal
a genuinely hydroelastic phenomenon: the entire Wilton-type resonant
regime disappears once $b\geq b_*:=1/14$ or
$\kappa\geq\kappa_*:=1/2$. Finally, fixed-depth slices on either side of
$\tth_{\mathrm{WB}}$ display the interaction between capillarity and bending
before and after the classical pure-gravity stability transition.

\paragraph{Main analytical ingredients and comparison with earlier work.}
The analysis developed in this paper builds on the Hamiltonian and reversible
approach used in the complete descriptions of the Benjamin--Feir spectrum for
classical Stokes waves in deep and finite depth~\cite{BMV1,BMV3},
and in the gravity--capillary setting~\cite{HM2025}.
Already in finite depth, the small-spectrum problem is considerably more
delicate than a regular eigenvalue perturbation problem. The four Bloch
eigenvalues issuing from the defective eigenvalue at the origin have the same
leading size $\mathcal{O}(\mu)$. Consequently, the Benjamin--Feir pair cannot
be isolated from the long-wave pair by a direct perturbative expansion: one
must first reduce the Bloch operator to a four-dimensional invariant subspace
and subsequently separate two interacting $2\times 2$ Hamiltonian blocks. This
finite-depth coupling mechanism persists in the hydroelastic problem.

The elastic bending term introduces a further layer of difficulty. The
restoring force generated by the bending energy is
\[
-b\left(\partial_s^2 \sigma(\eta)+\frac{1}{2}\sigma(\eta)^3\right).
\]
Thus, whereas the capillary contribution gives rise to a second-order
geometric operator, elastic bending produces a fourth-order nonlinear
operator. After linearization about a non-flat Stokes wave and conformal
flattening of the free boundary, this term becomes a variable-coefficient
fourth-order operator. In particular, the hydroelastic calculation is not
obtained by formally inserting the parameter $b$ into the gravity--capillary
formulae: the Stokes-wave expansion, the flattened linearized operator, the
Bloch family, and the reduced matrix all acquire additional contributions
that must be tracked while preserving the Hamiltonian and reversible
structures. The bending term also affects the nonlinear traveling-wave problem itself.
Before carrying out the spectral analysis, we establish, in a scale of
exponentially weighted Sobolev spaces $H^{\alpha,m}$, a real-analytic
family
\begin{align*}
\varepsilon\longmapsto
(\eta_\varepsilon,\psi_\varepsilon,c_\varepsilon)
\in
H^{\alpha,m}_{\mathtt{ev}} (\T)\times H^{\alpha,m}_{\mathtt{odd}}(\T)\times \R
\end{align*}
of small-amplitude hydroelastic Stokes waves for parameters outside the
Wilton-type resonance set $\fR$.
 The exclusion of $\fR$ has an intrinsic
bifurcation-theoretic meaning: on $\fR$, the fundamental mode resonates with
a higher harmonic, so that the simple-kernel mechanism underlying the
standard local bifurcation argument breaks down. As an analytic foundation for this construction, we prove a tame
analyticity result for the finite-depth Dirichlet--Neumann operator in
exponentially weighted Sobolev spaces. This provides the analytic
dependence on $\varepsilon$ needed for the subsequent expansions of the
Bloch family and the reduced matrix. 

At the linear level, bending also modifies the mode-dependent
phase velocities according to
\[
c_n^2=\frac{(1+\kappa n^2+b n^4)\tanh(n\tth)}{n},
\]
and therefore reshapes the Wilton-type resonance loci. The coefficients entering the final instability index consequently retain their full dependence on the three parameters $(\tth,\kappa,b)$.

Our result is complementary to the existing stability theory for
hydroelastic waves. Previous work on periodic flexural--gravity waves derived
nonlinear Schr\"odinger approximations and compared their modulational
predictions with numerical Floquet computations~\cite{TMPVB},
while the recent work~\cite{WY} proves nonlinear modulational instability
in deep pure hydroelasticity through a focusing cubic NLS approximation. Here
we instead work directly with the linearized finite-depth Euler equations in
the presence of both capillarity and bending, and resolve all four Bloch
eigenvalues bifurcating from the origin. After a symplectic Kato reduction, we perform a singular symplectic
rescaling under which the reduced matrix extends analytically to the zero
Floquet exponent. A first symplectic and reversibility-preserving Lie
transform, determined by a Sylvester equation, lowers the off-diagonal
coupling from $\mathcal{O}(\mu\varepsilon)$ to
$\mathcal{O}(\mu\varepsilon^3)$. We then eliminate the residual coupling
exactly by a second analytic, structure-preserving conjugation. Within the non-resonant regime, the obstruction to this exact block
separation is precisely the characteristic-collision surface
$\mathfrak D$. Thus, the exclusion of $\mathfrak D$
is not merely technical: on this surface the Benjamin--Feir transport
velocity coincides with one of the long-wave characteristic velocities, and
the normalized Sylvester operator loses invertibility.

The bending term also creates degeneracies that are absent from the
pure-gravity and gravity--capillary problems: a coefficient entering the
reduced Benjamin--Feir block, which is non-vanishing in those settings, may
vanish in the hydroelastic regime. The resulting three-parameter diagram
therefore records not only the sign change of the Benjamin--Feir index, but
also the interaction between modulational instability, Wilton-type
resonances, and an intrinsic characteristic collision.

\section{The full Benjamin--Feir spectrum of hydroelastic waves} \label{sec2 full}
In this section we formulate the two-dimensional hydroelastic water-wave problem in
Zakharov variables, construct the small-amplitude hydroelastic Stokes branch, and
introduce the Bloch--Floquet spectral problem associated with its modulational
stability. The formulation is the finite-depth analogue of the Hamiltonian
free-surface description for classical water waves, with the additional restoring
forces generated by capillarity and elastic bending of the upper interface.

\paragraph{The problem of hydroelastic waves.} 
Let $\tth>0$ denote the fluid depth and let $\mathbb T := \mathbb R / 2\pi \mathbb Z$.
At time $t$, the fluid occupies the domain
\begin{align*}
    \Omega_{\tth,\eta}
    :=
    \bigl\{ (x,y)\in \mathbb T\times \mathbb R :
    -\tth < y < \eta(t,x) \bigr\},
\end{align*}
whose upper boundary is the graph
\[
    S_t := \bigl\{ (x,y)\in \mathbb T\times\mathbb R :
    y=\eta(t,x) \bigr\}.
\]
The function $\eta=\eta(t,x)$ describes the vertical displacement of a thin
elastic cover lying on the upper surface of the fluid.

We consider an inviscid, incompressible, homogeneous fluid of constant density
equal to one. If $\mathbf v=(v_1,v_2)$ and $p$ denote respectively the fluid
velocity and pressure, then the motion is governed by the free-boundary Euler
equations
\begin{equation*}
    (\partial_t+\mathbf v\cdot\nabla)\mathbf v
    =
    -\nabla p - g e_2,
    \qquad
    \nabla\cdot \mathbf v =0,
    \qquad
    (x,y)\in \Omega_{\tth,\eta},
\end{equation*}
where $g>0$ is the gravitational constant and $e_2=(0,1)$. We impose the
irrotationality condition
\begin{equation}
\label{eq:irrotationality}
    \operatorname{curl}\mathbf v =0
    \qquad\text{in }\Omega_{\tth,\eta},
\end{equation}
as well as the impermeable bottom condition
\begin{equation*}
    v_2(t,x,-\tth)=0.
\end{equation*}
The upper interface moves with the normal component of the fluid velocity.
Equivalently, the kinematic boundary condition reads
\begin{equation*}
    \partial_t \eta
    =
    v_2-v_1\eta_x
    \qquad\text{on }S_t.
\end{equation*}
The elastic cover is assumed to be massless and to contribute both capillary
and bending energies. We introduce the arclength derivative and the curvature
of the graph by
\begin{equation}
\label{eq:curvature-arclength}
    \partial_s
    :=
    \frac{1}{\sqrt{1+\eta_x^2}}\partial_x,
    \qquad
    \sigma(\eta)
    :=
    \partial_x
    \left(
        \frac{\eta_x}{\sqrt{1+\eta_x^2}}
    \right).
\end{equation}
The surface energy is
\begin{equation}
\label{eq:surface-energy}
    \mathcal E(\eta)
    :=
    \kappa
    \underbrace{\int_{\mathbb T}
        \bigl(\sqrt{1+\eta_x^2}-1\bigr)\,dx}_{\mathcal{E}_\kappa(\eta)}
    +
    b
    \underbrace{\int_{\mathbb T}
        \frac{1}{2}\sigma(\eta)^2 \sqrt{1+\eta_x^2}\,dx}_{\mathcal{E}_b(\eta)},
\end{equation}
where $\kappa\geq 0$ is the surface-tension coefficient and $b>0$ is the
bending rigidity.

Thus, after normalizing the constant atmospheric pressure to zero, the dynamic
boundary condition can be written as
\begin{equation}
\label{eq:dynamic-pressure-condition}
    p\big|_{S_t}
    = -\kappa \sigma(\eta)
    +b\left(
        \partial_s^2\sigma(\eta)
        +
        \frac12 \sigma(\eta)^3
    \right).
\end{equation}

The elastic contribution in \eqref{eq:dynamic-pressure-condition} is the geometric
quadratic-bending specialization of the nonlinear membrane law derived by \cite[Toland, Section 1]{To2008}. More precisely, Toland's steady hydroelastic boundary condition contains
the term
\[
    b_0 \sigma
    -
    b_2 \partial_s^2 \sigma
    -
    \frac{b_2}{2}\sigma^3
\]
when the stored bending energy density is quadratic in the curvature. Taking
$b_0=\kappa$ and $b_2=b$ gives precisely the capillary--bending restoring force
in \eqref{eq:dynamic-pressure-condition}. In contrast with Toland's steady infinite-depth
setting, here we retain the full time-dependent finite-depth formulation required
for the spectral analysis of periodic traveling waves.

By \eqref{eq:irrotationality}, there exists a velocity potential
$\Psi=\Psi(t,x,y)$ such that
\[
    \mathbf v = \nabla_{x,y}\Psi.
\]
For a given surface profile $\eta$ and boundary trace
\[
    \psi(t,x):=\Psi(t,x,\eta(t,x)),
\]
the potential $\Psi$ is determined by
\begin{equation}
\label{eq:harmonic-extension}
\left\{
\begin{aligned}
    \Delta \Psi &=0
        &&\text{in }\Omega_{\tth,\eta},\\
    \Psi(x,\eta(t,x)) &= \psi(t,x)
        &&\text{on }S_t,\\
    \partial_y\Psi(x,-\tth)&=0
        &&\text{on }y=-\tth.
\end{aligned}
\right.
\end{equation}
The associated Dirichlet--Neumann operator is defined by
\begin{equation}
\label{eq:DN-definition}
    G(\eta)\psi:=G(\eta;\tth)\psi
    :=\Psi_y(x,\eta(x))-\Psi_x(x,\eta(x))\eta_x(x).
\end{equation}
The equations of motion reduce to a
system of equations determined by the non-local quasi-linear equations (see, \cite[Zakharov]{Zak68} and \cite[Craig--Sulem]{CS})
\begin{equation} \label{Craig-Sulem formula}
\left\{\begin{aligned}
    \eta_t&=G(\eta)\psi\\
    \psi_t&=-g\eta-\frac{\psi_x^2}{2}+\frac{1}{2(1+\eta^2_x)}\left(G(\eta)\psi+\eta_x\psi_x\right)^2+\kappa\sigma(\eta)-b\left(\partial^2_s \sigma(\eta)+\frac{1}{2}\sigma(\eta)^3\right).
\end{aligned}\right.
\end{equation}
We pause to remark that the Dirichlet--Neumann operator $[G(\eta)\psi](x):=\Psi_y(x,\eta(x))-\Psi_x(x,\eta(x))\eta_x(x)$ is symmetric (see Appendix \ref{sec:AppG}); in the sequel, without loss of generality, we set the gravity constant $g=1$.

System \eqref{Craig-Sulem formula} is Hamiltonian and can be written as
\begin{equation} \label{pa_t eta psi}
    \begin{aligned}
        \partial_t \begin{bmatrix}
\eta\\
\psi 
\end{bmatrix}=
\mathcal{J}\begin{bmatrix}
\nabla_\eta \mathcal{H}\\
\nabla_\psi \mathcal{H} 
\end{bmatrix}
    \end{aligned},~~\mathcal{J}:=
    \begin{bmatrix}
0& \mathrm{Id}\\
-\mathrm{Id} &0
\end{bmatrix},
\end{equation}
where $\nabla$ denotes the $L^2$-gradient, namely
\begin{align}\label{defL2gra}
    \de_\eta \mathcal{H}[\tilde{\eta}] = \int_\T \nabla_\eta \mathcal{H}  \, \tilde{\eta} \, \de x, \qquad 
    \de_\psi \mathcal{H} [\tilde{\psi}] = \int_\T \nabla_\psi \mathcal{H}  \, \tilde{\psi} \, \de x,
\end{align}
and the Hamiltonian 
\begin{align} \label{mathcal H}
\mathcal{H}(\eta,\psi):=\int_{\mathbb{T}} \left(\frac{1}{2}\psi G(\eta)\psi+\frac{g}{2}\eta^2\right) \de x +\mathcal{E}(\eta)    
\end{align}
is the sum of the kinetic energy of the fluid and the
gravitational, capillary, and elastic bending potential energies of the
fluid--elastic-interface system. Note that we can choose the function space of $\eta$ and $\psi$. For example, we can choose $(\eta,\psi) \in H^{\alpha,m}_{\mathtt{ev}} (\T)\times H^{\alpha,m}_{\mathtt{odd}}(\T)$ with $\alpha \ge 0$ and $m > \frac{11}{2}$.
In addition,  the water waves system \eqref{Craig-Sulem formula} is reversible with respect to the involution
\begin{equation} \label{rho involution}
    \begin{aligned}
        \rho \begin{bmatrix}
            \eta(x)\\
            \psi(x)
        \end{bmatrix}:=
        \begin{bmatrix}
            \eta(-x)\\
            -\psi(-x)
        \end{bmatrix},~~\mathrm{i.e.}~\mathcal{H}\circ \rho=\mathcal{H},
    \end{aligned}
\end{equation}
and it is space invariant.

 \paragraph{Hydroelastic Stokes waves.} In a moving reference frame with constant speed $c$ (and with normalized gravity $g=1)$, the water waves system \eqref{Craig-Sulem formula} becomes
 \begin{equation} \label{in the reference frame}
\left\{\begin{aligned}
    \eta_t&=c\eta_x+G(\eta)\psi\\
    \psi_t&=c\psi_x-\eta-\frac{\psi_x^2}{2}+\frac{1}{2(1+\eta^2_x)}\left(G(\eta)\psi+\eta_x\psi_x\right)^2+\kappa\sigma(\eta)-b\left(\partial^2_s \sigma(\eta)+\frac{1}{2}\sigma(\eta)^3\right).
\end{aligned}\right.
\end{equation}
We consider small-amplitude hydroelastic Stokes waves, namely stationary solutions of \eqref{in the reference frame} which we further require to be  $2\pi$-periodic in space.
 
\begin{theorem} [Hydroelastic Stokes waves] \label{Thm: Stokes expansion} 
 Let $(\tth,\kappa,b) \in (\R_{>0}\times\R_{\geq 0} \times\R_{> 0}) \setminus \fR$, with  $\fR$  in \eqref{def:fR}. There 
 exist $\e_*:=\e_*(\tth,\kappa,b) >0$ and a unique family  of real analytic 
 solutions $(\eta_{\e}(x), \psi_{\e}(x), c_{\e})$, parameterized by the amplitude $|\e| \leq \e_*$, to  
\begin{equation}\label{travelingWWstokes}
c \, \eta_x+G(\eta)\psi = 0 \, , \quad 
c \, \psi_x -  \eta - \dfrac{\psi_x^2}{2} + 
\dfrac{1}{2(1+\eta_x^2)} \big( G(\eta) \psi + \eta_x \psi_x \big)^2  +\kappa\sigma(\eta)-b\left(\partial^2_s \sigma(\eta)+\frac{1}{2}\sigma(\eta)^3\right)= 0 \, , 
\end{equation}
  such that
 $ \eta_\e (x), \psi_\e (x) $ are $2\pi$-periodic;  $\eta_\e (x) $ is even
and $\psi_\e (x) $ is odd. Moreover, they admit the expansions where the remainder terms are understood in the corresponding
function-space norms.
 \begin{equation}\label{exp:Sto}
 \begin{aligned}
  & \eta_{\e} (x) = \e \cos(x)+\e^2\left(\eta_{2}^{[0]}+\eta_{2}^{[2]}\cos(2x)\right) +\cO(\e^3), \\ 
  & \psi_{\e} (x)  =  \e c_{\tth}^{-1}\ckb \sin(x)+\e^2 \psi_{2}^{[2]}\sin(2x)+\cO(\e^3),  \\
  & c_{\e} = c_{0}  +\e^2 c_{2}+\cO(\e^3)\quad \text{where} \quad c_{0} := \ckb c_{\tth}, \quad \ckb:=\sqrt{1+\kappa+b},\quad c_{\tth}:=\sqrt{\tanh(\tth)} \, . 
   \end{aligned}
  \end{equation}
\end{theorem}

For any  $ \alpha \geq  0 $ and $m + \frac12 \in \N$, $m > \frac{11}{2}$, there exists $ \e_*>0 $ such that
the map $\e \mapsto (\eta_\e, \psi_\e, c_\e)$ is analytic from $B(\e_*) \to H^{\alpha,m}_{\mathtt{ev}} (\T)\times H^{\alpha,m}_{\mathtt{odd}}(\T)\times \R$, where 
$ H^{\alpha,m}_{\mathtt{ev}}(\T) $, respectively $ H^{\alpha,m}_{\mathtt{odd}}(\T) $, denote the  space of even, respectively odd, 
 real valued $ 2 \pi $-periodic analytic functions
$ u(x) = \sum_{k \in \mathbb{Z}} u_k e^{\im k x} $
such that $ \| u \|_{\alpha,m}^2 := \sum_{k \in \mathbb{Z}} |u_k|^2 \langle k \rangle^{2m} 
e^{2 \alpha |k|} < + \infty$. We pause to remark that the bifurcation of small-amplitude Stokes waves from the trivial solution was first studied for pure gravity water waves by \cite[Stokes]{stokes}, \cite[Levi-Civita]{LC}, \cite[Nekrasov]{Nek}, and \cite[Struik]{Struik}. 
 In our setting, the existence and analyticity of a bifurcating branch of traveling waves follows from the Crandall--Rabinowitz theorem. Also, we denote $\mathbb{T}:=\mathbb{R}/ 2\pi\mathbb{Z}$ and $B(r):=\{x\in\mathbb{R}: |x|<r\}$ the open ball of radius $r$ centered at zero. 

The existence of the hydroelastic Stokes waves and their analytic dependence on $\e$ are established in Appendix~\ref{sec:AppG}. The expansions in 
\eqref{exp:Sto} are proved in Appendix \ref{sec:App2}. 
The condition $(\tth,\kappa,b) \not \in \fR$ is used in Lemma \ref{lem:B0} to ensure that the kernel of the linearized operator at the flat surface is one-dimensional. 
The condition $(\mathtt{h},\kappa,b)\notin\mathfrak R$ is equivalent to
requiring that the phase velocity of the fundamental mode is distinct
from that of every higher integer mode, namely
\begin{align*}
    \frac{(1+\kappa n^2+b n^4)\tanh(n\tth)}{n}
    \neq
    (1+\kappa+b)\tanh(\tth),
    \qquad n\ge2.
\end{align*}
This is precisely the non-resonance condition ensuring that the kernel
associated with the fundamental Fourier mode is simple in the
Crandall--Rabinowitz bifurcation argument. We also note that the excluded set $\mathfrak R$ is not merely a technical artifact.  When
$(\tth,\kappa,b)\in\mathfrak R$, the interaction between finite depth, surface tension, and the
higher-order bending term may produce higher-order resonances among the linear modes.  In this
regime one expects the emergence of Wilton-type traveling waves, in analogy with the classical
capillary-gravity Wilton ripples~\cite[Wilton]{Wilton}.  It would be an interesting problem to construct
such resonant hydroelastic periodic waves rigorously, possibly by combining the present
hydroelastic formulation with the Lyapunov--Schmidt-type bifurcation method developed in
\cite[Reeder--Shinbrot]{reeder}.

We pause to remark that, recalling \eqref{def:fRn}, for each fixed $\tth>0$ and integer $n\geq2$, one has
\[
1<\rho_n(\tth)<n<n^2.
\]
Consequently, the resonance condition
\[
\kappa+\frac{n^4-\rho_n(\tth)}{n^2-\rho_n(\tth)}\,b
=
\frac{\rho_n(\tth)-1}{n^2-\rho_n(\tth)}
\]
defines a straight line in the $(\kappa,b)$-plane with negative slope and positive $\kappa$- and $b$-intercepts. Hence, the resonance set forms a line segment in the first quadrant:
\[
0< b\leq \frac{\rho_n(\tth)-1}{n^4-\rho_n(\tth)}.
\]
In particular, when $n=2$, the resonance set is given by
\begin{align} \label{rhkb}
\mathfrak R_2
=
\left\{
r_{\tth,\kappa,b}=0
\right\}, \qquad r_{\tth,\kappa,b}
:=\ch^4(1+\kappa+b)-3\kappa-15b.
\end{align}
Moreover, in the deep-water limit $\tth\to\infty$, one has $\rho_n(\tth)\to n$, and therefore the resonance condition becomes
\[
\kappa+(n^2+n+1)b=\frac1n.
\]
Finally, for every $n \geq 2$, the resonance condition in \eqref{def:fRn} and the
inequalities $1<\rho_n(\tth)<n$ imply
\[
\kappa
\leq
\frac{\rho_n(\tth)-1}{n^2-\rho_n(\tth)}
<
\frac{n-1}{n^2-n}
=
\frac{1}{n}
\leq
\frac{1}{2},
\]
and
\[
b
\leq
\frac{\rho_n(\tth)-1}{n^4-\rho_n(\tth)}
<
\frac{n-1}{n^4-n}
=
\frac{1}{n(n^2+n+1)}
\leq
\frac{1}{14}.
\]
Consequently, every Wilton-type resonance is confined to the region
\[
0 \leq \kappa < \frac{1}{2},
\qquad
0 < b < \frac{1}{14}.
\]
In particular,
\begin{align} \label{bc kc}
b \geq b_*:=\frac{1}{14}
\quad\text{or}\quad
\kappa \geq \kappa_*:=\frac{1}{2}
\qquad\text{implies}\qquad
(\tth,\kappa,b)\notin \mathfrak{R}
\end{align}
for every finite depth $\tth>0$.

\paragraph{Linearization.}
We linearize system
\eqref{in the reference frame} at the Stokes waves $(\eta_\varepsilon(x),\psi_\varepsilon(x))$ given in Theorem \ref{Thm: Stokes expansion} and evaluate $c$ at $c_\varepsilon$. 
Using the shape derivative formula \cite[Lannes]{LD, LD book} 
$
\mathrm{d}_\eta G(\eta)[\hat \eta]\psi = - G(\eta)(B\hat \eta) - \pa_x( V \hat \eta), 
$
where
the functions $(V(x),B(x))$ are the horizontal and vertical components of the velocity field $(\Psi_x,\Psi_y)$ at the free surface and are given by
\begin{align} \label{espV}
   V:= V(x)&:=-B(\eta_\varepsilon)_x+(\psi_\varepsilon)_x,\\ \label{espB}
  B:=  B(x)&:=\frac{G(\eta_\varepsilon)\psi_\varepsilon+(\psi_\varepsilon)_x(\eta_\varepsilon)_x}{1+(\eta_\varepsilon)^2_x}=\frac{(\psi_\varepsilon)_x-c_\varepsilon}{1+(\eta_\varepsilon)^2_x} (\eta_\varepsilon)_x,
\end{align}
one obtains the real, autonomous linearized  system
\small
\begin{equation} \label{first linear eq}
    \begin{aligned}
        \begin{bmatrix}
            \hat{\eta}_t\\
            \hat{\psi}_t
        \end{bmatrix}=
        \left[\begin{array}{c|c} 
	 -G(\eta_\varepsilon)B-\partial_x\circ(V-c_\varepsilon) & G(\eta_\varepsilon)  \\ 
	\hline 
	-1+B(V-c_\varepsilon)\partial_x-B\partial_x\circ(V-c_\varepsilon)-BG(\eta_\varepsilon)\circ B+\kappa\mathrm{d}\sigma(\eta_\e)-b\mathrm{d}F(\eta_\e) & -(V-c_\varepsilon)\partial_x+B G(\eta_\varepsilon)
\end{array}\right] 
        \begin{bmatrix}
            \hat{\eta}\\
            \hat{\psi}
        \end{bmatrix},
    \end{aligned}
\end{equation}
\normalsize
where
\begin{align} \label{esp l}
   F(\eta):=\partial^2_s \sigma(\eta)+\frac{1}{2}\sigma(\eta)^3,\qquad \mathrm{d} F(\eta_\e)=\pa^4_x+\e^2\left(F^{[1]}_2\pa_x+F^{[2]}_2\pa^2_x+F^{[3]}_2\pa^3_x+F^{[4]}_2\pa^4_x\right)+\cO(\e^3).
\end{align}
and
\begin{equation} \label{F12 F22 F32 F42}
\begin{aligned}
    &F^{[1]}_2(x):=-5(\eta_1)_x(\eta_1)_{xxxx}-10(\eta_1)_{xx}(\eta_1)_{xxx}=\frac{15}{2}\sin(2x),\\
    &F^{[2]}_2(x):=-\frac{15}{2}(\eta_1)^2_{xx}-10(\eta_1)_x(\eta_1)_{xxx}=\frac{5}{4}-\frac{35}{4}\cos(2x),\\
    &F^{[3]}_2(x):=-10(\eta_1)_x(\eta_1)_{xx}=-5\sin(2x),\qquad F^{[4]}_2(x):=-\frac{5}{2}(\eta_1)^2_{x}=\frac{5}{4}\cos(2x)-\frac{5}{4}.
\end{aligned}
\end{equation}
We remark that the operator $\mathrm{d} F$ is symmetric. In fact, from \eqref{defL2gra} we know
\[
\de \mathcal{E}_b(\eta)[\tilde{\eta}] = \int_\T F(\eta)  \, \tilde{\eta} \, \de x \qquad\text{where}\qquad \mathcal{E}_b(\eta):=\frac{1}{2}\int_{\mathbb{T}} \sigma(\eta)^2\sqrt{1+\eta^2_x} \,\de x \quad (cf. \,\eqref{eq:surface-energy}).   
\]
Taking another differential w.r.t. $\eta$ we obtain
\[
\de^2 \mathcal{E}_b(\eta)[\tilde{\eta},\hat{\eta}] = \int_\T \de F(\eta)[\hat{\eta}]  \, \tilde{\eta} \, \de x.
\]
As a result, 
\[
\int_\T \de F(\eta)[\hat{\eta}]  \, \tilde{\eta} \, \de x = \de^2 \mathcal{E}_b(\eta)[\tilde{\eta},\hat{\eta}] = \de^2 \mathcal{E}_b(\eta)[\hat{\eta},\tilde{\eta}] = \int_\T \de F(\eta)[\tilde{\eta}]  \, \hat{\eta} \, \de x.
\]

Similarly, the linearized curvature operator $\mathrm{d}\sigma(\eta_\varepsilon)[\hat{\eta}] = \partial_x \left( (1+(\eta_\varepsilon)_x^2)^{-3/2} \hat{\eta}_x \right)$ expands as
\begin{align} \label{dSigma expansion}
    \mathrm{d}\sigma(\eta_\varepsilon) = \partial_x^2 + \varepsilon^2 \left( \Sigma_2^{[1]} \partial_x + \Sigma_2^{[2]} \partial_x^2 \right) + \mathcal{O}(\varepsilon^3),
\end{align}
where the coefficients are given by
\begin{equation} \label{Sigma coefficients}
\begin{aligned}
    \Sigma_2^{[1]}(x) := -3(\eta_1)_x(\eta_1)_{xx} = -\frac{3}{2}\sin(2x), \quad \Sigma_2^{[2]}(x) := -\frac{3}{2}(\eta_1)_x^2 = -\frac{3}{4} + \frac{3}{4}\cos(2x).
\end{aligned}
\end{equation}
As before, $\de \sigma$ is symmetric. The map  $\e \to (V, B)$ is analytic as a map $B(\e_0) \to H^{\alpha, m-1}_{\mathtt{ev}}(\T) \times H^{\alpha, m-1}_{\mathtt{odd}}(\T)$. 
The real system \eqref{first linear eq} is Hamiltonian, i.e. of the form
$\cJ \cA$ with $\cA$ symmetric with respect to the real scalar product of $L^2(\T, \R^2) = L^2(\T, \R) \times L^2(\T, \R)$.
Moreover the linear operator in \eqref{first linear eq} is reversible, namely it anti-commutes with the involution $\rho$ in \eqref{rho involution}. 

Next, we conjugate \eqref{first linear eq} by using the time-independent ``good unknown of Alinhac'' linear transformation 
\begin{align}\label{alinhac}
    \begin{bmatrix}
        \hat{\eta}\\
        \hat{\psi}
    \end{bmatrix}:= Z
    \begin{bmatrix}
        u\\
        v
    \end{bmatrix},~~Z=\begin{bmatrix}
        1&0\\
        B&1
    \end{bmatrix},
    ~~Z^{-1}=\begin{bmatrix}
        1&0\\
        -B&1
    \end{bmatrix},
\end{align}
yielding the linear system 
\begin{align} \label{second linear eq}
    \begin{bmatrix}
        u_t\\
        v_t
    \end{bmatrix}=  \widetilde{\cL}_\e 
    \begin{bmatrix}
        u\\
        v
    \end{bmatrix} \ , \qquad 
    \widetilde{\cL}_\e := \begin{bmatrix}
        -\partial_x\circ (V-c_\varepsilon)& G(\eta_\varepsilon)\\
        -1-(V-c_\varepsilon)B_x+\kappa\,\mathrm{d}\sigma(\eta_\e)-b\,\mathrm{d}F(\eta_\e) ~~& -(V-c_\varepsilon)\partial_x
    \end{bmatrix} \ , 
\end{align}
which  is Hamiltonian and reversible since the transformation $Z$ is symplectic, $\mathrm{i.e.}$ $Z^T\mathcal{J} Z=\mathcal{J}$, and satisfies $Z\circ \rho=\rho\circ Z$.

Next, 
we perform a conformal change of variables to flatten 
the water surface. 
By \cite[Appendix A]{BBHM}, 
 there exists a diffeomorphism of $\mathbb{T}$,
 $ x\mapsto x+\mathfrak{p}(x)$, with a small $2\pi$-periodic  function $\mathfrak{p}(x)$, 
 and a small constant $\ttf $, such that, by defining the associated composition operator $ (\mathfrak{P}u)(x) := u(x+\mathfrak{p}(x))$, the Dirichlet--Neumann operator writes as \cite[Lemma A.5]{BBHM}
\begin{equation}\label{Gneta}
 G(\eta_\e) = \pa_x \circ \mathfrak{P}^{-1} \circ {\mathfrak H} \circ
 \tanh\big((\tth+\ttf)|D| \big)
 \circ \mathfrak{P} \,,\quad\text{where}\quad D:=\frac{1}{\im} \,\pa_x. 
\end{equation}
where $ {\mathfrak H} $ is the Hilbert transform, i.e. the  Fourier multiplier operator
$$
 \mathfrak{H}(e^{\im j x}):= \frac{1}{\im}\, \textup{sign}(j) e^{\im j x} \, , 
 \quad  \forall j \in \Z \setminus \{0\} \, , 
 \quad \mathfrak{H}(1) := 0 \, . 
$$
The function $\mathfrak p(x)$ and the constant $\ttf $ are  determined as a fixed point  of 
(see  \cite[formula (A.15)]{BBHM})
\begin{equation} \label{def:ttf}
\mathfrak{p}  =  \frac{\mathfrak{H}}{\tanh \big((\tth + \ttf)|D| \big)}[\eta_\e ( x + \mathfrak{p}(x))] \, , 
 \qquad 
 \ttf:= \frac{1}{2\pi} \int_\T \eta_\e (x +  \mathfrak{p}(x)) \de x \, . 
  \end{equation}
  As proved in \cite{BMV3}, the map $\e \to (\mathfrak{p}, \ttf)$ is analytic as a map $B(\e_0) \to H^m_{\mathtt{odd}}(\T) \times \R$.
In addition, in  Appendix \ref{sec:App2 2} we prove the expansion
\begin{equation}
 \label{expfe}
 \begin{aligned}
&  \mathfrak p(x)  = \e \ch^{-2} \sin(x) +\e^2 \frac{\left(\ch^4+1\right)\,\left(b\,\ch^4-3\kappa -27 b+\ch^4 \kappa +\ch^4+3\right)}{8\ch^4\,r_{\tth,\kappa,b}}\sin(2x)+\cO(\e^3) \, ,\\  
&   \ttf =
   \e^2 \frac{\ckb^2(\ch^4-1)-2}{4\ch^2} +
   \cO(\e^3) \, .
   \end{aligned}
 \end{equation}
Under the symplectic and reversibility-preserving change of variables 
\begin{align}\label{LC}
    h=\mathcal{P}_\e\begin{bmatrix}
        u\\
        v
\end{bmatrix},~~\mathcal{P}_\e=\begin{bmatrix}
    (1+\mathfrak{p}_x)\mathfrak{P} & 0\\
    0 & \mathfrak{P}
\end{bmatrix} \ , 
\end{align}
one  transforms the system \eqref{second linear eq} into the linear system $h_t=L_\varepsilon h$ where $L_\varepsilon$ is the Hamiltonian and reversible real operator
\begin{equation} \label{mathcal L e}
\begin{aligned}
    L_\varepsilon:=& L_{\varepsilon}(\tth,\kappa,b) :=\mathcal{P}_\e \widetilde{\cL}_\e \mathcal{P}_\e^{-1} \\=& \begin{bmatrix}
\partial_x\circ(c_0+q_\varepsilon(x)) &  |D| \tanh((\tth + \ttf)|D|)\\
        -(1+a_\varepsilon(x))+\kappa\tau_{\e}-b\,\beta_\e & ~~(c_0+q_\varepsilon(x))\partial_x
    \end{bmatrix}\\
    =&\begin{bmatrix}
        0& \mathrm{Id}\\
        -\mathrm{Id}& 0
    \end{bmatrix}\underbrace{\begin{bmatrix}
 (1+a_\varepsilon(x))-\kappa\tau_{\e}+b\,\beta_\e& ~~-(c_0+q_\varepsilon(x))\partial_x\\
 \partial_x\circ(c_0+q_\varepsilon(x))& ~~|D| \tanh((\tth + \ttf)|D|)
    \end{bmatrix}}_{=: K_\e}, 
\end{aligned}    
\end{equation}
where the functions $q_\e(x)$ and $a_\e(x)$ are given by
\begin{equation}\label{def:pa}
c_0+q_\e(x) :=  \displaystyle{\frac{ c_\e-V(x+\mathfrak{p}(x))}{ 1+\mathfrak{p}_x(x)}} \, , \quad 1+a_\e(x):=   \displaystyle{\frac{1+ (V(x + \mathfrak{p}(x)) - c_\e)
 B_x(x + \mathfrak{p}(x))  }{1+\mathfrak{p}_x(x)}} \,, 
\end{equation}
and the operators $\beta_\e:=\mathfrak{P}\circ\mathrm dF(\eta_\e)\circ\mathfrak{P}^{-1}\circ\frac{1}{1+\mathfrak{p}_x}$ and $\tau_\e:=\mathfrak{P}\circ\mathrm d\sigma(\eta_\e)\circ\mathfrak{P}^{-1}\circ\frac{1}{1+\mathfrak{p}_x}$ are given by
\begin{equation} \label{def:Sigma g}
\begin{aligned}
     \beta_\e =&\frac{1}{1+\mathfrak{p}_x}\left(\widetilde{\pa_x}\right)^4+\e^2 \sum_{j=1}^4 \frac{F^{[j]}_2(x+\mathfrak{p}(x))}{1+\mathfrak{p}_x}\left(\widetilde{\pa_x}\right)^j+\cO(\e^3)\,,
\end{aligned}    
\end{equation}
and
\begin{equation} \label{tau operator}
\begin{aligned}
    \tau_\varepsilon &= \frac{1}{1+\mathfrak{p}_x} \left(\widetilde{\pa_x}\right)^2 + \varepsilon^2 \sum_{j=1}^2 \frac{\Sigma_2^{[j]}(x+\mathfrak{p}(x))}{1+\mathfrak{p}_x} \left(\widetilde{\pa_x}\right)^j + \mathcal{O}(\varepsilon^3)\,,
\end{aligned}
\end{equation}
where $\widetilde{\pa_x}:=\pa_x\circ\frac{1}{1+\mathfrak{p}_x}$. The operators $\beta_\varepsilon$ and $\tau_\varepsilon$ are also symmetric. To see this we apply a change of variable to obtain 
\[
\int_\T \mathfrak{P} f(y)  g(y) \de y = \int_\T f(z)  \mathfrak{P}^{-1}\left[\frac{g}{1 + \mathfrak{p}_x} \right](z) \de y.
\]
By the analyticity result of the map $\e \mapsto (V, B)$ given above,  the map
$\e \to (q_\e, a_\e)$ is analytic as a map $B(\e_0) \to H^{m - 1}_{\mathtt{ev}}(\T)\times H^{m - 2}_{\mathtt{ev}}(\T)$.
In Appendix \ref{sec:App2 2}, we provide their Taylor expansions, that we collect in the following lemmas.

\begin{lemma}\label{lem:pa.exp}
The analytic functions $q_\e (x) $ and $a_\e (x) $  in \eqref{def:pa} 
are even in $ x $, and
\begin{equation}\label{SN1}
q_\e (x)  
= \e q_1 (x) + \e^2 q_2 (x)  + \cO(\e^3) \, , \qquad   
a_\e (x)  
= \e a_1(x) +\e^2 a_2 (x) + \cO(\e^3) \, , 
\end{equation}
where
\begin{align}\label{pino1fd}
    & q_1(x)  =   q_1^{[1]} \cos(x)\, , \qquad \quad q_1^{[1]} := -2\ch^{-1}\ckb,\\
 \label{pino2fd}
    & q_2(x) =q_2^{[0]}+q_2^{[2]}\cos(2x),\, q_2^{[0]} := c_{2}-\frac{\left(\ch^4-3\right)\ckb}{2\ch^3},\, q_2^{[2]} :=\frac{\ckb \left(14 b \ch^4+2\ch^4\kappa-18 b-6\kappa-\ch^4-3\right)}{2\ch^3 \rhkb},
\end{align}
and 
\begin{align} 
& a_1(x) \label{aino1fd} = a_1^{[1]}\cos(x)\, , \qquad \qquad 
a_1^{[1]}:= -(\ch^{-2}+c_0^2)\, , \\
& a_2(x) \label{aino2fd} = a_2^{[0]}+a_2^{[2]}\cos(2x)\, ,\quad  \, a_2^{[0]}:=\frac{3\ckb^2\ch^4+1}{2\ch^4}, \\
\label{a[2]2}
& a_2^{[2]} :=\frac{6\ch^4\left(4 b-1\right)+(22\ch^4-72 b\ch^4-\ch^8-3)\ckb^2+(10\ch^4\left(\ch^4-3\right))\ckb^4}{4\ch^4\rhkb}.
\end{align}  
\end{lemma}

\begin{lemma} \label{tau}
    The symmetric operator $\tau_\varepsilon$, defined in \eqref{tau operator}, admits the expansion
 \begin{align} \label{taueps}
    \tau_\varepsilon = \partial_x^2 + \varepsilon \tau_1 + \varepsilon^2 \tau_2 + \mathcal{O}(\varepsilon^3),
\end{align}
where
\begin{align} \label{tauj}
     \tau_j :=& \tau_{j,0}(x) + \tau_{j,1}(x)\partial_x + \tau_{j,2}(x)\partial_x^2, \qquad j = 1, 2,
\end{align}
with coefficients given by
\begin{align} \label{tau10}
    \tau_{1,0} &= \tau_{1,0}^{[1]}\cos(x), \qquad \tau_{1,0}^{[1]} := \ch^{-2},  \\
    \tau_{1,1} &= \tau_{1,1}^{[1]}\sin(x), \qquad \tau_{1,1}^{[1]} := 3\ch^{-2}, \label{tau11} \\
    \tau_{1,2} &= \tau_{1,2}^{[1]}\cos(x), \qquad \tau_{1,2}^{[1]} := -3\ch^{-2}, \label{tau12}
\end{align}
and at second order:
\begin{align}
    &\tau_{2,0} = \tau_{2,0}^{[0]} + \tau_{2,0}^{[2]}\cos(2x), \qquad \tau_{2,0}^{[0]} := -\frac{1}{2}\ch^{-4}, \\
    &\tau_{2,0}^{[2]} := \frac{51 b+15 \kappa -59 b\ch^4+2 b\ch^8-11\ch^4 \kappa +2\ch^8\kappa +\ch^4+2\ch^8+6}{2\ch^4\rhkb}, \label{tau20} \\
    &\tau_{2,1} = \tau_{2,1}^{[2]}\sin(2x), \qquad \tau_{2,1}^{[2]} := \frac{99 b + 27\kappa - 45 b\ch^4 - 9\ch^4\kappa + 9}{2\ch^4\rhkb}, \label{tau21} \\
    &\tau_{2,2} = \tau_{2,2}^{[0]} + \tau_{2,2}^{[2]}\cos(2x), \qquad \tau_{2,2}^{[0]} := \frac{12-3\ch^4}{4\ch^4}, \qquad \tau_{2,2}^{[2]} := \frac{45 b\ch^4 - 27\kappa - 99 b + 9\ch^4\kappa - 9}{4\ch^4\rhkb}. \label{tau22}
\end{align}
\end{lemma}

\begin{lemma} \label{beta}
    The symmetric operator $\beta_\varepsilon$, defined in \eqref{def:Sigma g}, admits the expansion
 \begin{align} \label{betaeps}
\beta_\e=\pa^4_x+\e\beta_1+\e^2\beta_2+\cO(\e^3),
\end{align}
where
\begin{align} \label{betaj}
    \beta_j:=&b_{j,0}(x)+b_{j,1}(x)\pa_x+b_{j,2}(x)\pa^2_x+b_{j,3}(x)\pa^3_x+b_{j,4}(x)\pa^4_x,\qquad j=1,2,
\end{align}
with coefficients given by
\begin{align} \label{b110}
    &b_{1,0}=b^{[1]}_{1,0}\cos(x),\qquad b^{[1]}_{1,0}:=-\ch^{-2},\\ \label{b111}
    &b_{1,1}=b^{[1]}_{1,1}\sin(x),\qquad  b^{[1]}_{1,1}:=-5\ch^{-2},\\ \label{b112}
    &b_{1,2}=b^{[1]}_{1,2}\cos(x),\qquad  b^{[1]}_{1,2}:=10\ch^{-2},\\ \label{b113}
    &b_{1,3}=b^{[1]}_{1,3}\sin(x),\qquad  b^{[1]}_{1,3}:=10\ch^{-2},\\ \label{b114}
    &b_{1,4}=b^{[1]}_{1,4}\cos(x),\qquad  b^{[1]}_{1,4}:=-5\ch^{-2},
\end{align}
and at second order:
\begin{align}
&b_{2,0}:=b_{2,0}^{[0]}+b_{2,0}^{[2]}\cos(2x),\qquad b^{[0]}_{2,0}:=\frac{1}{2} \ch^{-4}, \\ 
    & b^{[2]}_{2,0}:=\frac{239 b \ch^4-69\kappa-249 b-8 b\ch^8+47\ch^4\kappa-8\ch^8\kappa-\ch^4-8\ch^8-24}{2\ch^4 \rhkb},\\
&b_{2,1}:=b^{[2]}_{2,1}\sin(2x),\qquad b^{[2]}_{2,1}:=\frac{385 b \ch^4-210\kappa - 810 b - 5 b\ch^8+85\ch^4\kappa-5\ch^8\kappa + 10\ch^4 - 5\ch^8 - 60}{2\ch^4 \rhkb},\\
&b_{2,2}:=b^{[0]}_{2,2}+b^{[2]}_{2,2}\cos(2x),\qquad b^{[0]}_{2,2}:=\frac{5\,\ch^4-30}{4\,\ch^4}\,, \\
&b^{[2]}_{2,2}:=\frac{2070b+510\kappa-725 b \ch^4+5 b \ch^8-185\ch^4\kappa+5\ch^8\kappa-50\ch^4+5\ch^8+120}{4\ch^4\rhkb},\\
&b_{2,3}:=b^{[2]}_{2,3}\sin(2x),\qquad b^{[2]}_{2,3}:=\frac{315 b+75\kappa-85 b\ch^4-25\ch^4\kappa-10 \ch^4+15}{\ch^4\rhkb},\\ \label{b224}
&b_{2,4}:=b^{[0]}_{2,4}+b^{[2]}_{2,4}\cos(2x),\quad b^{[0]}_{2,4}:=-\frac{5\,\ch^4-30}{4\,\ch^4},\quad b^{[2]}_{2,4}:=\frac{85 b\ch^4-75\kappa-315 b+25\ch^4\kappa+10\ch^4-15}{4\ch^4\rhkb}.
\end{align}
\end{lemma}

\paragraph{Bloch--Floquet expansions.} 

In the following we regard $L_\e$ as an operator on $L^2(\R,\mathbb{C}^2)$ instead of on $L^2(\T,\mathbb{C}^2)$. To be more precise, we interpret the $D$ as the multiplier $\xi$ on the Fourier transform, namely the operator $m(D)$ is given by the following formula
\[
m(D) u := (m(\xi) \widehat{u}(\xi))^\vee,
\]
where $\widehat{u}$ denotes the Fourier transform of $u$ and $\vee$ denotes the inverse Fourier transform. The following Bloch transform gives an isometry between $L^2(\R,\mathbb{C}^2)$ and $L^2(Q,L^2(\T,\mathbb{C}^2)) \cong L^2(Q) \otimes L^2(\T,\mathbb{C}^2)$:
\begin{align}
    \label{Blocht}
    \mathcal{U} f(\mu,x) = \sqrt{2\pi} \sum_{k \in \Z} f(x + 2\pi k) e^{-\im \mu (2 \pi k + x)} = \frac{1}{\sqrt{2\pi}} \sum_{k \in \Z} \widehat{f}(k + \mu) e^{\im kx},
\end{align}
where $\mu \in Q := \left[-\frac{1}{2},\frac{1}{2}\right)$, $x \in \T$. Here the second equality in \eqref{Blocht} is given by Poisson summation formula first by smooth functions with compact support then by density (see the isometry below). In fact, we have 
\[
\int_Q \int_\T |\mathcal{U} f(\mu,x)|^2 \de \mu \de x = \frac{1}{2\pi} \int_Q \sum_{k \in \Z} |\widehat{f}(k + \mu)|^2 \de \mu  = \frac{1}{2\pi} \int_\R |\widehat{f}(\mu)|^2 \de \mu = \int_\R |f(x)|^2 \de x.
\]

Now for fixed $\mu \in Q$, we define the following operator $L_{\mu,\e}:=e^{-\im\mu x}L_\e e^{\im\mu x}$, namely
\begin{equation} \label{mathcal L mu e}
\begin{aligned}
    L_{\mu,\e}:=&\begin{bmatrix}
(\partial_x+\im\mu)\circ(c_0+q_\varepsilon(x)) &  |D+\mu| \tanh((\tth + \ttf)|D+\mu|)\\
        -(1+a_\varepsilon(x))+\kappa\tau_{\mu,\varepsilon}-b\beta_{\mu,\e} & ~~(c_0+q_\varepsilon(x))(\partial_x+\im\mu)
    \end{bmatrix}\\
    =&\underbrace{\begin{bmatrix}
        0& \operatorname{Id}\\
        -\operatorname{Id}& 0
    \end{bmatrix}}_{=:\mathcal{J}}\underbrace{\begin{bmatrix}
 (1+a_\varepsilon(x))-\kappa\tau_{\mu,\e}+b\beta_{\mu,\e}& ~~-(c_0+q_\varepsilon(x))(\partial_x+\im\mu)\\
 (\partial_x+\im\mu)\circ(c_0+q_\varepsilon(x))& ~~|D+\mu| \tanh((\tth + \ttf)|D+\mu|)
    \end{bmatrix}}_{=:\mathfrak{K}_{\mu,\e}},
\end{aligned}    
\end{equation}
where $\beta_{\mu,\e}:=e^{-\im\mu x}\beta_\e e^{\im\mu x}$ and $\tau_{\mu,\e}:=e^{-\im\mu x}\tau_\e e^{\im\mu x}$ are given by the symmetric operator (since $\beta_\e, \tau_\e$ are symmetric)
\begin{equation} \label{Sigma}
\begin{aligned}
    \beta_{\mu,\e}&=(\pa_x+\im \mu)^4+\e \sum_{j=0}^4 b_{1,j}(x)\left(\pa_x+\im\mu\right)^{j}+\e^2 \sum_{j=0}^4 b_{2,j}(x)\left(\pa_x+\im\mu\right)^{j}+\cO(\e^3)\,,\\
    \tau_{\mu,\e}&=(\pa_x+\im \mu)^2+\e \sum_{j=0}^2 \tau_{1,j}(x)\left(\pa_x+\im\mu\right)^{j}+\e^2 \sum_{j=0}^2 \tau_{2,j}(x)\left(\pa_x+\im\mu\right)^{j}+\cO(\e^3).
\end{aligned}
\end{equation}

In fact, if we regard $L_{\mu,\e}$ as an operator with domain $Y:=H^4(\mathbb{T},\mathbb{C})\times H^1(\mathbb{T},\mathbb{C})$ in the space $X:=L^2(\mathbb{T},\mathbb{C})\times L^2(\mathbb{T},\mathbb{C})$, equipped with the complex scalar product
\begin{align} \label{complex product}
    (f,g):=\frac{1}{2\pi}\int_0^{2\pi} \left(f_1\overline{g_1}+f_2\overline{g_2} \right)\,\mathrm{d}x,~~\forall~f=\vet{f_1}{f_2},~~g=\vet{g_1}{g_2}\in L^2(\mathbb{T},\mathbb{C}^2),
\end{align}
one may verify that $\mathfrak{K}_{\mu,\e}$ is self-adjoint. By definition of \eqref{Blocht} we can check that (noticing that the operator $L_\e$ in \eqref{mathcal L e} has $2\pi$-periodic coefficients and if $A=Op(a)$ is a pseudo-differential operator with symbol $a(x,\xi)$, which is $2\pi$-periodic in $x$, then $A_\mu:=e^{-\im\mu x}A e^{\im\mu x}=Op(a(x,\xi+\mu))$)
\[
\mathcal{U}(K_\e f)(\mu, \cdot) = \mathfrak{K}_{\mu,\e} \, \mathcal{U}f(\mu, \cdot).
\]
In other words, if we define $\mathfrak{K}_{\e}$ on $L^2(Q,L^2(\T,\mathbb{C}^2))$ by
\[
(\mathfrak{K}_{\e}f) (\mu)  = \mathfrak{K}_{\mu,\e} (f(\mu)),
\]
we have $K_\e = \mathcal{U}^{-1} \mathfrak{K}_{\e} \mathcal{U}$. Now we can apply the Bloch--Floquet theory  \cite[(d) of Theorem XIII.85]{RS78}. Note that our family is continuous w.r.t. parameter $\mu$, we actually get
\[
\sigma_{L^2(\R,\mathbb{C}^2)}(K_\e) = \bigcup_{\mu\in[-\frac{1}{2},\frac{1}{2})} \sigma_{L^2(\mathbb{T},\mathbb{C}^2)}(\mathfrak{K}_{\mu,\e}),
\]
which in turn implies 
\begin{align*}
    \sigma_{L^2(\mathbb{R},\mathbb{C}^2)}(L_\e)=\bigcup_{\mu\in[-\frac{1}{2},\frac{1}{2})} \sigma_{L^2(\mathbb{T},\mathbb{C}^2)}(L_{\mu,\e}).
\end{align*}



We pause to make some remarks that: (i) If $A$ is a real operator and  $A_\mu:=e^{-\im\mu x}A e^{\im\mu x}$, then $\overline{A_\mu}=A_{-\mu}$. As a consequence the spectrum $\sigma(A_{-\mu})=\overline{\sigma(A_\mu)}$ and we can study $\sigma(A_\mu)$ just for $\mu\ge0$. (ii) $\sigma(A_\mu)$ is further a $1$-periodic set with respect to $\mu$, so one can restrict to $\mu\in[0,\frac{1}{2})$.

The complex operator $L_{\mu,\e}$ in \eqref{mathcal L mu e} is complex Hamiltonian and reversible. Recall that if $L : Y \to X$ is a complex linear operator, we say that it is 
\begin{itemize}
    \item \textbf{Complex Hamiltonian:} if there exists a self-adjoint operator, namely $\mathcal{K}=\mathcal{K}^*$, where $\mathcal{K}^*$ (with domain $Y$) is the adjoint with respect to the complex scalar product \eqref{complex product} such that $L = \mathcal{J} \mathcal{K}$.

    \item \textbf{Reversible:} if 
    \begin{align} \label{reversible}
    L \circ \overline{\rho} = -\overline{\rho} \circ L, \quad \text{where} \quad 
    \overline{\rho} \begin{bmatrix} \eta(x) \\ \psi(x) \end{bmatrix} := 
    \begin{bmatrix} \overline{\eta}(-x) \\ -\overline{\psi}(-x) \end{bmatrix} \ . 
    \end{align}
\end{itemize}

The property \eqref{reversible} for $L_{\mu,\e}$ follows because $L_\e$ is a real operator which is reversible with respect to the involution $\rho$ in \eqref{rho involution}. Equivalently, since $\mathcal{J}\circ \overline{\rho}=-\overline{\rho}\circ \mathcal{J}$, the self-adjoint operator $\mathfrak{K}_{\mu,\e}$ is reversibility-preserving, $\mathrm{i.e.}$
\begin{align}\label{B rho=rho B}
    \mathfrak{K}_{\mu,\e}\circ\overline{\rho}=\overline{\rho}\circ \mathfrak{K}_{\mu,\e}.
\end{align}
In addition $(\mu,\e)\rightarrow L_{\mu,\e}\in \mathcal{L}( Y,X)$ is analytic, since the functions $\e \mapsto a_\e, ~q_\e$ and $\beta_\e$ defined in \eqref{SN1}, \eqref{betaeps} are analytic and $ L_{\mu,\varepsilon}$ is analytic with respect to $\mu$, since, for any $\mu\in[-\frac{1}{2},\frac{1}{2})$,
\begin{align} \label{DtanhD}
    |D+\mu|\tanh\Big((\tth+\ttf)|D+\mu|\Big)=(D+\mu)\tanh\Big((\tth+\ttf)(D+\mu)\Big).
\end{align}
Recall also the identity \cite[Section 5.1]{NS}
\begin{align} \label{D+mu}
    |D+\mu|=|D|+\mu(\mathrm{sgn}(D)+\Pi_0),~~{\forall \mu\in\left[0,\frac{1}{2}\right)},
\end{align}
where $\mathrm{sgn}(D)$ is the Fourier multiplier operator, acting on $2\pi$-periodic functions, with symbol
\begin{equation} \label{sgn D}
\mathrm{sgn}(k):=1 \ \ \forall k >0 \ , \ \
\mathrm{sgn}(0):=0  , \ \ \ 
\mathrm{sgn}(k):= -1 \ \ \forall k < 0 \  , 
\end{equation}
and $\Pi_0$ is the projector operator on the zero mode, $\Pi_0 f(x):=\frac{1}{2\pi}\int_\mathbb{T} f(x) dx$.

\smallskip 

Our goal is to prove the existence of eigenvalues of $ L_{\mu,\e}$ in \eqref{mathcal L mu e} with nonzero real part. We remark that the Hamiltonian structure of $ L_{\mu,\e}$ implies that eigenvalues with nonzero real part may arise only from multiple eigenvalues of $ L_{\mu,0}$ (``Krein criterion''), because if $\lambda$ is an eigenvalue of $ L_{\mu,\e}$ then also $-\overline{\lambda}$ is, and the total algebraic multiplicity of the eigenvalues is conserved under small perturbation. We now describe the spectrum of $ L_{\mu,0}$.

\paragraph{The spectrum of $ L_{\mu,0}$.} The spectrum of the Fourier multiplier matrix operator 
\begin{equation} \label{mathcal L mu 0}
\begin{aligned}
 L_{\mu,0}:=&\begin{bmatrix}
c_0(\partial_x+\im\mu)&  |D+\mu| \tanh(\tth|D+\mu|)\\
        -1+\kappa (\pa_x+\im\mu)^2-b (\pa_x + \im \mu)^4 & ~~c_0(\partial_x+\im\mu)
    \end{bmatrix}
\end{aligned}    
\end{equation}
consists of the purely imaginary eigenvalues $\{\lambda^{\pm}_k(\mu) : \,~k\in\mathbb{Z}\}$, where
\begin{align} \label{eigenvalues of Lmu0}
    \lambda^{\pm}_k(\mu):=\im\left(c_0(\pm k+\mu)\mp \sqrt{\left(1+\kappa(k\pm\mu)^2+b(k\pm\mu)^4\right)\, |k\pm \mu|\tanh(\tth| k\pm \mu|)}\right) \ . 
\end{align}
For $(\tth,\kappa,b) \not\in \mathfrak{R}$ (cf. \eqref{def:fR}), at $\mu=0$ the real operator $ L_{0,0}$ possesses the eigenvalue $0$ with algebraic multiplicity four, 
\begin{align*}
    \lambda^+_0(0)=\lambda^-_0(0)=\lambda^+_1(0)=\lambda^-_1(0)=0,
\end{align*}
and geometric multiplicity $3$. A real basis of the kernel of $ L_{0,0}$ is 
\begin{align} \label{eigenfunc of mathcall L00}
    \phi^+_1:=\vet{(\ch/\ckb)^{\frac{1}{2}}\cos(x)}{(\ch/\ckb)^{-\frac{1}{2}}\sin(x)}, ~~\phi^-_1:=\vet{-(\ch/\ckb)^{\frac{1}{2}}\sin(x)}{(\ch/\ckb)^{-\frac{1}{2}}\cos(x)}, ~~\phi^-_0:=\vet{0}{1},
\end{align}
together with the generalized eigenvector
\begin{align} \label{eigenfunc f+0}
    \phi^+_0:=\vet{1}{0}, ~~ L_{0,0} \phi^+_0=-\phi^-_0.
\end{align}
Furthermore $0$ is an isolated eigenvalue for $ L_{0,0}$, namely the spectrum $\sigma( L_{0,0})$ decomposes in two separated parts,
\begin{align} \label{sigma decomposition}
    \sigma( L_{0,0})=\sigma'( L_{0,0})\cup\sigma''( L_{0,0}), ~~\mbox{where}~~\sigma'( L_{0,0}):=\{0\},
\end{align}
and $\sigma''( L_{0,0}):=\{\lambda^{\sigma}_k(0):~k\neq 0,1, ~\sigma=\pm\}$.

\begin{lemma}[Generalized kernel of $L_{0,\varepsilon}$] \label{237}
Let $(\tth,\kappa,b)\in
\big(\mathbb{R}_{>0}\times\mathbb{R}_{\geq 0}\times\mathbb{R}_{>0}\big)
\setminus \mathfrak R$.
For $|\varepsilon|$ sufficiently small, the eigenvalue $0$ of
$L_{0,\varepsilon}$ has algebraic multiplicity four. More precisely,
recall $Z$ and $\mathcal{P}_\e$ defined in \eqref{alinhac} and \eqref{LC} respectively, and define
\[
U_1 :=
\begin{bmatrix}
0\\
1
\end{bmatrix},
\qquad
U_2 :=
\varepsilon^{-1}\mathcal P_\varepsilon Z^{-1}
\left.
\begin{bmatrix}
\partial_x\eta_{\varepsilon,p}\\
\partial_x\psi_{\varepsilon,p}
\end{bmatrix}
\right|_{p=0},
\]
\[
U_3 :=
\mathcal P_\varepsilon Z^{-1}
\left.
\begin{bmatrix}
\partial_\varepsilon\eta_{\varepsilon,p}\\
\partial_\varepsilon\psi_{\varepsilon,p}
\end{bmatrix}
\right|_{p=0},
\qquad
U_4 :=
\mathcal P_\varepsilon Z^{-1}
\left.
\begin{bmatrix}
\partial_p\eta_{\varepsilon,p}\\
\partial_p\psi_{\varepsilon,p}
\end{bmatrix}
\right|_{p=0}.
\]
Then
\[
L_{0,\varepsilon}U_1=0,
\qquad
L_{0,\varepsilon}\widetilde U_2=0,
\]
and
\[
L_{0,\varepsilon}U_3
=
-\varepsilon
\left.\partial_\varepsilon c_{\varepsilon,p}\right|_{p=0}
\widetilde U_2,
\qquad
L_{0,\varepsilon}U_4
=
-U_1
-
\varepsilon
\left.\partial_p c_{\varepsilon,p}\right|_{p=0}
\widetilde U_2.
\]
Also, the four vectors $\{U_1, U_2,U_3,U_4\}$ form a basis of $\mathcal{V}_{0,\e}$ and $L_{0,\e}^{2}=0$ on $\mathcal{V}_{0,\e}$.
\end{lemma}
\begin{proof}
   The proof is analogous to that of \cite[Lemma 4.1]{NS}, but uses the expansion of
$(\eta_{\e,p}(x),\psi_{\e,p}(x),c_{\e,p})$ given in Appendix~\ref{sec:App2}.
\end{proof}

By Kato's perturbation theory for any $\mu,\e\neq 0$ sufficiently small, the perturbed spectrum $\sigma( L_{\mu,\e})$ admits a disjoint decomposition as 
\begin{align} \label{disjoint decomposition of spectrum}
    \sigma( L_{\mu,\e})=\sigma'( L_{\mu,\e})\cup \sigma''( L_{\mu,\e}),
\end{align}
where $\sigma'( L_{\mu,\e})$ consists of $4$ eigenvalues close to $0$. We denote by $\mathcal{V}_{\mu,\e}$ the spectral subspace associated with $\sigma'( L_{\mu,\e})$, which has dimension $4$ and it is invariant by $ L_{\mu,\e}$. Our goal is to prove that, for $\e$ small, for values of the Floquet exponent $\mu$ in an interval of order $\e$, the $4\times 4$ matrix which represents the operator $ L_{\mu,\e}:\mathcal{V}_{\mu,\e}\rightarrow \mathcal{V}_{\mu,\e}$ possesses a pair of eigenvalues close to zero with opposite nonzero real parts. 

Before stating our main result, let us introduce a notation that we shall use throughout the paper.

\begin{itemize}
\item
\textbf{Notation:} we denote by $\mathcal{O}(\mu^{m_1} \varepsilon^{n_1}, \dots, \mu^{m_p} \varepsilon^{n_p})$, $m_j, n_j \in \mathbb{N}$ (for us $\mathbb{N} := \{0,1,2, \dots \}$), analytic functions of $(\mu, \varepsilon)$ with values in a Banach space $X$ which satisfy, for some $C > 0$ uniform for $\tth$ in any compact set of $(0,+\infty)$, the bound
\[
\left\|
\mathcal O(
\mu^{m_1}\varepsilon^{n_1},
\dots,
\mu^{m_p}\varepsilon^{n_p}
)
\right\|_X \leq C \sum_{j=1}^{p} |\mu|^{m_j} |\varepsilon|^{n_j}
\]
for small values of $(\mu, \varepsilon)$. Similarly we denote $r_k(\mu^{m_1} \varepsilon^{n_1}, \dots, \mu^{m_p} \varepsilon^{n_p})$ scalar functions $\mathcal{O}(\mu^{m_1} \varepsilon^{n_1}, \dots, \mu^{m_p} \varepsilon^{n_p})$ which are also \textit{real} analytic.
\end{itemize}

Our main spectral result is the following:

\begin{theorem}[Complete Benjamin--Feir spectrum] \label{Complete BF thm}
Let $(\tth,\kappa,b) \in (\R_{>0}\times\R_{\geq 0} \times \R_{> 0}) \setminus (\fR \cup \mathfrak{D})$, where $\fR$ and $\mathfrak{D}$ are defined in \eqref{def:fR} and \eqref{def:fD} respectively. There exist $\e_{0}, \mu_{0} > 0$, such that, for any $0 < \mu < \mu_{0}$ and $0 \leq \e < \e_{0}$, the operator $L_{\mu,\e} : \mathcal{V}_{\mu,\e} \to \mathcal{V}_{\mu,\e}$ can be represented by a $4 \times 4$ matrix of the form
\begin{align} \label{U S diag}
\begin{pmatrix}
U & 0 \\
0 & S
\end{pmatrix},
\end{align}
where $U$ and $S$ are $2 \times 2$ matrices, with identical diagonal entries each, of the form
\begin{equation} \label{U S}
\begin{aligned}
U &= 
\begin{pmatrix}
\im  \frac{1}{2} \breve{\mathtt{c}}_{\tth,\kappa,b}\mu + \im r_{2}(\mu \e^{2}, \mu^{2}\e, \mu^{3})
& -\mathsf{e}_{22}\tfrac{\mu}{8}(1 + r_{5}(\e,\mu)) \\
- \mu \e^2 \mathsf{e}_{\mathrm{WB}} + r_{1}'(\mu \e^{3}, \mu^{2}\e^2) + \mathsf{e}_{22}\tfrac{\mu^{3}}{8}(1 + r_{1}''(\e,\mu)) 
& \im  \frac{1}{2} \breve{\mathtt{c}}_{\tth,\kappa,b}\mu + \im r_{2}(\mu \e^{2}, \mu^{2}\e, \mu^{3})
\end{pmatrix},\\
S &= 
\begin{pmatrix}
\im c_0 \mu + \im r_{9}(\mu \e^{2}, \mu^{2}\e) & \tanh(\tth\mu) + r_{10}(\mu \e) \\[1ex]
- \mu + r_{8}(\mu \e^{2}, \mu^{3} ) & \im c_0 \mu + \im r_{9}(\mu \e^{2}, \mu^{2}\e)
\end{pmatrix},
\end{aligned}
\end{equation}
where $\mathsf{e}_{\mathrm{WB}}, \mathsf{e}_{12}, \mathsf{e}_{22}$ are defined in \eqref{eWB} and \eqref{e11 f11}.
The eigenvalues of $U$ have the form
\begin{align} \label{lambda 1}
\lambda_{1}^{\pm}(\mu,\e) 
= \im \frac{1}{2} \breve{\mathtt{c}}_{\tth,\kappa,b}\mu + \im r_{2}(\mu \e^{2}, \mu^{2}\e, \mu^{3}) 
\pm \frac{1}{8} \mu \sqrt{1+r_{5}(\e,\mu)} \, \sqrt{\Delta_{\mathrm{BF}}(\tth,\kappa,b ;\mu,\e)},
\end{align}
where $\breve{\mathtt{c}}_{\tth,\kappa,b} := 2 c_0 - \mathsf{e}_{12}(\tth,\kappa,b)$ and $\Delta_{\mathrm{BF}}(\tth,\kappa,b;\mu,\e)$ is the Benjamin--Feir discriminant function
\begin{align}\label{BFDF}
    \Delta_{\mathrm{BF}}(\tth,\kappa,b;\mu,\e):=8\e^2 \mathsf{e}_{22}(\tth,\kappa,b)\mathsf{e}_{\mathrm{WB}}(\tth,\kappa,b)+r_1(\e^3,\mu\e^2)-\mathsf{e}^2_{22}(\tth,\kappa,b)\mu^2(1+r_1''(\e,\mu)).
\end{align}
The eigenvalues  in \eqref{lambda 1} have nonzero real part if and only if $\Delta_{\mathrm{BF}}(\tth,\kappa,b;\mu,\e)>0$.

The eigenvalues of the matrix $S$ are a pair of purely imaginary eigenvalues of the form
\begin{align} \label{lambda zero}
\lambda_{0}^{\pm}(\mu,\e) = \im c_0 \mu (1 + r_{9}(\e^{2}, \mu \e)) \mp \im \sqrt{\mu \tanh(\tth \mu)} (1 + r(\mu^2,\e)).
\end{align}
For $\e = 0$ the eigenvalues $\lambda_{1}^{\pm}(\mu,0), \lambda_{0}^{\pm}(\mu,0)$ coincide with those in \eqref{eigenvalues of Lmu0}.
\end{theorem}

\begin{remark}
At $\varepsilon = 0$, the eigenvalues in \eqref{lambda 1} have the Taylor expansion
\begin{align*}
\lambda_{1}^{\pm}(\mu,0) 
= \im \frac{1}{2} \breve{\mathtt{c}}_{\tth,\kappa,b}\mu 
\pm \im \frac{1}{8}|\mathsf{e}_{22}(\tth,\kappa,b)| \mu^{2} + \mathcal{O}(\mu^{3}),
\end{align*}
which coincides with that of $\lambda_{1}^{\pm}(\mu)$ in \eqref{eigenvalues of Lmu0}, in view of the coefficients $\mathsf{e}_{12}(\tth,\kappa,b)$ and $\mathsf{e}_{22}(\tth,\kappa,b)$ defined in \eqref{e11 f11} respectively.
\end{remark}

\section{Perturbative Approach to the Separated Eigenvalues} \label{sec3 PA}

In this section, we analyze the splitting of the eigenvalues of $ L_{\mu, \varepsilon}$ close to 0 for small values of $\mu$ and $\varepsilon$, using Kato’s similarity transformation theory \cite[I-§4-6, II-§4]{Kato1966} and \cite{BMV1, BMV3, BMV_ed}. To this end, it is convenient to rewrite the operator $ L_{\mu, \varepsilon}$ in \eqref{mathcal L mu e} as
\begin{equation} \label{mathcal L= ichmu+mathscr L}
 L_{\mu, \varepsilon} = \im c_0 \mu + \mathscr{L}_{\mu, \varepsilon}, \quad \mu > 0,
\end{equation}
where, using also \eqref{D+mu}, $\mathscr{L}_{\mu, \varepsilon}$ is the Hamiltonian operator
\begin{equation} \label{mathscr L mu e}
\mathscr{L}_{\mu, \varepsilon} = \mathcal{J} \mathcal{K}_{\mu, \varepsilon},
\end{equation}
with $\mathcal{K}_{\mu, \varepsilon}$ the self-adjoint operator
\begin{equation} \label{mathcal B mu e}
\mathcal{K}_{\mu, \varepsilon} := \begin{bmatrix} 1 + a_{\varepsilon}(x)-\kappa\tau_{\mu,\e}+b\beta_{\mu,\e}& -(c_0 + q_{\varepsilon}(x)) \partial_x - \im \mu q_{\varepsilon}(x) \\ \partial_x \circ(c_0 + q_{\varepsilon}(x)) + \im \mu q_{\varepsilon}(x) & |D + \mu| \tanh ((\tth + \tf_\varepsilon)|D + \mu|) \end{bmatrix} \ ,
\quad \beta_{\mu,\e}, \tau_{\mu,\e} \mbox{ in } \eqref{Sigma} \ . 
\end{equation}
In addition  $\mathscr{L}_{\mu, \varepsilon}$   is also complex-reversible, namely it satisfies, by \eqref{reversible},
\begin{equation} \label{mathscr rho=-rho mathscr}
\mathscr{L}_{\mu, \varepsilon} \circ \bar{\rho} = - \bar{\rho} \circ \mathscr{L}_{\mu, \varepsilon},
\end{equation}
whereas $\mathcal{K}_{\mu, \varepsilon}$ is reversibility-preserving, i.e. fulfills \eqref{B rho=rho B}. Note also that $\mathcal{K}_{0, \varepsilon}$ is a real operator.

The scalar operator $\im c_0 \mu \equiv \im c_0 \mu \,\operatorname{Id}$ just translates the spectrum of $ L_{\mu, \varepsilon}$ along the imaginary axis of the quantity $\im c_0 \mu$, that is, in view of \eqref{mathcal L= ichmu+mathscr L},
\begin{equation}
\sigma( L_{\mu, \varepsilon}) = \im c_0 \mu + \sigma(\mathscr{L}_{\mu, \varepsilon}).
\end{equation}
Thus in the sequel we focus on studying the spectrum of $\mathscr{L}_{\mu, \varepsilon}$.

Note also that $\mathscr{L}_{0, \varepsilon} =  L_{0, \varepsilon}$ for any $\varepsilon \geq 0$. In particular $\mathscr{L}_{0,0}$ has zero as an isolated eigenvalue with algebraic multiplicity 4, geometric multiplicity 3 and generalized kernel spanned by the vectors $\{\phi_1^{+}, \phi_1^{-}, \phi_0^{+}, \phi_0^{-}\}$ in \eqref{eigenfunc of mathcall L00}, \eqref{eigenfunc f+0}; furthermore, its spectrum is separated as in \eqref{sigma decomposition}. For any $\varepsilon \neq 0$ small, $\mathscr{L}_{0, \varepsilon}$ has zero as an isolated eigenvalue with geometric multiplicity 2, and two generalized eigenvectors satisfying Lemma~\ref{237}.

We remark that the operator $\mathscr{L}_{\mu, \varepsilon}$ is analytic with respect to $\mu$. To be more precise, the most difficult part is to check the analyticity of the map 
\[
(\mu, \varepsilon) \mapsto |D + \mu| \tanh ((\tth + \tf_\varepsilon)|D + \mu|)
\]
in the space $\mathcal{L}(H^1(\T,\mathbb{C}), L^2(\T,\mathbb{C}))$.
In fact, since $\mu$ is real, we can write $|D + \mu| \tanh ((\tth + \tf_\varepsilon)|D + \mu|)$ as $(D + \mu) \tanh ((\tth + \tf_\varepsilon)(D + \mu))$. As a consequence, since the part $D + \mu$ is clearly an analytic map in $\mathcal{L}(H^1(\T,\mathbb{C}), L^2(\T,\mathbb{C}))$, it is sufficient to prove the map
\[
(\mu, \varepsilon) \mapsto \tanh ((\tth + \tf_\varepsilon)(D + \mu))
\]
is analytic in $\mathcal{L}(L^2(\T,\mathbb{C}))$. Since there is a natural embedding from $\ell^\infty(\Z) \to \mathcal{L}(L^2(\T,\mathbb{C}))$, we reduce to prove the following map 
\[
(\mu, \varepsilon) \mapsto (\tanh ((\tth + \tf_\varepsilon)(k + \mu)))_{k \in \Z}
\]
is analytic in $\ell^\infty(\Z)$. Recall that $\tf_\varepsilon$ is analytic in $\varepsilon$ in $B(\varepsilon_0)$. We can extend it (by power series) as a holomorphic function on $B(2\varepsilon_0')$ (by abuse of notation we still use the same notation $B$ to denote the balls in $\mathbb{C}$). Note that the function $\tanh(z)$ is uniformly bounded on the region where the real part of $z$ is lower bounded. So there exists a $\mu_0' > 0$ (shrinking $\varepsilon$ is necessary) such that when $|\mathrm{Im} (\mu)| < 2 \mu_0'$ and $\varepsilon \in B(2\varepsilon_0')$, we have 
\[
\sup_{k \in \Z}|\tanh ((\tth + \tf_\varepsilon)(k + \mu))| \le M.
\]
Now fix a point $(\mu_*,\varepsilon_*)$ such that $|\mathrm{Im} (\mu_*)| <  \mu_0'$ and $\varepsilon_* \in B(\varepsilon_0')$, by Cauchy's formula we obtain
\[
\tanh ((\tth + \tf_\varepsilon)(k + \mu)) = \sum_{i = 0}^\infty \sum_{j = 0}^\infty a_{ij}^k (\mu - \mu_*)^i (\varepsilon - \varepsilon_*)^j
\]
with the bound 
\[
|a_{ij}^k| \le \frac{M}{(\mu_0')^{i}(\varepsilon_0')^{j}}, \qquad \forall i,j \ge 0, k \in \Z.
\]
Since the estimates above are independent of $k$, we know the power series converges also in $\ell^\infty(\Z)$ for $|\mu - \mu_*|< \mu_0'$ and $|\varepsilon - \varepsilon_*| < \varepsilon_0'$.

The operator $\mathscr{L}_{\mu, \varepsilon}: Y \subset X \to X$ has domain $Y := H^4(\mathbb{T},\mathbb{C})\times H^1(\mathbb{T},\mathbb{C})$ and range $X := L^2(\mathbb{T},\mathbb{C})\times L^2(\mathbb{T},\mathbb{C})$.

\begin{lemma} \label{kato thm}
Let $\Gamma$ be a closed, counterclockwise-oriented curve around $0$ in the complex plane separating $\sigma' (\mathscr{L}_{0,0}) = \{0\}$ and the other part of the spectrum $\sigma'' (\mathscr{L}_{0,0})$ in \eqref{sigma decomposition}. There exist $\mu_0,\,\varepsilon_0,  > 0$ such that for any $(\mu, \varepsilon) \in B(\mu_0) \times B(\varepsilon_0)$ the following statements hold:

\begin{enumerate}
    \item \textit{The curve $\Gamma$ belongs to the resolvent set of the operator $\mathscr{L}_{\mu,\varepsilon} : Y \subset X \to X$ defined in \eqref{mathscr L mu e}.}
    \item \textit{The operators}
    \begin{equation} \label{Projection P mu e}
        P_{\mu,\varepsilon} := - \frac{1}{2\pi \im} \oint_\Gamma (\mathscr{L}_{\mu,\varepsilon} - \lambda)^{-1} \ \de\lambda : X \to Y
    \end{equation}
    \textit{are well-defined projectors commuting with $\mathscr{L}_{\mu,\varepsilon}$, i.e., $P_{\mu,\varepsilon}^2 = P_{\mu,\varepsilon}$ and $P_{\mu,\varepsilon} \mathscr{L}_{\mu,\varepsilon} = \mathscr{L}_{\mu,\varepsilon} P_{\mu,\varepsilon}$. The map $(\mu, \varepsilon) \mapsto P_{\mu,\varepsilon}$ is analytic from $B(\mu_0) \times B(\varepsilon_0)$ to $ \mathcal{L}(X,Y)$.}
    \item \textit{The domain $Y$ of the operator $\mathscr{L}_{\mu,\varepsilon}$ decomposes as the direct sum}
    \begin{equation} \label{Y=V+ker P}
        Y = \mathcal{V}_{\mu,\varepsilon} \oplus \ker(P_{\mu,\varepsilon}), \quad \mathcal{V}_{\mu,\varepsilon} := \operatorname{Rg}(P_{\mu,\varepsilon}) = \ker(\operatorname{Id} - P_{\mu,\varepsilon}),
    \end{equation}
    \textit{of closed invariant subspaces, namely $\mathscr{L}_{\mu,\varepsilon} : \mathcal{V}_{\mu,\varepsilon} \to \mathcal{V}_{\mu,\varepsilon}$, $\mathscr{L}_{\mu,\varepsilon} : \ker(P_{\mu,\varepsilon}) \to \ker(P_{\mu,\varepsilon})$. Moreover}
    \begin{equation} \label{spectrum separated by Gamma}
    \begin{aligned}
        \sigma(\mathscr{L}_{\mu,\varepsilon}) \cap \{z \in \mathbb{C}: z  \text{ is inside } \Gamma\} &= \sigma(\mathscr{L}_{\mu,\varepsilon} |_{\mathcal{V}_{\mu,\varepsilon}}) = \sigma'(\mathscr{L}_{\mu,\varepsilon}), \\
        \sigma(\mathscr{L}_{\mu,\varepsilon}) \cap \{z \in \mathbb{C}: z \text{ is outside } \Gamma\} &= \sigma(\mathscr{L}_{\mu,\varepsilon} |_{\ker(P_{\mu,\varepsilon})}) = \sigma''(\mathscr{L}_{\mu,\varepsilon}).
    \end{aligned}
    \end{equation}
    \item \textit{The projectors $P_{\mu,\varepsilon}$ are similar to each other; the transformation operators}
    \begin{equation} \label{U transformation operators}
        U_{\mu,\varepsilon} := (\operatorname{Id} - (P_{\mu,\varepsilon} - P_{0,0})^2)^{-\frac{1}{2}} \big[P_{\mu,\varepsilon} P_{0,0} + (\operatorname{Id} - P_{\mu,\varepsilon})(\operatorname{Id} - P_{0,0}) \big]
    \end{equation}
    \textit{are bounded and invertible in $Y$ and in $X$, with inverse}
    \begin{equation} \label{U inverse}
        U_{\mu,\varepsilon}^{-1} = \big[P_{0,0} P_{\mu,\varepsilon} + (\operatorname{Id} - P_{0,0})(\operatorname{Id} - P_{\mu,\varepsilon})\big](\operatorname{Id} - (P_{\mu,\varepsilon} - P_{0,0})^2)^{-\frac{1}{2}},
    \end{equation}
    \textit{and $U_{\mu,\varepsilon} P_{0,0} U_{\mu,\varepsilon}^{-1} = P_{\mu,\varepsilon}$ as well as $U_{\mu,\varepsilon}^{-1} P_{\mu,\varepsilon} U_{\mu,\varepsilon} = P_{0,0}$. The map $(\mu, \varepsilon) \mapsto U_{\mu,\varepsilon}$ is analytic from $B(\mu_0) \times B(\varepsilon_0)$ to $ \mathcal{L}(Y)$.}
    \item \textit{The subspaces $\mathcal{V}_{\mu,\varepsilon} = \operatorname{Rg}(P_{\mu,\varepsilon})$ are isomorphic to each other: $\mathcal{V}_{\mu,\varepsilon} = U_{\mu,\varepsilon} \mathcal{V}_{0,0}$. In particular $\dim \mathcal{V}_{\mu,\varepsilon} = \dim \mathcal{V}_{0,0} = 4$, for any $(\mu, \varepsilon) \in B(\mu_0) \times B(\varepsilon_0)$.}
\end{enumerate}
\end{lemma}

The proof of Lemma \ref{kato thm} is similar to the one of \cite[Lemma 3.1]{BMV1} and we skip it.  Recalling \eqref{mathscr L mu e}-\eqref{mathscr rho=-rho mathscr}, the Hamiltonian and reversible nature of the operator $\mathscr{L}_{\mu,\e}$ imply additional algebraic properties for spectral projectors $P_{\mu,\e}$ and the transformation operators $U_{\mu,\e}$ as follows.  

\begin{lemma} \label{properties of U and P}
For any $(\mu, \varepsilon) \in B(\mu_0) \times B(\varepsilon_0)$, the following holds true:

\begin{itemize}
    \item[(i)] The projectors $P_{\mu,\varepsilon}$ defined in \eqref{Projection P mu e} are skew-Hamiltonian, namely $\mathcal{J} P_{\mu,\varepsilon} = P_{\mu,\varepsilon}^* \mathcal{J}$, and reversibility preserving, i.e. $\bar{\rho} P_{\mu,\varepsilon} = P_{\mu,\varepsilon} \bar{\rho}$.
    
    \item[(ii)] The transformation operators $U_{\mu,\varepsilon}$ in \eqref{U transformation operators} are symplectic, namely $U_{\mu,\varepsilon}^* \mathcal{J} U_{\mu,\varepsilon} = \mathcal{J}$, and reversibility preserving.
    
    \item[(iii)] $P_{0,\varepsilon}$ and $U_{0,\varepsilon}$ are real operators, i.e. $\bar{P}_{0,\varepsilon} = P_{0,\varepsilon}$ and $\bar{U}_{0,\varepsilon} = U_{0,\varepsilon}$.
\end{itemize}
    
\end{lemma} 

See \cite[Lemma 3.2]{BMV1} for details. By the previous lemma, the linear involution $\bar{\rho}$ commutes with the spectral projectors $P_{\mu,\varepsilon}$ and then $\bar{\rho}$ leaves invariant the subspace $\mathcal{V}_{\mu,\varepsilon} = \mathrm{Rg}(P_{\mu,\varepsilon})$.

\paragraph{Symplectic and reversible basis of $\mathcal{V}_{\mu,\varepsilon}$.} It is convenient to represent the Hamiltonian and reversible operator $\mathscr{L}_{\mu,\varepsilon}: \mathcal{V}_{\mu,\varepsilon} \to \mathcal{V}_{\mu,\varepsilon}$ in a basis which is symplectic and reversible, according to the following definition:

\begin{definition}[Symplectic and reversible basis]  \label{Symplectic and reversible basis}
    A basis $\{ \Phi_1^+, \Phi_1^-, \Phi_0^+, \Phi_0^- \}$ of $\mathcal{V}_{\mu,\varepsilon}$ is \textit{symplectic} if, for any $k, k' = 0,1$,
\begin{equation} \label{basis is symplectic}
    \begin{aligned}
    (\mathcal{J} \Phi_k^\mp, \Phi_k^\pm) = \pm 1, ~~(\mathcal{J} \Phi_k^\sigma, \Phi_k^\sigma) = 0, \quad \forall \,\sigma = \pm; \\
    \text{if } k \neq k', \text{ then } (\mathcal{J} \Phi_k^\sigma, \Phi_{k'}^{\sigma'}) = 0, \quad \forall\, \sigma, \sigma' = \pm.
\end{aligned}
\end{equation}   

This is \textit{reversible} if
\begin{equation} \label{basis is reversible}
    \begin{aligned}
    \bar{\rho} \Phi_1^+ &= \Phi_1^+,  ~\bar{\rho} \Phi_1^- = -\Phi_1^-,~\bar{\rho} \Phi_0^+ = \Phi_0^+,  ~\bar{\rho} \Phi_0^- = -\Phi_0^-, \\
    &\text{i.e. } \bar{\rho} \Phi_k^\sigma = \sigma \Phi_k^\sigma, \quad \forall\, \sigma = \pm, k = 0,1.
\end{aligned}
\end{equation}
\end{definition} 

We use the following notation along the paper: we denote by $\mathit{even}(x)$ a real $2\pi$-periodic function which is even in $x$, and by $\mathit{odd}(x)$ a real $2\pi$-periodic function which is odd in $x$.

We pause to remark that any vector $\Phi\in\mathcal{V}_{\mu,\e}$ can be expressed as
\begin{align} \label{expasion of f by using symplectic basis}
    \Phi=-(\mathcal{J}\Phi,\Phi^-_1)\Phi^+_1+(\mathcal{J}\Phi,\Phi^+_1)\Phi^-_1-(\mathcal{J}\Phi,\Phi^-_0)\Phi^+_0+(\mathcal{J}\Phi,\Phi^+_0)\Phi^-_0\in\mathcal{V}_{\mu,\e}.
\end{align}
Also, following from the reversibility of the problem, specifically from the involution \( \overline{\rho} \) defined in equation \eqref{reversible}, the elements of a reversible basis \( \{\Phi^+_1,\Phi^-_1,\Phi^+_0,\Phi^-_0\} \) satisfy specific parity properties: 
\begin{align} \label{Parity structure}
\Phi_k^+(x) = 
\begin{bmatrix}
\mathit{even}(x) + \im\,\mathit{odd}(x) \\
\mathit{odd}(x) + \im\,\mathit{even}(x)
\end{bmatrix}, \quad
\Phi_k^-(x) = 
\begin{bmatrix}
\mathit{odd}(x) + \im\,\mathit{even}(x) \\
\mathit{even}(x) + \im\,\mathit{odd}(x)
\end{bmatrix}.
\end{align}

We now compute the matrix representation of $\mathscr{L}_{\mu,\e}$ in a symplectic and reversible basis of $\mathcal{V}_{\mu,\e}$.

\begin{lemma} \label{matrix representation mathsf L}
The $4 \times 4$ matrix that represents the Hamiltonian and reversible operator $\mathscr{L}_{\mu,\varepsilon} = \mathcal{J} \mathcal{K}_{\mu,\varepsilon} : \mathcal{V}_{\mu,\varepsilon} \to \mathcal{V}_{\mu,\varepsilon}$ with respect to a symplectic and reversible basis $\{ \Phi_1^{+}, \Phi_1^{-}, \Phi_0^{+}, \Phi_0^{-} \}$ of $\mathcal{V}_{\mu,\varepsilon}$ is
\begin{align} \label{J4Bmuepsilon}
  \mathsf{L}_{\mu,\varepsilon}:=  \mathsf{J}_4 \mathsf{K}_{\mu,\varepsilon}, \quad \mathsf{J}_4 := 
\left(
\begin{array}{c|c}
\mathsf{J}_2 & 0 \\
\hline
0 & \mathsf{J}_2
\end{array}
\right), \quad
    \mathsf{J}_2 := 
    \begin{pmatrix}
        0 & 1 \\
        -1 & 0
    \end{pmatrix}, \quad \text{where} \quad \mathsf{K}_{\mu,\varepsilon} = \mathsf{K}_{\mu,\varepsilon}^{*}.
\end{align}
The self-adjoint matrix
\begin{align} \label{Bmuepsilon matrix representation}
    \mathsf{K}_{\mu,\varepsilon} = \left( 
    \begin{array}{cccc}
        (\mathcal{K}_{\mu,\varepsilon} \Phi_1^{+}, \Phi_1^{+}) & (\mathcal{K}_{\mu,\varepsilon} \Phi_1^{-}, \Phi_1^{+}) & (\mathcal{K}_{\mu,\varepsilon} \Phi_0^{+}, \Phi_1^{+}) & (\mathcal{K}_{\mu,\varepsilon} \Phi_0^{-}, \Phi_1^{+}) \\
        (\mathcal{K}_{\mu,\varepsilon} \Phi_1^{+}, \Phi_1^{-}) & (\mathcal{K}_{\mu,\varepsilon} \Phi_1^{-}, \Phi_1^{-}) & (\mathcal{K}_{\mu,\varepsilon} \Phi_0^{+}, \Phi_1^{-}) & (\mathcal{K}_{\mu,\varepsilon} \Phi_0^{-}, \Phi_1^{-}) \\
        (\mathcal{K}_{\mu,\varepsilon} \Phi_1^{+}, \Phi_0^{+}) & (\mathcal{K}_{\mu,\varepsilon} \Phi_1^{-}, \Phi_0^{+}) & (\mathcal{K}_{\mu,\varepsilon} \Phi_0^{+}, \Phi_0^{+}) & (\mathcal{K}_{\mu,\varepsilon} \Phi_0^{-}, \Phi_0^{+}) \\
        (\mathcal{K}_{\mu,\varepsilon} \Phi_1^{+}, \Phi_0^{-}) & (\mathcal{K}_{\mu,\varepsilon} \Phi_1^{-}, \Phi_0^{-}) & (\mathcal{K}_{\mu,\varepsilon} \Phi_0^{+}, \Phi_0^{-}) & (\mathcal{K}_{\mu,\varepsilon} \Phi_0^{-}, \Phi_0^{-})
    \end{array}
    \right).
\end{align}
The entries of the matrix $\mathsf{K}_{\mu,\varepsilon}$ are alternatively real or purely imaginary: for any $\sigma = \pm, k = 0, 1$,
\begin{align} \label{B are alternatively real or imaginary}
    (\mathcal{K}_{\mu,\varepsilon} \Phi_k^{\sigma}, \Phi_{k'}^{\sigma}) \text{ is real}, \quad
    (\mathcal{K}_{\mu,\varepsilon} \Phi_k^{\sigma}, \Phi_{k'}^{-\sigma}) \text{ is purely imaginary}.
\end{align}  
\end{lemma} 
\begin{proof}
    Recalling \eqref{mathscr L mu e} and \eqref{expasion of f by using symplectic basis}, for $\sigma=\pm$, $k=0,1$, we obtain
    \begin{equation*} 
        \begin{aligned}
            \mathscr{L}_{\mu,\varepsilon}\Phi^{\sigma}_k&=-(\mathcal{J}\mathscr{L}_{\mu,\varepsilon}\Phi^{\sigma}_k,\Phi^-_1)\Phi^+_1+(\mathcal{J}\mathscr{L}_{\mu,\varepsilon}\Phi^{\sigma}_k,\Phi^+_1)\Phi^-_1-(\mathcal{J}\mathscr{L}_{\mu,\varepsilon}\Phi^{\sigma}_k,\Phi^-_0)\Phi^+_0+(\mathcal{J}\mathscr{L}_{\mu,\varepsilon}\Phi^{\sigma}_k,\Phi^+_0)\Phi^-_0\\
            &=(\mathcal{K}_{\mu,\varepsilon}\Phi^{\sigma}_k,\Phi^-_1)\Phi^+_1-(\mathcal{K}_{\mu,\varepsilon}\Phi^{\sigma}_k,\Phi^+_1)\Phi^-_1+(\mathcal{K}_{\mu,\varepsilon}\Phi^{\sigma}_k,\Phi^-_0)\Phi^+_0-(\mathcal{K}_{\mu,\varepsilon}\Phi^{\sigma}_k,\Phi^+_0)\Phi^-_0.
        \end{aligned}
    \end{equation*}
This verifies that the matrix representation of $\mathscr{L}_{\mu,\varepsilon}$ with respect to $\mathsf{F}$ is $\mathsf{J}_4\mathsf{K}_{\mu,\varepsilon}$. Also, the matrix $\mathsf{K}_{\mu,\e}$ is self-adjoint because $\mathcal{K}_{\mu,\e}$ is self-adjoint. Next, recall from \eqref{complex product} and \eqref{reversible} that the inner product satisfies
\begin{align} \label{(f,g)=(rho f, rho g)}
    (\Phi,\tilde{\Phi})=\overline{(\overline{\rho}\Phi,\overline{\rho}\tilde{\Phi})}.
\end{align}
Then, since $\mathcal{K}_{\mu,\e}$ is both self-adjoint and reversibility-preserving \eqref{B rho=rho B} and \eqref{basis is reversible}, we compute:
\begin{align*}
    (\mathcal{K}_{\mu,\varepsilon} \Phi_k^{\sigma}, \Phi_{k'}^{\sigma'})=\overline{(\overline{\rho}\mathcal{K}_{\mu,\e} \Phi^{\sigma}_k,\overline{\rho} \Phi^{\sigma'}_{k'})}=\overline{(\mathcal{K}_{\mu,\e}\overline{\rho} \Phi^{\sigma}_k,\overline{\rho} \Phi^{\sigma'}_{k'})}=\sigma\sigma'\overline{(\mathcal{K}_{\mu,\e}\Phi^{\sigma}_k,\Phi^{\sigma'}_{k'})},
\end{align*}
which proves \eqref{B are alternatively real or imaginary}.
\end{proof}

We conclude this section recalling some notation. A $2n \times 2n$, $n = 1, 2$, matrix of the form $\mathsf{L} = \mathsf{J}_{2n} \mathsf{K}$ is \textit{Hamiltonian} if $\mathsf{K}$ is a self-adjoint matrix, i.e. $\mathsf{K} = \mathsf{K}^*$. It is \textit{reversible} if $\mathsf{K}$ is reversibility-preserving, i.e.
\[
\rho_{2n} \circ \mathsf{K} = \mathsf{K} \circ \rho_{2n},
\]
where
\[
\rho_4 := \begin{pmatrix} \rho_2 & 0 \\ 0 & \rho_2 \end{pmatrix}, \quad
\rho_2 := \begin{pmatrix} \mathfrak{c} & 0 \\ 0 & -\mathfrak{c} \end{pmatrix}, 
\]
and $\mathfrak{c} : z \mapsto \bar{z}$ is the conjugation of the complex plane.
Equivalently, $\rho_{2n} \circ \mathsf{L} = - \mathsf{L} \circ \rho_{2n}.$

The transformations preserving the Hamiltonian structure are called \textit{symplectic}, and satisfy
\begin{align} \label{YstarJ4Y=J4}
Q^* \mathsf{J}_4 Q = \mathsf{J}_4.
\end{align}
If $Q$ is symplectic then $Q^*$ and $Q^{-1}$ are symplectic as well. A Hamiltonian matrix $\mathsf{L} = \mathsf{J}_4 \mathsf{K}$, with $\mathsf{K} = \mathsf{K}^*$, is conjugated through a symplectic matrix $Q$ in a new Hamiltonian matrix
\begin{align} \label{L1=J4B}
    \mathsf{L}_1=Q^{-1}\mathsf{L}Q=Q^{-1}\mathsf{J}_4(Q^-)^*Q^*\mathsf{K}Q=\mathsf{J}_4 \mathsf{K}_1\,\,\textit{where}\,\,\mathsf{K}_1:=Q^*\mathsf{K}Q=\mathsf{K}^*_1.
\end{align}
A $4\times 4$ matrix $\mathsf{K} = (\mathsf{K}_{ij})_{i,j=1,\ldots,4}$ is reversibility-preserving if and only if its entries are alternatively real and purely imaginary, namely $\mathsf{K}_{ij}$ is real when $i + j$ is even and purely imaginary otherwise, as in \eqref{B are alternatively real or imaginary}. A $4\times 4$ complex matrix $\mathsf{L} = (\mathsf{L}_{ij})_{i,j=1,\ldots,4}$ is reversible if and only if $\mathsf{L}_{ij}$ is purely imaginary when $i + j$ is even and real otherwise.

We finally mention that the flow of a Hamiltonian reversibility-preserving matrix is symplectic and reversibility-preserving (see \cite[Lemma 3.8]{BMV1}).

\section{Matrix Representation of $\mathscr{L}_{\mu,\e}$ on $\mathcal{V}_{\mu,\e}$}\label{sec4 MP} Using the transformation operator $U_{\mu,\e}$ in \eqref{U transformation operators}, we construct the basis of $\mathcal{V}_{\mu,\e}$
\begin{equation} \label{F basis set and f}
    \begin{aligned}
        \Phi_{\mu,\e}&:=\{\phi^+_1(\mu,\e),\phi^-_1(\mu,\e),\phi^+_0(\mu,\e),\phi^-_0(\mu,\e)\}, \qquad 
        \phi^\sigma_k(\mu,\e):=U_{\mu,\e} \phi^{\sigma}_k,~~\sigma=\pm,~k=0,1,
    \end{aligned}
\end{equation}
where 
\begin{align} \label{eigenfunc of mathcall L00 2}
    \phi^+_1=\vet{(\ch/\ckb)^{\frac{1}{2}}\cos(x)}{(\ch/\ckb)^{-\frac{1}{2}}\sin(x)}, ~~\phi^-_1=\vet{-(\ch/\ckb)^{\frac{1}{2}}\sin(x)}{(\ch/\ckb)^{-\frac{1}{2}}\cos(x)}, ~~\phi^+_0=\vet{1}{0},~~\phi^-_0=\vet{0}{1},
\end{align}
form a basis of $\mathcal{V}_{0,0}=\mathrm{Rg}(P_{0,0})$, cf. \eqref{eigenfunc of mathcall L00}-\eqref{eigenfunc f+0}. Note that the real valued vectors $\{\phi^\pm_1,\phi^\pm_0\}$ form a symplectic and reversible basis for $\mathcal{V}_{0,0}$, according to Definition \ref{Symplectic and reversible basis}. Then, by Lemma \ref{kato thm} and Lemma \ref{properties of U and P} we deduce that:
\begin{lemma} \label{F is symplectic and reversible}
    The basis $\Phi_{\mu,\e}$ of $\mathcal{V}_{\mu,\e}$ defined in \eqref{F basis set and f}, is symplectic and reversible, i.e. satisfies \eqref{basis is symplectic} and \eqref{basis is reversible}. Each map $(\mu,\e) \mapsto \phi^\sigma_k(\mu,\e)$ is analytic as a map $B(\mu_0)\times B(\mu_0)\rightarrow H^4(\mathbb{T})\times H^1(\mathbb{T})$.
\end{lemma}
\begin{proof}
By Lemma \ref{properties of U and P}-(ii), the operators $U_{\mu,\varepsilon}$ are symplectic and reversibility preserving, i.e. 
$U_{\mu,\varepsilon}^* \mathcal{J} U_{\mu,\varepsilon} = \mathcal{J}$ and $U_{\mu,\varepsilon}\circ\overline{\rho}=\overline{\rho}\circ U_{\mu,\varepsilon}$. Recall \eqref{F basis set and f}, i.e. $\phi_k^\sigma(\mu,\varepsilon):=U_{\mu,\varepsilon} \phi_k^\sigma, \sigma=\pm, k=0,1.$
Then for any $k,k',\sigma,\sigma'$,
\begin{align*}
(\mathcal{J} \phi_k^\sigma(\mu,\varepsilon), \phi_{k'}^{\sigma'}(\mu,\varepsilon))
= (\mathcal{J} \phi_k^\sigma, \phi_{k'}^{\sigma'}),
\end{align*}
so the symplectic relations \eqref{basis is symplectic} are preserved. Moreover,
\begin{align*}
\overline{\rho} \phi_k^\sigma(\mu,\varepsilon)
= U_{\mu,\varepsilon}\circ \overline{\rho} \phi_k^\sigma
= \sigma \phi_k^\sigma(\mu,\varepsilon),
\end{align*}
which shows that the reversibility conditions \eqref{basis is reversible} hold as well. 
Finally, the analyticity of $\phi_k^\sigma(\mu,\varepsilon)$ follows from the analyticity of 
$U_{\mu,\varepsilon}$ (Lemma \ref{kato thm}). 
\end{proof}

We then expand the vectors $\phi^\sigma_k(\mu,\e)$ in $(\mu,\e)$. We denote by $\mathit{even}_0(x)$ a real, even, $2\pi$-periodic function with zero space average. In the sequel $\cO(\mu^m\e^n)\vet{\mathit{even}(x)}{\mathit{odd}(x)}$ denotes an analytic map in $(\mu,\e)$ with values in $Y=H^4(\mathbb{T},\mathbb{C})\times H^1(\mathbb{T},\mathbb{C})$, whose first component is $\mathrm{even}(x)$ and the second one $\mathit{odd}(x)$; we have a similar meaning for $\cO(\mu^m\e^n)\vet{\mathit{odd}(x)}{\mathit{even}(x)}$, etc $\ldots$.

Next, we provide the expansion of the basis $\phi_k^{\pm}(\mu,\varepsilon) = U_{\mu,\varepsilon} \phi_k^{\pm}$, $k = 0,1$, in \eqref{F basis set and f}, where $\phi_k^{\pm}$ defined in \eqref{eigenfunc of mathcall L00 2} belong to the subspace $\mathcal{V}_{0,0} := \operatorname{Rg}(P_{0,0})$. We first Taylor-expand the transformation operators $U_{\mu,\varepsilon}$ defined in \eqref{U transformation operators}. We denote $\partial_{\varepsilon}$ with a prime and $\partial_{\mu}$ with a dot.
The next lemma follows as \cite[Lemma A.1]{BMV1}.
\begin{lemma} \label{U mu epsilon P00}
    The first jets of $U_{\mu,\varepsilon} P_{0,0}$ are
\begin{align} \label{A1}
    &U_{0,0} P_{0,0} = P_{0,0}, \quad\quad U'_{0,0} P_{0,0} = P'_{0,0} P_{0,0}, \quad\quad \dot{U}_{0,0} P_{0,0} = \dot{P}_{0,0} P_{0,0},\\ \label{A2}
    &\dot{U}'_{0,0} P_{0,0} = \left( \dot{P}'_{0,0} - \frac{1}{2} P_{0,0} \dot{P}'_{0,0} \right) P_{0,0}, 
\end{align}
where
\begin{align} \label{A3}
    P'_{0,0} &= \frac{1}{2\pi \im} \oint_{\Gamma} (\mathscr{L}_{0,0} - \lambda)^{-1} \mathscr{L}'_{0,0} (\mathscr{L}_{0,0} - \lambda)^{-1} \de \lambda,  \\ \label{A4}
    \dot{P}_{0,0} &= \frac{1}{2\pi \im} \oint_{\Gamma} (\mathscr{L}_{0,0} - \lambda)^{-1} \dot{\mathscr{L}}_{0,0} (\mathscr{L}_{0,0} - \lambda)^{-1} \de \lambda, 
\end{align}
and
\begin{align} \label{A5a}
    \dot{P}'_{0,0} &= -\frac{1}{2\pi \im} \oint_{\Gamma} (\mathscr{L}_{0,0} - \lambda)^{-1} \dot{\mathscr{L}}_{0,0} (\mathscr{L}_{0,0} - \lambda)^{-1} \mathscr{L}'_{0,0} (\mathscr{L}_{0,0} - \lambda)^{-1} \de \lambda \\ \label{A5b}
    &\quad -\frac{1}{2\pi \im} \oint_{\Gamma} (\mathscr{L}_{0,0} - \lambda)^{-1} \mathscr{L}'_{0,0} (\mathscr{L}_{0,0} - \lambda)^{-1} \dot{\mathscr{L}}_{0,0} (\mathscr{L}_{0,0} - \lambda)^{-1} \de \lambda \\ \label{A5c}
    &\quad + \frac{1}{2\pi \im} \oint_{\Gamma} (\mathscr{L}_{0,0} - \lambda)^{-1} \dot{\mathscr{L}}'_{0,0} (\mathscr{L}_{0,0} - \lambda)^{-1} \de \lambda.
\end{align}
The operators $\mathscr{L}'_{0,0}$ and $\dot{\mathscr{L}}_{0,0}$ are
\begin{align} \label{A6}
    \mathscr{L}'_{0,0}=\begin{bmatrix}
        \partial_x\circ q_1(x)  &0\\
        -a_1(x)+\kappa\tau_1-b\beta_1 &q_1(x)\circ \partial_x
    \end{bmatrix},~~\dot{\mathscr{L}}_{0,0}=\begin{bmatrix}
        0 &\mathrm{sgn}(D)m(D)\\
        2\im\kappa\pa_x-4\im b \partial^3_{x} &0
    \end{bmatrix},
\end{align}
where $\mathrm{sgn}(D)$ is defined in \eqref{D+mu}, \eqref{sgn D} and $m(D)$ is the real, even operator
\begin{align} \label{mD}
    m(D):=\tanh(\tth|D|)+\tth|D|(1-\tanh^2(\tth|D|))
\end{align}
and $a_1(x)$, $q_1(x)$ are given in Lemma \ref{lem:pa.exp}, $\tau_1$ is given in Lemma \ref{tau}, and $\beta_1$ is given in Lemma \ref{beta}.

The operator $\dot{\mathscr{L}}'_{0,0}$ is
\begin{align} \label{A8}
    \dot{\mathscr{L}}'_{0,0}=\begin{bmatrix}
        \im q_1(x) & 0\\
    \im\kappa\left(\tau_{1,1}(x)+2\tau_{1,2}(x)\pa_x\right)-\im b \left(b_{1,1}(x) +2b_{1,2}(x)\pa_x+3b_{1,3}(x)\pa^2_x+4b_{1,4}(x)\pa^4_x\right)       & \im q_1(x)
    \end{bmatrix}
\end{align}
with $\tau_{1,1}(x)$, $\tau_{1,2}(x)$ in \eqref{tau11}, \eqref{tau12} and $b_{1,1}(x)$, $b_{1,2}(x)$, $b_{1,3}(x)$, $b_{1,4}(x)$ in \eqref{b111}, \eqref{b112}, \eqref{b113}, \eqref{b114}. 
\end{lemma}
\begin{proof}
   For $\|P_{\mu,\varepsilon}-P_{0,0}\|_{ \mathcal{L}(Y)}<1$, we have binomial series
   \begin{equation} \label{binomial}
       \begin{aligned}
       (\mathrm{Id}-(P_{\mu,\varepsilon}-P_{0,0})^2)^{-\frac{1}{2}}&=\sum_{k=0}^{\infty} \binom{-\frac{1}{2}}{k}(-1)^k (P_{\mu,\varepsilon}-P_{0,0})^{2k}\\
       &=\mathrm{Id}+\frac{1}{2}(P_{\mu,\varepsilon}-P_{0,0})^2+\frac{3}{8} (P_{\mu,\varepsilon}-P_{0,0})^4+\cO((P_{\mu,\varepsilon}-P_{0,0})^6),
   \end{aligned}
   \end{equation}  
where 
\[
\mathcal{O}(P_{\mu, \varepsilon} - P_{0,0})^6 = \mathcal{O}(\varepsilon^6, \varepsilon^5 \mu, \varepsilon^4 \mu^2, \varepsilon^3 \mu^3, \varepsilon^2\mu^4, \varepsilon \mu^5, \mu^6) \in  \mathcal{L}(Y).
\]
By \eqref{U transformation operators} and \eqref{binomial} one has the Taylor expansion in $ \mathcal{L}(Y)$
\begin{align*}
    U_{\mu, \varepsilon} P_{0,0} &= P_{\mu, \varepsilon} P_{0,0} + \frac{1}{2} (P_{\mu, \varepsilon} - P_{0,0})^2 P_{\mu, \varepsilon} P_{0,0} + \mathcal{O}(P_{\mu, \varepsilon} - P_{0,0})^4.
\end{align*}
Consequently, one derives \eqref{A1}, \eqref{A2}, using also the identity 
\[
\dot{P}_{0,0} P'_{0,0} P_{0,0} + P'_{0,0} \dot{P}_{0,0} P_{0,0} = - P_{0,0} \dot{P}'_{0,0} P_{0,0},
\]
which follow differentiating $P_{\mu, \varepsilon}^2 = P_{\mu, \varepsilon}$. Differentiating \eqref{Projection P mu e} one gets \eqref{A3}, \eqref{A4}, \eqref{A5a}, \eqref{A5b}, and \eqref{A5c}. Formulas \eqref{A6}, \eqref{mD}, and \eqref{A8} follow by \eqref{mathcal B mu e} using also that the Fourier multiplier 
$\Pi_0 \big( \tanh(\tth|D|) + \tth|D| (1 - \tanh^2(\tth|D|)) \big) = 0$.
\end{proof}
By the Lemma \ref{U mu epsilon P00}, we have the Taylor expansion
\begin{equation} \label{the expandsion of f sigma mu}
    \begin{aligned}
    \phi_k^\sigma (\mu, \varepsilon) &= \phi_k^\sigma + \varepsilon P'_{0,0} \phi_k^\sigma + \mu \dot{P}_{0,0} \phi_k^\sigma  + \mu \varepsilon \big( \dot{P}'_{0,0} - \frac{1}{2} P_{0,0} \dot{P}'_{0,0} \big) \phi_k^\sigma + \mathcal{O}(\mu^2, \varepsilon^2).
\end{aligned}
\end{equation}

\begin{lemma}  [Expansion of the basis $\Phi_{\mu,\e}$] \label{expansion of the basis F} For small values of $(\mu,\e)$ the basis $\Phi_{\mu,\e}$ in \eqref{F basis set and f} has the expansion

\begin{equation} \label{43 f+1}
\begin{aligned} 
\phi_1^+(\mu,\varepsilon) &= 
\vet{(\ch/\ckb)^{\frac{1}{2}}\cos(x)}{(\ch/\ckb)^{-\frac{1}{2}}\sin(x)}
+ \im \frac{\mu}{4} \gamma_{\tth,\kappa,b}
\vet{(\ch/\ckb)^{\frac{1}{2}}\sin(x)}{(\ch/\ckb)^{-\frac{1}{2}}\cos(x)}
+ \varepsilon
\begin{bmatrix}
\alpha_{\tth,\kappa,b} \cos(2x) \\
\beta_{\tth,\kappa,b} \sin(2x)
\end{bmatrix} \\
& \quad + \mathcal{O}(\mu^2)
\begin{bmatrix}
\mathit{even}_0(x) + \im \mathit{odd}(x) \\
\mathit{odd}(x) + \im \mathit{even}_0(x)
\end{bmatrix}
+ \mathcal{O}(\varepsilon^2)
\begin{bmatrix}
\mathit{even}_0(x) \\
\mathit{odd}(x)
\end{bmatrix} + \im \mu \varepsilon
\begin{bmatrix}
\mathit{odd}(x) \\
\mathit{even}(x)
\end{bmatrix}
+ \mathcal{O}(\mu^2 \varepsilon, \mu \varepsilon^2),
\end{aligned}
\end{equation}

\begin{equation} \label{44 f-1}
\begin{aligned}
\phi_1^-(\mu,\varepsilon) &= 
\vet{-(\ch/\ckb)^{\frac{1}{2}}\sin(x)}{(\ch/\ckb)^{-\frac{1}{2}}\cos(x)}
+ \im \frac{\mu}{4}\gamma_{\tth,\kappa,b}
\vet{(\ch/\ckb)^{\frac{1}{2}}\cos(x)}{-(\ch/\ckb)^{-\frac{1}{2}}\sin(x)}
+ \varepsilon
\begin{bmatrix}
- \alpha_{\tth,\kappa,b} \sin(2x) \\
\beta_{\tth,\kappa,b} \cos(2x)
\end{bmatrix} \\
& \quad + \mathcal{O}(\mu^2)
\begin{bmatrix}
\mathit{odd}(x) + \im \mathit{even}_0(x) \\
\mathit{even}_0(x) + \im \mathit{odd}(x)
\end{bmatrix}
+ \mathcal{O}(\varepsilon^2)
\begin{bmatrix}
\mathit{odd}(x) \\
\mathit{even}(x)
\end{bmatrix}   + \im \mu \varepsilon
\begin{bmatrix}
\mathit{even}(x) \\
\mathit{odd}(x)
\end{bmatrix}
+ \mathcal{O}(\mu^2 \varepsilon, \mu \varepsilon^2),
\end{aligned}
\end{equation}

\begin{equation} \label{45 f+0}
\begin{aligned}
\phi_0^+(\mu,\varepsilon) &=
\begin{bmatrix}
1 \\
0
\end{bmatrix}
+ \varepsilon \delta_{\tth,\kappa,b}
\vet{(\ch/\ckb)^{\frac{1}{2}}\cos(x)}{-(\ch/\ckb)^{-\frac{1}{2}}\sin(x)}
+ \mathcal{O}(\varepsilon^2)
\begin{bmatrix}
\mathit{even}_0(x) \\
\mathit{odd}(x)
\end{bmatrix}  + \im \mu \varepsilon
\begin{bmatrix}
\mathit{odd}(x) \\
\mathit{even}(x)
\end{bmatrix}
+ \mathcal{O}(\mu^2 \varepsilon, \mu \varepsilon^2),
\end{aligned}
\end{equation}

\begin{equation} \label{46 f-0}
\begin{aligned}
\phi_0^-(\mu,\varepsilon) &=
\begin{bmatrix}
0 \\
1
\end{bmatrix}
+ \im \mu \varepsilon
\begin{bmatrix}
\mathit{even}_0(x) \\
\mathit{odd}(x)
\end{bmatrix}
+ \mathcal{O}(\mu^2 \varepsilon, \mu \varepsilon^2),
\end{aligned}
\end{equation}
where the remainders $\mathcal{O}()$ are vectors in $Y=H^4(\mathbb{T},\mathbb{C})\times H^1(\mathbb{T},\mathbb{C})$ and
\begin{equation} \label{47 coefficients}
\begin{aligned}
\alpha_{\tth,\kappa,b} &:= \frac{b\ch^4-3\kappa -27 b+\ch^4\kappa +\ch^4+3}{2\ch^\frac{3}{2}\ckb^\frac{1}{2}\,\left(b\ch^4-3\kappa -15 b+\ch^4\kappa +\ch^4\right)}, \quad
\beta_{\tth,\kappa,b} := -\frac{\left(\ch^4+1\right)\,\left(\ch^4-3\right)\ckb^\frac{5}{2}}{4\ch^\frac{5}{2}\,\left(b\ch^4-3\kappa -15 b+\ch^4\kappa +\ch^4\right)}, \\
\gamma_{\tth,\kappa,b} &:= 1+\tth(\ch^{-2}-\ch^2)-(4b+2\kappa)\ckb^{-2}, \quad
\delta_{\tth,\kappa,b} := \frac{\left(\ch^4+3\right)\,\ckb^\frac{1}{2}}{4\ch^\frac{5}{2}}.
\end{aligned}
\end{equation}
For $\mu = 0$, the basis $\{ \phi_j^\pm(0,\varepsilon),\, \varepsilon = 0 \}$ is real and
\begin{equation} \label{48 f-0=01}
\begin{aligned}
\phi_1^+(0,\varepsilon) = 
\begin{bmatrix}
\mathit{even}_0(x) \\
\mathit{odd}(x)
\end{bmatrix}, \quad
\phi_1^-(0,\varepsilon) = 
\begin{bmatrix}
\mathit{odd}(x) \\
\mathit{even}(x)
\end{bmatrix},  \quad \phi_0^+(0,\varepsilon) =\vet{1}{0}+
\begin{bmatrix}
\mathit{even}_0(x) \\
\mathit{odd}(x)
\end{bmatrix}, \quad
\phi_0^-(0,\varepsilon) = 
\begin{bmatrix}
0 \\
1
\end{bmatrix}.
\end{aligned}
\end{equation}
\end{lemma}
\begin{proof}
    The proof can be found in Appendix \ref{secA1}.
\end{proof}

We now state the main result of this section.

\begin{lemma} \label{B decomposition}
    The matrix that represents the Hamiltonian and reversible operator $\mathscr{L}_{\mu,\varepsilon} : \mathcal{V}_{\mu,\varepsilon} \to \mathcal{V}_{\mu,\varepsilon}$ in the symplectic and reversible basis $\Phi_{\mu,\e}$ of $\mathcal{V}_{\mu,\varepsilon}$ defined in \eqref{F basis set and f}, is a Hamiltonian matrix $\mathsf{L}_{\mu,\varepsilon} = \mathsf{J}_4 \mathsf{K}_{\mu,\varepsilon}$, where $\mathsf{K}_{\mu,\varepsilon}$ is a self-adjoint and reversibility preserving (i.e., satisfying \eqref{B are alternatively real or imaginary}) $4 \times 4$ matrix of the form
\begin{align} \label{B mu epsilon}
    \mathsf{K}_{\mu,\varepsilon} = 
\begin{pmatrix}
\mathsf E & \mathsf F \\
\mathsf{F}^* & \mathsf G
\end{pmatrix}, \quad \mathsf E = \mathsf{E}^*, \quad \mathsf G = \mathsf{G}^*,
\end{align}
where $\mathsf E, \mathsf F, \mathsf G$ are the $2 \times 2$ matrices
\begin{equation} \label{E}
\begin{aligned} 
\mathsf E &:= 
\begin{pmatrix}
\mathsf{e}_{11} \varepsilon^2 (1 + \hat{r}_1(\varepsilon,\mu\varepsilon)) - \mathsf{e}_{22} \frac{\mu^2}{8} (1 + \check{r}_1(\varepsilon,\mu)) &
\im\left(\frac{1}{2} \mathsf{e}_{12} \mu + r_2(\mu \varepsilon^2, \mu^2 \varepsilon, \mu^3)\right) \\
- \im\left(\frac{1}{2} \mathsf{e}_{12} \mu + r_2(\mu \varepsilon^2, \mu^2 \varepsilon, \mu^3)\right) &
- \mathsf{e}_{22} \frac{\mu^2}{8}(1 + r_5(\varepsilon, \mu))
\end{pmatrix}, 
\end{aligned}
\end{equation}

\begin{equation} \label{F}
\begin{aligned} 
\mathsf F &:= 
\begin{pmatrix}
\mathsf{f}_{11} \varepsilon + r_3(\varepsilon^3, \mu \varepsilon^2, \mu^2 \varepsilon) &
\im \,\mu \varepsilon \ch^{-\frac{1}{2}}\ckb^{\frac{1}{2}} + \im\, r_4(\mu \varepsilon^2, \mu^2 \varepsilon) \\
\im\, r_6(\mu \varepsilon) &
r_7(\mu^2 \varepsilon)
\end{pmatrix}, 
\end{aligned}
\end{equation}

\begin{equation} \label{G}
\begin{aligned} 
\mathsf G &:= 
\begin{pmatrix}
1 +\kappa\mu^2+ r_8(\varepsilon^2, \mu^2 \varepsilon) &
- \im r_9(\mu \varepsilon^2, \mu^2 \varepsilon) \\
\im r_9(\mu \varepsilon^2, \mu^2 \varepsilon) &
\mu \tanh(\tth \mu) + r_{10}(\mu^2 \varepsilon)
\end{pmatrix}, 
\end{aligned}
\end{equation}
where 
\begin{equation} \label{e11 f11}
    \begin{aligned} 
\mathsf{e}_{11}&:=\mathsf{e}_{11}(\tth,\kappa,b) \\
&:= \Big(-b^2\ch^8+20 b^2\ch^4+69 b^2+5 b\ch^8 \kappa +8 b\ch^8-89 b\ch^4 \kappa -140 b\ch^4+90 b\kappa +78 b\\
&\,+6\ch^8 \kappa ^2+15\ch^8 \kappa +9\ch^8-25\ch^4 \kappa ^2-44\ch^4 \kappa -10\ch^4+21 \kappa ^2+30\kappa +9\Big)\Big(8\ch^3\ckb r_{\tth,\kappa,b}\Big)^{-1},\\
\mathsf{e}_{12} &:=\mathsf{e}_{12}(\tth,\kappa,b):=\frac{\left(4 b+2\kappa \right)\ch}{\ckb}+\frac{\left(\ch^2-\tth \left(\ch^4-1\right)\right)\ckb}{\ch}\,,\\
\mathsf{e}_{22} &:=\mathsf{e}_{22}(\tth,\kappa,b)\\
&:= \Big( 2b\ch^4+2b\tth^2-2\ch^2\tth+2\ch^6\tth+2\ch^4\kappa +2\tth^2\kappa +\ch^4+\tth^2+17b^2\ch^4+b^2\tth^2+2\ch^4\tth^2\\
&\,-3\ch^8\tth^2+5\ch^4\kappa^2+\tth^2\kappa^2-2b^2\ch^2\tth+4b\ch^4\tth^2+2b^2\ch^6\tth-6b\ch^8\tth^2-2\ch^2\tth\kappa^2+4\ch^4\tth^2\kappa\\ 
&\,+2\ch^6\tth\kappa^2-6\ch^8\tth^2\kappa +2b^2\ch^4\tth^2-3b^2\ch^8\tth^2-32b\ch^4\ckb^2+2\ch^4\tth^2\kappa^2-3\ch^8\tth^2\kappa^2-8\ch^4\kappa\ckb^2\\
&\,-4b\ch^2\tth+4b\ch^6\tth+18 b\ch^4\kappa +2b\tth^2\kappa -4\ch^2\tth\kappa +4\ch^6\tth\kappa -8b\ch^2\tth\ckb^2+8b\ch^6\tth\ckb^2-4\ch^2\tth\kappa\ckb^2\\
&\,+4\ch^6\tth\kappa\ckb^2-4b\ch^2\tth\kappa +4b\ch^6\tth\kappa +4b\ch^4\tth^2\kappa -6b\ch^8\tth^2\kappa\Big)\Big(\ch^3\ckb^3\Big)^{-1}\,,\\
\mathsf{f}_{11}&:=\mathsf{f}_{11} (\tth,\kappa,b):=-\frac{1}{2}\left(\ch^{\frac{5}{2}}-\ch^{-\frac{3}{2}}\right)\ckb^{\frac{3}{2}}. 
\end{aligned}
\end{equation}
\end{lemma}
\begin{proof}
    The proof can be found in Appendix \ref{secD}.
\end{proof}

\section{Block-Decoupling}\label{sec:BD}
\paragraph{Singular symplectic rescaling.} Our goal here is to block-decouple the $4\times 4$ Hamiltonian matrix $\mathsf{L}_{\mu,\varepsilon} = \mathsf{J}_4 \mathsf{K}_{\mu,\varepsilon}$ in Lemma~\ref{B decomposition}. To this end, let us first introduce a symplectic and reversibility-preserving matrix:
\begin{equation*}
\begin{aligned} \mathsf{Q}:=
\left(
\begin{array}{cc|cc}
\mu^{\frac{1}{2}} & 0 & 0 & 0\\
0 & \mu^{-\frac{1}{2}} & 0 & 0\\
\hline
0 & 0 & \mu^{\frac{1}{2}} & 0\\
0 & 0 & 0 & \mu^{-\frac{1}{2}}
\end{array}
\right), \qquad\text{with}\qquad \mu>0.
\end{aligned}
\end{equation*}
Let us define
\begin{align} \label{L1mue YLY}
    \mathsf{L}_{\mu,\varepsilon}^{(1)} := \mathsf{Q}^{-1} \mathsf{L}_{\mu,\varepsilon} \mathsf{Q}.
\end{align}
A straightforward calculation reveals that 
\begin{align*} 
    \mathsf{L}_{\mu,\varepsilon}^{(1)} = \mathsf{J}_4 \mathsf{K}_{\mu, \varepsilon}^{(1)} 
= \begin{pmatrix}
\mathsf{J}_2 \mathsf{E}^{(1)} & \mathsf{J}_2 \mathsf{F}^{(1)} \\
\mathsf{J}_2 [\mathsf{F}^{(1)}]^* & \mathsf{J}_2 \mathsf{G}^{(1)}
\end{pmatrix}\,,
\end{align*}
where $\mathsf{K}_{\mu, \varepsilon}^{(1)}$ is a self-adjoint and reversibility-preserving $4 \times 4$ matrix
\begin{align*}
    \mathsf{K}_{\mu, \varepsilon}^{(1)} = 
\begin{pmatrix}
\mathsf{E}^{(1)} & \mathsf{F}^{(1)} \\
[\mathsf{F}^{(1)}]^* & \mathsf{G}^{(1)}
\end{pmatrix}, 
\quad 
\mathsf{E}^{(1)} = [\mathsf{E}^{(1)}]^*, \quad \mathsf{G}^{(1)} = [\mathsf{G}^{(1)}]^*,
\end{align*}
where the $2 \times 2$ reversibility-preserving matrices $\mathsf{E}^{(1)}$, $\mathsf{F}^{(1)}$, and $\mathsf{G}^{(1)}$ extend analytically at $\mu = 0$ with the following expansion
\begin{align} \label{E1}
\mathsf{E}^{(1)} &= 
\begin{pmatrix}
\mathsf{e}_{11} \mu \e^2 (1 + \hat{r}_1(\e, \mu \varepsilon)) - \mathsf{e}_{22} \frac{\mu^3}{8} (1 + \check{r}_1(\e, \mu)) & \im\, \left( \frac{1}{2} \mathsf{e}_{12} \mu + r_2(\mu \varepsilon^2, \mu^2 \varepsilon, \mu^3) \right) \\
-\im\, \left( \frac{1}{2} \mathsf{e}_{12} \mu + r_2(\mu \varepsilon^2, \mu^2 \varepsilon, \mu^3) \right) & -\mathsf{e}_{22} \frac{\mu}{8} (1 + r_5(\varepsilon, \mu))
\end{pmatrix},
\\ \label{F1}
\mathsf{F}^{(1)} &= 
\begin{pmatrix}
\mathsf{f}_{11} \mu \e + r_3(\mu \varepsilon^3, \mu^2 \varepsilon^2, \mu^3 \varepsilon) & \im\, \mu \varepsilon \ch^{-\frac{1}{2}}\ckb^{\frac{1}{2}} + \im\, r_4(\mu \varepsilon^2, \mu^2 \varepsilon) \\
\im\, r_6(\mu \varepsilon) & r_7(\mu \varepsilon)
\end{pmatrix},
\\ \label{G1}
\mathsf{G}^{(1)} &= 
\begin{pmatrix}
\mu +\kappa\mu^3+ r_8(\mu \varepsilon^2, \mu^3 \varepsilon) & -\im\, r_9(\mu \varepsilon^2, \mu^2 \varepsilon) \\
\im\, r_9(\mu \varepsilon^2, \mu^2 \varepsilon) & \tanh(\tth \mu) + r_{10}(\mu \varepsilon)
\end{pmatrix},
\end{align}
where $\mathsf{e}_{11}, \mathsf{f}_{11}, \mathsf{e}_{12}, \mathsf{e}_{22}, $ are defined in \eqref{e11 f11}.

We pause to remark that the Hamiltonian and reversible matrix $\mathsf{L}_{\mu, \varepsilon}^{(1)}$, a priori defined only for $\mu \neq 0$, extends analytically to the zero matrix at $\mu = 0$. For $\mu \neq 0$ the spectrum of $\mathsf{L}_{\mu, \varepsilon}^{(1)}$ coincides with the spectrum of $\mathsf{L}_{\mu, \varepsilon}$.

\paragraph{Block-decoupling.} We perform a 
symplectic and reversibility-preserving Lie transform, determined by a 
Sylvester equation, to reduce the off-diagonal coupling from 
$\cO(\mu\varepsilon)$ to $\cO(\mu\varepsilon^3)$. Recalling \eqref{E1}, \eqref{F1}, and \eqref{G1}, let us introduce the $4\times 4$ matrix 
\begin{align} \label{Ax=f}
\mathsf{A}^{(1)}:=\begin{pmatrix}
a & b & c & 0 \\
d & a & 0 & -c \\
e & 0 & a & -b \\
0 & -e & -d & a
\end{pmatrix}:=\begin{pmatrix}
\mathsf{G}^{(1)}_{12}-\mathsf{E^{(1)}_{12}} & \mathsf{G}^{(1)}_{11} & \mathsf{E}^{(1)}_{22} & 0 \\
\mathsf{G}^{(1)}_{22} & \mathsf{G}^{(1)}_{12}-\mathsf{E^{(1)}_{12}} & 0 & -\mathsf{E}^{(1)}_{22} \\
\mathsf{E}^{(1)}_{11} & 0 & \mathsf{G}^{(1)}_{12}-\mathsf{E^{(1)}_{12}} & -\mathsf{G}^{(1)}_{11} \\
0 & -\mathsf{E}^{(1)}_{11} & -\mathsf{G}^{(1)}_{22} & \mathsf{G}^{(1)}_{12}-\mathsf{E^{(1)}_{12}}
\end{pmatrix}.
\end{align}
Assume that $(\tth,\kappa,b) \notin \mathfrak{D}$, where 
$\mathfrak{D}$ is defined in \eqref{def:fD}. A direct computation yields 
\begin{equation} \label{det A}
\begin{aligned}
\det \mathsf{A}^{(1)} 
=  \mu^4 \textup{D}_{\tth,\kappa,b}^2 (1 + r(\varepsilon, \mu^2))\neq 0 ,\quad\text{with}\quad \mu\neq 0 \quad\text{and}\quad |\mu|, |\e|\ll 1.
\end{aligned}    
\end{equation}
The inverse of $\mathsf{A}^{(1)}$ is computed in \cite[formula (5.23)]{BMV3} as 
\begin{align} \label{A inverse}
[\mathsf{A}^{(1)}]^{-1} = \frac{1}{\det\mathsf{A}^{(1)} }
\begin{pmatrix}
a (a^2 - bd - ce) & b(-a^2 + bd - ce) & -c(a^2 + bd - ce) & -2abc \\[6pt]
d(-a^2 + bd - ce) & a(a^2 - bd - ce) & 2acd & -c(-a^2 - bd + ce) \\[6pt]
-e(a^2 + bd - ce) & 2abe & a(a^2 - bd - ce) & b(a^2 - bd + ce) \\[6pt]
-2ade & -e(-a^2 - bd + ce) & d(a^2 - bd + ce) & a(a^2 - bd - ce)
\end{pmatrix}
\end{align}
and by an elementary (but tedious) computation, we have
\begin{align} \label{mathsf A -1}
[\mathsf{A}^{(1)}]^{-1} = (1 + r(\varepsilon, \mu)) 
\frac{1}{\mu \textup{D}_{\tth,\kappa,b}^2}\left(
\begin{array}{cccc}
\frac{1}{2} \mathsf{e}_{12} \textup{D}_{\tth,\kappa,b} & \textup{D}_{\tth,\kappa,b} & \frac{1}{32} \mathsf{e}_{22}(\mathsf{e}_{12}^2 + 4\tth) & -\frac{1}{8} \mathsf{e}_{12} \mathsf{e}_{22} \\
\tth \textup{D}_{\tth,\kappa,b} & \frac{1}{2} \mathsf{e}_{12} \textup{D}_{\tth,\kappa,b} & \frac{1}{8} \mathsf{e}_{12} \mathsf{e}_{22} \tth & -\frac{1}{32} \mathsf{e}_{22} (\mathsf{e}_{12}^2 + 4\tth) \\
r(\varepsilon^2, \mu^2) & r(\varepsilon^2, \mu^2) & \frac{1}{2} \mathsf{e}_{12} \textup{D}_{\tth,\kappa,b} & -\textup{D}_{\tth,\kappa,b} \\
r(\varepsilon^2, \mu^2) & r(\varepsilon^2, \mu^2) & -\tth \textup{D}_{\tth,\kappa,b} & \frac{1}{2} \mathsf{e}_{12} \textup{D}_{\tth,\kappa,b} \\
\end{array}
\right).
\end{align}
Let us define 
\begin{align} \label{mathsf V}
    \mathsf{V}:=\begin{pmatrix}
v_{11} & \im \,v_{12} \\
\im \,v_{21} & v_{22}
\end{pmatrix}, \qquad\text{with}\qquad 
\begin{pmatrix}
v_{11} \\ v_{12} \\ v_{21} \\ v_{22}
\end{pmatrix}
:=
[\mathsf{A}^{(1)}]^{-1}
\begin{pmatrix}
- \mathsf{F}_{21}^{(1)} \\  \mathsf{F}_{22}^{(1)} \\ - \mathsf{F}_{11}^{(1)} \\  \mathsf{F}_{12}^{(1)} 
\end{pmatrix}=\begin{pmatrix}
\cO(\e) \\  \cO(\e) \\ \cO(\e) \\  \cO(\e) 
\end{pmatrix}\in\mathbb{R}^4.
\end{align}
We are ready to introduce the new Hamiltonian and reversibility-preserving matrix
\begin{align} \label{L2mueps}
\mathsf{L}^{(2)}_{\mu, \varepsilon} := \exp\left(\mathsf{S}^{(1)}\right) \mathsf{L}^{(1)}_{\mu, \varepsilon} \exp\left(-\mathsf{S}^{(1)}\right) .
\end{align}
Here the symplectic and reversibility-preserving $4 \times 4$ matrix
\begin{align} \label{S1}
    \exp\left(\mathsf{S}^{(1)}\right), \quad \text{where } \quad \mathsf{S}^{(1)} := \mathsf{J}_4 \begin{pmatrix}
0 & \mathsf{M} \\
\mathsf{M}^* & 0
\end{pmatrix}, \quad \mathsf{M} := \mathsf{J}_2 \mathsf{V}. 
\end{align}
By direct computation (see Appendix~\ref{secF}), we obtain a self-adjoint and reversibility-preserving $4\times 4$ matrix 
\begin{align}
 \mathsf{K}^{(2)}_{\mu, \varepsilon} = [ \mathsf{K}^{(2)}_{\mu, \varepsilon}]^*=  \begin{pmatrix}
\mathsf{J}_2 \mathsf{E}^{(2)} & \mathsf{J}_2 \mathsf{F}^{(2)} \\
\mathsf{J}_2 [\mathsf{F}^{(2)}]^* & \mathsf{J}_2 \mathsf{G}^{(2)}
\end{pmatrix}=-\mathsf{J}_4\mathsf{L}^{(2)}_{\mu, \varepsilon}.
\end{align}
Here the reversibility-preserving $2 \times 2$ self-adjoint matrix $\mathsf{E}^{(2)}=[\mathsf{E}^{(2)}]^* $ has the form
\begin{align} \label{E2 in lemma}
\mathsf{E}^{(2)} = \begin{pmatrix}
\mu \varepsilon^2 \mathsf{e}_{\mathrm{WB}} + r_1'(\mu \varepsilon^3, \mu^2 \varepsilon^2) - \mathsf{e}_{22} \tfrac{\mu^3}{8}(1 + r_1''(\varepsilon, \mu)) & \im\,\left( \frac{1}{2} \mathsf{e}_{12} \mu + r_2(\mu \varepsilon^2, \mu^2 \varepsilon, \mu^3) \right) \\
-\im\,\left( \frac{1}{2}  \mathsf{e}_{12} \mu + r_2(\mu \varepsilon^2, \mu^2 \varepsilon, \mu^3) \right) & -\mathsf{e}_{22} \frac{\mu}{8}(1 + r_5(\varepsilon, \mu))
\end{pmatrix}, 
\end{align}
where
\begin{align} \label{eWB}
\mathsf{e}_{\mathrm{WB}}:= \mathsf{e}_{\mathrm{WB}}(\tth,\kappa,b) := \mathsf{e}_{11} - \textup{D}_{\tth,\kappa,b}^{-1}\left(\ch^{-1}\ckb  + \tth \mathsf{f}^2_{11} + \mathsf{e}_{12} \mathsf{f}_{11} \ch^{-\frac{1}{2}}\ckb^{\frac{1}{2}}\right), 
\end{align}
with constants $\mathsf{e}_{11},  \mathsf{f}_{11}, \mathsf{e}_{12}$ defined in \eqref{e11 f11}. The reversibility-preserving $2 \times 2$ self-adjoint matrix $\mathsf{G}^{(2)}=[\mathsf{G}^{(2)}]^* $ has the form
\begin{align} \label{G2}
\mathsf{G}^{(2)} = \begin{pmatrix}
\mu+\kappa \mu^3 + r_8(\mu \varepsilon^2, \mu^3\e ) & -\im r_9(\mu \varepsilon^2, \mu^2 \varepsilon) \\
\im r_9(\mu \varepsilon^2, \mu^2 \varepsilon) & \tanh(\tth \mu) + r_{10}(\mu \varepsilon)
\end{pmatrix}, 
\end{align}
and
\begin{align} \label{F2}
\mathsf{F}^{(2)} = \begin{pmatrix}
r_3(\mu \varepsilon^3) & \im r_4(\mu \varepsilon^3) \\
\im r_6(\mu \varepsilon^3) & r_7(\mu \varepsilon^3)
\end{pmatrix}. 
\end{align}

\paragraph{Complete block-decoupling and proof of the main result.}
It remains to eliminate the residual off-diagonal coupling in 
$\mathsf{L}^{(2)}_{\mu,\e}$. We first decompose
\begin{equation} \label{L2mueps2}
\mathsf{L}^{(2)}_{\mu,\e}
=
\mathsf{D}^{(2)}+\mathsf{R}^{(2)}, 
\qquad
\mathsf{D}^{(2)}
:=
\begin{pmatrix}
\mathsf{J}_2\mathsf{E}^{(2)} & 0\\
0 & \mathsf{J}_2\mathsf{G}^{(2)}
\end{pmatrix},
\qquad
\mathsf{R}^{(2)}
:=
\begin{pmatrix}
0 & \mathsf{J}_2\mathsf{F}^{(2)}\\
\mathsf{J}_2[\mathsf{F}^{(2)}]^* & 0
\end{pmatrix}.
\end{equation}
By \eqref{E2 in lemma}, \eqref{G2}, and \eqref{F2}, there exist analytic 
Hamiltonian and reversibility-preserving matrices 
$\overline{\mathsf{D}}^{(2)}$ and $\overline{\mathsf{R}}^{(2)}$ such that
$$
\mathsf{D}^{(2)}
=
\mu\overline{\mathsf{D}}^{(2)},
\qquad
\mathsf{R}^{(2)}
=
\mu\overline{\mathsf{R}}^{(2)},
\qquad
\overline{\mathsf{R}}^{(2)}
=
\cO(\e^3).
$$

Let $\Pi_D$ and $\Pi_R$ denote the projections onto the 
block-diagonal and block-off-diagonal matrices, respectively. For 
$\vec{x}=(x_{11},x_{12},x_{21},x_{22})^\top\in\mathbb{R}^4$, define
$$
\mathsf{W}(\vec{x})
:=
\begin{pmatrix}
x_{11} & \im x_{12}\\
\im x_{21} & x_{22}
\end{pmatrix},
\qquad
\mathsf{M}(\vec{x})
:=
\mathsf{J}_2\mathsf{W}(\vec{x}), \qquad
\mathsf{S}^{(2)}(\vec{x})
:=
\mathsf{J}_4
\begin{pmatrix}
0 & \mathsf{M}(\vec{x})\\
\mathsf{M}(\vec{x})^* & 0
\end{pmatrix}.
$$
Then $\mathsf{S}^{(2)}(\vec{x})$ is Hamiltonian and reversibility-preserving, 
and hence $\exp(\mathsf{S}^{(2)}(\vec{x}))$ is symplectic and 
reversibility-preserving. We seek $\vec{x}=\vec{x}(\mu,\e)$ such that
$$
\Pi_R
\left(
\exp(\mathsf{S}^{(2)}(\vec{x}))
\left(
\mathsf{D}^{(2)}+\mathsf{R}^{(2)}
\right)
\exp(-\mathsf{S}^{(2)}(\vec{x}))
\right)
=
0.
$$
Since both $\mathsf{D}^{(2)}$ and $\mathsf{R}^{(2)}$ contain the common 
factor $\mu$, this equation is equivalent, for $\mu\neq0$, to
\begin{equation} \label{normalized off diagonal equation}
\Pi_R
\left(
\exp(\mathsf{S}^{(2)}(\vec{x}))
\left(
\overline{\mathsf{D}}^{(2)}
+
\overline{\mathsf{R}}^{(2)}
\right)
\exp(-\mathsf{S}^{(2)}(\vec{x}))
\right)
=
0,
\end{equation}
and the latter equation extends analytically to $\mu=0$.

Identifying the space of block-off-diagonal Hamiltonian and 
reversibility-preserving matrices with $\mathbb{R}^4$, define
$$
f(\mu,\e,\vec{x})
:=
\Pi_R
\left(
\exp(\mathsf{S}^{(2)}(\vec{x}))
\left(
\overline{\mathsf{D}}^{(2)}
+
\overline{\mathsf{R}}^{(2)}
\right)
\exp(-\mathsf{S}^{(2)}(\vec{x}))
\right).
$$
Since $\overline{\mathsf{R}}^{(2)}=\cO(\e^3)$, one has
$$
f(\mu,\e,0)
=
\overline{\mathsf{R}}^{(2)}
=
\cO(\e^3),
\qquad
f(0,0,0)=0.
$$
Moreover,
$$
\partial_{\vec{x}}f(0,0,0)[\vec{h}]
=
\Pi_R
\left[
\mathsf{S}^{(2)}(\vec{h}),
\overline{\mathsf{D}}^{(2)}(0,0)
\right].
$$
This is precisely the normalized Sylvester operator associated with the 
block-diagonal matrix $\overline{\mathsf{D}}^{(2)}(0,0)$. Since the 
corrections appearing in $\mathsf{E}^{(2)}$ and $\mathsf{G}^{(2)}$ vanish 
at $(\mu,\e)=(0,0)$, its coordinate matrix has the same leading determinant 
as the normalized matrix in \eqref{det A}; namely,
$$
\det\left(\partial_{\vec{x}}f(0,0,0)\right)
=
\textup{D}_{\tth,\kappa,b}^{\,2}.
$$
Since $(\tth,\kappa,b)\notin\mathfrak{D}$, one has
$\textup{D}_{\tth,\kappa,b}\neq0$. Therefore,
$\partial_{\vec{x}}f(0,0,0)$ is invertible.

By the analytic implicit function theorem, there exists a unique analytic 
map
$$
\vec{x}
=
\vec{x}(\mu,\e),
\qquad
\vec{x}(0,0)=0,
$$
defined for $|\mu|,|\e|\ll1$, such that
$$
f(\mu,\e,\vec{x}(\mu,\e))=0.
$$
Furthermore, since $f(\mu,\e,0)=\cO(\e^3)$ and 
$\partial_{\vec{x}}f(0,0,0)$ is invertible, one has
$$
\vec{x}(\mu,\e)=\cO(\e^3).
$$
Hence, setting
$$
\mathsf{S}^{(2)}
:=
\mathsf{S}^{(2)}(\vec{x}(\mu,\e)),
$$
we obtain a Hamiltonian and reversibility-preserving matrix, analytic in 
$(\mu,\e)$, satisfying
$$
\mathsf{S}^{(2)}=\cO(\e^3)\qquad\text{and}\qquad
\Pi_R
\left(
\exp(\mathsf{S}^{(2)})
\mathsf{L}^{(2)}_{\mu,\e}
\exp(-\mathsf{S}^{(2)})
\right)
=
0.
$$
Thus,
\begin{equation} \label{complete block diagonalization}
\exp(\mathsf{S}^{(2)})
\mathsf{L}^{(2)}_{\mu,\e}
\exp(-\mathsf{S}^{(2)})
=
\mathsf{D}^{(2)}+\mathsf{P},
\end{equation}
where
$$
\mathsf{P}
:=
\Pi_D
\left(
\exp(\mathsf{S}^{(2)})
\mathsf{L}^{(2)}_{\mu,\e}
\exp(-\mathsf{S}^{(2)})
\right)
-
\mathsf{D}^{(2)}=\cO(\mu\e^6)
$$
is block-diagonal, Hamiltonian, reversibility-preserving, and analytic in 
$(\mu,\e)$.

Consequently, there exist self-adjoint and 
reversibility-preserving $2\times2$ matrices $\mathsf{E}^{(3)}$ and 
$\mathsf{G}^{(3)}$ such that
$$
\mathsf{D}^{(2)}+\mathsf{P}
=
\begin{pmatrix}
\mathsf{J}_2\mathsf{E}^{(3)} & 0\\
0 & \mathsf{J}_2\mathsf{G}^{(3)}
\end{pmatrix},
$$
with
$$
\mathsf{E}^{(3)}
=
\mathsf{E}^{(2)}
+
\cO(\mu\e^6),
\qquad
\mathsf{G}^{(3)}
=
\mathsf{G}^{(2)}
+
\cO(\mu\e^6).
$$
In particular, $\mathsf{E}^{(3)}$ and $\mathsf{G}^{(3)}$ have the same 
expansions as $\mathsf{E}^{(2)}$ and $\mathsf{G}^{(2)}$ at the orders 
displayed in \eqref{E2 in lemma} and \eqref{G2}.

\begin{proof} [{Proof of Theorem \ref{Complete BF thm}.}]
Recalling \eqref{mathcal L= ichmu+mathscr L} the operator $ L_{\mu, \varepsilon} : \mathcal{V}_{\mu, \varepsilon} \to \mathcal{V}_{\mu, \varepsilon}$ is represented by the $4 \times 4$ Hamiltonian and reversible matrix
\[
\im c_0 \mu + \exp(\mathsf{S}^{(2)}) \mathsf{L}^{(2)}_{\mu, \varepsilon} \exp(-\mathsf{S}^{(2)}) 
= \im c_0 \mu + 
\begin{pmatrix}
\mathsf{J}_2 \mathsf{E}^{(3)} & 0 \\
0 & \mathsf{J}_2 \mathsf{G}^{(3)}
\end{pmatrix}
=: 
\begin{pmatrix}
U & 0 \\
0 & S
\end{pmatrix},
\]
where the matrices $\mathsf{E}^{(3)}$ and $\mathsf{G}^{(3)}$ expand as in \eqref{E2 in lemma}, \eqref{G2}. Consequently, the matrices $U$ and $S$ expand as in \eqref{U S}. Theorem \ref{Complete BF thm} is proved. 
\end{proof}

\appendix

\section{Analyticity of the Dirichlet--Neumann operator for finite depth}\label{sec:AppG}

In this appendix we establish a tame analyticity result for the
finite-depth Dirichlet--Neumann operator in a scale of exponentially
weighted Sobolev spaces. The argument is formulated on $\mathbb T^d$ and therefore applies beyond the one-dimensional periodic
setting used in the main body of the paper. This result provides the
analytic foundation for the construction of the hydroelastic Stokes branch
and for the subsequent Bloch perturbation analysis. Our proof follows the
general strategy of~\cite{BMV2}, but the finite-depth geometry and the
regularity requirements of the present problem require several
modifications. We therefore give a self-contained argument.

Fix $\tth>0$ and $d\in\mathbb N$. Recall that on the domain $\Omega_{\tth,\eta} := \{(x,y) \in \T^d \times \R : -\tth < y < \eta(x)\}$, given a periodic Dirichlet datum $\psi(x)$, we consider the unique harmonic function $\Psi(x,y)$ solving the system
\begin{align}\label{defining}
\begin{cases}
    \Delta_{x,y} \Psi = 0 \quad & \mbox{in} \ \Omega_{\tth,\eta} \\
    \Psi(x,y) = \psi(x) \quad & \mbox{at} \ y = \eta(x) \\
    \partial_y \Psi(x,y)  =0  \quad & \mbox{at} \ y = -\tth.
\end{cases}
\end{align}

Then the Dirichlet--Neumann operator $G(\eta)$ is defined by the linear operator
\[
[G(\eta)\psi](x) := (\partial_y \Psi)(x, \eta(x)) - \nabla_x\eta(x) \cdot (\nabla_x \Psi) (x,\eta(x))  = \sqrt{1 + |\nabla_x \eta (x)|^2} \partial_n \Psi,
\]
where $\partial_n$ denotes the outward normal derivative. It follows from Green's formula that $G$ is symmetric. In fact, if $\tilde{\Psi}$ is the solution of the following system 
\begin{align}\label{defining22}
\begin{cases}
    \Delta_{x,y} \tilde{\Psi} = 0 \quad & \mbox{in} \ \Omega_{\tth,\eta} \\
    \tilde{\Psi}(x,y) = \tilde{\psi}(x) \quad & \mbox{at} \ y = \eta(x) \\
    \partial_y \tilde{\Psi}(x,y)  =0  \quad & \mbox{at} \ y = -\tth.
\end{cases}
\end{align}
Then with suitable regularity on $\psi$ and $\tilde{\psi}$ we obtain
\[
\int_{\T^d} [G(\eta)\psi](x) \tilde{\psi}(x) \de x = 
\int_{\Omega_{\tth,\eta}} \nabla_{x,y} \Psi(x,y) \cdot \nabla_{x,y}  \tilde{\Psi}(x,y) \de x \de y
= \int_{\T^d} \psi(x) [G(\eta)\tilde{\psi}](x) \de x.
\]
The terminology ``Dirichlet--Neumann''  reflects the fact that the operator $G(\eta)$ maps the Dirichlet datum $\psi(x)$ of the harmonic function $\Psi(x,y)$ into its weighted outward normal derivative at the boundary. Here is the main result of this appendix. 

\begin{theorem}\label{thmG}
    Let $\alpha \ge 0$ and $m, m_0$ such that $m + \frac{1}{2},m_0 \in \N$, and $m - \frac{3}{2} \ge m_0 > \frac{d + 1}{2}$. Then there exists $\varepsilon_0 > 0$ such that the Dirichlet--Neumann operator map 
    \[
    \eta \mapsto G(\eta), \qquad H^{\alpha,m} \cap B^{\alpha, m_0 + \frac{3}{2}}(\varepsilon_0) \to  \mathcal{L}(H^{\alpha,m}, H^{\alpha,m - 1})
    \]
is analytic and the following tame estimate holds
\[
\|G(\eta) \psi\|_{H^{\alpha,m - 1}} \le C (\|\psi\|_{H^{\alpha,m}} + \|\eta\|_{H^{\alpha,m}} \|\psi\|_{H^{\alpha,m_0 + \frac{3}{2}}}  )
\]
for some $C > 0$. 
\end{theorem}
Here the function space
\begin{align}\label{defHss}
H^{\alpha,m} = H^{\alpha, m}(\T^d) := \left\{u(x) = \sum_{k \in \Z^d} u_k e^{\im k \cdot x} : \|u\|_{H^{\alpha,m}}^2 := \sum_{k \in \Z^d} e^{2\alpha|k|_1} \langle k \rangle^{2 m} |u_k|^2 < +\infty \right\},    
\end{align}
where for any $k = (k_1, \ldots, k_d) \in \Z^d$, we set
\[
|k|_1 := \sum_{j = 1}^d |k_j|, \qquad \langle k \rangle := \sqrt{1 + |k|^2}, \qquad |k| := \sqrt{\sum_{j = 1}^d |k_j|^2}.
\]
Furthermore, we use the notation $B^{\alpha,m}(r)$ to denote the ball centered at $0$ with radius $r$ in the function space $H^{\alpha, m}$ and the notation $ \mathcal{L}(H_1,H_2)$ are reserved to denote the space of bounded linear maps between two Hilbert spaces $H_1$ and $H_2$.

Roughly speaking, the proof of Theorem \ref{thmG} is composed of the following three steps: (i) since the information of $\eta$ is hidden in the region, we cannot say much about it. The first thing to do is to use a change of variables to flatten the boundary. If we do so, the information of $\eta$ will be represented by the operator; (ii) we solve the equation \eqref{defining} but after the change of variables and show that the solution depends analytically on both $\eta$ and $\psi$. Furthermore, we need to find a suitable function space on which the solution lives. Here we will use the perturbative argument; (iii) we go back to the operator $G(\eta)$ by taking trace and track all the estimates of the norms.

\paragraph{Function spaces.} In this subsection we collect basic properties of the function spaces that we will use in the proof. Recall that the space $H^{\alpha, m}$ is defined in \eqref{defHss}. When $m > \frac{d}{2}$, it is well-known that $H^{\alpha, m}$ is a Banach algebra. Moreover, the following tame estimate holds (cf. \cite[Lemma B.2]{BMV2}). 

\begin{lemma}[Tame]\label{lemT}
Let $\alpha \ge 0$ and $m \ge m_0 > \frac{d}{2}$. There exist positive constants $C_{m}$ (increasing in $m$) such that for any $f,g \in H^{\alpha, m}$, one has
\begin{align}\label{tamee}
    \|fg\|_{H^{\alpha, m}} \le C_m (\|f\|_{H^{\alpha, m}} \|g\|_{H^{\alpha, m_0}} + \|f\|_{H^{\alpha, m_0}} \|g\|_{H^{\alpha, m}}).
\end{align}
In particular, for any $j \ge 1$,
\begin{align}\label{tamee2}
    \|f^j\|_{H^{\alpha, m}} \le (2 C_m \|f\|_{H^{\alpha, m_0}})^{j - 1} \|f\|_{H^{\alpha, m}}.
\end{align}
\end{lemma}

For the solution after flattening, for given $\alpha \ge 0$ and $m \in \N$, we introduce the following space:
\begin{align}
    \label{defcHss}
    \cH^{\alpha,m} := \left\{ u(x,y) = \sum_{k \in \Z^d} u_k(y) e^{\im k \cdot x} : \T^d \times (-\tth, 0 ) \to \mathbb{C}: \|u\|_{\cH^{\alpha,m}} < +\infty\right\},
\end{align}
where
\begin{align}
    \label{defnorm}
  \|u\|_{\cH^{\alpha,m}}^2 := \sum_{j = 0}^m \|\partial_y^j u \|_{L^2((-\tth,0), H^{\alpha, m - j})}^2
  = \sum_{j = 0}^m \sum_{k \in \Z^d} e^{2 \alpha |k|_1} \langle k \rangle^{2(m - j)}\int_{-\tth}^0 |\partial_y^j u_k(y) |^2 \de y.
\end{align}

\begin{remark}
Since $\alpha \ge 0$ and $m \in \N$ and we observe that $\cH^{0,0} = L^2(\T^d \times (-\tth,0))$, we know that $\cH^{\alpha,m} \subset L^2(\T^d \times (-\tth,0))$. In the definition there are two weak derivatives: one appears as the weak derivative of $u$ and the other is the weak derivative of $u_k$. In fact, we have the following relation:
\[
(\partial^j_y u)_k = \partial^j_y (u_k), \qquad \forall \, k \in \Z^d, j \in \N
\]
so there is no ambiguity if we write $\partial^j_y u_k$.
\end{remark}

\begin{remark}
In \cite{BMV2} the authors use $\cH^{\alpha,m,a}$ (with a exponential decay $e^{-a|y|}$ at infinity) instead of our $\cH^{\alpha,m}$. Since we are in the finite depth case, there is no need to consider the decay at infinity, which simplifies the space a lot. Furthermore, the constant function is also included in our spaces so there is no need to use the space $\mathbb{C} \oplus \cH^{\alpha,m,a}$ and the corresponding projection any more.
\end{remark}

Here are some properties of the space $ \cH^{\alpha,m}$.

\begin{lemma}
    \label{lemder}
    Let $\alpha \ge 0$ and $m \in \N$. The linear maps
    \[
    \partial_{x_j}, \quad \forall \, j = 1, \ldots, d, \qquad \mbox{and} \qquad \partial_y
    \]
    are bounded from $\cH^{\alpha,m}$ to $\cH^{\alpha,m - 1}$.
\end{lemma}

\begin{lemma}[Trace]
    \label{lemtra}
    Let $\alpha \ge 0$ and $m \in \R$. Then there exists a constant $C > 0$ such that
    \[
    \sup_{y \in (-\tth,0)}\|u(\cdot,y)\|_{H^{\alpha,m}} \le C \left( \|u\|_{L^2((-\tth,0), H^{\alpha,m + \frac{1}{2}})} + \|\partial_y u\|_{L^2((-\tth,0), H^{\alpha,m - \frac{1}{2}})}\right).
    \]
    In particular, the trace operators
    \begin{align}\label{deftra}
    \Gamma(u) := u(\cdot, 0) := u|_{y = 0}, \qquad \widetilde{\Gamma}(u) := u(\cdot, -\tth) := u|_{y = -\tth}
    \end{align}
    are, for any $m \in \N$, a bounded map between $\cH^{\alpha,m + 1} \to H^{\alpha, m + \frac{1}{2}}$, satisfying
    \begin{align}
        \|\Gamma(u)\|_{H^{\alpha, m + \frac{1}{2}}} + \|\widetilde{\Gamma}(u)\|_{H^{\alpha, m + \frac{1}{2}}} \le C' \|u\|_{\cH^{\alpha,m + 1}} 
    \end{align}
    for some $C' > 0$.
\end{lemma}

\begin{proof}
It suffices to prove the first assertion since the rest part follows from it. In the following we focus on the points on $(-\frac{3}{4} \tth,0 )$ and for the points belonging to $(-\tth, - \frac{1}{4}\tth) $ the proof is similar with a different cutoff function. Fix a cutoff function $\chi \in C^\infty_c(\R)$ such that $\chi$ is $1$ on $(-\frac{3}{4} \tth,0 )$ and $0$ on $(-\infty, -\frac{7}{8} \tth)$. Now for any $u \in C^\infty(\T^d \times (-\tth, 0))$. It is not hard to prove that there exists a constant $C_* > 0$ such that 
\[
 \|u \chi\|_{L^2((-\tth,0), H^{\alpha,m + \frac{1}{2}})} + \|\partial_y (u \chi)\|_{L^2((-\tth,0), H^{\alpha,m - \frac{1}{2}})} \le C_* \left( \|u\|_{L^2((-\tth,0), H^{\alpha,m + \frac{1}{2}})} + \|\partial_y u\|_{L^2((-\tth,0), H^{\alpha,m - \frac{1}{2}})}\right).
\]
Therefore, without loss of generality, we can assume that the function $u$ is $0$ for $y$ (so as $u_k$) near $-\tth$, which allows us to write
\[
|u_k(y_0)|^2 = \int_{-\tth}^{y_0} \partial_y |u_k(y)|^2 \de y \le 2 \int_{-\tth}^{0}  |\partial_y u_k(y)| |u_k(y)| \de y\,,
\]
which implies
\begin{align*}
&\|u(\cdot,y_0)\|_{H^{\alpha,m}}^2 \le  \sum_{k \in \Z^d} e^{2\alpha |k|_1} 2 \langle  k \rangle^{2m} \int_{-\tth}^{0}  |\partial_y u_k(y)| |u_k(y)| \de y \\
\le \, &  \sum_{k \in \Z^d} e^{2\alpha |k|_1} \langle k \rangle^{2m + 1}  \int_{-\tth}^{0} |u_k(y)|^2 \de y + \sum_{k \in \Z^d} e^{2\alpha |k|_1} \langle k \rangle^{2m - 1}  \int_{-\tth}^{0} |\partial_y u_k(y)|^2 \de y  \\
=  \, &  \|u\|_{L^2((-\tth,0), H^{\alpha,m + \frac{1}{2}})}^2 + \|\partial_y u\|_{L^2((-\tth,0), H^{\alpha,m - \frac{1}{2}})}^2,
\end{align*}
where we have used the inequality $2 \langle  k \rangle^{2m} |\partial_y u_k(y)| |u_k(y)| \le \langle k \rangle^{2m + 1} |u_k(y)|^2 + \langle k \rangle^{2m - 1} |\partial_y u_k(y)|^2$. Finally, the general estimate holds by approximation by smooth functions.
\end{proof}

Similar to the space $H^{\alpha,m}$, when $m > \frac{d + 1}{2}$, the space $\cH^{\alpha,m}$ is also a Banach algebra and similar tame estimates hold.

\begin{proposition}[Tame]\label{propT}
Let $\alpha \ge 0$ and $m \ge m_0 > \frac{d + 1}{2}$, $m,m_0 \in \N$. There exist positive constants $C_{m}$ (increasing in $m$) such that for any $u,v \in \cH^{\alpha, m}$, one has
\begin{align}\label{tamee3}
    \|uv\|_{\cH^{\alpha, m}} \le C_m (\|u\|_{\cH^{\alpha, m}} \|v\|_{\cH^{\alpha, m_0}} + \|u\|_{\cH^{\alpha, m_0}} \|v\|_{\cH^{\alpha, m}}).
\end{align}
In particular, for any $j \ge 1$,
\begin{align}\label{tamee4}
    \|u^j\|_{\cH^{\alpha, m}} \le (2 C_m \|u\|_{\cH^{\alpha, m_0}})^{j - 1} \|u\|_{\cH^{\alpha, m}}.
\end{align}
\end{proposition}

To give a proof of Proposition \ref{propT} we follow the strategy of \cite{BMV2}: we first extend the 
functions $u,v$ to $\T^d \times \R$, keeping the Sobolev norm up to a constant uniformly. On $\T^d \times \R$ we can use the Fourier transform to characterize such space which leads to the estimate above for the Sobolev spaces on $\T^d \times \R$. Since the extension will not enlarge the norm too much, we can go back to the original space. 

To realize the proof, we introduce the following space: given $\alpha \ge 0$ and $m \in \N$\
\begin{align}
    \label{defcHss2}
    \cH^{\alpha,m,\R} := \left\{ u(x,y) = \sum_{k \in \Z^d} u_k(y) e^{\im k \cdot x} : \T^d \times \R \to \mathbb{C}: \|u\|_{\cH^{\alpha,m,\R}} < +\infty\right\},
\end{align}
where
\begin{align}
    \label{defnorm2}
  \|u\|_{\cH^{\alpha,m,\R}}^2 := 
  \sum_{j = 0}^m \sum_{k \in \Z^d} e^{2 \alpha |k|_1} \langle k \rangle^{2(m - j)}\int_\R |\partial_y^j u_k(y) |^2 \de y.
\end{align}

\begin{remark}
Using Fourier transform, we can write down an equivalent norm as follows
\[
\sqrt{\sum_{k \in \Z^d} \int_\R e^{2 \alpha |k|_1} \langle k,\xi \rangle^{2m} |\widehat{u}_k(\xi)|^2 \de \xi},
\]
where
\[
\langle k,\xi \rangle = \sqrt{1 + |k|^2 + |\xi|^2}
\]
and $\widehat{u}_k$ can be understood as both the Fourier transform of $u_k$ or the $k$-th Fourier coefficient (w.r.t. $x$-variable) of the partial Fourier transform w.r.t. $y$ variable.
\end{remark}

Then we need the following extension lemma. 

\begin{lemma}[Extension operator]
    \label{lemext}
Let $\alpha \ge 0$ and $m \in \N$. There exists a linear bounded extension operator $\mathcal{E}_m : \cH^{\alpha,m} \to \cH^{\alpha,m,\R}$ such that $\mathcal{E}_m u = u$ a.e. on $(-\tth,0)$. Moreover, there exists a constant $C > 0$ such that 
\begin{align} \label{extest}
    \|u\|_{\cH^{\alpha,m}} \le \|\mathcal{E}_m u\|_{\cH^{\alpha,m,\R}} \le C \|u\|_{\cH^{\alpha,m}}.
\end{align}
\end{lemma}

\begin{proof}
Fix $\alpha \ge 0$ and $m \in \N$. In this proof, we define another function space 
\begin{align}
    \label{defcHss3}
    \cH^{\alpha,m,*} := \left\{ u(x,y) = \sum_{k \in \Z^d} u_k(y) e^{\im k \cdot x} : \T^d \times \left(-\tth - \frac{\tth}{m + 1}, \frac{\tth}{m + 1} \right) \to \mathbb{C}: \|u\|_{\cH^{\alpha,m,*}} < +\infty\right\},
\end{align}
where
\begin{align}
    \label{defnorm3}
  \|u\|_{\cH^{\alpha,m,*}}^2 := 
  \sum_{j = 0}^m \sum_{k \in \Z^d} e^{2 \alpha |k|_1} \langle k \rangle^{2(m - j)}\int_{-\tth - \frac{\tth}{m + 1}}^{\frac{\tth}{m + 1}} |\partial_y^j u_k(y) |^2 \de y.
\end{align}
Our method for extension is still the standard reflection, but we need to take care of two boundaries at the same time. To be more precise, for any function $u \in C_c^\infty(\R, H^{\alpha,m})$, we define
\begin{align}
    \label{defext}
    (\widetilde{\mathcal{E}}_m u)(y) = \begin{cases}
        \sum_{j = 0}^m \alpha_j^{(m)} u (- (j + 1) y), \qquad & y \in \left(0, \frac{\tth}{m + 1}\right), \\
        u(y), \qquad & y \in [-\tth,0], \\
        \sum_{j = 0}^m \beta_j^{(m)} u ((j + 1) (y + \tth)), \qquad & y \in \left(- \tth -\frac{\tth}{m + 1},-\tth\right),
    \end{cases}
\end{align}
where the coefficients $\alpha_j^{(m)},  \beta_j^{(m)}, 0 \le j \le m$ are determined by the following relations
\[
\partial_y^j (\widetilde{\mathcal{E}}_m u) (0) = \partial_y^j u (0), \qquad \partial_y^j (\widetilde{\mathcal{E}}_m u)(-\tth) = \partial_y^j u (-\tth), \qquad \forall \, j = 0,\ldots, m.
\]
This is possible because for example for the choice of $\alpha_j^{(m)}, 0 \le j \le m$ we need to solve the following linear system
\[
\begin{bmatrix}
    1 & 1 & \cdots & 1 \\
    -1 & -2 & \cdots & -m - 1 \\
    \vdots & \vdots & \ddots & \vdots \\
    (-1)^m & (-2)^m & \cdots & (- m - 1)^m
\end{bmatrix}
\begin{bmatrix}
   \alpha_0^{(m)} \\
   \alpha_2^{(m)} \\
   \vdots \\
   \alpha_m^{(m)}
\end{bmatrix}
= 
\begin{bmatrix}
    1 \\
    1\\
    \vdots \\
    1
\end{bmatrix},
\]
which is always solvable since the coefficient matrix is a Vandermonde matrix. Furthermore, by a direct calculation, one can find that there is a constant $C_{*} > 0$ such that 
\[
\|\widetilde{\mathcal{E}}_m u\|_{\cH^{\alpha,m,*}} \le C_* \|u\|_{\cH^{\alpha,m}}.
\]
Now we fix a cutoff function $\chi \in C_c^\infty(\R)$ such that $\chi$ is $1$ on $[-\tth,0]$ and $0$ outside $\left[ -\tth - \frac{\tth}{2(m + 1)}, \frac{\tth}{2(m + 1)} \right]$ and define 
\[
\mathcal{E}_m u := \chi \cdot \widetilde{\mathcal{E}}_m u.
\]
It is not hard to show that $\mathcal{E}_m u \in \cH^{\alpha,m,\R}$ and \eqref{extest} holds.
\end{proof}

Recall the following tame estimate of $\cH^{\alpha,m,\R}$ established in \cite[Lemma B.5]{BMV2}.

\begin{lemma}
    \label{lemtamlemB}
    Let $\alpha \ge 0$ and $m,m_0 \in \N$ with $m \ge m_0 > \frac{d + 1}{2}$. Then there exists a constant $K > 0$ such that 
    \[
    \|uv\|_{\cH^{\alpha, m,\R}} \le K (\|u\|_{\cH^{\alpha, m,\R}} \|v\|_{\cH^{\alpha, m_0,\R}} + \|u\|_{\cH^{\alpha, m_0,\R}} \|v\|_{\cH^{\alpha, m,\R}}).
    \]
\end{lemma}

\begin{proof}[Proof of Proposition \ref{propT}]
    Since $\mathcal{E}_m u$ and $\mathcal{E}_m v$ are extensions of $u$ and $v$, we have
\begin{align*}
    &\|uv\|_{\cH^{\alpha, m}} \le \|\mathcal{E}_m u \, \mathcal{E}_m v\|_{\cH^{\alpha, m,\R}} \le 
    K (\|\mathcal{E}_m u\|_{\cH^{\alpha, m,\R}} \|\mathcal{E}_m v\|_{\cH^{\alpha, m_0,\R}} + \|\mathcal{E}_m u\|_{\cH^{\alpha, m_0,\R}} \|\mathcal{E}_m v\|_{\cH^{\alpha, m,\R}}) \\
    \le \, & K'  (\|u\|_{\cH^{\alpha, m}} \|v\|_{\cH^{\alpha, m_0}} + \|u\|_{\cH^{\alpha, m_0}} \|v\|_{\cH^{\alpha, m}}),
\end{align*}
where in the second ``$\le$'' we have used Lemma \ref{lemtamlemB} and in the last ``$\le$'' we have used Lemma \ref{lemext}. Finally we chose $K'$ in a way that it is increasing in $m$ and we end the proof.
\end{proof}

\paragraph{Flattening diffeomorphism.} As mentioned above, we introduce the following change of variables
\[
x = x', \qquad y = \rho(x',y'), \quad \mbox{where} \quad \rho(x',y') := y' + \frac{\sinh((y' + \tth )|D|)}{\sinh(\tth |D|)} \eta(x').
\]
Here $\frac{\sinh((y' + \tth )|D|)}{\sinh(\tth |D|)}$ denotes the Fourier multiplier
\begin{align}\label{defsinh}
\frac{\sinh((y' + \tth )|D|)}{\sinh(\tth |D|)} g(x) := \sum_{k \in \Z^d} g_k \frac{\sinh((y' + \tth )|k|)}{\sinh(\tth |k|)} e^{\im k \cdot x}, \qquad \mbox{if} \quad g(x) = \sum_{k \in \Z^d} g_k e^{\im k \cdot x}.
\end{align}
Similarly, we can also define the Fourier multiplier
\begin{align}\label{defcosh}
\frac{\cosh((y' + \tth )|D|)}{\cosh(\tth |D|)} g(x) := \sum_{k \in \Z^d} g_k \frac{\cosh((y' + \tth )|k|)}{\cosh(\tth |k|)} e^{\im k \cdot x}, \qquad \mbox{if} \quad g(x) = \sum_{k \in \Z^d} g_k e^{\im k \cdot x}.
\end{align}

Let us take it for granted that $\rho$ is $C^2$ and for every $x' \in \T^d$, we have $\partial_{y'} \rho (x',y') > 0$. In fact, in the following we will choose a suitable function space with small norm of $\eta$ to realize it (we will not mention it later but one can check this fact). From our construction, we know that for fixed $x' \in \T^d$, $\rho^{-1}$ maps the point $(x', \eta(x'))$ in the upper boundary $\{(x,y) \in \T^d \times \R : y = \eta(x)\}$ to $(x',0)$ and fixes the point $(x',-\tth)$ in the lower boundary. In conclusion, this change of variables is a $C^2$-diffeomorphism and its inverse maps $\Omega_{\tth,\eta}$ to $\Omega_{\tth,0}$. Furthermore, the derivatives $\partial_y$ and $\nabla_x$ become
\[
\Lambda_1 = \frac{1}{\partial_{y'} \rho}\partial_{y'}, \qquad \Lambda_2  = \nabla_{x'} - \frac{\nabla_{x'} \rho}{\partial_{y'} \rho} \partial_{y'}.
\]
Set $\varphi(x',y') = \Psi(x', \rho(x',y'))$. After the change of variables, \eqref{defining} becomes (to ease the notation in the following we still use $x,y$ instead of $x',y'$) the problem
\begin{align}\label{defining2}
\begin{cases}
    (\Lambda_1^2 + \Lambda_2^2) \varphi = 0 \quad & \mbox{in} \ \Omega_{\tth,0} \\
    \varphi(x,0) = \psi(x) \quad & \mbox{at} \ y = 0 \\
    \partial_y \varphi(x,y)  =0  \quad & \mbox{at} \ y = -\tth.
\end{cases}
\end{align}

For our proof we can further rewrite it as the following perturbed elliptic problem
\begin{align}\label{defining3}
\begin{cases}
    \Delta_{x,y} \varphi = F(\eta)\varphi \quad & \mbox{in} \ \Omega_{\tth,0} \\
    \varphi(x,0) = \psi(x) \quad & \mbox{at} \ y = 0 \\
    \partial_y \varphi(x,y)  =0  \quad & \mbox{at} \ y = -\tth,
\end{cases}
\end{align}
where
\begin{align}
    \label{defFe}
    F(\eta)\varphi := (\alpha(\eta) \partial_y^2 + \beta(\eta)\Delta_x + \gamma(\eta) \cdot \nabla_x \partial_y + \delta(\eta) \partial_y) \varphi
\end{align}
with
\begin{equation}\label{defcoeff}
    \begin{aligned}
        \alpha(\eta) &:= 1 - \frac{1 + |\nabla_x \rho|^2}{\partial_y \rho}, \\
        \beta(\eta) &:= 1 - \partial_y \rho, \\
        \gamma(\eta) &:= 2 \nabla_x \rho, \\
        \delta(\eta) &:= \frac{1}{\partial_y \rho} \left( - 2 \nabla_x \rho \cdot \nabla_x \partial_y \rho + \partial_y \rho \Delta_x \rho +  \frac{1 + |\nabla_x \rho|^2}{\partial_y \rho} \partial_y^2 \rho\right).
    \end{aligned}
\end{equation}

And the Dirichlet--Neumann operator becomes
\begin{align}
    \label{defDNO}
    G(\eta)\psi = -\nabla_x \eta \cdot \nabla_x\varphi(\cdot, 0) + \frac{1 + |\nabla_x\eta|^2(\cdot)}{1 + (|D| \coth(h|D|)\eta)(\cdot)} (\partial_y \varphi)(\cdot, 0).
\end{align}

Finally we give the following mapping property.

\begin{lemma}[Harmonic propagator]
    \label{lemHp}
    Let $\alpha \ge 0$ and $m + \frac{1}{2} \in \N$. Given $g \in H^{\alpha
    ,m}$, write
    \[
    g_c(x,y) := \frac{\cosh((y + \tth )|D|)}{\cosh(\tth |D|)} g(x), \qquad g_s(x,y) := \frac{\sinh((y + \tth )|D|)}{\sinh(\tth |D|)} g(x).
    \]
    Then the following estimates hold:
    \[
    \|g_s\|_{\cH^{\alpha, m + \frac{1}{2}}} + \|g_c\|_{\cH^{\alpha, m + \frac{1}{2}}} \le C \|g\|_{H^{\alpha,m}}, \qquad \forall \, g \in H^{\alpha,m},
    \]
    for some $C > 0$.
\end{lemma}

\begin{proof}
We only provide the proof for $g_c$ and the proof for $g_s$ is similar. Since $|\sinh(z)| \le \cosh(z)$, we have
\begin{align*}
   & \|g_c\|^2_{\cH^{\alpha,m + \frac{1}{2}}} = \sum_{j = 0}^{m + \frac{1}{2}} \sum_{k \in \Z^d} e^{2 \alpha |k|_1} \langle k \rangle^{2(m - j) + 1} \int_{-\tth}^0 |\partial_y^j \cosh((y + \tth)|k|)|^2 \frac{|g_k|^2}{|\cosh(\tth |k|)|^2} \de y \\
   \le \, & \sum_{j = 0}^{m + \frac{1}{2}} \sum_{k \in \Z^d} e^{2 \alpha |k|_1} \langle k \rangle^{2m  + 1} \int_{-\tth}^0 | \cosh((y + \tth)|k|)|^2 \frac{|g_k|^2}{|\cosh(\tth |k|)|^2} \de y \\
   \le \, & \sum_{j = 0}^{m + \frac{1}{2}} \sum_{k \in \Z^d} e^{2 \alpha |k|_1} \langle k \rangle^{2m  + 1} \frac{|g_k|^2}{|\cosh(\tth |k|)|^2} \frac{1}{2|k|} (\cosh(\tth |k|) \sinh(\tth |k|) + \tth |k|) \\
   \le \, & C' \sum_{j = 0}^{m + \frac{1}{2}} \sum_{k \in \Z^d} e^{2 \alpha |k|_1} \langle k \rangle^{2m  } |g_k|^2  \le C'' \|g\|^2_{H^{\alpha, m}}.
\end{align*}
\end{proof}

\paragraph{Perturbative argument.} We are now ready to solve the equation \eqref{defining3}. In fact, by a perturbative argument we decompose \eqref{defining3} into the following two problems. To be more precise, we write the solution $\varphi$ in \eqref{defining3} as
\[
\varphi(x,y) =  u(x,y) + v(x,y),
\]
where the function $v$ is the solution of 
\begin{align}\label{defining4}
\begin{cases}
    \Delta_{x,y} v = 0 \quad & \mbox{in} \ \Omega_{\tth,0}, \\
    v(x,0) = \psi(x) \quad & \mbox{at} \ y = 0, \\
    \partial_y v(x,y)  =0  \quad & \mbox{at} \ y = -\tth,
\end{cases}
\end{align}
and the function $u$ is the solution of 
\begin{align}\label{defining5}
\begin{cases}
    \Delta_{x,y} u = F(\eta)(v + u) \quad & \mbox{in} \ \Omega_{\tth,0}, \\
    u(x,0) = 0 \quad & \mbox{at} \ y = 0, \\
    \partial_y u(x,y)  =0  \quad & \mbox{at} \ y = -\tth.
\end{cases}
\end{align}

For \eqref{defining4} it is not hard to show that the solution (similar to the proof of Lemma \ref{lemER} below and we omit the details) is given by 
\begin{align}
    \label{repv}
    v(x,y) = \psi_c(x,y) := \frac{\cosh((y + \tth )|D|)}{\cosh(\tth |D|)} \psi(x).
\end{align}

With the information of $v$ at hand, we are in a position to solve the function $u$ in \eqref{defining5}. To this end, we first consider the following linear elliptic problem
\begin{align}\label{defining6}
\begin{cases}
    \Delta_{x,y} u = g \quad & \mbox{in} \ \Omega_{\tth,0}, \\
    u(x,0) = 0 \quad & \mbox{at} \ y = 0, \\
    \partial_y u(x,y)  =0  \quad & \mbox{at} \ y = -\tth.
\end{cases}
\end{align}

One key property of the solution of \eqref{defining6} is the following.

\begin{lemma}[Elliptic Regularity]
\label{lemER}
Let $\alpha \ge 0$ and $m \in \N$. For any $g \in \cH^{\alpha,m}$, the elliptic problem \eqref{defining6} has a unique solution $L g \in \cH^{\alpha,m + 2}$. Furthermore, there exists a constant $C > 0$ such that 
\[
\|L g\|_{\cH^{\alpha,m + 2}} \le C \|g\|_{\cH^{\alpha,m }}, \qquad \forall \, g \in \cH^{\alpha,m }.
\]
\end{lemma}

\begin{proof}
We write
\[
u(x,y)  = \sum_{k \in \Z^d} u_k(y) e^{\im k \cdot x}, \qquad g(x,y)  = \sum_{k \in \Z^d} g_k(y) e^{\im k \cdot x}.
\]
Inserting this into \eqref{defining6}, we obtain that for any $k \in \Z^d$, the following equation holds:
\begin{align}\label{defining7}
\begin{cases}
    u_k''(y) - |k|^2 u_k(y) = g_k(y),  \\
    u_k(0) = 0,  \\
    u_k'(-\tth)  =0.
\end{cases}
\end{align}

Regarding \eqref{defining7} we first solve the case $g_k = 0$ and use the variation of constants method, the solution can be uniquely written as 
\begin{align}
    \label{repuk}
    u_k(y) = \int_{-\tth}^y \frac{\cosh((t + \tth)|k|)}{\cosh(\tth |k|)} \frac{\sinh(y |k|)}{|k|} g_k(t) \de t + \int_y^0 \frac{\cosh((y + \tth)|k|)}{\cosh(\tth |k|)} \frac{\sinh(t |k|)}{|k|} g_k(t) \de t.
\end{align}
As a result, we can take derivative and obtain
\begin{align}
    \label{repukp}
    u_k'(y) = \int_{-\tth}^y \frac{\cosh((t + \tth)|k|)}{\cosh(\tth |k|)} \cosh(y |k|) g_k(t) \de t + \int_y^0 \frac{\sinh((y + \tth)|k|)}{\cosh(\tth |k|)} \sinh(t |k|) g_k(t) \de t.
\end{align}

In the following to ease the notation, we use the notation $A \lesssim B$ if there exists a constant $C > 0$ (independent of $k$, $t$, $x$ and $y$ but may depends on $\tth$, $\alpha$ and $m,m_0$) such that $A \le C B$. The value of the implicit constant may change from one line to the next. We also write the first term in \eqref{repuk} as $I_k^1(y)$ and the second term \eqref{repuk} as $I_k^2(y)$. 

We claim that 
\begin{align}
    \label{claimapp}
    \int_{-\tth}^0 |u_k(y)|^2 \de y \lesssim \frac{1}{\langle k \rangle^4} \int_{-\tth}^0 |g_k(y)|^2 \de y, \qquad \int_{-\tth}^0 |u_k'(y)|^2 \de y \lesssim \frac{1}{\langle k \rangle^2} \int_{-\tth}^0 |g_k(y)|^2 \de y.
\end{align}

We only focus on the first inequality of \eqref{claimapp} and the proof for the other is almost the same. In fact, it suffices to show that 
\begin{align}
    \label{claimapp2}
    \int_{-\tth}^0 |I_k^1(y)|^2 \de y \lesssim \frac{1}{\langle k \rangle^4} \int_{-\tth}^0 |g_k(y)|^2 \de y, \qquad  \int_{-\tth}^0 |I_k^2(y)|^2 \de y \lesssim \frac{1}{\langle k \rangle^4} \int_{-\tth}^0 |g_k(y)|^2 \de y.
\end{align}

We split the proof into two cases.

\paragraph{Case 1: $k = 0$.} In such case it follows from the Cauchy--Schwarz inequality that 
\[
|I_0^1(y)|^2 + |I_0^2(y)|^2 \lesssim  \int_{-\tth}^0 |g_0(y)|^2 \de y
\]
so the estimate \eqref{claimapp2} is direct.

\paragraph{Case 2: $k \ne 0$.} We first deal with $I_k^1$. Notice that if $-\tth \le t \le y \le 0$, we have 
$(t + \tth)|k| \ge 0$ and $y |k| \le 0$, which implies
\[
\frac{\cosh((t + \tth)|k|)}{\cosh(\tth |k|)} \sinh(y |k|) \lesssim e^{(t - y) |k|}
\]
and thus by the Cauchy--Schwarz inequality we obtain
\begin{align*}
& \int_{-\tth}^0 |I_k^1(y)|^2 \de y \lesssim \frac{1}{|k|^2} \int_{-\tth}^0 \left( \int_{-\tth}^y e^{(t - y) |k|} |g(t)| \de t\right)^2 \de y   \\
 \le \, & \frac{1}{|k|^3}  \int_{-\tth}^0  \int_{-\tth}^y e^{(t - y) |k|}  |g(t)|^2 \de t \de y =  \frac{1}{|k|^3}   \int_{-\tth}^0 \int_t^0 e^{(t - y) |k|} \de y  |g(t)|^2 \de t 
 \le  \frac{1}{|k|^4}\int_{-\tth}^0 |g(t)|^2 \de t,
\end{align*}
which is our desired estimate. The treatment of $I_k^2$ is similar. In fact, if $-\tth \le y \le t \le 0$, we have 
$(y + \tth)|k| \ge 0$ and $t |k| \le 0$, which implies
\[
\frac{\cosh((y + \tth)|k|)}{\cosh(\tth |k|)} \sinh(t |k|) \lesssim e^{(y - t) |k|}
\]
and thus again by the Cauchy--Schwarz inequality we obtain
\begin{align*}
& \int_{-\tth}^0 |I_k^2(y)|^2 \de y \lesssim \frac{1}{|k|^2} \int_{-\tth}^0 \left( \int_y^0 e^{(y - t) |k|} |g(t)| \de t\right)^2 \de y   \\
 \le \, & \frac{1}{|k|^3}  \int_{-\tth}^0  \int_y^0 e^{(y - t) |k|}  |g(t)|^2 \de t \de y =  \frac{1}{|k|^3}   \int_{-\tth}^0 \int_{-\tth}^t e^{(y - t) |k|} \de y  |g(t)|^2 \de t 
 \le  \frac{1}{|k|^4}\int_{-\tth}^0 |g(t)|^2 \de t.
\end{align*}
So we have proved the claim \eqref{claimapp}. Now we turn to the proof of Lemma \ref{lemER}. Note that from the equation \eqref{defining7} we can obtain 
\begin{align}
    \label{deruk}
    u_k^{(\ell)}(y) = \sum_{j = 1}^{\lfloor \frac{\ell}{2} \rfloor} |k|^{2j - 2} g^{(\ell - 2j)}_k(y) + |k|^{2 \lfloor \frac{\ell}{2} \rfloor} u_k^{(\ell - 2 \lfloor \frac{\ell}{2} \rfloor)}(y).
\end{align}
Here $\lfloor \cdot \rfloor$ is the floor function. Using \eqref{deruk}, we can prove that for $\ell \ge 2$,
\begin{align}
    \label{deruk2}
    \langle k \rangle^{2(m + 2 - \ell)} \int_{-\tth}^0 |u_k^{(\ell)}(y)|^2 \de y \lesssim \sum_{j  =0}^\ell  \langle k \rangle^{2(m - j)}  \int_{-\tth}^0 |g_k^{(j)}(y)|^2 \de y.
\end{align}
Then multiplying $e^{2\alpha |k|_1}$ on both side and summing over $k \in \Z^d$ and $\ell \in \{2,\ldots, m + 2\}$ (together with \eqref{claimapp}) yield our desired estimate. To prove \eqref{deruk2} we only give the proof for the case $\ell = 2N$ with $N \ge 1$. Then \eqref{deruk} becomes
\begin{align}
    \label{deruk3}
    u_k^{(2N)}(y) = \sum_{j = 1}^{N} |k|^{2j - 2} g^{(2N - 2j)}_k(y) + |k|^{2 N} u_k(y),
\end{align}
which implies
\begin{align*}
 \langle k \rangle^{2(m + 2 - 2N)}    \int_{-\tth}^0 |u_k^{(2N)}(y)|^2 \de y \lesssim 
  \sum_{j = 1}^{N} \langle k \rangle^{2(m - 2(N - j))} \int_{-\tth}^0 | g^{(2N - 2j)}_k(y)|^2 \de y + \langle k \rangle^{2(m + 2)} \int_{-\tth}^0 |u_k (y)|^2 \de y.
\end{align*}
Then an application of \eqref{claimapp} gives the estimate \eqref{deruk2}.
\end{proof}

\paragraph{The proof of Theorem \ref{thmG}.} We begin with the last preparation that the map $\eta \mapsto F(\eta)$ is analytic. 

\begin{lemma}
    \label{lemanaF}
    Let $\alpha \ge 0$ and $m + \frac{1}{2}, m_0 \in \N$ with $m - \frac{3}{2} \ge m_0 > \frac{d + 1}{2}$. Then there exist $\varepsilon_0 > 0$ and $C > 0$ such that the nonlinear map
\begin{align*}
F : H^{\alpha,m} \cap B^{\alpha,m_0 + \frac{3}{2}} (\varepsilon_0) &\to  \mathcal{L}(\cH^{\alpha, m + \frac{1}{2}},\cH^{\alpha, m- \frac{3}{2}})\\
\eta &\mapsto F(\eta)
\end{align*}
is analytic and the following estimate holds
\begin{align}\label{estF}
    \|F(\eta)\theta\|_{\cH^{\alpha, m- \frac{3}{2}}} \le C (\|\eta\|_{H^{\alpha,m_0 + \frac{3}{2}}} \|\theta\|_{\cH^{\alpha, m + \frac{1}{2}}}  + \|\eta\|_{H^{\alpha,m}} \|\theta\|_{\cH^{\alpha, m_0 + 2}}).
\end{align}
\end{lemma}

\begin{proof}
Recall the definition of $F(\eta)$ in \eqref{defFe} and we can thus rewrite it as 
\begin{align*}
    F(\eta)\theta = \mathcal{F}_1(\alpha(\eta),\theta) + \mathcal{F}_2(\beta(\eta),\theta) + \sum_{j = 1}^d \mathcal{F}_{3j} (\gamma_j(\eta),\theta) + \mathcal{F}_4(\delta(\eta),\theta),
\end{align*}
with the bilinear maps
\begin{align} \label{defbilin}
    \mathcal{F}_1(g,\theta) := g \partial_y^2 \theta, \qquad 
    \mathcal{F}_2(g,\theta) := g \Delta_x \theta, \qquad 
    \mathcal{F}_{3j}(g,\theta) := g \partial_{x_j} \partial_y\theta, \quad \forall \, j = 1, \ldots, d, \qquad 
    \mathcal{F}_4(g,\theta) := g \partial_y \theta.
\end{align}
Then it follows from Proposition \ref{propT} and Lemma \ref{lemder} that for every $\mathcal{F} \in \{\mathcal{F}_1, \mathcal{F}_2, \mathcal{F}_{31}, \ldots, \mathcal{F}_{3d}, \mathcal{F}_4\}$ the following tame estimates hold:
\[
\|\mathcal{F}(g,\theta)\|_{\cH^{\alpha, m- \frac{3}{2}}} \lesssim (\|g\|_{\cH^{\alpha,m_0 }} \|\theta\|_{\cH^{\alpha, m + \frac{1}{2}}}  + \|g\|_{\cH^{\alpha,m - \frac{3}{2}}} \|\theta\|_{\cH^{\alpha, m_0 + 2}}).
\]
Since for every $\mathcal{F} \in \{\mathcal{F}_1, \mathcal{F}_2, \mathcal{F}_{31}, \ldots, \mathcal{F}_{3d}, \mathcal{F}_4\}$, $g \mapsto \mathcal{F}(g,\cdot)$ is linear, thus analytic. It suffices to prove that the maps
\begin{align*}
    H^{\alpha,m} \cap B^{\alpha, m_0 + \frac{1}{2}}(\varepsilon_0) &\to \cH^{\alpha, m- \frac{1}{2}}, \qquad \eta \mapsto \alpha(\eta),\beta(\eta),\gamma_j(\eta), \quad \forall \, j = 1, \ldots, d, \\
    H^{\alpha,m} \cap B^{\alpha, m_0 + \frac{3}{2}}(\varepsilon_0) &\to \cH^{\alpha, m- \frac{3}{2}}, \qquad \eta \mapsto \delta(\eta)\,,
\end{align*}
are analytic and the following estimate holds:
\begin{align}\label{targete}
   \|\alpha(\eta)\|_{\cH^{\alpha, m- \frac{1}{2}}} +  \|\beta(\eta)\|_{\cH^{\alpha, m- \frac{1}{2}}} + \sum_{j = 1}^d \|\gamma_j(\eta)\|_{\cH^{\alpha, m- \frac{1}{2}}} + \|\delta(\eta)\|_{\cH^{\alpha, m- \frac{3}{2}}} \lesssim \|\eta\|_{H^{\alpha, m}}. 
\end{align}

Recall that the those maps are defined in \eqref{defcoeff}. For the proof of $\alpha(\eta)$ we rewrite it as 
\[
\alpha(\eta) = 1 - \frac{1}{\partial_y \rho} + \left( 1 - \frac{1}{\partial_y \rho} \right) |\nabla_x \rho|^2 - |\nabla_x \rho|^2.
\]

To prove the analyticity of $\alpha(\eta)$ we notice that it follows from Lemma \ref{lemder} and Lemma \ref{lemHp} that 
\[
\|\partial_y \rho  - 1\|_{\cH^{\alpha , m  - \frac{1}{2}}} \lesssim \|\eta_s\|_{\cH^{\alpha , m  + \frac{1}{2}}} \lesssim  \|\eta\|_{H^{\alpha,m}}.
\]
Here by definition $\eta_s(x,y) := \frac{\sinh((y + \tth )|D|)}{\sinh(\tth |D|)} \eta(x)$. Then it follows from Proposition \ref{propT} that if we choose $\varepsilon_0$ small enough, the series
\[
\sum_{j = 1}^{+\infty} (\partial_y \rho  - 1)^j
\]
converges and to the function $1 - \frac{1}{\partial_y \rho}$. By linearity the map
\[
\eta \mapsto \partial_y \rho  - 1
\]
is analytic and so as $ 1 - \frac{1}{\partial_y \rho}$ now.  Furthermore, we have the estimate
\[
\left\| 1 - \frac{1}{\partial_y \rho} \right\|_{\cH^{\alpha,m - \frac{1}{2}}} \lesssim \|\eta\|_{H^{\alpha,m}}.
\]

Since now $\alpha(\eta)$ is a polynomial of $1 - \frac{1}{\partial_y \rho}$ and $\partial_{x_j} \rho, \ j = 1, \ldots, d$, it follows that $\alpha(\eta)$ is analytic in $\eta$ and the \eqref{targete} follows from the previous estimates, Lemma \ref{lemder}, Lemma \ref{lemHp} and Proposition \ref{propT}. Similar arguments also work for $\beta(\eta)$ and $\gamma_j(\eta), \ j = 1, \ldots, d$. Finally for 
$\delta(\eta)$, it suffices to know from \eqref{defcoeff} that it is a polynomial of
\[
1 - \frac{1}{\partial_y \rho}, \partial_{x_j} \rho, \partial_y \rho, \partial_{x_j}^2 \rho, \partial_{x_j} \partial_y \rho, \partial_y^2\rho
\]
and the maximum order of derivatives is $2$. So finally we can obtain the desired esimate \eqref{targete} and end the proof.
\end{proof}

Now we show that the solution of \eqref{defining3} depends analytically in $\eta$.

\begin{proposition}\label{propA}
    Let $\alpha \ge 0$ and $m + \frac{1}{2}, m_0 \in \N$ with $m - \frac{3}{2} \ge m_0 > \frac{d + 1}{2}$. Then there exist $\varepsilon_0 > 0$ and $C > 0$ such that for any $\eta \in H^{\alpha,m} \cap B^{\alpha,m_0 + \frac{3}{2}} (\varepsilon_0)$ and $\psi \in H^{\alpha,m}$, we have a unique solution $\Psi(\eta)\psi \in \cH^{\alpha , m + \frac{1}{2}}$ of \eqref{defining3} satisfying the following estimate
\begin{align}\label{estPsi}
    \|\Psi(\eta)\psi\|_{\cH^{\alpha, m+ \frac{1}{2}}} \le C ( \|\psi\|_{H^{\alpha, m}}  + \|\eta\|_{H^{\alpha,m}} \|\psi\|_{H^{\alpha, m_0 + \frac{3}{2}}}).
\end{align}
Furthermore, $\psi \mapsto \Psi(\eta)\psi$ is linear and the map $\Psi: H^{\alpha,m} \cap B^{\alpha,m_0 + \frac{3}{2}} (\varepsilon_0) \to  \mathcal{L}(H^{\alpha,m}, \cH^{\alpha, m+ \frac{1}{2}})$ is analytic.
\end{proposition}

\begin{proof}
Recall the operator $L$ appearing in Lemma \ref{lemER} and write $P(\eta) = L \circ F(\eta)$. It follows from Lemma \ref{lemanaF} and Lemma \ref{lemER} that there exist $\varepsilon_0 > 0$ and $C > 0$ such that the nonlinear map
\begin{align*}
P : H^{\alpha,m} \cap B^{\alpha,m_0 + \frac{3}{2}} (\varepsilon_0) &\to  \mathcal{L}(\cH^{\alpha, m + \frac{1}{2}},\cH^{\alpha, m +  \frac{1}{2}})\\
\eta &\mapsto P(\eta)
\end{align*}
is analytic and the following estimate holds
\begin{align}\label{estP}
    \|P(\eta)\theta\|_{\cH^{\alpha, m + \frac{1}{2}}} \le C (\|\eta\|_{H^{\alpha,m_0 + \frac{3}{2}}} \|\theta\|_{\cH^{\alpha, m + \frac{1}{2}}}  + \|\eta\|_{H^{\alpha,m}} \|\theta\|_{\cH^{\alpha, m_0 + 2}}).
\end{align}
In particular, for the special case $m =  m_0 + \frac{3}{2}$ we obtain 
\begin{align}\label{estP2}
    \|P(\eta)\theta\|_{\cH^{\alpha, m_0 + 2 }} \le C' \|\eta\|_{H^{\alpha,m_0 + \frac{3}{2}}}\|\theta\|_{\cH^{\alpha, m_0 + 2}}.
\end{align}
Without loss of generality we can assume $C' \le C$. Let us prove that for small $\varepsilon'_0 \le \varepsilon_0$, when $\eta \in B^{\alpha ,m_0 + \frac{3}{2} }(\varepsilon_0)$, the series 
\[
\sum_{j = 0}^{+\infty} P(\eta)^j
\]
converges in $ \mathcal{L}(\cH^{\alpha, m + \frac{1}{2}},\cH^{\alpha, m +  \frac{1}{2}})$. Firstly, it follows from \eqref{estP2} that 
\begin{align}
    \label{estP3}
    \|P(\eta)^j \theta\|_{\cH^{\alpha, m_0 + 2}} \le (C' \|\eta\|_{H^{\alpha,m_0 + \frac{3}{2}}} )^{j - 1} \|\theta\|_{\cH^{\alpha, m_0 + 2}}.
\end{align}
By induction, we can show that 
\begin{align}
    \label{estP4}
    \|P(\eta)^j \theta\|_{\cH^{\alpha, m + \frac{1}{2}}} \le C^j \|\eta\|_{H^{\alpha,m_0 + \frac{3}{2}}}^{j - 1}(\|\eta\|_{H^{\alpha,m_0 + \frac{3}{2}}} \|\theta\|_{\cH^{\alpha,m + \frac{1}{2}}} + j  \|\eta\|_{H^{\alpha,m}} \|\theta\|_{\cH^{\alpha, m_0 + 2}}).
\end{align}
In fact, for the case $j = 1$, it is exactly \eqref{estP} and from the case $j$ to $j + 1$ it suffices to use \eqref{estP} again and \eqref{estP3}. So as a result, for small $\varepsilon'_0 \le \varepsilon_0$, when $\eta \in B^{\alpha ,m_0 + \frac{3}{2} }(\varepsilon_0)$, we have $(\operatorname{Id} - P(\eta))^{-1}$ is well-defined, analytically in $\eta$, and 
\[
\|(\operatorname{Id} - P(\eta))^{-1}\|_{ \mathcal{L}(\cH^{\alpha, m + \frac{1}{2}},\cH^{\alpha, m +  \frac{1}{2}})} \lesssim 1.
\]
Now if $u$ is a solution of \eqref{defining5} we must have
\[
u = P(\eta)(v + u),
\]
which implies $u = (\operatorname{Id} - P(\eta))^{-1} P(\eta)[v]$. Thus the solution of \eqref{defining3} can be represented by
\[
\Psi(\eta)\psi = (\operatorname{Id} - P(\eta))^{-1} P(\eta)\psi_c + \psi_c,
\]
which is analytic in $\eta$ and the estimate \eqref{estPsi} follows from Lemma \ref{lemHp} and \eqref{estP}.
\end{proof}

\begin{proof}[Proof of Theorem \ref{thmG}]
We rewrite $G(\eta)[\psi]$ as
\[
G(\eta)\psi = - \nabla_x \eta \cdot \nabla_x\psi + \Gamma(\partial_y \Psi(\eta)\psi) +  \frac{|\nabla_x\eta|^2 - |D|\coth(h|D|)\eta}{1 + |D|\coth(h|D|)\eta}\Gamma(\partial_y \Psi(\eta)\psi) = \sum_{j = 1}^3 \mathbf{G}_j(\eta)\psi.
\]
 The first term $\mathbf{G}_1(\eta)\psi$ is bilinear so it is analytic and the estimate follows from Lemma \ref{lemder} and    Lemma \ref{lemT}. For the second term, we know from Proposition \ref{propA} that $\Psi(\eta)\psi$ is analytic. Since both $\Gamma$ and $\partial_y$ are linear, $\mathbf{G}_2(\eta)\psi$ is also analytic. The estimate follows from \eqref{estPsi}, Lemma \ref{lemder} and Lemma \ref{lemtra}. For the remaining term $\mathbf{G}_3(\eta)\psi$ we note that 
 $\mathbf{G}_3(\eta)[\psi] = f(\eta) \mathbf{G}_2(\eta)\psi$ with $f(\eta) := \frac{|\nabla_x\eta|^2 - |D|\coth(h|D|)\eta}{1 + |D|\coth(h|D|)\eta}$. As before, by using Lemma \ref{lemT}, for $\varepsilon_0$ small and $\eta \in B^{\alpha,m_0 + \frac{3}{2}}(\varepsilon_0)$, the following series
 \[
 \sum_{j = 0}^{+\infty} (-|D|\coth(h|D|)\eta)^j
 \]
 converges in $H^{\alpha, m- 1}$, which means $f(\eta)$ is analytic. Thus by Lemma \ref{lemT} again we have the estimate
 \[
 \|f(\eta)\|_{H^{\alpha, m- 1}} \lesssim \|\eta\|_{H^{\alpha,m}}.
 \]
 Finally the estimate of $\mathbf{G}_3(\eta)\psi$ follows from that of $\mathbf{G}_3(\eta)\psi$, $f(\eta)$ and Lemma \ref{lemT} again.
\end{proof}

\begin{theorem}[Existence of Stokes waves]
Given $\alpha\geq 0$, $m + \frac12 \in \N$, $m > \frac{11}{2}$ and $(\tth,\kappa,b) \in (\R_{>0}\times\R_{\geq 0} \times\R_{> 0}) \setminus \fR$, with  $\fR$  in \eqref{def:fR}. There exist
$\e_*=\e_*(\alpha,m,\tth,\kappa,b)>0$ and a unique family of solutions to \eqref{travelingWWstokes}
\[
(\eta_\e(x),\psi_\e(x),c_\e)\in H^{\alpha,m}_{\mathtt{ev}} (\T)\times H^{\alpha,m}_{\mathtt{odd}}(\T)\times \R,
\qquad |\e|<\e_*
\]
such that
\[
c_0:=\sqrt{1+\kappa+b}\sqrt{\tanh(\tth)},
\]
where  $ H^{\alpha,m}_{\mathtt{ev}}(\T) $ (resp. $ H^{\alpha,m}_{\mathtt{odd}}(\T) $) denotes the space of even (resp. odd) functions in $ H^{\alpha,m}(\T) $. 
\end{theorem}

\begin{proof}
This result follows from an application of the classical Crandall--Rabinowitz bifurcation theorem, see for example \cite[Theorem 3.1]{BMV2}. For the sake of completeness, we include a sketched proof here. In fact, by an application of Theorem \ref{thmG}, there exists an $\epsilon_0 >0 $ such that the following map is analytic
\begin{align*}
    F:  &(H^{\alpha,m}_{\mathtt{ev}} (\T) \cap B^{\alpha, m}(\epsilon_0))\times H^{\alpha,m}_{\mathtt{odd}}(\T)\times \R \to H^{\alpha,m - 1}_{\mathtt{odd}}(\T) \times H^{\alpha,m - 4}_{\mathtt{ev}} (\T), \\
    F(\eta,\psi,c) &:= \begin{bmatrix}
        c \, \eta_x+G(\eta)\psi \\
        c \, \psi_x -  \eta - \dfrac{\psi_x^2}{2} + 
\dfrac{1}{2(1+\eta_x^2)} \big( G(\eta) \psi + \eta_x \psi_x \big)^2  +\kappa\sigma(\eta)-b\left(\partial^2_s \sigma(\eta)+\frac{1}{2}\sigma(\eta)^3\right).
    \end{bmatrix}
\end{align*}

It is not hard to show that $F(0,0,c) = 0$ for all $c \in \R$. A direct calculation gives that 
\begin{align*}
    \de_{\eta,\psi}F(0,0,c): & H^{\alpha,m}_{\mathtt{ev}} (\T)\times H^{\alpha,m}_{\mathtt{odd}}(\T) \to H^{\alpha,{m - 1}}_{\mathtt{odd}}(\T) \times H^{\alpha,{m - 4}}_{\mathtt{ev}} (\T), \\
    \de_{\eta,\psi}F(0,0,c) \begin{bmatrix}
        \hat{\eta} \\
        \hat{\psi}
    \end{bmatrix} & := \begin{bmatrix}
        c \partial_x & |D| \tanh(h|D|) \\
        -1 + \kappa \partial_x^2 - b \partial_x^4 & c \partial_x
    \end{bmatrix} \begin{bmatrix}
        \hat{\eta} \\
        \hat{\psi}
    \end{bmatrix}.
\end{align*}

Our assumptions guarantee that with the choice $c = c_0$ given above, we have
\[
\mathrm{ker} (\de_{\eta,\psi}F(0,0,c_0)) = \mathrm{span} \left\{ \begin{bmatrix}
    a \cos(x) \\
    a' \sin(x)
\end{bmatrix} \right\}, \qquad \mbox{such that} \qquad 
\begin{bmatrix}
    - c_0 & \tanh(\tth) \\
    -(1 + \kappa + b) & c_0
\end{bmatrix}
\begin{bmatrix}
    a  \\
    a' 
\end{bmatrix}
= 0
\]
and 
\begin{align*}
\mathrm{ran} (\de_{\eta,\psi}F(0,0,c_0)) = \mathrm{span} \left\{ \begin{bmatrix}
   0 \\
    1
\end{bmatrix} \right\} &\oplus  \mathrm{span} \left\{ \begin{bmatrix}
    c_0 \sin(x) \\
    (1 + \kappa + b) \cos(x)
\end{bmatrix} \right\} \\
&\oplus 
     \overline{\mathrm{span} \left\{ \begin{bmatrix}
    \sin(jx) \\
    0
\end{bmatrix},\begin{bmatrix}
    0 \\
    \cos(jx)
\end{bmatrix}, j \ge 2 \right\}}^{H^{\alpha,{m - 1}}_{\mathtt{odd}}(\T) \times H^{\alpha,{m - 4}}_{\mathtt{ev}} (\T)}.
\end{align*}
So it is clear that the kernel of $\de_{\eta,\psi}F(0,0,c_0)$ is $1$-dimensional, the range of $\de_{\eta,\psi}F(0,0,c_0)$ is closed and the cokernel of $\de_{\eta,\psi}F(0,0,c_0)$ is also $1$-dimensional. Finally we show that $\partial_c \de_{\eta,\psi}F(0,0,c_0) [u^*]$ does not belong to the range of $\de_{\eta,\psi}F(0,0,c_0)$, where $u^*$ is a nonzero element in the $\mathrm{ker} (\de_{\eta,\psi}F(0,0,c_0))$. We assume on the contrary that $\partial_c \de_{\eta,\psi}F(0,0,c_0) [u^*] \in \mathrm{ran} (\de_{\eta,\psi}F(0,0,c_0))$, namely
\[
\begin{bmatrix}
         \partial_x & 0 \\
        0 &  \partial_x
    \end{bmatrix} \begin{bmatrix}
    a \cos(x) \\
    a' \sin(x)
\end{bmatrix} \in \mathrm{ran} (\de_{\eta,\psi}F(0,0,c_0))
\]
for some $(a,a') \ne (0,0)$ satisfying $-(1 + \kappa + b) a + c_0 a'  = 0$. We can see from the expression of  $\mathrm{ran} (\de_{\eta,\psi}F(0,0,c_0))$ that 
\[
\begin{bmatrix}
    -a \sin(x) \\
    a' \cos(x)
\end{bmatrix} \parallel 
\begin{bmatrix}
    c_0 \sin(x) \\
    (1 + \kappa + b) \cos(x)
\end{bmatrix},
\]
which implies $(1 + \kappa + b) a + c_0 a' = 0$. Combining this with $-(1 + \kappa + b) a + c_0 a'  = 0$, we get $(a,a') = (0,0)$, which leads to a contradiction. Then we can conclude that  $F$ satisfies all the conditions of Crandall--Rabinowitz bifurcation theorem.
\end{proof}

\section{Hydroelastic Stokes Waves}\label{sec:App2}

In this Appendix we provide the expansions in \eqref{exp:Sto}. In fact, we prove slightly more. We actually compute the expansion with sufficiently small pressure $p$.
\\[1mm]
\noindent
{\bf Proof of \eqref{exp:Sto}.}
Writing
 \begin{equation}\label{etapsic}
\begin{aligned}
 & \eta_{\e,p}(x) =\eta_{0,p}+\tilde{\eta}_{\e,p}(x)= \eta_{0,p}+\e \eta_{1,p}(x) + \e^2 \eta_{2,p}(x) + \cO(\e^3) \, , 
  \end{aligned}
\end{equation}
\begin{equation} \label{B2 expansion}
\psi_{\e,p}(x) = \e \psi_{1,p}(x) + \e^2 \psi_{2,p}(x) + \cO(\e^3)\,, \quad c_{\e,p} = c_{0,p} + \e c_{1,p} + \e^2 c_{2,p}+ \cO(\e^3) \,,   
\end{equation}
and
\begin{align} \label{exp:Sto 2}
    \quad c_{0,p} := \ckb c_{\tth,p}, \quad \ckb:=\sqrt{1+\kappa+b},\quad c_{\tth,p}:=\sqrt{\tanh(\tth+p)}\,,
\end{align}
where  $\eta_{i,p}$ is $\mathit{even}(x)$ and $\psi_{i,p}$ is $\mathit{odd}(x)$ for $i=1,2$, 
we solve order by order in $ \e $ for the system as follows:
\begin{equation}
\label{Sts0}
\begin{cases}
-c \, \psi_x +  \eta   + \dfrac{\psi_x^2}{2} - 
\dfrac{\eta_x^2}{2(1+\eta_x^2)} \left( c  -  \psi_x \right)^2-\kappa\sigma(\eta)+b\left(\partial^2_s \sigma(\eta)+\frac{1}{2}\sigma(\eta)^3\right)  = p\,, \\
c \, \eta_x+G(\eta;\tth)\psi = 0 	\, ,
\end{cases}
\end{equation}
having substituted $G(\eta;\tth)\psi $ with $-c \, \eta_x $  in the first equation.
First, we observe that $\eta_{0,p}=p$. Next, let $\tilde{\tth}:=\tth+p$. We expand the Dirichlet--Neumann operator 
$ G(\eta;\tth)= G(\tilde{\eta};\tilde{\tth})= G_0+ G_1(\tilde{\eta}) + G_2(\tilde{\eta}) + \cO(\tilde{\eta}^3)  $
where, according to \cite[formula (2.14)]{CS}, 
\begin{equation} \label{expDiriNeu}
\begin{aligned}
 G_0 & := D\tanh(\tilde{\tth} D) = |D| \tanh(\tilde{\tth} |D|) \,, \\
 G_1(\tilde{\eta}) & := D \big( \tilde{\eta} -  \tanh(\tilde{\tth} D)\tilde{\eta} \tanh(\tilde{\tth} D) \big)D = -\pa_x \tilde{\eta} \pa_x - |D| \tanh(\tilde{\tth}|D|)\tilde{\eta}|D| \tanh(\tilde{\tth}|D|), \\
 G_2(\tilde{\eta}) & := -\frac12 D \Big( 
 D{\tilde{\eta}}^2 \tanh(\tilde{\tth} D) +\tanh(\tilde{\tth} D){\tilde{\eta}}^2 D - 2\tanh(\tilde{\tth} D)\tilde{\eta} D\tanh(\tilde{\tth}  D)\tilde{\eta}\tanh(\tilde{\tth} D) \Big)D \, .
\end{aligned}
\end{equation}
Using the fact that 
\begin{align}
    G(\eta;\tth)=G(\tilde{\eta};\tilde{\tth}), \quad \sigma(\eta)=\sigma(\tilde{\eta}),
\end{align}
one may rewrite \eqref{Sts0} as
\begin{equation}
\label{Sts}
\begin{cases}
-c \, \psi_x +  \tilde{\eta}   + \dfrac{\psi_x^2}{2} - 
\dfrac{\tilde{\eta}_x^2}{2(1+\tilde{\eta}_x^2)} \left( c  -  \psi_x \right)^2-\kappa\sigma(\tilde{\eta})+b\left(\partial^2_s \sigma(\tilde{\eta})+\frac{1}{2}\sigma(\tilde{\eta})^3\right)  = 0\,, \\
c \, \tilde{\eta}_x+G(\tilde{\eta};\tilde{\tth})\psi = 0 	\, .
\end{cases}
\end{equation}
\noindent
{\bf First order in $ \e $.}  Substituting the expansions in \eqref{etapsic} into \eqref{Sts}, we get the linear system 
\begin{equation}
 \left\{\begin{matrix} -c_{0,p} \,\pa_x(\psi_{1,p}) + \eta_{1,p}-\kappa\pa_x^2(\eta_{1,p}) +b \pa_x^4(\eta_{1,p})= 0\,, \\
 c_{0,p} \,\pa_x(\eta_{1,p}) + G_0\psi_{1,p} =0 \, ,  \end{matrix}\right.
\end{equation}
or, equivalently, 
\begin{equation} \label{cB0}
\begin{aligned}
   \vet{\eta_{1,p}}{\psi_{1,p}} \in \mathrm{ker }\,\cB_0  \text{ with } \cB_0 := \begin{bmatrix} 1-\kappa\pa^2_{x}+b\pa^4_{x} & -c_{0,p}\,\pa_x \\ c_{0,p}\,\pa_x & G_0  \end{bmatrix}\,,
\end{aligned}
\end{equation}
where $\eta_{1,p}$ is $\mathit{even}(x)$ and $\psi_{1,p}$ is $\mathit{odd}(x)$.
\begin{lemma}\label{lem:B0}
If $(\tth+p,\kappa,b) \in \R^3_+ \setminus \fR$,  the kernel of 
 the linear operator $\cB_0$ in \eqref{cB0} is one dimensional and given by
\begin{equation}\label{chk}
\mathrm{ker }\,\cB_0= \mathrm{span}\,\left\{\vet{\cos(x)}{\frac{\ckb}{c_{\tth,p}}\sin(x)} \right\}.
\end{equation}
\end{lemma}

\begin{proof} 
The action of $\cB_0$ on each subspace spanned by $\footnotesize{\,\Big\{\vet{\cos(kx)}{0}, \vet{0}{\sin(kx)}\Big\}} $, $k\in \N$, is represented by the $2\times 2$ matrix $\footnotesize{ \begin{bmatrix} 1+\kappa k^2+b k^4 & -c_{0,p}\, k \\ -c_{0,p}\, k & k\tanh((\tth+p) k)  \end{bmatrix}}$. Its determinant is given by
 $$ (1+\kappa k^2+b k^4)k \tanh((\tth+p) k) - c_{0,p}^2 k^2\stackrel{\eqref{exp:Sto 2}}{=} k^2 \left(\frac{(1+\kappa k^2+b k^4)\tanh((\tth+p) k)}{k} -(1+\kappa+b)\tanh(\tth+p) \right)$$
 and, provided $(\tth+p,\kappa,b) \in \R^3_+ \setminus \fR$, 
 vanishes if and only if $k=1$.
Indeed, for $k$ large enough, the determinant goes asymptotically to $+\infty$ so it is uniformly bounded away from zero,  whereas if for some $k \in \N$ it vanishes, it implies that $(\tth+p,\kappa,b) \in \fR_k \subset \fR$, absurd.
  For $k=1$ we obtain the kernel of $\cB_0$ given in \eqref{chk}. For $k=0$ it has no kernel since $\psi_{1,p}(x)$ is odd.
\end{proof}
We set
$
\eta_{1,p}(x) := \cos(x)$, $\psi_{1,p}(x) := c_{\tth,p}^{-1}\ckb \sin(x) 
$
in agreement with 
\eqref{exp:Sto} when $p=0$. Also, we denote $c_{\tth}:=c_{\tth,0}$ and $c_0:=c_{0,0}$ for notational simplicity. 
\\[1mm]
{\bf Second order in $ \e $.} 
By \eqref{Sts}, and since 
$ (c_{0,p} (\eta_{1,p})_x)^2 = (G_0\psi_{1,p})^2  $, we get  the linear system
\begin{equation}\label{syslin2}
 \cB_0 \vet{\eta_{2,p}}{\psi_{2,p}} = \vet{c_{1,p}(\psi_{1,p})_x-\frac12 (\psi_{1,p})_x^2 + \frac12 (G_0\psi_{1,p})^2 }{-c_{1,p}(\eta_{1,p})_x - G_1(\eta_{1,p})\psi_{1,p}} \, , 
\end{equation}
where  $\cB_0$ is the self-adjoint operator in \eqref{cB0}. 
System \eqref{syslin2} admits a solution if and only if its right-hand term is orthogonal to the kernel  of $\cB_0$ in \eqref{chk}, namely
\begin{equation}\label{orth1}
 \left(\vet{c_{1,p}(\psi_{1,p})_x-\frac12 (\psi_{1,p})_x^2 + \frac12 (G_0\psi_{1,p})^2 }{-c_{1,p}(\eta_{1,p})_x - G_1(\eta_{1,p})\psi_{1,p}}\;,\;\vet{\cos(x)}{\frac{\ckb}{c_{\tth,p}}\sin(x)}\right)=0 \, . 
\end{equation}
In view of the first order expansion \eqref{exp:Sto}, \eqref{expDiriNeu} and the identity
 $
  \tanh(2(\tth+p))  = 2c_{\tth,p}^2(1+c_{\tth,p}^4)^{-1}
 $,
  we obtain $
 [G_0\psi_{1,p}](x)= c_{\tth,p} \ckb\sin(x)$, $\big[G_1(\eta_{1,p})\psi_{1,p}\big](x) 
 =\frac{(1-c_{\tth,p}^4)\ckb}{c_{\tth,p}(1+c_{\tth,p}^4)}\sin(2x)$
so that \eqref{orth1} implies
 $c_{1,p}=0$, in agreement with \eqref{exp:Sto}. Equation
 \eqref{syslin2} reduces to 
\begin{align}\label{sisto2}
\begin{bmatrix} 1-\kappa\pa^2_{x}+b\pa^4_{x} & -c_{0,p}\,\pa_x \\ c_{0,p}\,\pa_x & G_0  \end{bmatrix}
\vet{\eta_{2,p}}{\psi_{2,p}} 
  = \vet{\frac{\ckb^2}{4}\left(c_{\tth,p}^2-c_{\tth,p}^{-2}\right)-\frac{\ckb^2}{4}\left(c_{\tth,p}^2+c_{\tth,p}^{-2}\right)\cos(2x) }{ -\frac{(1-c_{\tth,p}^4)\ckb}{c_{\tth,p}(1+c_{\tth,p}^4)}\sin(2x)}.
\end{align}
Setting $ \eta_{2,p} = \eta_{2,p}^{[0]} +  \eta_{2,p}^{[2]} \cos(2x) $ and  $ \psi_{2,p} =
 \psi_{2,p}^{[2]}  \sin (2x) $, system 
\eqref{sisto2} amounts to 
\begin{align*}
\left\{ \begin{matrix} \eta_{2,p}^{[0]} +\big( (1+4\kappa+16 b)\eta_{2,p}^{[2]} -2c_{0,p}  \psi_{2,p}^{[2]} \big) \cos(2x)  =  \frac{\ckb^2}{4}\left(c_{\tth,p}^2-c_{\tth,p}^{-2}\right)-\frac{\ckb^2}{4}\left(c_{\tth,p}^2+c_{\tth,p}^{-2}\right)\cos(2x)\,,   \\ (-2c_{0,p} \eta_{2,p}^{[2]}   + 2  \psi_{2,p}^{[2]} \tanh(2(\tth+p)))\sin(2x)  =   -\frac{(1-c_{\tth,p}^4)\ckb}{c_{\tth,p}(1+c_{\tth,p}^4)}\sin(2x) \, , \end{matrix}\right.
\end{align*}
which leads to the expansions as follow:
\begin{align}\label{expcoef}
 &\eta_{2,p}^{[0]} := \frac{\ckb^2}{4}\left(c_{\tth,p}^2-c_{\tth,p}^{-2}\right), \qquad 
\eta_{2,p}^{[2]} := -\frac{\left(c_{\tth,p}^4-3\right)\,\ckb^2}{4\,c_{\tth,p}^2\,(b\,c_{\tth,p}^4-3\,\kappa -15\,b+c_{\tth,p}^4\,\kappa +c_{\tth,p}^4)}\,, \\ \label{expcoef1}
&\psi_{2,p}^{[2]} :=\frac{\ckb\left(33\,b+9\,\kappa -30\,b\,c_{\tth,p}^4+b\,c_{\tth,p}^8-6\,c_{\tth,p}^4\,\kappa +c_{\tth,p}^8\,\kappa +c_{\tth,p}^8+3\right)}{8\,c_{\tth,p}^3\,(b\,c_{\tth,p}^4-3\,\kappa -15\,b+c_{\tth,p}^4\,\kappa +c_{\tth,p}^4)}\,.
\end{align}
In particular, $\eta^{[0]}_{2}:=\eta^{[0]}_{2,0}$, $\eta^{[2]}_{2}:=\eta^{[2]}_{2,0}$, and $\psi^{[2]}_{2}:=\psi^{[2]}_{2,0}$.

\noindent
{\bf Third order in $ \e $.} 
It remains to determine 
$ c_{2,p} $ in \eqref{B2 expansion} (cf. \eqref{expc2}).
 We get the linear system 
\begin{equation}\label{syslin3}
 \cB_0 \vet{\eta_{3,p}}{\psi_{3,p}} = \vet{c_{2,p}(\psi_{1,p})_x 
 - (\psi_{1,p})_x (\psi_{2,p})_x - (\eta_{1,p})_x^2 (\psi_{1,p})_x c_{0,p} + 
 (\eta_{1,p})_x (\eta_{2,p})_x c_{0,p}^2-\kappa\mathfrak{C}(\eta_{1,p})+b\mathfrak{D}(\eta_{1,p})
 }{-c_{2,p}(\eta_{1,p})_x - G_1(\eta_{1,p})\psi_{2,p}- G_1(\eta_{2,p})\psi_{1,p} - G_2(\eta_{1,p})\psi_{1,p}} \, , 
\end{equation}
where 
\begin{align}
    \mathfrak{C}(\eta):=\frac{3}{2}(\eta)^2_x(\eta)_{xx},\quad\mathfrak{D}(\eta):=\frac{5}{2}\left((\eta)^3_{xx}+(\eta)^2_x(\eta)_{xxxx}\right)+10(\eta)_x(\eta)_{xx}(\eta)_{xxx}.
\end{align}
System \eqref{syslin3} has 
a solution if and only if the right hand side is orthogonal to the kernel of
$ \cB_0 $ given in \eqref{chk}. This condition determines uniquely $ c_{2,p} $.
Denoting  $\Pi_1$ the $L^2$-orthogonal projector on span$\, \{\cos(x),\sin(x)\} $, we obtain 
\begin{align*} 
& c_{2,p} (\psi_{1,p})_x = c_{2,p} c_{\tth,p}^{-1}\ckb\cos(x)\, , \quad c_{2,p} (\eta_{1,p})_x = -c_{2,p} \sin(x) \, , \quad \Pi_1[ (\psi_{1,p})_x (\psi_{2,p})_x] = 
\psi_{2,p}^{[2]} c_{\tth,p}^{-1} \ckb \cos(x)\, ,\\ 
& \Pi_1 [c_{0,p} (\eta_{1,p})_x^2 (\psi_{1,p})_x ] = \frac14 \ckb^2\cos(x)\, , \quad \Pi_1[c_{0,p}^2 (\eta_{1,p})_x (\eta_{2,p})_x] = \eta_{2,p}^{[2]} c_{\tth,p}^2\ckb^2\cos(x) \, ,\\
& \Pi_1[\mathfrak{C}(\eta_{1,p})]=\Pi_1\left[-\frac{3}{2}\cos(x)\sin^2(x)\right]=-\frac{3}{8}\cos(x) ,\quad\Pi_1[\mathfrak{D}(\eta_{1,p})]=\Pi_1\left[\frac{25}{2}\cos(x)-15\cos^3(x)\right]=\frac{5}{4}\cos(x)\, ,
\end{align*}
and, in view of \eqref{expDiriNeu}, and \eqref{exp:Sto}, \eqref{expcoef}, 
\begin{align*}
 G_1(\eta_{1,p})\psi_{2,p}  &= \psi_{2,p}^{[2]}\frac{1-c_{\tth,p}^4}{1+c_{\tth,p}^4}\sin(x)+\psi_{2,p}^{[2]}\frac{3(1-c_{\tth,p}^8)}{3 c_{\tth,p}^4+1}\sin(3x)\, , \\
  G_2(\eta_{1,p})\psi_{1,p}
 &=  \frac{\ckb(3 c_{\tth,p}^5-c_{\tth,p})}{4(1+c_{\tth,p}^4)} \sin(x)-\frac{3\ckb\left(c_{\tth,p}^9-4c_{\tth,p}^5+3c_{\tth,p}\right)}{4\,\left(3c_{\tth,p}^8+4c_{\tth,p}^4+1\right)}\sin(3x) \, , \\
G_1(\eta_{2,p})\psi_{1,p}
&= c_{\tth,p}^{-1}\ckb 
 \Big( \eta_{2,p}^{[0]}(1-c_{\tth,p}^4) +   \frac12 \eta_{2,p}^{[2]} (1+c_{\tth,p}^4) \Big)\sin(x)+\frac{3\,\eta_{2,p}^{[2]}\ckb(1-c_{\tth,p}^8)}{c_{\tth,p}(6c_{\tth,p}^4+2)}\sin(3x) \,, 
\end{align*}
and, hence,
\begin{align*}
 \Pi_1[ G_1(\eta_{1,p})\psi_{2,p}]  &=\psi_{2,p}^{[2]}\frac{1-c_{\tth,p}^4}{1+c_{\tth,p}^4} \sin(x) \, , \quad \Pi_1[G_2(\eta_{1,p})\psi_{1,p}] 
 =\frac{\ckb(3c_{\tth,p}^5-c_{\tth,p})}{4(1+c_{\tth,p}^4)} \sin(x)\, , \\
 \Pi_1[G_1(\eta_{2,p})\psi_{1,p}] 
&= c_{\tth,p}^{-1}\ckb 
 \Big( \eta_{2,p}^{[0]}(1-c_{\tth,p}^4) +   \frac12 \eta_{2,p}^{[2]} (1+c_{\tth,p}^4) \Big)\sin(x) \, . 
\end{align*}
Imposing the orthogonality condition, we arrive at
\begin{equation}\label{expc2}
    \begin{aligned}
c_{2,p} :=& \Bigg(3c_{\tth,p}^4\left(28 b^2-19 b+3\right)+\left(48 b-39 bc_{\tth,p}^4-7 bc_{\tth,p}^8+6c_{\tth,p}^4+3c_{\tth,p}^8-12\right)\ckb^2\\&+\left(24 b-48 bc_{\tth,p}^4+24 bc_{\tth,p}^8-13c_{\tth,p}^4+15\right)\ckb^4+\left(-2\,{\left(c_{\tth,p}^4-1\right)}^2\,\left(c_{\tth,p}^4-3\right)\right)\ckb^6\Bigg)\Bigg(16c_{\tth,p}^3\ckb\,r_{\tth,p,\kappa,b}\Bigg)^{-1}\,,
\end{aligned}
\end{equation}
with
\begin{align*}
r_{\tth,p,\kappa,b}:=b\,c_{\tth,p}^4-3\,\kappa -15\,b+c_{\tth,p}^4\,\kappa +c_{\tth,p}^4\,.
\end{align*}
In particular, $c_2:=c_{2,0}$ and $r_{\tth,\kappa,b}:=r_{\tth,0,\kappa,b}$.

\section{Expansions of $\frak{p}(x)$, $\mathsf{f}_\e$, Lemma \ref{lem:pa.exp}, Lemma \ref{tau}, and Lemma \ref{beta}} \label{sec:App2 2}

\begin{proof}[Proof of \eqref{expfe}] 
We expand  the function $\mathfrak{p}(x)  = \e\mathfrak{p}_1(x) + \e^2 \mathfrak{p}_2(x) + \cO(\e^3)  $ defined by the fixed point equation \eqref{def:ttf}.   We first 
note that  the constant $\mathtt{f}_\e=\cO(\e^2)$  because
 $\eta_1(x) = \cos(x)$ has zero average. 
Then 
$ \mathfrak{p}(x)
 = \frac{\mathfrak H}{\tanh(\tth |D|)} \big[\e\eta_1 +\e^2\big(\eta_2 + 
 (\eta_1)_x \mathfrak{p}_1 \big)+\cO(\e^3)\big] $, 
and, using that $\mathfrak H \cos(kx) = \sin(kx)$, for any  $ k \in \N$, we get 
\begin{align}
 \mathfrak{p}_1(x) 
 & = \frac{\mathfrak H}{\tanh(\tth |D|)}\cos(x)  = \ch^{-2}\sin(x) \, , \label{pfra1} \\
 \mathfrak{p}_2(x) &= \frac{\mathfrak H}{\tanh(\tth |D|)}((\eta_1)_x
 \mathfrak{p}_1 +\eta_2 ) 
 = \left(\frac{\ch^{-2}+2\eta_2^{[2]}}{2\tanh(2\tth)}\right)\sin(2x) \, . \label{pfra2}
\end{align}
More explicitly, we obtain
\begin{align}
    \mathfrak{p}_2(x)=\frac{\left(\ch^4+1\right)\,\left(b\,\ch^4-3\kappa -27 b+\ch^4 \kappa +\ch^4+3\right)}{8\ch^4\,r_{\tth,\kappa,b}}\sin(2x).
\end{align}
Finally
\begin{align*}
 \tf_\e & = \frac{\e^2}{2\pi} \int_\T \big( \eta_2 + (\eta_1)_x\mathfrak{p}_1 \big) \de x + \cO(\e^3) 
 = \e^2\left(\eta_2^{[0]} -\frac1{2} \ch^{-2} \right)+\cO(\e^3 ) \stackrel{\eqref{expcoef}}{=} \e^2 \tf^{[2]}_2+\cO(\e^3 ) \, .
\end{align*}
The expansion \eqref{expfe} is proved.
\end{proof}

\begin{proof}[Proof of Lemma \ref{lem:pa.exp}, Lemma \ref{tau}, and Lemma \ref{beta}] 
In view of \eqref{exp:Sto}-\eqref{expcoef}, the expansions of the functions $B$, $V$  in \eqref{espV} and \eqref{espB} are
\begin{align}
 B&= : \e B_1(x) + \e^2 B_2(x) + \cO(\e^3) 
\notag \\
&= \e c_0\, \sin(x) + \e^2 \frac{\ckb\,\left(18 b+6\kappa -2 b\ch^4-2\ch^4 \kappa -2\ch^4+3\right)}{2\ch r_{\tth,\kappa,b}}\sin(2x) + \cO(\e^3) \label{espB1}  
  \end{align}
  and
  \begin{align}
 V&= : \e V_1(x) + \e^2 V_2(x)  + \cO(\e^3) \notag\\
 &= \e \ch^{-1} \ckb\cos(x) + \e^2 \left[\frac{c_0}{2}  +\left(\frac{\ckb \left(33b+9\kappa-b\ch^8-\ch^8\kappa-\ch^8+3\right)}{4 \ch^3 \rhkb}\right)\cos(2x)\right]+\cO(\e^3)
  \, .\label{espV1} 
\end{align}
In  view of  \eqref{def:pa}, denoting derivatives w.r.t $x$ with  a prime and suppressing dependence on $x$ when trivial, we have
\begin{align}
c_0+q_\e(x)  &= (c_0+\e^2 c_2 - V(x) - V'(x)\mathfrak{p}(x)+\cO(\e^3)) (1-\mathfrak{p}'(x)+(\mathfrak{p}'(x))^2+\cO(\e^3))\notag \\
&= c_0 + \e \underbrace{(- V_1 -c_0 \mathfrak{p}_1')}_{=: q_1}+\e^2 \underbrace{\big( c_2 + V_1\mathfrak{p}_1' - V_2 - V_1'\mathfrak{p}_1 -c_0 \mathfrak{p}_2' +c_0 (\mathfrak{p}_1')^2 \big)}_{=:q_2} + \, \cO(\e^3) \, .\label{pino12imp}
\end{align} 
Similarly by \eqref{def:pa}
\begin{align}
 1+a_\e(x) : =& \frac{1}{1+\mathfrak{p}_x(x)} - (c_0+q_\e(x))B_x(x+\mathfrak{p}(x)) \notag \\
   = &  1+\e\underbrace{\big(-\mathfrak{p}_1'-c_0 B_1'\big)}_{=:a_1} +\e^2\underbrace{\big((\mathfrak{p}_1')^2-\mathfrak{p}_2'- c_0 B_2'-c_0 B_1''\mathfrak{p}_1(x)+ B_1' V_1 + c_0  B_1'\mathfrak{p}_1' \big)}_{=:a_2}+\cO(\e^3)\, . \label{aino12imp}
 \end{align}
By \eqref{espV1}, \eqref{pfra1}, \eqref{exp:Sto}, \eqref{pfra2}, \eqref{espB1} we 
deduce that the functions $q_1 $, $q_2 $, $a_1 $, $a_2 $ in \eqref{pino12imp} and \eqref{aino12imp} have an expansion as in \eqref{pino1fd}-\eqref{aino2fd}.

Finally the operator $\beta_\e$ in \eqref{def:Sigma g} expands as
$\beta_\e = \pa^4_{x} + \e \beta_1 + \e^2 \beta_2 + \cO(\e^3)$
where
\begin{align}\label{Sigma 1 expansion}
    \beta_1:=& -5\mathfrak{p}_1'\pa^4_x-10\mathfrak{p}_1''\pa^3_x-10\mathfrak{p}_1'''\pa^2_x-5\mathfrak{p}_1''''\pa_x-\mathfrak{p}_1''''' 
\end{align}
and
\begin{equation} \label{Sigma 2}
\begin{aligned}
\beta_2:=&\left(F^{[4]}_2(x)-5\mathfrak{p}_2'+15(\mathfrak{p}_1')^2\right)\pa^4_x+\left(F^{[3]}_2(x)+60\mathfrak{p}_1'\mathfrak{p}_1''-10\mathfrak{p}_2''\right)\pa^3_x\\ 
    &+\left(F^{[2]}_2(x)+45(\mathfrak{p}_1'')^2+60\mathfrak{p}_1'\mathfrak{p}_1'''-10\mathfrak{p}_2'''\right)\pa^2_x+\left(F^{[1]}_2(x)+30\mathfrak{p}_1'\mathfrak{p}_1''''+60\mathfrak{p}_1''\mathfrak{p}_1'''-5\mathfrak{p}_2''''\right)\pa_x\\
    &+6\mathfrak{p}_1'\mathfrak{p}_1'''''+15\mathfrak{p}_1''\mathfrak{p}_1''''+10(\mathfrak{p}_1''')^2-\mathfrak{p}_2'''''.
\end{aligned}
\end{equation}
Inserting the expansions of $F^{[1]}_2,F^{[2]}_2,F^{[3]}_2,F^{[4]}_2$ in \eqref{F12 F22 F32 F42} and those of $\mathfrak{p}_1$, $\mathfrak{p}_2$ in \eqref{pfra1}, \eqref{pfra2} proves the claimed expansions 
in \eqref{betaj} with coefficients in \eqref{b110}--\eqref{b224}.

Similarly, the operator $\tau_\e$ in \eqref{tau operator} expands as 
$\tau_\e = \pa^2_{x} + \e \tau_1 + \e^2 \tau_2 + \cO(\e^3)$
where
\begin{align}\label{tau 1 expansion}
    \tau_1:=& -3\mathfrak{p}_1'\pa^2_x-3\mathfrak{p}_1''\pa_x-\mathfrak{p}_1''' 
\end{align}
and
\begin{equation} \label{tau 2 expansion}
\begin{aligned}
\tau_2:=&\left(\Sigma^{[2]}_2(x)-3\mathfrak{p}_2'+6(\mathfrak{p}_1')^2\right)\pa^2_x+\left(\Sigma^{[1]}_2(x)-3\mathfrak{p}_2''+12\mathfrak{p}_1'\mathfrak{p}_1''\right)\pa_x+3(\mathfrak{p}_1'')^2+4\mathfrak{p}_1'\mathfrak{p}_1'''-\mathfrak{p}_2'''.
\end{aligned}
\end{equation}
Inserting the expansions of $\Sigma^{[1]}_2,\Sigma^{[2]}_2$ in \eqref{Sigma coefficients} and those of $\mathfrak{p}_1$, $\mathfrak{p}_2$ in \eqref{pfra1}, \eqref{pfra2} proves the claimed expansions 
in \eqref{tauj} with coefficients in \eqref{tau10}--\eqref{tau22}.
\end{proof}

\section{Kato Transformation}\label{secA1}

In this appendix, we prove Lemma \ref{expansion of the basis F}. In order to compute the vectors $P'_{0,0} \phi_k^\sigma$ and $\dot{P}_{0,0} \phi_k^\sigma$ using \eqref{A3} and \eqref{A4}, it is useful to know the action of $(\mathscr{L}_{0,0} - \lambda)^{-1}$ on the vectors
\begin{equation} \label{A10}
    \begin{aligned}
        \phi^+_k&:=\vet{(\ch/\ckb)^{\frac{1}{2}}\cos(kx)}{(\ch/\ckb)^{-\frac{1}{2}}\sin(kx)},~~\phi^-_k:=\vet{-(\ch/\ckb)^{\frac{1}{2}}\sin(kx)}{(\ch/\ckb)^{-\frac{1}{2}}\cos(kx)},\\
        \phi^+_{-k}&:=\vet{(\ch/\ckb)^{\frac{1}{2}}\cos(kx)}{-(\ch/\ckb)^{-\frac{1}{2}}\sin(kx)},~~\phi^-_{-k}:=\vet{(\ch/\ckb)^{\frac{1}{2}}\sin(kx)}{(\ch/\ckb)^{-\frac{1}{2}}\cos(kx)},~~k\in\mathbb{N}.
    \end{aligned}
\end{equation}

\begin{lemma}
    The space $Y=H^4(\mathbb{T},\mathbb{C})\times H^1(\mathbb{T},\mathbb{C})$ decomposes as $Y = \mathcal{V}_{0,0} \oplus \mathcal{U} \oplus \mathcal{W}_{Y}$, with
\begin{equation*}
    \mathcal{W}_{Y} = \overline{\bigoplus_{k=2}^{\infty} \mathcal{W}_k}^{Y}
\end{equation*}
where the subspaces $\mathcal{V}_{0,0}, \mathcal{U}$, and $\mathcal{W}_k$, defined below, are invariant under $\mathscr{L}_{0,0}$ and the following properties hold:

\begin{itemize}
\item[(i)] $\mathcal{V}_{0,0} = \text{span}\{\phi_1^+, \phi_1^-, \phi_0^+, \phi_0^-\}$ is the generalized kernel of $\mathscr{L}_{0,0}$. For any $\lambda \neq 0$ the operator $\mathscr{L}_{0,0} - \lambda : \mathcal{V}_{0,0} \to \mathcal{V}_{0,0}$ is invertible and
    \begin{equation} \label{A11-0}
        (\mathscr{L}_{0,0} - \lambda)^{-1} \phi_1^+ = -\frac{1}{\lambda} \phi_1^+, \quad (\mathscr{L}_{0,0} - \lambda)^{-1} \phi_1^- = -\frac{1}{\lambda} \phi_1^-,
    \end{equation}
    \begin{equation} \label{A11}
        (\mathscr{L}_{0,0} - \lambda)^{-1} \phi_0^- = -\frac{1}{\lambda} \phi_0^-,
    \end{equation}
    \begin{equation} \label{A12}
        (\mathscr{L}_{0,0} - \lambda)^{-1} \phi_0^+ = -\frac{1}{\lambda} \phi_0^+ + \frac{1}{\lambda^2} \phi_0^- .
    \end{equation}

    \item[(ii)] $\mathcal{U} := \text{span}\{ \phi_{-1}^+, \phi_{-1}^- \}$. For any $\lambda \neq \pm \im \,2c_0$ the operator $\mathscr{L}_{0,0} - \lambda : \mathcal{U} \to \mathcal{U}$ is invertible and
    \begin{equation} \label{A13}
    \begin{aligned}
        (\mathscr{L}_{0,0} - \lambda)^{-1} \phi_{-1}^+ &= \frac{1}{\lambda^2 + 4c_0^2} \left(-\lambda \phi_{-1}^+ + 2c_0 \phi_{-1}^-\right),\\
        (\mathscr{L}_{0,0} - \lambda)^{-1} \phi_{-1}^- &= \frac{1}{\lambda^2 + 4c_0^2} \left(-2c_0 \phi_{-1}^+ - \lambda \phi_{-1}^-\right).
    \end{aligned}
    \end{equation}

    \item[(iii)] Each subspace $\mathcal{W}_k := \text{span}\{ \phi_k^+, \phi_k^-, \phi_{-k}^+, \phi_{-k}^- \}$ is invariant under $\mathscr{L}_{0,0}$. Let
    \begin{equation*}
        \mathcal{W}_{L^2} = \overline{\bigoplus_{k=2}^{\infty} \mathcal{W}_k}^{ L^2}.
    \end{equation*}
    For any $|\lambda| < \delta(\tth)$ small enough, the operator $\mathscr{L}_{0,0} - \lambda : \mathcal{W}_{Y} \to \mathcal{W}_{L^2}$ is invertible and for any $\phi \in \mathcal{W}_{L^2}$,
    \begin{equation} \label{A14}
        (\mathscr{L}_{0,0} - \lambda)^{-1} \phi = \left( c_0^2 \partial^2_{x} + (1-\kappa\pa^2_x+b \pa^4_{x})G_0 \right)^{-1} \begin{bmatrix} c_0 \partial_x & -G_0 \\ 1-\kappa\pa^2_x+b \pa^4_{x} & c_0 \partial_x \end{bmatrix} \phi
        + \lambda \varphi_\phi(\lambda, x),
    \end{equation}
    for some analytic function $\lambda \mapsto \varphi_f(\lambda, \cdot) \in Y=H^4(\mathbb{T},\mathbb{C})\times H^1(\mathbb{T},\mathbb{C})$.
\end{itemize}
\end{lemma}

\begin{proof}
    By inspection the spaces $\mathcal{V}_{0,0}, \mathcal{U}$ and $\mathcal{W}_k$ are invariant under $\mathscr{L}_{0,0}$ and, by Fourier series, they decompose $Y=H^4(\mathbb{T},\mathbb{C})\times H^1(\mathbb{T},\mathbb{C})$. Formulas \eqref{A11}-\eqref{A12} follow using that $\phi_1^+, \phi_1^-, \phi_0^-$ are in the kernel of $\mathscr{L}_{0,0}$, and $\mathscr{L}_{0,0} \phi_0^+ = -\phi_0^-$. Formula \eqref{A13} follows using that $\mathscr{L}_{0,0} \phi_{-1}^+ = -2c_0 \phi_{-1}^-$ and $\mathscr{L}_{0,0} \phi_{-1}^- = 2c_0 \phi_{-1}^+$. Let us prove item $(iii)$. Let $\mathcal{W} := \mathcal{W}_{Y}$. The operator $(\mathscr{L}_{0,0} - \lambda \operatorname{Id})|_{\mathcal{W}}$ is invertible for any 
$$\lambda \notin \left\{ \pm \im c_0 k \pm \im\sqrt{|k|\tanh(\tth| k|)\left(1+\kappa k^2+b k^4\right)}, k \geq 2, k \in \mathbb{N} \right\}$$
and 
$$
(\mathscr{L}_{0,0}|_{\mathcal{W}})^{-1} = \left( c_0^2 \partial^2_{x} + (1-\kappa\pa_x^2+b \pa^4_{x}) G_0\right)^{-1} \begin{bmatrix} c_0 \partial_x & -G_0 \\ 1-\kappa\pa_x^2+b\pa^4_{x} & c_0 \partial_x \end{bmatrix} \Big|_{\mathcal{W}}.
$$
By Neumann series, for any $\lambda$ such that 
$$|\lambda| \| (\mathscr{L}_{0,0}|_{\mathcal{W}})^{-1} \|_{ \mathcal{L}(\mathcal{W}, Y)} < 1$$
we have 
$$
(\mathscr{L}_{0,0}|_{\mathcal{W}} - \lambda)^{-1} = (\mathscr{L}_{0,0}|_{\mathcal{W}})^{-1} (\operatorname{Id} - \lambda (\mathscr{L}_{0,0}|_{\mathcal{W}})^{-1})^{-1} 
= (\mathscr{L}_{0,0}|_{\mathcal{W}})^{-1} \sum_{k \geq 0} ((\mathscr{L}_{0,0}|_{\mathcal{W}})^{-1} \lambda)^k.
$$
Formula \eqref{A14} follows with 
$$\varphi_\phi(\lambda, x) := (\mathscr{L}_{0,0}|_{\mathcal{W}})^{-1} \sum_{k \geq 1} \lambda^{k-1} [(\mathscr{L}_{0,0}|_{\mathcal{W}})^{-1}]^k \phi.$$ 
\end{proof}

To prove Lemma \ref{expansion of the basis F}, we shall also use the following formulas obtained by \eqref{A6}, \eqref{mD}, and \eqref{eigenfunc of mathcall L00 2}:

\begin{equation} \label{A15}
    \begin{aligned}
        &\mathscr{L}'_{0,0} \phi^+_1=\vet{2\left(\ch/\ckb\right)^{-\frac{1}{2}}\sin(2x)}{b^{[1]}_{\tth,\kappa,b}+b^{[2]}_{\tth,\kappa,b}\cos(2x)},\quad \mathscr{L}'_{0,0} \phi^-_1=\vet{2\left(\ch/\ckb\right)^{-\frac{1}{2}}\cos(2x)}{-b^{[2]}_{\tth,\kappa,b}\sin(2x)},\\
        &\mathscr{L}'_{0,0} \phi^+_0=\vet{2\ch^{-1}\ckb\sin(x)}{\left(\ch^2+\ch^{-2}\right)\ckb^2\cos(x)}, \quad \mathscr{L}'_{0,0}\phi^-_0=\vet{0}{0},\\
        &\dot{\mathscr{L}}_{0,0} \phi^+_1=\vet{-\im b^{[3]}_{\tth,\kappa,b}\cos(x)}{-2\im\ch^{\frac{1}{2}}\ckb^{-\frac{1}{2}}(2b+\kappa)\sin(x)}, \quad \dot{\mathscr{L}}_{0,0} \phi^-_1=\vet{\im b^{[3]}_{\tth,\kappa,b} \sin(x)}{-2\im\ch^{\frac{1}{2}}\ckb^{-\frac{1}{2}}(2b+\kappa)\cos(x)},\\
        &\dot{\mathscr{L}}_{0,0} \phi^+_0=\vet{0}{0}, \quad \dot{\mathscr{L}}_{0,0}\phi^-_0=\vet{0}{0},
    \end{aligned}
\end{equation}
where
\begin{equation}  \label{b1 b2 b3}
\begin{aligned}
b^{[1]}_{\tth,\kappa,b}&:=\frac{\ckb^\frac{3}{2}(\ch^4-1)}{2\ch^{\frac{3}{2}}}, \ \ 
    b^{[2]}_{\tth,\kappa,b}:=\frac{\ckb^2\ch^4+29b+5\kappa-1}{2\ch^{\frac{3}{2}}\ckb^{\frac{1}{2}}}, \ \ 
b^{[3]}_{\tth,\kappa,b}:=\ch^{-\frac{1}{2}}\ckb^{\frac{1}{2}}\left(\ch^2+\tth(1-\ch^4)\right).
\end{aligned}
\end{equation}
We finally calculate $P'_{0,0} \phi^\sigma_k$ and $\dot{P}_{0,0} \phi^\sigma_k$.

\begin{lemma} One has
    \begin{equation} \label{A16}
        \begin{aligned}
    P'_{0,0}\phi^+_1&=\vet{\alpha_{\tth,\kappa,b}\cos(2x)}{\beta_{\tth,\kappa,b}\sin(2x)},\quad P'_{0,0}\phi^-_1=\vet{-\alpha_{\tth,\kappa,b}\sin(2x)}{\beta_{\tth,\kappa,b}\cos(2x)}, \quad 
            P'_{0,0}\phi^+_0= \delta_{\tth,\kappa,b}\phi^+_{-1}, \\
            P'_{0,0}\phi^-_0& =\vet{0}{0}, \quad \dot{P}_{0,0}\phi^+_0=\vet{0}{0}, \quad \dot{P}_{0,0}\phi^-_0=\vet{0}{0}, \quad 
            \dot{P}_{0,0}\phi^+_1=\im\frac{\gamma_{\tth, \kappa}}{4}\phi^-_{-1}, \quad  \dot{P}_{0,0}\phi^-_1=\im\frac{\gamma_{\tth, \kappa}}{4}\phi^+_{-1},
        \end{aligned}
    \end{equation}
where $\alpha_{\tth,\kappa,b}$,  $\beta_{\tth,\kappa,b}$, ${\gamma_{\tth, \kappa,b}}$ and $\delta_{\tth, \kappa,b}$ are defined in \eqref{47 coefficients}.
\end{lemma}
\begin{proof}
    We first calculate $P'_{0,0} \phi^+_1$. By \eqref{A3}, \eqref{A11-0}, \eqref{A11}, and \eqref{A15} we deduce
    \begin{align*}
        P'_{0,0} \phi^+_1=-\frac{1}{2\pi \im}\oint_{\Gamma} \frac{1}{\lambda} \left(\mathscr{L}_{0,0}-\lambda\right)^{-1} \vet{2\left(\ch/\ckb\right)^{-\frac{1}{2}}\sin(2x)}{b^{[1]}_{\tth,\kappa,b}+b^{[2]}_{\tth,\kappa,b}\cos(2x)} \de \lambda.
    \end{align*}
We note that 
\begin{equation*}
    \vet{2\left(\ch/\ckb\right)^{-\frac{1}{2}}\sin(2x)}{b^{[1]}_{\tth,\kappa,b}+b^{[2]}_{\tth,\kappa,b}\cos(2x)}=b^{[1]}_{\tth,\kappa,b}\phi^-_0+\vet{2\left(\ch/\ckb\right)^{-\frac{1}{2}}\sin(2x)}{b^{[2]}_{\tth,\kappa,b}\cos(2x)} , \qquad b^{[1]}_{\tth,\kappa,b},\,b^{[2]}_{\tth,\kappa,b} \mbox{ in } \eqref{b1 b2 b3} \ . 
\end{equation*}
Therefore by \eqref{A11} and \eqref{A14} there is an analytic function $\lambda \mapsto \varphi(\lambda,\cdot)\in Y=H^4(\mathbb{T},\mathbb{C})\times H^1(\mathbb{T},\mathbb{C})$ so that, by exploiting the identity $\tanh(2\tth)=2\ch^2/(1+\ch^4)$ in applying \eqref{A14}, we obtain
\begin{align*}
    P'_{0,0} \phi^+_1=-\frac{1}{2\pi \im} \oint_\Gamma \frac{1}{\lambda}\left(-\frac{b^{[1]}_{\tth,\kappa,b}}{\lambda}\phi^-_0+\vet{-\alpha_{\tth,\kappa,b}\cos(2x)}{-\beta_{\tth,\kappa,b}\sin(2x)}+\lambda\varphi(\lambda)\right) \de \lambda.
\end{align*}
Thus, by means of the Residue Theorem we obtain the first identity in \eqref{A16}. Similarly, one may calculate $P'_{0,0} \phi^-_1$. By \eqref{A3}, \eqref{A11}, and \eqref{A15}, one has $P'_{0,0}\phi^-_0=0$. Next, we calculate $P'_{0,0}\phi^+_0$. By \eqref{A3}, \eqref{A11}, \eqref{A12}, and \eqref{A15} we get 
\begin{align} \label{Pprimefp0}
P'_{0,0} \phi^+_0 = -\frac{1}{2\pi \im} \oint_{\Gamma} \frac{1}{\lambda} (\mathscr{L}_{0,0} - \lambda)^{-1} \vet{2\ch^{-1}\ckb\sin(x)}{\left(\ch^2+\ch^{-2}\right)\ckb^2\cos(x)} \de \lambda.
\end{align}

Next, we decompose
\begin{equation} \label{alphabetac21}
    \begin{aligned}
       \vet{2\ch^{-1}\ckb\sin(x)}{\left(\ch^2+\ch^{-2}\right)\ckb^2\cos(x)}&= \underbrace{\frac{1}{2}\ckb^{\frac{3}{2}}\ch^{\frac{1}{2}}\left(\ch^2+3\ch^{-2}\right)}_{=: \alpha} \phi^-_{-1} + \underbrace{\frac{1}{2}\ckb^{\frac{3}{2}}\ch^{\frac{1}{2}}\left(\ch^2-\ch^{-2}\right)}_{=: \beta} \phi^-_1. 
    \end{aligned}
\end{equation}
By \eqref{A15}, \eqref{A13}, \eqref{Pprimefp0}, and \eqref{alphabetac21}, using the Residue Theorem, we get
\[
P'_{0,0} \phi^+_0 = -\frac{1}{2\pi \im} \oint_{\Gamma} 
\left( -\frac{2\alpha c_0}{\lambda (\lambda^2 + 4c_0^2)} \phi^+_{-1} 
- \frac{\alpha}{\lambda^2 + 4c_0^2} \phi^-_{-1} 
- \frac{\beta}{\lambda^2} \phi^-_{1} \right) \de \lambda= \frac{\alpha}{2 c_0} \phi^+_{-1}.
\]
Thus, we obtain the third identity in \eqref{A16}. Now, we calculate $\dot{P}_{0,0} \phi^+_1$. First, we have 
\begin{align*}
    \dot{P}_{0,0} \phi^+_1=\frac{1}{2\pi }\oint_\Gamma \frac{1}{\lambda}\left(\mathscr{L}_{0,0}-\lambda\right)^{-1} \vet{b^{[3]}_{\tth,\kappa,b} \cos(x)}{2\ch^{\frac{1}{2}}\ckb^{-\frac{1}{2}}(2b+\kappa)\sin(x)} \de \lambda,
\end{align*}
where $b^{[3]}_{\tth,\kappa,b}$ is in \eqref{A15}, and then, writing
\begin{align*}
    &\vet{\cos(x)}{0}=\frac{1}{2}\left(\frac{\ch}{\ckb}\right)^{-\frac{1}{2}}\left(\phi^+_1+\phi^+_{-1}\right)\,, \qquad \vet{0}{\sin(x)}=\frac{1}{2}\left(\frac{\ch}{\ckb}\right)^{\frac{1}{2}}\left(\phi^+_1-\phi^+_{-1}\right),
\end{align*}
and using \eqref{A13}, we conclude using again the residue theorem 
\begin{align*}
    \dot{P}_{0,0}\phi^+_1=\frac{\im}{4}\left(b^{[3]}_{\tth,\kappa,b}\ch^{-\frac{3}{2}}\ckb^{-\frac{1}{2}}-(2\kappa+4b)\ckb^{-2}\right)  \phi^-_{-1}.
\end{align*}
Similarly, we have
\begin{align*}
    \dot{P}_{0,0}\phi^-_1=\frac{\im}{4}\left(b^{[3]}_{\tth,\kappa,b}\ch^{-\frac{3}{2}}\ckb^{-\frac{1}{2}}-(2\kappa+4b)\ckb^{-2}\right) \phi^+_{-1}.
\end{align*}
Finally, in view of \eqref{A15}, we have
\begin{equation*}
    \begin{aligned}
        \dot{P}_{0,0}\phi^+_0=&\frac{1}{2\pi \im}\oint_\Gamma \left(\mathscr{L}_{0,0}-\lambda\right)^{-1}\dot{\mathscr{L}}_{0,0}\left(\frac{1}{\lambda^2}\phi^-_0-\frac{1}{\lambda}\phi^+_0\right) \de \lambda=0,\\
        \dot{P}_{0,0}\phi^-_0=&\frac{1}{2\pi \im}\oint_\Gamma \left(\mathscr{L}_{0,0}-\lambda\right)^{-1}\dot{\mathscr{L}}_{0,0}\left(\frac{-1}{\lambda}\phi^-_0\right) \de \lambda=0.
    \end{aligned}
\end{equation*}
In conclusion, all the formulas in \eqref{A16} are proven.
\end{proof}

So far we have obtained the linear terms of the expansions \eqref{43 f+1}, \eqref{44 f-1}, \eqref{45 f+0}, and \eqref{46 f-0}. We now provide further information about the expansion of the basis at $\mu = 0$. 

\begin{lemma} 
     The basis $\{ \phi_k^{\sigma}(0,\e), \quad k = 0,1, \quad \sigma = \pm \}$ is real. For any $\varepsilon$ we obtain $\phi_0^-(0,\varepsilon) \equiv \phi_0^-$. The property \eqref{48 f-0=01} holds.
\end{lemma}

\begin{proof}
    The reality of the basis $\phi_k^\sigma(0,\varepsilon)$ is a consequence of Lemma \ref{properties of U and P}-(iii). Since, recalling \eqref{mathscr L mu e} and \eqref{mathcal B mu e}, $\mathscr{L}_{0,\varepsilon} \phi_0^- = 0$ for any $\varepsilon$ (cf. Lemma~\ref{237}), we deduce $(\mathscr{L}_{0,\varepsilon} - \lambda)^{-1} \phi_0^- = -\frac{1}{\lambda} \phi_0^-$ and then, using also the Residue Theorem,
\[
P_{0,\varepsilon} \phi_0^- = -\frac{1}{2\pi \im} \oint_\Gamma (\mathscr{L}_{0,\varepsilon} - \lambda)^{-1} \phi_0^- \, \de \lambda = \phi_0^-.
\]
In particular $P_{0,\varepsilon} \phi_0^- = P_{0,0} \phi_0^-$, for any $\varepsilon$ and we get, by \eqref{U transformation operators}, $\phi_0^-(0,\varepsilon) = U_{0,\varepsilon} \phi_0^- = \phi_0^-$, for any $\varepsilon$.

Let us prove property \eqref{48 f-0=01}. In view of \eqref{Parity structure} and since the basis is real, we know that
\[
\phi_k^+(0,\varepsilon) = \begin{bmatrix}\mathit{even}(x)\\ \mathit{odd}(x)\end{bmatrix}, \quad \phi_k^-(0,\varepsilon) = \begin{bmatrix}\mathit{odd}(x)\\ \mathit{even}(x)\end{bmatrix},
\]
for any $k = 0, 1$. By Lemma \ref{F is symplectic and reversible} the basis $\{\phi_k^\sigma(0,\varepsilon)\}$ is symplectic (cf. \eqref{basis is symplectic} and, since
\[
\mathcal{J} \phi_0^-(0,\varepsilon) = \mathcal{J} \phi_0^- = \begin{bmatrix}1\\0\end{bmatrix},
\]
for any $\varepsilon$, we get
\[
0 = (\mathcal{J} \phi_0^-(0,\varepsilon),\, \phi_1^+(0,\varepsilon)) = \left(\begin{bmatrix}1\\0\end{bmatrix}, \phi_1^+(0,\varepsilon)\right), \qquad 
1 = (\mathcal{J} \phi_0^-(0,\varepsilon),\, \phi_0^+(0,\varepsilon)) = \left(\begin{bmatrix}1\\0\end{bmatrix}, \phi_0^+(0,\varepsilon)\right).
\]
Thus the first component of both $\phi_1^+(0,\varepsilon)$ and $\phi_0^+(0,\varepsilon) - \begin{bmatrix}1\\0\end{bmatrix}$ has zero average, proving \eqref{48 f-0=01}.
\end{proof}

\begin{lemma} \label{fmu0 eq f0} 
For any small $\mu$, we have $\phi_0^+(\mu,0) \equiv \phi_0^+$ and $\phi_0^-(\mu,0) \equiv \phi_0^-$. Moreover, the vectors $\phi_1^+(\mu,0)$ and $\phi_1^-(\mu,0)$ have both components with zero space average.
\end{lemma}
\begin{proof}
The operator $\mathscr{L}_{\mu,0} = 
\begin{bmatrix}
c_0\partial_x & |D+\mu|\tanh(\tth|D+\mu|) \\
-1+\kappa(\pa_{x}+\im \mu)^2-b(\pa_{x}+\im \mu)^4 & c_0\partial_x
\end{bmatrix}$ 
leaves invariant the subspace $\mathcal{Z} := \text{span}\{ \phi_0^+, \phi_0^- \}$ since 
$\mathscr{L}_{\mu,0} \phi_0^+ = -(1+\kappa\mu^2+b\mu^4)\phi_0^-$ and $\mathscr{L}_{\mu,0} \phi_0^- = \mu\tanh(\tth\mu) \phi_0^+$. 
The operator $\mathscr{L}_{\mu,0}|_{\mathcal{Z}}$ has the two eigenvalues $\pm \im \sqrt{(1+\kappa\mu^2+b\mu^4)\mu\tanh(\tth\mu)}$, 
which, for small $\mu$, lie inside the loop $\Gamma$ around $0$ in \eqref{Projection P mu e}. Then, by \eqref{spectrum separated by Gamma}, 
we have $\mathcal{Z} \subseteq \mathcal{V}_{\mu,0} = \text{Rg}(P_{\mu,0})$ and 

\[
P_{\mu,0} \phi_0^\pm = \phi_0^\pm, \quad 
\phi_0^\pm (\mu,0) = U_{\mu,0} \phi_0^\pm = \phi_0^\pm, 
\quad \text{for any $\mu$ small.}
\]

The basis $\{ \phi_k^\sigma (\mu,0) \}$ is symplectic (cf. \eqref{basis is symplectic}). Then, since 
$\mathcal{J} \phi_0^+ = \begin{bmatrix} 0 \\ -1 \end{bmatrix}$ 
and $\mathcal{J} \phi_0^- = \begin{bmatrix} 1 \\ 0 \end{bmatrix}$, we have

\[
\begin{aligned}
0 &= \big( \mathcal{J} \phi_0^+ (\mu,0), \phi_1^\sigma (\mu,0) \big) 
= \Big( \begin{bmatrix} 0 \\ -1 \end{bmatrix}, \phi_1^\sigma (\mu,0) \Big), \quad 
0 = \big( \mathcal{J} \phi_0^- (\mu,0), \phi_1^\sigma (\mu,0) \big) 
= \Big( \begin{bmatrix} 1 \\ 0 \end{bmatrix}, \phi_1^\sigma (\mu,0) \Big),
\end{aligned}
\]
namely both the components of $\phi_1^\pm (\mu,0)$ have zero average.
\end{proof}

We finally consider the $\mu \varepsilon$ term in the expansion \eqref{the expandsion of f sigma mu}.

\begin{lemma} 
The derivatives $ (\partial_\mu \partial_\varepsilon \phi_k^{\sigma}(0,0) = \left( \dot{P}_{0,0}^{\prime} - \frac{1}{2} P_{0,0} \dot{P}_{0,0}^{\prime} \right) \phi_k^{\sigma}$ satisfy
\begin{equation} \label{A17}
\begin{aligned}
(\partial_\mu \partial_\varepsilon \phi_1^+(0,0)) &= \im \begin{bmatrix} \mathit{odd}(x) \\ \mathit{even}(x) \end{bmatrix}, & (\partial_\mu \partial_\varepsilon \phi_1^-(0,0)) &= \im \begin{bmatrix} \mathit{even}(x) \\ \mathit{odd}(x) \end{bmatrix}, \\
(\partial_\mu \partial_\varepsilon \phi_0^+(0,0)) &= \im \begin{bmatrix} \mathit{odd}(x) \\ \mathit{even}_0(x) \end{bmatrix}, & (\partial_\mu \partial_\varepsilon \phi_0^-(0,0)) &= \im \begin{bmatrix} \mathit{even}_0(x) \\ \mathit{odd}(x) \end{bmatrix}.
\end{aligned}
\end{equation}
\end{lemma}

 \begin{proof}
 We prove that $ \dot{P}_{0,0}^{\prime} = \eqref{A5a} + \eqref{A5b} + \eqref{A5c}$ is purely imaginary. This follows since the operators in \eqref{A5a}, \eqref{A5b}, and \eqref{A5c} are purely imaginary because $\dot{\mathscr{L}}_{0,0}$ is purely imaginary, $\mathscr{L}_{0,0}^{\prime}$ in \eqref{A6} is real, and $\dot{\mathscr{L}}_{0,0}^{\prime}$ in \eqref{A8} is purely imaginary (as argued in Lemma \ref{properties of U and P}-(iii) of \cite{BMV1}). Then, when applied to the real vectors $\phi_k^{\sigma}$, for $k = 0,1$ and $\sigma = \pm$, they produce purely imaginary vectors.

The property \eqref{Parity structure} implies that $ (\partial_\mu \partial_\varepsilon \phi_k^{\sigma}(0,0)) $ has the claimed parity structure in \eqref{A17}. We shall now prove that $ (\partial_\mu \partial_\varepsilon \phi_0^+(0,0)) $ has zero average. We have, by \eqref{A12} and \eqref{A15},
\begin{equation}
\eqref{A5a} \phi_0^+ := \frac{1}{2\pi i} \oint_{\Gamma} \frac{1}{\lambda}(\mathscr{L}_{0,0} - \lambda)^{-1} \dot{\mathscr{L}}_{0,0}(\mathscr{L}_{0,0} - \lambda)^{-1} \vet{2\ch^{-1}\ckb\sin(x)}{\left(\ch^2+\ch^{-2}\right)\ckb^2\cos(x)} \de \lambda.
\end{equation}

Since the operators $(\mathscr{L}_{0,0} - \lambda)^{-1}$ and $\mathscr{L}_{0,0}$ are both Fourier multipliers, they preserve the absence of average in the vectors. Thus, \eqref{A5a} $\phi_0^+$ has zero average.
Next, \eqref{A5b} $\phi_0^+ = 0$ since $\dot{\mathscr{L}}_{0,0} \phi_0^{\pm} = 0$, cf. \eqref{A15}. Finally, by \eqref{A12} and \eqref{A8}, where $q_1(x) = q_1^{[1]} \cos(x)$, $\tau_{1,1}(x)=\tau_{1,1}^{[1]}\sin(x)$, and $b_{1,1}=b^{[1]}_{1,1}\sin(x)$
\begin{equation}
\eqref{A5c} \phi_0^+ := \frac{1}{2 \pi \im} \oint_{\Gamma} (\mathscr{L}_{0,0} - \lambda)^{-1} \left(\frac{-1}{\lambda}\vet{\im q_1(x)}{\im \kappa\, \tau_{1,1}(x)-\im b \,b_{1,1}(x)} +\frac{1}{\lambda^2}\vet{0}{\im q_1(x)}\right) \de \lambda
\end{equation}
is a vector with zero average. We conclude that $\dot{P}_{0,0}^{\prime} \phi_0^+$ is an imaginary vector with zero average, as well as $(\partial_\mu \partial_\varepsilon \phi_0^+(0,0))$ since $P_{0,0}$ sends zero average functions into zero average functions. Finally, by \eqref{Parity structure}, $(\partial_\mu \partial_\varepsilon \phi_0^+(0,0))$ has the claimed structure in \eqref{A17}.

We finally consider $(\partial_\mu \partial_\varepsilon \phi_0^-(0,0))$. By \eqref{A11} and $\mathscr{L}_{0,0}^{\prime} \phi_0^- = 0$ (cf. \eqref{A15}), we obtain
\begin{equation}
\eqref{A5a} \phi_0^- = \frac{1}{2\pi i} \oint_{\Gamma} \frac{1}{\lambda}(\mathscr{L}_{0,0} - \lambda)^{-1} \dot{\mathscr{L}}_{0,0} (\mathscr{L}_{0,0} - \lambda)^{-1} \mathscr{L}_{0,0}^{\prime} \phi_0^- \de \lambda = 0.
\end{equation}
Next, by \eqref{A11} and $\dot{\mathscr{L}}_{0,0} \phi_0^- = 0$, we get \eqref{A5b} $\phi_0^- = 0$. Finally, by \eqref{A11} and \eqref{A8},
\begin{equation}
\eqref{A5c} \phi_0^- = -\frac{1}{2\pi \im} \oint_{\Gamma} (\mathscr{L}_{0,0} - \lambda)^{-1} \frac{1}{\lambda} \begin{bmatrix} 0 \\ \im q_1(x)\end{bmatrix} \de \lambda
\end{equation} has zero average since $(\mathscr{L}_{0,0} - \lambda)^{-1}$ is a Fourier multiplier (and thus preserves average absence).

This completes the proof of Lemma \ref{expansion of the basis F}.
\end{proof}

\section{Matrix Representation of $\mathcal{K}_{\mu,\e}$ on $\mathcal{V}_{\mu,\e}$}\label{secD}

We decompose $\mathcal{K}_{\mu,\varepsilon}$ in \eqref{mathcal B mu e} as
\[
\mathcal{K}_{\mu,\varepsilon} = \mathcal{K}^\varepsilon +\mathcal{K}^{\flat} + \mathcal{K}^\sharp+\mathcal{K}^\tau+\mathcal{K}^\beta,
\]
where $\mathcal{K}^\varepsilon$, $\mathcal{K}^{\flat}$, $\mathcal{K}^\sharp$, $\mathcal{K}^\tau$, abd $\mathcal{K}^\beta$ are the self-adjoint and reversibility preserving operators:
\begin{align} \label{B epsilon}
\mathcal{K}^\varepsilon := \mathcal{K}_{0,\varepsilon} := 
\begin{bmatrix}
1 + a_\varepsilon(x)-\kappa\tau_{0,\e}+b\beta_{0,\e} & -(c_0 + q_\varepsilon(x))\partial_x \\
\partial_x \circ (c_0 + q_\varepsilon(x)) & |D| \tanh((\tth + \ttf)|D|)
\end{bmatrix}, 
\end{align}

\begin{align} \label{B b}
\mathcal{K}^{\flat} := 
\begin{bmatrix}
0 & 0 \\
0 & |D + \mu| \tanh((\tth + \ttf)|D + \mu|) - |D| \tanh((\tth + \ttf)|D|)
\end{bmatrix}, 
\end{align}

\begin{align} \label{B sharp}
\mathcal{K}^\sharp :=  
\begin{bmatrix}
0 & -\im\, \mu q_\varepsilon \\
\im\, \mu q_\varepsilon & 0
\end{bmatrix}, 
\qquad 
\mathcal{K}^\tau :=  
\begin{bmatrix}
-\kappa \left(\tau_{\mu,\e}-\tau_{0,\e}\right) & 0 \\
0 & 0
\end{bmatrix}, \qquad 
\mathcal{K}^\beta :=  
\begin{bmatrix}
b \left(\beta_{\mu,\e}-\beta_{0,\e}\right) & 0 \\
0 & 0
\end{bmatrix}, 
\end{align}
where $\beta_{\mu,\e}$ and $\tau_{\mu,\e}$ in \eqref{Sigma}, $\ttf$ in \eqref{def:ttf}, $a_\e(x)$ and $q_\e(x)$ in \eqref{SN1}. In view of \eqref{DtanhD}, the operator $\mathcal{K}^{\flat}$ is analytic in $\mu$.

\begin{lemma} [Expansion of $\mathsf{K}^\varepsilon$] \label{Expansion of B epsilon}
Let $\mathsf{K}^\varepsilon(\mu)$ be the matrix representation, in the basis
$\Phi_{\mu,\varepsilon}$ of $\mathcal V_{\mu,\varepsilon}$ defined in
\eqref{F basis set and f}, of the operator $\mathcal{K}^\varepsilon$ in
\eqref{B epsilon}. Then $\mathsf{K}^\varepsilon(\mu)$ is self-adjoint and
preserves the reversibility symmetry. Moreover, for $|\mu|$ and $|\varepsilon|$
sufficiently small, it has the expansion
\begin{equation} \label{Be90}
\begin{aligned}
\mathsf{K}^\varepsilon =
\left(
\begin{NiceArray}{cc|ccc}[code-for-first-col=\scriptstyle]
\mathsf{e}_{11} \varepsilon^2 + \zeta_{\tth,\kappa,b} \mu^2 + r_1(\varepsilon^3, \mu \varepsilon^3)
  & \im\, r_2(\mu \varepsilon^2)
  & \mathsf{f}_{11} \varepsilon + r_3(\varepsilon^3, \mu \varepsilon^2)
  & \im\, r_4(\mu \varepsilon^3)
  &  \\
- \im\, r_2(\mu \varepsilon^2)
  & \zeta_{\tth,\kappa,b} \mu^2
  & \im\, r_6(\mu \varepsilon)
  & 0
  &  \\
\hline
\mathsf{f}_{11} \varepsilon + r_3(\varepsilon^3, \mu \varepsilon^2)
  & - \im\, r_6(\mu \varepsilon)
  & 1 + r_8(\varepsilon^2, \mu \varepsilon^2)
  & \im\, r_9(\mu \varepsilon^2)
  & \\
- \im\, r_4(\mu \varepsilon^3)
  & 0
  & - \im\, r_9(\mu \varepsilon^2)
  & 0
  &  \\
\end{NiceArray}
\right)
\end{aligned}+\cO(\mu^2\e,\mu^3)\,,
\end{equation}
where $\mathsf{e}_{11}$, $\mathsf{f}_{11}$ are defined respectively in \eqref{e11 f11}, and
\begin{align} \label{zeta h}
    \zeta_{\tth,\kappa,b}:=\frac{1}{8}c_0 \left(1+\tth(\ch^{-2}-\ch^2)-(4b+2\kappa)\ckb^{-2}\right)^2.
\end{align}
\end{lemma}

\begin{proof}
    We expand the matrix $\mathsf{K}^\varepsilon(\mu)$ as

\begin{equation} \label{expansion of B_e in mu}
\mathsf{K}^\varepsilon(\mu) = \mathsf{K}^\varepsilon(0) + \mu(\partial_\mu \mathsf{K}^\varepsilon)(0) + \frac{\mu^2}{2} (\partial_\mu^2 \mathsf{K}^0)(0) + \mathcal{O}(\mu^2 \varepsilon, \mu^3).
\end{equation}

\textbf{The matrix $\mathsf{K}^\varepsilon(0)$.}
The main result of this long paragraph is to prove that the matrix $\mathsf{K}^\varepsilon(0)$ has the expansion \eqref{Be(0)}. We start recalling that the self-adjoint $\mathsf{K}^\e(0)$ has the form 
\begin{align} \label{B epsilon 0}
    \mathsf{K}^{\varepsilon}(0) = 
\left(
\begin{array}{cccc}
\left( \mathcal{K}^{\varepsilon} \phi_1^{+}(\varepsilon),\ \phi_{1}^{+}(\varepsilon) \right) & 
\left( \mathcal{K}^{\varepsilon} \phi_1^{-}(\varepsilon),\ \phi_{1}^{+}(\varepsilon) \right) & 
\left( \mathcal{K}^{\varepsilon} \phi_0^{+}(\varepsilon),\ \phi_{1}^{+}(\varepsilon) \right) & 
\left( \mathcal{K}^{\varepsilon} \phi_0^{-}(\varepsilon),\ \phi_{1}^{+}(\varepsilon) \right) \\
\left( \mathcal{K}^{\varepsilon} \phi_1^{+}(\varepsilon),\ \phi_{1}^{-}(\varepsilon) \right) & 
\left( \mathcal{K}^{\varepsilon} \phi_1^{-}(\varepsilon),\ \phi_{1}^{-}(\varepsilon) \right) & 
\left( \mathcal{K}^{\varepsilon} \phi_0^{+}(\varepsilon),\ \phi_{1}^{-}(\varepsilon) \right) & 
\left( \mathcal{K}^{\varepsilon} \phi_0^{-}(\varepsilon),\ \phi_{1}^{-}(\varepsilon) \right) \\
\left( \mathcal{K}^{\varepsilon} \phi_1^{+}(\varepsilon),\ \phi_{0}^{+}(\varepsilon) \right) & 
\left( \mathcal{K}^{\varepsilon} \phi_1^{-}(\varepsilon),\ \phi_{0}^{+}(\varepsilon) \right) & 
\left( \mathcal{K}^{\varepsilon} \phi_0^{+}(\varepsilon),\ \phi_{0}^{+}(\varepsilon) \right) & 
\left( \mathcal{K}^{\varepsilon} \phi_0^{-}(\varepsilon),\ \phi_{0}^{+}(\varepsilon) \right) \\
\left( \mathcal{K}^{\varepsilon} \phi_1^{+}(\varepsilon),\ \phi_{0}^{-}(\varepsilon) \right) & 
\left( \mathcal{K}^{\varepsilon} \phi_1^{-}(\varepsilon),\ \phi_{0}^{-}(\varepsilon) \right) & 
\left( \mathcal{K}^{\varepsilon} \phi_0^{+}(\varepsilon),\ \phi_{0}^{-}(\varepsilon) \right) & 
\left( \mathcal{K}^{\varepsilon} \phi_0^{-}(\varepsilon),\ \phi_{0}^{-}(\varepsilon) \right)
\end{array}
\right),
\end{align}
where $\phi^{\sigma}_k(\e):=\phi^{\sigma}_k(0,\e)$, for $\sigma=\pm$, $k=0,1$. The matrix $\mathsf{K}^\varepsilon(0)$ is real, because the operator $\mathcal{K}^\varepsilon$ is real and the basis $\{\phi_k^{\pm}(0, \varepsilon)\}_{k=0,1}$ is real.
On the other hand, by \eqref{B are alternatively real or imaginary}, its matrix elements $(\mathsf{K}^\varepsilon(0))_{i,j}$ vanish for $i + j$ odd. In addition $\phi_0^{-}(0,\varepsilon) = 
\begin{bmatrix}
0 \\ 1
\end{bmatrix}
$ by \eqref{48 f-0=01}, and, by \eqref{B epsilon}, we get $\mathcal{K}^\varepsilon \phi_0^{-}(0,\varepsilon) = 0$, for any $\varepsilon$.
We deduce that the self-adjoint matrix $\mathsf{K}^\varepsilon(0)$ in \eqref{B epsilon 0} has the form
\begin{equation} \label{B_e 0}
\mathsf{K}^\varepsilon(0) =\left(
\begin{NiceArray}{cc|ccc}[code-for-first-col=\scriptstyle]
\mathsf{E}_{11}(0,\varepsilon)
  & 0
  & \mathsf{F}_{11}(0,\varepsilon)
  & 0
  &  \\
0
  & \mathsf{E}_{22}(0,\varepsilon)
  & 0
  & 0
  &  \\
\hline
\mathsf{F}_{11}(0,\varepsilon)
  & 0
  & \mathsf{G}_{11}(0,\varepsilon)
  & 0
  & \\
0
  & 0
  & 0
  & 0
  &  \\
\end{NiceArray}
\right),
\end{equation}
where $ \mathsf{E}_{11}(0,\varepsilon), \mathsf{E}_{22}(0,\varepsilon), \mathsf{G}_{11}(0,\varepsilon), \mathsf{F}_{11}(0,\varepsilon) $ are real.
We claim that \( \mathsf{E}_{22}(0,\varepsilon) \equiv 0 \) for any \( \varepsilon \). As a first step, following \cite{BMV1}, we prove that
\begin{align} \label{E22=0 or E11=0=F11}
    \text{either } \mathsf{E}_{22}(0,\varepsilon) \equiv 0, \quad \text{or} \quad \mathsf{E}_{11}(0,\varepsilon) \equiv 0 \equiv \mathsf{F}_{11}(0,\varepsilon).
\end{align}
Indeed, by Lemma~\ref{237}, the operator \( \mathscr{L}_{0,\varepsilon} \equiv  L_{0,\varepsilon} \) possesses, for any sufficiently small \( \varepsilon \neq 0 \), the eigenvalue 0 with a four dimensional generalized kernel 
\[
\mathcal{W}_{\varepsilon} := \text{span}\{ U_1, \tilde{U}_2, U_3, U_4 \},
\]
spanned by \( \varepsilon \)-dependent vectors. By Lemma \ref{kato thm} we obtain that \( \mathcal{W}_\varepsilon = \mathcal{V}_{0,\varepsilon} = \text{Rg}(P_{0,\varepsilon}) \) and by Lemma~\ref{237} we have \( \mathscr{L}_{0,\varepsilon}^2 = 0 \) on \( \mathcal{V}_{0,\varepsilon} \). Thus the matrix
\begin{equation} \label{Le0}
\mathsf{L}_\varepsilon(0) := \mathsf{J}_4 \mathsf{K}^\varepsilon(0) = 
\left(
\begin{NiceArray}{cc|ccc}[code-for-first-col=\scriptstyle]
0
  & \mathsf{E}_{22}(0,\varepsilon)
  & 0
  & 0
  &  \\
-\mathsf{E}_{11}(0,\varepsilon)
  & 0
  & -\mathsf{F}_{11}(0,\varepsilon)
  & 0
  &  \\
\hline
0
  & 0
  & 0
  & 0
  & \\
-\mathsf{F}_{11}(0,\varepsilon)
  & 0
  & -\mathsf{G}_{11}(0,\varepsilon)
  & 0
  &  \\
\end{NiceArray}
\right),
\end{equation}
which represents \( \mathscr{L}_{0,\varepsilon} : \mathcal{V}_{0,\varepsilon} \to \mathcal{V}_{0,\varepsilon} \), satisfies \( \mathsf{L}_\varepsilon^2(0) = 0 \), namely

\begin{equation}
\mathsf{L}_\varepsilon^2(0) = - 
\left(
\begin{NiceArray}{cc|ccc}[code-for-first-col=\scriptstyle]
(\mathsf{E}_{11}\mathsf{E}_{22})(0,\varepsilon)
  & 0
  & (\mathsf{E}_{22}\mathsf{F}_{11})(0,\e)
  & 0
  &  \\
0
  & (\mathsf{E}_{11}\mathsf{E}_{22})(0,\varepsilon)
  & 0
  & 0
  &  \\
\hline
0
  & 0
  & 0
  & 0
  & \\
0
  & (\mathsf{E}_{22}\mathsf{F}_{11})(0,\varepsilon)
  & 0
  & 0
  &  \\
\end{NiceArray}
\right)=0,
\end{equation}
which implies \eqref{E22=0 or E11=0=F11}. We now prove that the matrix \( \mathsf{K}^\varepsilon(0) \) defined in \eqref{B_e 0} expands as

\begin{equation} \label{Be(0)}
\mathsf{K}^\varepsilon(0) = 
\left(
\begin{NiceArray}{cc|ccc}[code-for-first-col=\scriptstyle]
\mathsf{e}_{11}\e^2+r(\e^3)
  & 0
  & \mathsf{f}_{11}\e+r(\e^3)
  & 0
  &  \\
0
  & 0
  & 0
  & 0
  &  \\
\hline
\mathsf{f}_{11}\e+r(\e^3)
  & 0
  & 1+r(\e^2)
  & 0
  & \\
0
  & 0
  & 0
  & 0
  &  \\
\end{NiceArray}
\right),
\end{equation}
where $\mathsf{e}_{11}$ and $\mathsf{f}_{11}$ are in \eqref{E110epsilon} and \eqref{F110e}. We expand the operator $\mathcal{K}^{\e}$ in \eqref{B epsilon} as 
\begin{equation} \label{Be B0 B1 B2}
\begin{aligned}
    &\mathcal{K}^{\e}=\mathcal{K}_{0}+\e \mathcal{K}_{1}+\e^2 \mathcal{K}_{2}+\cO(\e^3), \quad \mathcal{K}_{0}:=\begin{bmatrix}
1-\kappa\pa^2_{x}+b\pa^4_x&  -c_0\pa_x\\
        c_0\pa_x & ~~G_0
    \end{bmatrix},\\
    &\mathcal{K}_{1}:=\begin{bmatrix}
a_1(x)-\kappa\tau_1+b\beta_1&  -q_1(x)\pa_x\\
       \pa_x\circ q_1(x) & ~~0
    \end{bmatrix}, \quad \mathcal{K}_{2}:=\begin{bmatrix}
a_2(x)-\kappa\tau_2+b\beta_2&  -q_2(x)\pa_x\\
        \pa_x \circ q_2(x) & ~~ -\tf^{[2]}_2 \pa^2_x(1-\tanh^2(\tth|D|))
    \end{bmatrix},
\end{aligned}
\end{equation}
where the remainder term $\cO(\e^3)\in \mathcal{L}(Y,X)$, the functions $a_1(x)$, $q_1(x)$, $a_2(x)$, $q_2(x)$,  are given in \eqref{SN1}-\eqref{a[2]2}, the operators $\tau_1$, $\tau_2$ are given in \eqref{tauj} and $\beta_1$, $\beta_2$ are given in \eqref{betaj},  in view of \eqref{expfe}, $\tf^{[2]}_2=(\ckb^2(\ch^4-1)-2)/(4\ch^2)$.

Our goal here is to determine the expansion of $\mathsf{F}_{11}(0,\e)=\mathsf{f}_{11}\e+r(\e^3)$. By \eqref{43 f+1} we rewrite the real function $\phi^+_1(0,\e)$ as
\begin{equation} \label{expansion fp1 0 epsilon}
    \begin{aligned}
        & \phi^+_1(0,\e)=\phi^+_1+\e \phi^+_{1_1}+\e^2 \phi^+_{1_2}+\cO(\e^3),\\
        & \phi^+_1=\vet{\left(\ch/\ckb\right)^{\frac{1}{2}}\cos(x)}{\left(\ch/\ckb\right)^{-\frac{1}{2}}\sin(x)}, \quad \phi^+_{1_1}:=\vet{\alpha_{\tth,\kappa,b} \cos(2x)}{\beta_{\tth,\kappa,b}\sin(2x)}, \quad \phi^+_{1_2}:=\vet{\mathit{even}_0(x)}{\mathit{odd}(x)},
    \end{aligned}
\end{equation}
where both $\phi^+_{1_2}$ and $\cO(\e^3)$ are vectors in $H^4(\mathbb{T},\mathbb{C})\times H^1(\mathbb{T},\mathbb{C})$. Also, by \eqref{45 f+0} we split the real-valued function $\phi_0^+(0, \varepsilon)$ as
\begin{equation} \label{fp0 expansion}
\begin{aligned}
&\phi_0^+(0, \varepsilon) = \phi_0^+ + \varepsilon \phi_{0_1}^+ + \varepsilon^2 \phi_{0_2}^+ + \mathcal{O}(\varepsilon^3), \quad 
\phi_0^+ = \begin{bmatrix} 1 \\ 0 \end{bmatrix},\\
&\phi_{0_1}^+ := \delta_{\tth,\kappa,b}
\begin{bmatrix}
\left(\ch/\ckb\right)^{\frac{1}{2}} \cos(x) \\
- \left(\ch/\ckb\right)^{-\frac{1}{2}} \sin(x)
\end{bmatrix}, \quad
\phi_{0_2}^+ := 
\begin{bmatrix}
\mathit{even}_0(x) \\
\mathit{odd}(x)
\end{bmatrix}. 
\end{aligned}
\end{equation}
Let us expand $\mathsf{F}_{11}(0, \varepsilon)$.  
By \eqref{Be B0 B1 B2}, \eqref{expansion fp1 0 epsilon}, \eqref{fp0 expansion}, using that $\mathcal{K}_0$, $\mathcal{K}_1$ are self-adjoint and real, and $\mathcal{K}_0 \phi_1^+ = 0$, $\mathcal{K}_0 \phi_0^+ = \phi_0^+$, we obtain
\begin{align*}
\mathsf{F}_{11}(0, \varepsilon) 
&= \varepsilon \left[ \left( \mathcal{K}_1 \phi_1^+, \phi_0^+ \right) + \left( \phi_{1_1}^+, \phi_0^+ \right) \right] \\
&\quad + \varepsilon^2 \left[ \left( \mathcal{K}_2 \phi_1^+, \phi_0^+ \right) + \left( \mathcal{K}_1 \phi_1^+, \phi_{0_1}^+ \right) + \left( \mathcal{K}_1 \phi_0^+, \phi_{1_1}^+ \right) \right. \\
&\qquad \left. + \left( \phi_{1_2}^+, \phi_0^+ \right) + \left( \mathcal{K}_0 \phi_{1_1}^+, \phi_{0_1}^+ \right) \right] + r(\varepsilon^3).
\end{align*}
A straightforward calculation reveals that
\begin{align} \label{F110e}
    \mathsf{F}_{11}(0, \varepsilon) = \mathsf{f}_{11} \varepsilon + r(\varepsilon^3), \quad\text{with}\quad \mathsf{f}_{11} :=-\frac{1}{2}\left(\ch^{\frac{5}{2}}-\ch^{-\frac{3}{2}}\right)\ckb^{\frac{3}{2}}.
\end{align}
Since $\mathsf{f}_{11}$ is nonzero,  the second alternative in \eqref{E22=0 or E11=0=F11} is ruled out, implying $\mathsf{E}_{22}(0, \varepsilon) \equiv 0$.

Our next goal is to determine the expansion of $\mathsf{E}_{11}(0,\e)=\mathsf{e}_{11}\e^2+r(\e^3)$. Since $\mathcal{K}_0 \phi^+_1=-\mathcal{J}\mathscr{L}_{0,0} \phi^+_1=0$, and both $\mathcal{K}_0$, $\mathcal{K}_1$ are self-adjoint real operators, we obtain
\begin{equation}
\begin{aligned}
\mathsf{E}_{11}(0, \varepsilon) &= \left( \mathcal{K}^\varepsilon \phi_1^+(0,\varepsilon), \phi_1^+(0,\varepsilon) \right) \\
&= \varepsilon \left( \mathcal{K}_1 \phi_1^+, \phi_1^+ \right) + \varepsilon^2 \left[ \left( \mathcal{K}_2 \phi_1^+, \phi_1^+ \right) + 2 \left( \mathcal{K}_1 \phi_1^+, \phi^+_{1_1} \right) + \left( \mathcal{K}_0 \phi^+_{1_1}, \phi^+_{1_1} \right) \right] + \mathcal{O}(\varepsilon^3).
\end{aligned}
\end{equation}
A straightforward calculation reveals that
\begin{equation} \label{E110epsilon}
\begin{aligned}
\mathsf{E}_{11}(0, \varepsilon) &= \mathsf{e}_{11} \varepsilon^2 + r(\varepsilon^3), \\
\mathsf{e}_{11} &:= \Big(-b^2\ch^8+20 b^2\ch^4+69 b^2+5 b\ch^8 \kappa +8 b\ch^8-89 b\ch^4 \kappa -140 b\ch^4+90 b\kappa +78 b\\
&+6\ch^8 \kappa ^2+15\ch^8 \kappa +9\ch^8-25\ch^4 \kappa ^2-44\ch^4 \kappa -10\ch^4+21 \kappa ^2+30\kappa +9\Big)\Big(8\ch^3\ckb r_{\tth,\kappa,b}\Big)^{-1}.
\end{aligned}
\end{equation}
We pause to remark that $\mathsf{e}_{11}$ may vanish due to the bending effect, in contrast to the pure-gravity case of \cite{BMV3} and the gravity--capillary case of \cite{HM2025}, where $\mathsf{e}_{11}\neq0$. 

Next let us determine the expansion of $\mathsf{G}_{11}(0,\e)$. By \eqref{45 f+0} we split the real-valued function $\phi_0^+(0, \varepsilon)$ as
\begin{equation} 
\begin{aligned}
&\phi_0^+(0, \varepsilon) = \phi_0^+ + \varepsilon \phi_{0_1}^+ + \varepsilon^2 \phi_{0_2}^+ + \mathcal{O}(\varepsilon^3), \quad 
\phi_0^+ = \begin{bmatrix} 1 \\ 0 \end{bmatrix},\\
&\phi_{0_1}^+ := \delta_{\tth,\kappa,b}
\begin{bmatrix}
\left(\ch/\ckb\right)^{\frac{1}{2}} \cos(x) \\
- \left(\ch/\ckb\right)^{-\frac{1}{2}} \sin(x)
\end{bmatrix}, \quad
\phi_{0_2}^+ := 
\begin{bmatrix}
\mathit{even}_0(x) \\
\mathit{odd}(x)
\end{bmatrix}. 
\end{aligned}
\end{equation}
Since, by \eqref{eigenfunc of mathcall L00} and \eqref{Be B0 B1 B2}, $\mathcal{K}_0 \phi_0^+ = \phi_0^+$, using that $\mathcal{K}_0$, $\mathcal{K}_1$ are self-adjoint real operators, and $(\phi_0^+,\phi_0^+) = 1$, $(\phi_0^+, \phi_{0_1}^+)=0$, we have
\[
\mathsf{G}_{11}(0, \varepsilon) = (\mathcal{K}^\varepsilon \phi_0^+(0, \varepsilon), \phi_0^+(0, \varepsilon)) = 1 + \varepsilon(\mathcal{K}_1 \phi_0^+, \phi_{0}^+) + \varepsilon(\mathcal{K}_0 \phi_{0_1}^+, \phi_{0}^+)+
r(\varepsilon^2).
\]
By \eqref{Be B0 B1 B2} and Lemma~\ref{lem:pa.exp}, Lemma~\ref{tau}, and Lemma~\ref{beta}, one has $ \left(\mathcal{K}_1 \phi_0^+,\phi^+_0\right)=0$, and $(\mathcal{K}_0 \phi_{0_1}^+, \phi_{0}^+)=( \phi_{0_1}^+,\mathcal{K}_0\phi_{0}^+)=0$.
Therefore, we deduce $\mathsf{G}_{11}(0, \varepsilon) = 1 + r(\varepsilon^2)$.

The next step is to compute the terms linear in $\mu$ of $\mathsf{K}^\varepsilon(\mu)$. We have 
\begin{align} \label{X+X}
    \partial_\mu \mathsf{K}^\varepsilon(0) = \widetilde{X} + \widetilde{X}^* \quad \text{where} \quad 
\widetilde{X} := \left( \left( \mathcal{K}^\varepsilon \phi_k^\sigma(0,\varepsilon), \, (\partial_\mu \phi_{k'}^{\sigma'})(0,\varepsilon) \right) \right)_{k,k'=0,1;\, \sigma,\sigma'=\pm}. 
\end{align}
and now prove that 
\begin{equation} \label{X}
\widetilde{X} = 
\left(
\begin{NiceArray}{cc|ccc}[code-for-first-col=\scriptstyle]
\cO(\e^3)
  & 0
  & \cO(\e^2)
  & 0
  &  \\
\cO(\e^2)
  & 0
  & \cO(\e)
  & 0
  &  \\
\hline
\cO(\e^3)
  & 0
  & \cO(\e^2)
  & 0
  & \\
\cO(\e^3)
  & 0
  & \cO(\e^2)
  & 0
  &  \\
\end{NiceArray}
\right)\,.
\end{equation}
Indeed consider the matrix $\mathsf{L}_\varepsilon(0)$ in \eqref{Le0}, where $\mathsf{E}_{22}(0, \varepsilon) = 0$, and recall that it represents the action of the operator $\mathscr{L}_{0,\varepsilon} : \mathcal{V}_{0,\varepsilon} \to \mathcal{V}_{0,\varepsilon}$ in the basis $\{\phi_k^\sigma(0,\varepsilon)\}$;  we deduce that
\[
\mathscr{L}_{0,\varepsilon} \phi_1^-(0,\varepsilon) = 0,\quad \mathscr{L}_{0,\varepsilon} \phi_0^-(0,\varepsilon) = 0.
\]
Thus also $\mathcal{K}^\varepsilon \phi_1^-(0,\varepsilon) =\mathcal{K}^\varepsilon \phi_0^-(0,\varepsilon) = 0$, and the second and fourth columns of the matrix $\widetilde{X}$ in \eqref{X} are zero. Now we compute $\pa_\mu \phi_k^\sigma(0,\e)$.  
In view of \eqref{43 f+1}--\eqref{47 coefficients} and by denoting with a dot the derivative w.r.t.\ $\mu$, one has
\begin{equation} \label{dotfp1 dotfp0 dotfn1 dotfn0}
    \begin{aligned}
\dot{\phi}_1^+(0,\varepsilon) &= \frac{\im}{4} \gamma_{\tth,\kappa,b} 
\begin{bmatrix}
(\ch/\ckb)^{\frac{1}{2}}\sin(x) \\
(\ch/\ckb)^{-\frac{1}{2}} \cos(x)
\end{bmatrix}
+ \im\varepsilon 
\begin{bmatrix}
\mathit{odd}(x) \\
\mathit{even}(x)
\end{bmatrix}
+ \mathcal{O}(\varepsilon^2), \\
\dot{\phi}_1^-(0,\varepsilon) &= \frac{\im}{4} \gamma_{\tth,\kappa,b} 
\begin{bmatrix}
(\ch/\ckb)^{\frac{1}{2}}\cos(x) \\
-(\ch/\ckb)^{\frac{1}{2}}\sin(x)
\end{bmatrix}
+ \im\varepsilon 
\begin{bmatrix}
\mathit{even}(x) \\
\mathit{odd}(x)
\end{bmatrix}
+ \mathcal{O}(\varepsilon^2),\\
\dot{\phi}_0^+(0,\varepsilon) &= \im\varepsilon 
\begin{bmatrix}
\mathit{odd}(x) \\
\mathit{even}_0(x)
\end{bmatrix}
+ \mathcal{O}(\varepsilon^2),  \qquad
\dot{\phi}_0^-(0,\varepsilon) = \im\varepsilon 
\begin{bmatrix}
\mathit{even}_0(x) \\
\mathit{odd}(x)
\end{bmatrix}
+ \mathcal{O}(\varepsilon^2). 
\end{aligned}
\end{equation}
In view of \eqref{pa_t eta psi}, \eqref{43 f+1}-\eqref{48 f-0=01}, \eqref{Le0}, \eqref{E110epsilon}, \eqref{F110e}, and since $\mathcal{K}^\varepsilon \phi_k^\sigma(0, \varepsilon) = -\mathcal{J} \mathscr{L}_\varepsilon \phi_k^\sigma(0, \varepsilon)$, we have
\begin{equation} \label{Befp1 Befp0}
   \begin{aligned}
    \mathcal{K}^\varepsilon \phi_1^+(0, \varepsilon) &= \mathsf{E}_{11}(0,\varepsilon) \, \mathcal{J} \phi_1^-(0,\varepsilon) + \mathsf{F}_{11}(0,\varepsilon) \, \mathcal{J} \phi_0^- 
    = \varepsilon 
    \begin{bmatrix}
        \mathsf{f}_{11} \\
        0
    \end{bmatrix}
    + \varepsilon^2 \mathsf{e}_{11}
    \begin{bmatrix}
     (\ch/\ckb)^{-\frac{1}{2}} \cos(x) \\
       (\ch/\ckb)^{\frac{1}{2}} \sin(x)
    \end{bmatrix}
    + \mathcal{O}(\varepsilon^3),
    \\
    \mathcal{K}^\varepsilon \phi_0^+(0, \varepsilon) &= \mathsf{F}_{11}(0,\varepsilon) \, \mathcal{J} \phi_1^-(0,\varepsilon) + \mathsf{G}_{11}(0,\varepsilon) \, \mathcal{J} \phi_0^- 
    = 
    \begin{bmatrix}
        1 \\
        0
    \end{bmatrix}
    + \varepsilon \mathsf{f}_{11}
    \begin{bmatrix}
     (\ch/\ckb)^{-\frac{1}{2}} \cos(x) \\
       (\ch/\ckb)^{\frac{1}{2}} \sin(x)
    \end{bmatrix}
    + \mathcal{O}(\varepsilon^2).
\end{aligned} 
\end{equation}
We deduce \eqref{X} by \eqref{dotfp1 dotfp0 dotfn1 dotfn0} and \eqref{Befp1 Befp0}.

Next we compute the terms quadratic in $\mu$. By denoting with a double dot the double derivative with respect to $\mu$, we have
\begin{align} \label{ddotB00}
\partial_\mu^2 \mathsf{K}^0(0) = \left( \mathcal{K}_0 \phi_k^\sigma, \ddot{\phi}_{k'}^{\sigma'}(0,0) \right)
+ \left( \ddot{\phi}_k^\sigma(0,0), \mathcal{K}_0 \phi_{k'}^{\sigma'} \right)
+ 2 \left( \mathcal{K}_0 \dot{\phi}_k^\sigma(0,0), \dot{\phi}_{k'}^{\sigma'}(0,0) \right)
= : \widetilde{Y} + \widetilde{Y}^* + 2\widetilde{Z}. 
\end{align}
We observe that the first, second and fourth column of $\widetilde{Y}$ are zero, since
$\mathcal{K}_0 \phi^\sigma_k=-\mathcal{J}\mathscr{L}_{0,0} \phi^\sigma_k=0$ for $\phi_k^\sigma \in \{ \phi_1^+, \phi_1^-, \phi_0^- \}$. The third column is also zero by noting that $\mathcal{K}_0 \phi_0^+ = \phi_0^+$ and recalling \eqref{43 f+1}--\eqref{48 f-0=01}
\begin{align*}
\ddot{\phi}_1^+(0,0) =
\begin{bmatrix}
\mathit{even}_0(x) + \im\,\mathit{odd}(x) \\
\mathit{odd}(x) + \im\,\mathit{even}_0(x)
\end{bmatrix}, \quad
\ddot{\phi}_1^-(0,0) =
\begin{bmatrix}
\mathit{odd}(x) + \im\,\mathit{even}_0(x) \\
\mathit{even}_0(x) + \im\,\mathit{odd}(x)
\end{bmatrix}, \quad \ddot{\phi}_0^+(0,0) = \ddot{\phi}_0^-(0,0) = 0.
\end{align*}
This implies $\widetilde{Y}=0$. Finally, by \eqref{dotfp1 dotfp0 dotfn1 dotfn0}, we have $\dot{\phi}_0^+(0,0) = \dot{\phi}_0^-(0,0) = 0$. By the definition of the matrix $\widetilde{Z}$, whose entries involve the vectors $\dot \phi_k^\sigma(0,0)$,
this implies that the last two columns of $\widetilde{Z}$ vanish. By self-adjointness of $\mathcal{K}_0$, the corresponding
last two rows also vanish. Also, we observe that $\left( \mathcal{K}_0 \dot{\phi}_1^+(0,0), \dot{\phi}_1^-(0,0) \right)=0$. Therefore, we have
\begin{align} \label{Z}
\widetilde{Z} = 
\left(
\begin{NiceArray}{cc|ccc}[code-for-first-col=\scriptstyle]
\zeta_{\tth,\kappa,b}
  & 0
  & 0
  & 0
  &  \\
0
  & \zeta_{\tth,\kappa,b}
  & 0
  & 0
  &  \\
\hline
0
  & 0
  & 0
  & 0
  & \\
0
  & 0
  & 0
  & 0
  &  \\
\end{NiceArray}
\right),\qquad\text{with}\qquad\zeta_{\tth,\kappa,b}:=\frac{1}{8}c_0\gamma^2_{\tth,\kappa,b}.
\end{align}
In conclusion, \eqref{expansion of B_e in mu}, \eqref{X+X}, \eqref{X}, \eqref{ddotB00}, and \eqref{Z} imply \eqref{Be90}, using also the self-adjointness of $\mathcal{K}^{\e}$ and \eqref{B are alternatively real or imaginary}. 
\end{proof}

Next, we list the matrix representations of $\mathcal{K}^{\flat}$, $\mathcal{K}^{\sharp}$ and $\mathcal{K}^{\tau}$, respectively.

\begin{lemma}
    [Expansion of $\mathsf{K}^{\flat}$] \label{Expansion of B flat} Let $\mathsf{K}^{\flat}$ be the matrix representation, in the basis
$\Phi_{\mu,\varepsilon}$ of $\mathcal V_{\mu,\varepsilon}$ defined in
\eqref{F basis set and f}, of the operator $\mathcal K^\flat$ in
\eqref{B b}. Then $\mathsf{K}^{\flat}$ is self-adjoint and
preserves the reversibility symmetry. Moreover, for $|\mu|$ and $|\varepsilon|$
sufficiently small, it has the expansion
\begin{equation} \label{mathsfBb}
\mathsf{K}^{\flat} = 
\left(
\begin{NiceArray}{cc|ccc}[code-for-first-col=\scriptstyle]
-\frac{\mu^2}{4}\flat_{\tth,\kappa,b}
  & \im\,(\frac{\mu}{2}\hat{\flat}_{\tth,\kappa,b}+r_2(\mu\e^2))
  & 0
  & 0
  &  \\
-\im\,(\frac{\mu}{2}\hat{\flat}_{\tth,\kappa,b}+r_2(\mu\e^2))
  & -\frac{\mu^2}{4}\flat_{\tth,\kappa,b}
  & \im\, r_6(\mu\e)
  & 0
  &  \\
\hline
0
  & -\im\, r_6(\mu\e)
  & 0
  & 0
  & \\
0
  & 0
  & 0
  & \mu \tanh(\tth \mu)
  &  \\
\end{NiceArray}
\right)
+ \mathcal{O}(\mu^2 \varepsilon, \mu^3),
\end{equation}
where 
\begin{align} \label{bh440}
\flat_{\tth,\kappa,b}& := \gamma_{\tth,\kappa,b}\ch\ckb+\tth\ch^{-1}\ckb(1-\ch^4)\left(\gamma_{\tth,\kappa,b}-2(1-\tth\ch^2)\right)\,,\qquad\hat{\flat}_{\tth,\kappa,b}:=\ch^{-1}\ckb (\ch^2 + (1 - \ch^4)\tth). 
\end{align}
\end{lemma}

\begin{proof}
    The proof is analogous to that of \cite[Lemma 4.5]{BMV3}.
\end{proof}
\begin{lemma} [Expansion of $\mathsf{K}^{\sharp}$] \label{Expansion of B sharp}
Let $\mathsf{K}^{\sharp}$ be the matrix representation, in the basis
$\Phi_{\mu,\varepsilon}$ of $\mathcal V_{\mu,\varepsilon}$ defined in
\eqref{F basis set and f}, of the operator $\mathcal K^\sharp$ in
\eqref{B sharp}. Then $\mathsf{K}^{\sharp}$ is self-adjoint and
preserves the reversibility symmetry. Moreover, for $|\mu|$ and $|\varepsilon|$
sufficiently small, it has the expansion
\begin{align} \label{mathsf B sharp}
\mathsf{K}^{\sharp} =
\left(
\begin{NiceArray}{cc|ccc}[code-for-first-col=\scriptstyle]
0
  & \im\, r_2(\mu\e^2)
  & 0
  & \im\,\mu\e\ch^{-\frac{1}{2}}\ckb^{\frac{1}{2}}+\im\, r_4(\mu\e^2)
  &  \\
-\im\, r_2(\mu\e^2)
  & 0
  & -\im\, r_6(\mu\e)
  & 0
  &  \\
\hline
0
  & \im\, r_6(\mu\e)
  & 0
  & -\im\, r_9(\mu\e^2)
  & \\
-\im\, \mu \e \ch^{-\frac{1}{2}}\ckb^{\frac{1}{2}}-\im r_4(\mu\e^2)
  & 0
  & \im\, r_9(\mu\e^2)
  & 0
  &  \\
\end{NiceArray}
\right)
+ \mathcal{O}(\mu^2 \varepsilon)\,.
\end{align}
\end{lemma}

\begin{proof}
    The proof is analogous to the one of \cite[Lemma 4.6]{BMV3}.
\end{proof}

\begin{lemma} [Expansion of $\mathsf{K}^{\tau}$] \label{Expansion of B s}
Let $\mathsf{K}^{\tau}$ be the matrix representation, in the basis
$\Phi_{\mu,\varepsilon}$ of $\mathcal V_{\mu,\varepsilon}$ defined in
\eqref{F basis set and f}, of the operator $\mathcal K^\tau$ in
\eqref{B sharp}. Then $\mathsf{K}^{\tau}$ is self-adjoint and
preserves the reversibility symmetry. Moreover, for $|\mu|$ and $|\varepsilon|$
sufficiently small, it has the expansion
\begin{equation}\label{mathsf B tau}
\begin{aligned} 
\mathsf{K}^{\tau} &=
\left(
\begin{NiceArray}{cc|ccc}[code-for-first-col=\scriptstyle]
\frac{\kappa\mu^2\ch}{2\ckb}(\gamma_{\tth,\kappa,b}+1)
  & \im\frac{\kappa\mu\ch}{\ckb}+\im \,r_2(\mu\e^2)
  & 0
  & \im \,r_4(\mu\e^2)
  &  \\
-\im\frac{\kappa\mu\ch}{\ckb}-\im \,r_2(\mu\e^2)
  & \frac{\kappa\mu^2\ch}{2\ckb}(\gamma_{\tth,\kappa,b}+1)
  & \im \,r_6(\mu\e)
  & 0
  &  \\
\hline
0
  & -\im \,r_6(\mu\e)
  & \kappa\mu^2
  & \im \,r_9(\mu\e^2)
  & \\
-\im \,r_4(\mu\e^2)
  & 0
  & -\im \,r_9(\mu\e^2)
  & 0
  &  \\
\end{NiceArray}
\right)+ \cO(\mu^2 \e,\mu^3),
\end{aligned}
\end{equation}
\end{lemma}
\begin{proof}
    The proof is analogous to the one of \cite[Lemma 4.8]{HM2025}.
\end{proof}

Finally, we consider the new term $\mathcal{K}^{\beta}$ that takes into the account the hydroelastic effect.

\begin{lemma} [Expansion of $\mathsf{K}^{\beta}$] \label{Expansion of B beta}
 Let $\mathsf{K}^{\beta}$ be the matrix representation, in the basis
$\Phi_{\mu,\varepsilon}$ of $\mathcal V_{\mu,\varepsilon}$ defined in
\eqref{F basis set and f}, of the operator $\mathcal K^\beta$ in
\eqref{B sharp}. Then $\mathsf{K}^{\beta}$ is self-adjoint and
preserves the reversibility symmetry. Moreover, for $|\mu|$ and $|\varepsilon|$
sufficiently small, it has the expansion
\begin{equation}\label{mathsf B s}
\begin{aligned} 
\mathsf{K}^{\beta} &=
\left(
\begin{NiceArray}{cc|ccc}[code-for-first-col=\scriptstyle]
\frac{b\mu^2\ch}{\ckb}(\gamma_{\tth,\kappa,b}+3)
  & \im\frac{2b\mu\ch}{\ckb}+\im \,r_2(\mu\e^2)
  & 0
  & \im \,r_4(\mu\e^2)
  &  \\
-\im\frac{2b\mu\ch}{\ckb}-\im \,r_2(\mu\e^2)
  & \frac{b\mu^2\ch}{\ckb}(\gamma_{\tth,\kappa,b}+3)
  & \im \,r_6(\mu\e)
  & 0
  &  \\
\hline
0
  & -\im \,r_6(\mu\e)
  & 0
  & \im \,r_9(\mu\e^2)
  & \\
-\im \,r_4(\mu\e^2)
  & 0
  & -\im \,r_9(\mu\e^2)
  & 0
  &  \\
\end{NiceArray}
\right)
+ \cO(\mu^2 \e,\mu^3),
\end{aligned}
\end{equation}
\end{lemma}

\begin{proof}
    We denote $\partial_{\mu}$ with a dot. Recalling \eqref{expfe} and Lemma~\ref{tau}, we write the operator $\mathcal{K}^{\beta}$ in \eqref{B sharp} as
    \begin{equation} \label{B beta expand}
    \begin{aligned}
        \mathcal{K}^\beta=\underbrace{\begin{bmatrix}
b \mu\dot{\beta}_{0,\e} & 0 \\
0 & 0
\end{bmatrix}}_{=:\mathcal{K}^\beta_1}+\underbrace{\begin{bmatrix}
\frac{1}{2}b \mu^2 \ddot{\beta}_{0,\e} & 0 \\
0 & 0
\end{bmatrix}}_{=:\mathcal{K}^\beta_2}+\cO(\mu^3)\,.
    \end{aligned}
    \end{equation}
We write the pure imaginary operator $\mathcal{K}^{\beta}_1$ in \eqref{B beta expand} as
\begin{align} \label{tau B11 B12}
    \mathcal{K}^{\beta}_1=\mathcal{K}^{\beta}_{11}+\mathcal{K}^{\beta}_{12},\quad \mathcal{K}^{\beta}_{11}:=4b\mu\im\begin{bmatrix}
\pa^3_x & 0 \\
0 & 0
\end{bmatrix},\quad\mathcal{K}^{\beta}_{12}:=b  \mu\e \im\begin{bmatrix}
\sum_{j=1}^4 j b_{1,j}\pa^{j-1}_x & 0 \\
0 & 0
\end{bmatrix}+\im \cO(\mu\e^2) \begin{bmatrix}
\operatorname{Id} & 0 \\
0 & 0
\end{bmatrix}.
\end{align}
In view of \eqref{43 f+1}-\eqref{46 f-0}, we have expansions as follows:
\begin{equation} \label{Lambda f pm 10 2}
\begin{aligned}
  \mathcal{K}^{\beta}_{11} \phi^+_1(\mu,\e)&= 4b \mu\im \vet{\ch^{\frac{1}{2}}\ckb^{-\frac{1}{2}}\sin(x)}{0}
+b\mu^2 \gamma_{\tth,\kappa,b}
\vet{\ch^{\frac{1}{2}}\ckb^{-\frac{1}{2}}\cos(x)}{0}
+32b  \alpha_{\tth,\kappa,b}\mu\e \im
\begin{bmatrix}
 \sin(2x) \\
0
\end{bmatrix} \\
& \quad
+\cO(\mu\e^2)\vet{\mathit{odd}(x)}{0}+ \mathcal{O}(\mu^3, \mu^2 \e),\\
 \mathcal{K}^{\beta}_{11}\phi^-_1(\mu,\e)&=4b \mu\im \vet{\ch^{\frac{1}{2}}\ckb^{-\frac{1}{2}}\cos(x)}{0}
- b\mu^2 \gamma_{\tth,\kappa,b}
\vet{\ch^{\frac{1}{2}}\ckb^{-\frac{1}{2}}\sin(x)}{0}
+32 b  \alpha_{\tth,\kappa,b}\mu\e \im
\begin{bmatrix}
 \cos(2x) \\
0
\end{bmatrix} \\
& \quad
+\cO(\mu\e^2)\vet{\mathit{even}(x)}{0}+ \mathcal{O}(\mu^3, \mu^2 \e),\\
 \mathcal{K}^{\beta}_{11} \phi^+_0(\mu,\e)&=4b \delta_{\tth,\kappa,b} \mu\e\im\vet{\ch^{\frac{1}{2}}\ckb^{-\frac{1}{2}}\sin(x)}{0}+\cO(\mu\e^2)\vet{\mathit{odd}(x)}{0}+\cO(\mu^2\e), \quad 
  \mathcal{K}^{\beta}_{11} \phi^-_0(\mu,\e)=\cO(\mu^2\e).
\end{aligned}
\end{equation}
A straightforward calculation reveals that the matrix
$\mathsf{K}^{\beta}_{11}$ representing the action of 
the operator $\mathcal{K}^{\beta}_{11}$ on the basis 
$\Phi_{\mu,\e}$ of $\mathcal{V}_{\mu,\e}$ in \eqref{F basis set and f}
admits the expansion
\begin{equation}\label{mathsf B s11} 
\begin{aligned} 
\mathsf{K}^{\beta}_{11} &=
\left(
\begin{NiceArray}{cc|ccc}[code-for-first-col=\scriptstyle]
\frac{b\gamma_{\tth,\kappa,b} \ch}{\ckb}\mu^2
  &\im\frac{2b\ch}{\ckb}\mu+\im \,r(\mu\e^2)
  & 0
  & 0
  &  \\
-\im\frac{2b\ch}{\ckb}\mu-\im \,r(\mu\e^2)
  & \frac{b\gamma_{\tth,\kappa,b}\ch}{\ckb}\mu^2
  & \im\, r(\mu\e)
  & 0
  &  \\
\hline
0
  & -\im\, r(\mu\e)
  & 0
  & 0
  & \\
0
  & 0
  & 0
  & 0
  &  \\
\end{NiceArray}
\right) + \mathcal{O}(\mu^2 \varepsilon,\mu^3).
\end{aligned}
\end{equation}
Consider now  $\mathcal{K}^{\beta}_{12}$, we have 
\begin{equation}
\left( \mathcal{K}^{\beta}_{12} \phi_k^\sigma (\mu, \varepsilon),\, \phi_{k'}^{\sigma'}(\mu, \varepsilon) \right)
= \left(\mathcal{K}^{\beta}_{12} \phi_k^\sigma (0, \varepsilon),\, \phi_{k'}^{\sigma'}(0, \e) \right)
+ \mathcal{O}(\mu^2 \varepsilon).
\end{equation}

The matrix entries $\left( \mathcal{K}^{\beta}_{12} \phi_k^\sigma (0, \varepsilon),\, \phi_{k'}^{\sigma}(0, \varepsilon) \right),\, k, k' = 0, 1,\, \sigma = \pm$ are zero, because they are simultaneously real by \eqref{B are alternatively real or imaginary}, and purely imaginary, being the operator $\mathcal{K}^{\beta}_{12}$ purely imaginary and the basis $\{ \phi_k^{\pm}(0, \varepsilon) \}_{k=0,1}$ real. Hence the matrix $\mathsf{K}^{\beta}_{12}$
 representing the action of 
the operator $\mathcal{K}^\beta_{12}$ on the basis 
$\Phi_{\mu,\e}$ of $\mathcal{V}_{\mu,\varepsilon}$
has the form
\begin{equation}
\mathsf{K}^{\beta}_{12} =
\left(
\begin{NiceArray}{cc|ccc}[code-for-first-col=\scriptstyle]
0 & \mathrm{i} a^{\beta}_{12} & 0 & \mathrm{i} b^{\beta}_{12} \\
-\mathrm{i} a^{\beta}_{12} & 0 & -\mathrm{i} c^{\beta}_{12} & 0 \\
\hline
0 & \mathrm{i} c^{\beta}_{12} & 0 & \mathrm{i} d^{\beta}_{12} \\
-\mathrm{i} b^{\beta}_{12} & 0 & -\mathrm{i} d^{\beta}_{12} & 0
\end{NiceArray}
\right)
+ \mathcal{O}(\mu^2 \varepsilon),
\end{equation}
where
\begin{equation}
\begin{aligned}
& \left( \mathcal{K}^{\beta}_{12} \phi_1^{-}(0, \e),\, \phi_1^{+}(0, \e) \right) := \mathrm{i} a^{\beta}_{12}, \quad \left( \mathcal{K}^{\beta}_{12} \phi_1^{-}(0, \e),\, \phi_0^{+}(0, \e) \right) := \mathrm{i} c^{\beta}_{12}, \\
& \left( \mathcal{K}^{\beta}_{12} \phi_0^{-}(0, \e),\, \phi_1^{+}(0, \e) \right) := \im b^{\beta}_{12}, \quad \left( \mathcal{K}^{\beta}_{12} \phi_0^{-}(0, \varepsilon),\, \phi_0^{+}(0, \e) \right) := \im d^{\beta}_{12},
\end{aligned}
\end{equation}
and $a^{\beta}_{12}, b^{\beta}_{12}, c^{\beta}_{12}, d^{\beta}_{12}$ are real numbers. As $\mathcal{K}^{\beta}_{12} = \mathcal{O}(\mu \varepsilon)$ in $ \mathcal{L}(Y, X)$, we deduce that $c^{\beta}_{12} = r(\mu \e)$. Let us compute the expansion of $a^{\beta}_{12}$, $b^{\beta}_{12}$ and $d^{\beta}_{12}$. In view of \eqref{43 f+1}-\eqref{46 f-0}, $\phi_1^{\pm}(0, \varepsilon) = \phi_1^{\pm} + \mathcal{O}(\varepsilon)$, $\phi_0^{+}(0, \varepsilon) = \phi_0^{+} + \mathcal{O}(\varepsilon)$, $\phi_0^{-}(0, \varepsilon) = \begin{bmatrix} 0 \\ 1 \end{bmatrix}$, where $\phi_k^\sigma$ are in \eqref{eigenfunc of mathcall L00 2}. Using \eqref{tau B11 B12}, we have
\begin{align*}
\mathcal{K}^{\beta}_{12} \phi_1^{-} =
\begin{bmatrix}
-\frac{5}{2}\im b\mu\e\ch^{-\frac{3}{2}}\ckb^{-\frac{1}{2}}(1+25\cos(2x)) \\
0
\end{bmatrix}+\cO(\mu\e^2),
\quad
\mathcal{K}^{\beta}_{12} \phi_0^{-} =
\begin{bmatrix}
0 \\
0
\end{bmatrix},
\end{align*}
and then
\begin{align*} 
a^{\beta}_{12}  = r(\mu \e^2), \quad 
b^{\beta}_{12} =  r(\mu \e^2), \quad d^{\beta}_{12}  = r(\mu \e^2).
\end{align*}
\color{black}
Next we consider the matrix
$\mathsf{K}^{\beta}_2$ representing the action of the operator
$ \mathcal{K}^{\beta}_2$ on the basis 
$\Phi_{\mu,\e}$.
Recalling \eqref{B beta expand}, we have 
\begin{equation}
\left( \mathcal{K}^{\beta}_2 \phi_k^\sigma (\mu, \varepsilon),\, \phi_{k'}^{\sigma'}(\mu, \varepsilon) \right)
= \left( -6b \mu^2\begin{bmatrix}
\pa^2_x & 0 \\
0 & 0
\end{bmatrix} \phi_k^\sigma (0, \varepsilon),\, \phi_{k'}^{\sigma'}(0, \varepsilon) \right)
+ \cO(\mu^3).
\end{equation}
The matrix entries $\left( \mathcal{K}^{\beta}_2 \phi_k^\sigma (0, \varepsilon),\, \phi_{k'}^{-\sigma}(0, \varepsilon) \right),\, k, k' = 0, 1,\, \sigma = \pm$ are zero, because they are simultaneously purely imaginary by \eqref{B are alternatively real or imaginary}, and real, being the operator $\mathcal{K}^{\beta}_2$ real and the basis $\{ \phi_k^{\pm}(0, \e) \}_{k=0,1}$ real. Hence $\mathcal{K}^{\beta}_2$, with respect to the basis $\Phi_{\mu,\e}$ of $\mathcal{V}_{\mu,\varepsilon}$ in \eqref{F basis set and f}, admits the expansion 
\begin{equation}
\mathsf{K}^{\beta}_2 =
\left(
\begin{NiceArray}{cc|ccc}[code-for-first-col=\scriptstyle]
a^\beta_2 & 0 & b^\beta_2 & 0 \\
0 & c^\beta_2 & 0 & d^\beta_2 \\
\hline
b^\beta_2 & 0 & e^\beta_2 & 0 \\
0 & d^\beta_2 & 0 & f^\beta_2
\end{NiceArray}
\right)
+ \cO(\mu^3).
\end{equation}
Using $\phi_k^{\sigma}(0, \varepsilon) = \phi_k^{\sigma} + \mathcal{O}(\varepsilon)$ defined in \eqref{43 f+1}-\eqref{46 f-0}, a straightforward calculation reveals that
\begin{align*}
a^{\beta}_2 = \frac{3b\mu^2\ch}{\ckb}+r(\mu^2\e), \quad b^{\beta}_2 =r(\mu^2\e),\quad c^{\beta}_2 = \frac{3b\mu^2\ch}{\ckb}+r(\mu^2\e), \quad d^{\beta}_2 =r(\mu^2\e),\quad e^{\beta}_2 =r(\mu^2\e),\quad f^{\beta}_2 =r(\mu^2\e).
\end{align*}
This proves \eqref{mathsf B s}. 
\end{proof}

Lemma \ref{Expansion of B epsilon}-Lemma \ref{Expansion of B beta} imply \eqref{B mu epsilon} where the matrix $\mathsf{E}$ has the form \eqref{E} and
\begin{equation*}
\begin{aligned}
\mathsf{e}_{12} &:=\mathsf{e}_{12}(\tth,\kappa,b):=\frac{(4b+2\kappa)\ch}{\ckb}+\hat{\flat}_{\tth,\kappa,b}\,,\\
\mathsf{e}_{22} &:=\mathsf{e}_{22}(\tth,\kappa,b):= 2\flat_{\tth,\kappa,b} - 8\zeta_{\tth,\kappa,b}-\frac{4\kappa\ch}{\ckb}\left(1+\gamma_{\tth,\kappa,b}\right)-\frac{8b\ch}{\ckb}\left(3+\gamma_{\tth,\kappa,b}\right)\,,
\end{aligned}
\end{equation*}
with $\flat_{\tth,\kappa,b}$, $\hat{\flat}_{\tth,\kappa,b}$ in \eqref{bh440} and $\zeta_{\tth,\kappa,b}$ in \eqref{zeta h}, $ \gamma_{\tth,\kappa,b}$ in \eqref{47 coefficients}. The term $\mathsf{e}_{22}$ has the expansion in \eqref{e11 f11}. 

Moreover, we obtain
\begin{equation} \label{F 2nd}
\begin{aligned} 
\mathsf{F} &:= \mathsf{F}(\mu,\e)=
\begin{pmatrix}
\mathsf{f}_{11} \varepsilon + r_3(\varepsilon^3, \mu \varepsilon^2, \mu^2 \varepsilon,\mu^3) &
\im \,\mu \varepsilon \ch^{-\frac{1}{2}}\ckb^{\frac{1}{2}} + \im\, r_4(\mu \varepsilon^2, \mu^2 \varepsilon,\mu^3) \\
\im\, r_6(\mu \varepsilon,\mu^3) &
r_7(\mu^2 \varepsilon,\mu^3)
\end{pmatrix}, 
\end{aligned}
\end{equation}
and
\begin{equation} \label{G 2nd}
\begin{aligned} 
\mathsf{G} &:= \mathsf{G}(\mu,\e)=
\begin{pmatrix}
1 +\kappa\mu^2+ r_8(\varepsilon^2, \mu^2 \varepsilon,\mu^3) &
- \im r_9(\mu \varepsilon^2, \mu^2 \varepsilon,\mu^3) \\
\im r_9(\mu \varepsilon^2, \mu^2 \varepsilon,\mu^3) &
\mu \tanh(\tth \mu) + r_{10}(\mu^2 \varepsilon,\mu^3)
\end{pmatrix} . 
\end{aligned}
\end{equation}
In order to deduce the expansion \eqref{F}-\eqref{G} of the matrices $\mathsf{F}$, $\mathsf{G}$ we exploit also the following lemma:
\begin{lemma} \label{Fmu0 eq 0 Gmu0}
At $ \varepsilon = 0 $ the matrices are $ \mathsf{F}(\mu, 0) = 0 $ and 
$
\mathsf{G}(\mu, 0) = \begin{pmatrix} 1+\kappa\mu^2+b\mu^4 & 0 \\ 0 & \mu \tanh(\tth\mu) \end{pmatrix}.
$
\end{lemma}

\begin{proof}
By Lemma \ref{fmu0 eq f0} and since $\mathcal{K}_{\mu,0} := \left[\begin{smallmatrix}
1-\kappa\left(\pa_x+\im \mu\right)^2+b\left(\pa_x+\im \mu\right)^4 & -\chk \partial_x \\
\chk\pa_x &|D + \mu| \tanh(\tth|D + \mu|)  
\end{smallmatrix}\right]$,  we obtain 
$\mathcal{K}_{\mu,0} \phi_0^+(\mu,0) = (1+\kappa\mu^2+b\mu^4)\phi_0^+$ and 
$\mathcal{K}_{\mu,0} \phi_0^-(\mu,0) = \mu \tanh(\tth\mu) \phi_0^-$,
for any \( \mu \). The lemma follows recalling \eqref{Bmuepsilon matrix representation} and the fact that 
\( \phi_1^+(\mu, 0) \) and \( \phi_1^-(\mu, 0) \) have zero space average by Lemma \ref{fmu0 eq f0}
\end{proof}

In view of Lemma \ref{Fmu0 eq 0 Gmu0} we deduce that the matrices \eqref{F 2nd} and \eqref{G 2nd} have the form \eqref{F} and \eqref{G}. This completes the proof of Lemma \ref{B decomposition}.

\section{Lie-Transform Block-Decoupling}\label{secF}
In this appendix we derive the expansion of
$\mathsf{L}^{(2)}_{\mu,\varepsilon}$ produced by the first structure-preserving Lie transform. Our argument follows the general
strategy of~\cite[Lemma 3.8]{BMV1}, but the hydroelastic coefficients and
the fourth-order bending contributions require a separate computation. We
therefore provide the details.

We split $\mathsf{L}^{(1)}_{\mu,\varepsilon}$ into its $2 \times 2$-diagonal and off-diagonal Hamiltonian and reversible matrices
\begin{equation} \label{L=D+R}
    \begin{aligned}
        &\mathsf{L}^{(1)}_{\mu,\varepsilon} = \mathsf{D}^{(1)} + \mathsf{R}^{(1)},\qquad \mathsf{D}^{(1)} := \begin{pmatrix} \mathsf{D}_1 & 0 \\ 0 & \mathsf{D}_0 \end{pmatrix}= \begin{pmatrix} \mathsf{J}_2 \mathsf{E}^{(1)} & 0 \\ 0 & \mathsf{J}_2 \mathsf{G}^{(1)} \end{pmatrix}, \quad 
\mathsf{R}^{(1)} := \begin{pmatrix} 0 & \mathsf{J}_2 \mathsf{F}^{(1)} \\ \mathsf{J}_2 [\mathsf{F}^{(1)}]^* & 0 \end{pmatrix}, 
    \end{aligned}
\end{equation}
and perform the Lie expansion
\begin{equation} \label{L2 mu e}
    \begin{aligned}
        \mathsf{L}^{(2)}_{\mu,\varepsilon} = \exp(\mathsf{S}^{(1)}) \mathsf{L}^{(1)}_{\mu,\varepsilon} \exp(-\mathsf{S}^{(1)}) &= \mathsf{D}^{(1)} + [\mathsf{S}^{(1)}, \mathsf{D}^{(1)}] + \frac{1}{2} [\mathsf{S}^{(1)}, [\mathsf{S}^{(1)}, \mathsf{D}^{(1)}]] + \mathsf{R}^{(1)} + [\mathsf{S}^{(1)}, \mathsf{R}^{(1)}]\\
        &+ \frac{1}{2} \int_0^1 (1 - t)^2 \exp\left(t \mathsf{S}^{(1)}\right) \mathrm{ad}_\mathsf{\mathsf{S}^{(1)}}^3\left(\mathsf{D}^{(1)}\right) \exp\left(-t \mathsf{S}^{(1)}\right) \, \de t\\
        &+ \int_0^1 (1 - t) \exp\left(t \mathsf{S}^{(1)}\right) \mathrm{ad}_\mathsf{\mathsf{S}^{(1)}}^2\left(\mathsf{R}^{(1)}\right) \exp(-t \mathsf{S}^{(1)}) \, \de t.
    \end{aligned}
\end{equation}
One may observe that
\begin{align*}
    \mathsf{R}^{(1)} + [\mathsf{S}^{(1)}, \mathsf{D}^{(1)}] = 0\qquad\text{if and only if}\qquad \mathsf{D}_1 \mathsf{V} - \mathsf{V} \mathsf{D}_0 = -\mathsf{J}_2 \mathsf{F}^{(1)}. 
\end{align*}
Recalling \eqref{mathsf V}, we know $\mathsf{V}$ solves the Sylvester equation $\mathsf{D}_1 \mathsf{V} - \mathsf{V} \mathsf{D}_0 = -\mathsf{J}_2 \mathsf{F}^{(1)}$, so the matrix $\mathsf{S}^{(1)}$ solves the homological equation $\mathsf{R}^{(1)}+[\mathsf{S}^{(1)}, \mathsf{D}^{(1)}]  = 0$. Therefore, identity \eqref{L2 mu e} simplifies to
\begin{align} \label{L2 mu epsilon}
\mathsf{L}_{\mu,\varepsilon}^{(2)} = \mathsf{D}^{(1)} + \frac{1}{2} [ \mathsf{S}^{(1)}, \mathsf{R}^{(1)} ]
+ \frac{1}{2} \int_0^1 (1 -t^2)\, \exp\left(t\mathsf{S}^{(1)}\right) \, \mathrm{ad}_{\mathsf{S}^{(1)}}^2 \left(\mathsf{R}^{(1)}\right) \, \exp\left(-t \mathsf{S}^{(1)}\right) \, \de t,
\end{align}
since
\[
\operatorname{ad}_{\mathsf{S}^{(1)}}^{n+1} \left(\mathsf{D}^{(1)}\right) = -\operatorname{ad}_{\mathsf{S}^{(1)}}^n \left(\mathsf{R}^{(1)}\right) \qquad\text{for}\qquad n\geq 1.
\]
Denoting $\mathrm{Sym}(A) := \frac{1}{2}( A + A^* )$, a direct calculation reveals that
\begin{align} \label{0.5 S R}
\frac{1}{2} [ \mathsf{S}^{(1)}, \mathsf{R}^{(1)} ]=
\begin{pmatrix}
\mathsf{J}_2 \widetilde{\mathsf{E}} & 0 \\
0 & \mathsf{J}_2 \widetilde{\mathsf{G}}
\end{pmatrix}, \qquad\text{with}\qquad \widetilde{\mathsf{E}} := \mathrm{Sym} ( \mathsf{J}_2 \mathsf{V} \mathsf{J}_2 [\mathsf{F}^{(1)}]^* ), 
\quad
\widetilde{\mathsf{G}} := \mathrm{Sym} ( \mathsf{V}^* \mathsf{F}^{(1)} ).
\end{align}
Also, we observe that 
\begin{align*}
    \mathsf{V}^* \mathsf{F}^{(1)}=\begin{pmatrix}
\cO(\mu\e^2) & \cO(\mu\e^2) \\
\cO(\mu\e^2) & \cO(\mu\e^2)
\end{pmatrix}
\end{align*}
and
\begin{equation*} 
\begin{aligned}
\mathsf{J}_2 \mathsf{V} \mathsf{J}_2 [\mathsf{F}^{(1)}]^* 
&= \begin{pmatrix}
v_{21} \mathsf{F}^{(1)}_{12} - v_{22} \mathsf{F}^{(1)}_{11} & \im ( v_{21} \mathsf{F}^{(1)}_{22} + v_{22} \mathsf{F}^{(1)}_{21} ) \\
\im ( v_{11} \mathsf{F}^{(1)}_{12} + v_{12} \mathsf{F}^{(1)}_{11} ) & - v_{11} \mathsf{F}^{(1)}_{22} + v_{12} \mathsf{F}^{(1)}_{21}
\end{pmatrix} 
= \begin{pmatrix}
\tilde{\mathsf{e}}_{11} \mu \varepsilon^2 + r(\mu \varepsilon^3, \mu^2 \varepsilon^2) & \im r(\mu \varepsilon^2) \\
\im r(\mu \varepsilon^2) & r(\mu \varepsilon^2)
\end{pmatrix},
\end{aligned}
\end{equation*}
where
\begin{align*}
    \tilde{\mathsf{e}}_{11} &:= - \textup{D}_{\tth,\kappa,b}^{-1} \left( \ch^{-1}\ckb+ \tth \mathsf{f}_{11}^2 + \mathsf{e}_{12} \mathsf{f}_{11} \ch^{-\frac{1}{2}}\ckb^{\frac{1}{2}} \right).
\end{align*}
Then, the self-adjoint and reversibility-preserving matrices $\widetilde{\mathsf{E}}$, $\widetilde{\mathsf{G}}$ have the form:
    \begin{equation}\label{tilde EG and tilde e11}
    \begin{aligned} 
\widetilde{\mathsf{E}} &= \begin{pmatrix}
\tilde{\mathsf{e}}_{11} \mu \varepsilon^2 + \tilde{r}_1 (\mu \varepsilon^3, \mu^2 \varepsilon^2) & \im\, \tilde{r}_2 (\mu \varepsilon^2) \\
- \im \,\tilde{r}_2 (\mu \varepsilon^2) & \tilde{r}_5 (\mu \varepsilon^2)
\end{pmatrix}, \quad
\widetilde{\mathsf{G}} = \begin{pmatrix}
\tilde{r}_8 (\mu \varepsilon^2) & \im\, \tilde{r}_9 (\mu \varepsilon^2) \\
- \im\, \tilde{r}_9 (\mu \varepsilon^2) & \tilde{r}_{10} (\mu \varepsilon^2)
\end{pmatrix}.
\end{aligned}
\end{equation}
 
Finally, we claim the last term in \eqref{L2 mu epsilon} is small. 

\begin{lemma} \label{high order terms lemma}
    The $4 \times 4$ Hamiltonian and reversibility matrix
\begin{align} \label{high order terms EG hat}
\frac{1}{2} \int_0^1 (1 - t^2) \exp\left(t \mathsf{S}^{(1)}\right) \mathrm{ad}_{\mathsf{S}^{(1)}}^2 \left(\mathsf{R}^{(1)}\right) \exp\left(-t \mathsf{S}^{(1)}\right) \, \de t
= \begin{pmatrix}
\mathsf{J}_2 \widehat{\mathsf{E}} & \mathsf{J}_2 \mathsf{F}^{(2)} \\
\mathsf{J}_2 [\mathsf{F}^{(2)}]^* & \mathsf{J}_2 \widehat{\mathsf{G}}
\end{pmatrix}\,,
\end{align}
where the $2 \times 2$ self-adjoint and reversible matrices $\widehat{\mathsf{E}}$, $\widehat{\mathsf{G}}$ have entries
\begin{align} \label{Eij Gij}
\widehat{\mathsf{E}}_{ij} = \widehat{\mathsf{G}}_{ij} = r(\mu \varepsilon^3), 
\quad i,j = 1,2,
\end{align}
and the $2 \times 2$ reversible matrix $\mathsf{F}^{(2)}$ admits an expansion as in \eqref{F2}.
\end{lemma}
\begin{proof}
Each matrix 
$\exp\left(t \mathsf{S}^{(1)}\right) 
\operatorname{ad}_{\mathsf{S}^{(1)}}^2 \left(\mathsf{R}^{(1)}\right) 
\exp\left(-t \mathsf{S}^{(1)}\right)$ 
is Hamiltonian and reversibility-preserving, and hence the representation in 
\eqref{high order terms EG hat} holds. To estimate its entries, we first compute 
$\operatorname{ad}_{\mathsf{S}^{(1)}}^2 \left(\mathsf{R}^{(1)}\right)$. 
Using the form of $\mathsf{S}^{(1)}$ in \eqref{S1} and the expression for 
$[\mathsf{S}^{(1)}, \mathsf{R}^{(1)}]$ in \eqref{0.5 S R}, we obtain $\operatorname{ad}_{\mathsf{S}^{(1)}}^2 \left(\mathsf{R}^{(1)}\right) = 
\begin{pmatrix}
0 & \mathsf{J}_2 \widetilde{\mathsf{F}} \\
\mathsf{J}_2 \widetilde{\mathsf{F}}^* & 0
\end{pmatrix}$,
where $\widetilde{\mathsf{F}} := 2 \left( \mathsf{M} \mathsf{J}_2 \widetilde{\mathsf{G}} 
- \widetilde{\mathsf{E}} \mathsf{J}_2 \mathsf{M} \right)$,
and $\widetilde{\mathsf{E}}$, $\widetilde{\mathsf{G}}$ are defined in 
\eqref{0.5 S R}. Since 
$\widetilde{\mathsf{E}}, \widetilde{\mathsf{G}} = \mathcal{O}(\mu \varepsilon^2)$ 
by \eqref{tilde EG and tilde e11}, and 
$\mathsf{M} = \mathsf{J}_2 \mathsf{V} = \mathcal{O}(\varepsilon)$ 
by \eqref{mathsf V}, we deduce that 
$\widetilde{\mathsf{F}} = \cO(\mu \varepsilon^3)$. Consequently, for any 
$t \in [0,1]$, $\exp\left(t \mathsf{S}^{(1)}\right) 
\operatorname{ad}_{\mathsf{S}^{(1)}}^2 \left(\mathsf{R}^{(1)}\right) 
\exp\left(-t \mathsf{S}^{(1)}\right) 
=
\operatorname{ad}_{\mathsf{S}^{(1)}}^2 \left(\mathsf{R}^{(1)}\right)
(1 + \mathcal{O}(\mu, \varepsilon))
$. In particular, the matrix $\mathsf{F}^{(2)}$ in 
\eqref{high order terms EG hat} has the same expansion as 
$\widetilde{\mathsf{F}}$, namely, $\mathsf{F}^{(2)} = \mathcal{O}(\mu \varepsilon^3)$, while the matrices $\widehat{\mathsf{E}}$ and $\widehat{\mathsf{G}}$ have 
entries as in \eqref{Eij Gij}.
\end{proof}

\footnotesize

\vspace{1em}
\noindent{\bf Data availability statement:}
Data will be made available on reasonable request\\

\noindent{\bf Conflict of interest:}
We do not have any conflict of interest.\\

\noindent{\bf Acknowledgments:}
T.-Y. Hsiao is  supported by the European Union  ERC CONSOLIDATOR GRANT 2023 GUnDHam, Project Number: 101124921. He would like to thank Vera Hur, Zhao Yang, and Alberto Maspero for their encouragement and helpful discussions. He is also grateful to Massimiliano Berti and Alberto Maspero for their illuminating lectures on Hamiltonian PDEs. Y. Zhang is supported by the European Union ERC STARTING GRANT 2020 GeoSub, Project Number: 945655.  Views and opinions expressed are however those of the authors only and do not necessarily reflect those of the European Union or the European Research Council. Neither the European Union nor the granting authority can be held responsible for them.

\end{document}